\documentclass[12pt]{amsart}
\usepackage{fullpage}

\usepackage{hyperref}
\usepackage{amsmath,amsfonts,amssymb}
\usepackage{mathrsfs}
\usepackage{verbatim}
\usepackage{amsthm}
\usepackage{mathrsfs} 
\usepackage{xcolor,colortbl}
\definecolor{BROWN}{gray}{0.9}

\newtheorem{theorem}{Theorem}[section]
 \newtheorem{corollary}[theorem]{Corollary}
 \newtheorem{lemma}[theorem]{Lemma}
 \newtheorem{proposition}[theorem]{Proposition}
 \theoremstyle{definition}
 \newtheorem{definition}[theorem]{Definition}
 
 \theoremstyle{remark}
 \newtheorem{remark}[theorem]{Remark}
  \newtheorem{ex}[theorem]{Example}
   \newtheorem{notation}[theorem]{Notation}
 \numberwithin{equation}{section}

\def \bC {\mathbb C}

\def \bN {\mathbb N}

\def \bP {\mathbb P}

\def \bR {\mathbb R}

\def \bT {\mathbb T}

\def \bV {\mathbb V}

\def \bX {\mathbb X}
\def \bY {\mathbb Y}
\def \bZ {\mathbb Z}

\def \cA {\mathcal A}

\def \cC {\mathcal C}
\def \cD {\mathcal D}

\def \cF {\mathcal F}
\def \cG {\mathcal G}
\def \cH {\mathcal H}

\def \cK {\mathcal K}

\def \cR {\mathcal R}
\def \cS {\mathcal S}
\def \cT {\mathcal T}

\def \cV {\mathcal V}

\def \fg {\mathfrak g}

\def \sL{\mathscr L}
\def \sE{\mathscr E}
\def \sU{\mathscr U}

\def \ad {{\rm ad}}
\def \id {\text{\rm I}}
\def \Op {{\rm Op}}
\def \AOp {{\rm AOp}}
\def \supp{{{\rm supp}}}
\def \Gh {{\widehat G}}
\def \eps {\varepsilon}

\def \Exp {{\rm Exp}}
\def \Ln  {{\rm Ln}}
\def \gr {{\rm gr}}
\def \jac {{\rm jac}}
\def \dom {{\rm dom}}
\def \DO {{\rm DO}}
\def \princ {{\rm Princ}}
\def\IM  {\text {\rm Im}}

\begin{document}
\author{Clotilde Fermanian-Kammerer, V\'eronique Fischer and Steven Flynn}

\title{Quantization  on filtered manifolds}

\medskip

%35P05 General topics in linear spectral theory of PDEs
%43A85 Haronic analysis on homogeneous spaces 
%47G30 Psudo diff op 
%53C30 Differential geometry on homogeneous spaces
%58J40 Pseudo diff and FIO on manifolds
%35A27 Microlocal methods and methods of sheaf theory and homological algebra applied to PDEs
%35H20 Subelliptic equations
%58J50, 58J35, 35K08, 35P05, 53C17, 43A85


\begin{abstract}
In this article, we develop a pseudodifferential calculus on a general filtered manifold $M$. The symbols are fields of operators $\sigma(x,\pi)$ parametrised by $x\in M$ and the unitary dual $\Gh_x M$  of the osculating Lie group $G_x M$. 
We define classes of symbols and a local quantization  formula associated to a local frame adapted to the filtration. 
We prove that the collection of operators on $M$ coinciding locally with the quantization of symbols 
enjoys the essential properties of a pseudodifferential calculus:
composition, adjoint, parametrices, continuity on adapted Sobolev spaces. Moreover, we show that the polyhomogeneous subcalculus coincides with the calculus constructed by  van Erp and Yuncken via groupoids.
\end{abstract}
\maketitle

\makeatletter
\renewcommand\l@subsection{\@tocline{2}{0pt}{3pc}{5pc}{}}
\makeatother

\tableofcontents

\section{Introduction}

\subsection{Symbolic pseudodifferential calculus and applications}
%This work is motivated by the  theory of pseudodifferential operators. 
Starting from about  1960, 
the seminal works of H\"ormander  \cite{ho67,ho} 
provided a  flexible framework for studying elliptic operators through pseudodifferential calculi, along with remarkable characterisations and properties for hypoellipticity.
For the former, many other contemporaneous contributions also played a significant role in this development, see e.g. 
\cite{KohnNirenberg65,Nirenberg68}. 
The latter was the motivation behind 
the research programme led by Folland and Stein and their collaborators in the 70's and 80's (see eg \cite{folland75,folland77,RoSt}) 
on subelliptic operators
modelled via nilpotent Lie groups.
Nowadays, 
the geometric contexts for these ideas
are called subRiemannian.
The main method revolved around the construction of parametrices of these operators, and more generally of algebras of operators adapted to these settings. 
Usually, the pseudodifferential calculi constructed then were in effect much closer to algebras of singular integral operators described via their integral kernels (in a similar spirit, see Melin's unpublished preprint \cite{melin} and Street's monographs \cite{Street2014,Street2023}); indeed, these calculi didn't use notions of symbols as in the Riemannian or Euclidean settings for elliptic operators.
Furthermore,
they were set on contact manifolds and their generalisations (e.g. Heisenberg manifolds)
\cite{Beals_Greiner,Epstein_Melrose,melin} or on nilpotent Lie groups
\cite{CGGP}. 
These ideas and methods were further deployed in the 90's in the studies of spectral or functional problems of operators such as subLaplacians on stratified nilpotent Lie groups in harmonic analysis (see the seminal paper  \cite{MullerRicciStein}).

\medskip 

In the past twenty years,  a pseudodifferential theory based on Alain Connes’ idea of tangent groupoid~\cite{Connes} emerged from works by Ponge \cite{Ponge} and by van Erp \cite{VanErp} on Heisenberg  and contact manifolds. 
It was then realised that the context to generalise these considerations were manifolds equipped with filtrations of the tangent bundle, i.e. filtered manifolds, see the works  of Choi and Ponge \cite{Choi_Ponge_groupoid,Choi_Ponge_1,Choi_Ponge_2} and  of van Erp and Yuncken \cite{VeY1,VeY2}.
Van Erp and Yuncken's results  can be considered as the development in the context of filtered manifolds of  Debord and Skandalis' works~\cite{DS_14,DS_15} which provided the first abstract characterisation of the classical pseudodifferential operators in terms of actions  on Connes' tangent groupoid. 
While the class of filtered manifolds contains equiregular subRiemannian manifolds, 
recent progress suggests that these ideas can  be further generalised to  manifolds whose space of vector fields admits a filtration
\cite{AMY}, encompassing also singular subRiemannian manifolds and quasi-subRiemannian settings. 
It is worth noting that the latter calculi constructed within the tangent groupoid framework are included in the more general construction developed by Street~\cite{Street2014,Street2023}.

The pseudodifferential theory via groupoids was devised by Alain Connes around 1990 as a framework to show index theorems. Since then, it has  become an efficient tool to study  hypoellipticity and Fredholmness in many contexts.
However, at the time of writing, it does not show the same breadth and depth of applications to PDE's as  H\"ormander's calculus. One reason is that H\"ormander's pseudodifferential operators enjoy an explicit symbolic calculus related to the symplectic structure of the cotangent bundle with an explicit quantization procedure; this  make it flexible to deal with solutions of PDEs and quantum dynamical problems via classical dynamics. Moreover, H\"ormander's calculus allows for  $(\rho,\delta)$-generalisations and   semiclassical calculus including symbolic smoothing operators \cite{FKLR} which aren't yet developed with the groupoid approach.
  
  One motivation for this work is to construct a noncommutative counterpart of H\"ormander’s symbolic pseudodifferential calculus on filtered manifolds, formulated in terms of the natural noncommutative cotangent bundle.
We show in Section \ref{sec:groupoid} that, upon restriction to the polyhomogeneous subcalculus, it agrees with the construction of van Erp and Yuncken~\cite{VeY1,VeY2}.
Moreover, we believe that this approach is as flexible as H\"ormander’s calculus, allowing for the incorporation of 
$(\rho,\delta)$-type generalizations and the development of a corresponding semiclassical pseudodifferential calculus.

\medskip 

In the mean time, the last decades have seen a rising interest of the PDE community for subelliptic operators~\cite{Arnaiz_Riviere,beauchard,Burq_Sun,CdVHT,CdV_Letrouit,Letrouit,Riviere}. 
The methods underlying these works have involved mainly Riemannian  tools based on 
H\"ormander's calculus. These tools are  very powerful and are  the basis for microlocal and semiclassical analysis with  various domains of applications, see for instance the seminal papers in control theory~\cite{BLR:92}, spectral geometry~\cite{CdV,zeldi}, homogeneisation~\cite{gerard,tartar}.
However, recent works \cite{Arnaiz_Riviere,CdVHT,Riviere} suggest that they can be  fruitful  for subelliptic or subRiemannian problems only  in the context of very low-dimensional contact manifolds.
Beyond this, 
the non-commutativity inherent to subRiemannian and filtered settings seems to make any semiclassical or microlocal question  rapidly intractable when using the traditional (Riemannian) H\"ormander calculus. However, other approaches for instance via groupoids have recently given impressive results such as the proof in full generality of the Helffer-Nourrigat conjecture~\cite{HN,AMY}. This reinforces the motivation for the construction of an adapted {\it symbolic} pseudodifferential calculus, with the goal of providing the tools 
%underlying future analysis 
for dealing with questions from PDEs   in subRiemannian and subelliptic contexts.

\smallskip

%Our approach to pseudodifferential calculus is {\it symbolic}. 
A notable exception to the literature of the non-symbolic subelliptic pseudodifferential calculi of the late twentieth century is Taylor's monograph \cite{Taylor}:  he proposed to define pseudodifferential operators   symbolically
using  the operator-valued group Fourier transform, focusing mainly on the Heisenberg group (see also \cite{BFG}). Using also this group Fourier transform together with many tools and methods from  harmonic analysis  (singular integral theory and spectral properties of subLaplacians), 
the recent monograph~\cite{R+F_monograph} 
laid the foundations of a symbolic and global pseudodifferential calculus on any graded nilpotent Lie group. 
We  call these pseudodifferential calculi {\it noncommutative} because, unlike H\"ormander's, their  symbols do not commute.

The pseudodifferential  calculus in~\cite{R+F_monograph}  has allowed for results similar to  the problems mentioned above in the Euclidean or Riemannian  settings. For instance,  the article~\cite{FF1} proves a div-curl Lemma in the context of nilpotent graded Lie groups, in the same spirit as the Euclidean work~\cite{gerard}. The article \cite{FFF1} emphasises invariance properties of quantum limits while~\cite{FFF2} deals with spectral geometry for subLaplacians on  2-step nilmanifolds. The article~\cite{FL} is a control result. 
These applications underscore the relevance of the noncommutative analogue of the cotangent bundle arising from the group Fourier transform. It also highlights the necessity of incorporating smoothing symbols when high-frequency phenomena appear and semiclassical analysis becomes essential.
Motivated by these considerations, we follow the framework developed in~\cite{R+F_monograph} to construct a \emph{noncommutative symbolic pseudodifferential calculus} in the more geometric context of filtered manifolds.
Our study further establishes 
the construction of parametrices for 
symbols which are invertible at high frequency
(Rockland-type conditions). We also introduce corresponding \emph{Sobolev spaces} of arbitrary real order.

\subsection{Strategy of the construction}
Let us sketch  now our strategy  - the rigorous definitions of the objects involved and their subsequent analysis will be  done later in the core of the paper. The readers will find in Section~\ref{subsec:organisation}  information about where to find  definitions and statements about each of the items mentioned here. 
 
\smallskip 

To the filtered manifold $M$, we associate its osculating Lie group~$GM$  and osculating Lie algebra bundles~$\mathfrak gM$. The symbols we introduce are fields of operators 
$$
\sigma = \{\sigma (x,\pi) :\ x\in M, \ \pi\in \Gh_x M\}
$$
with $\sigma(x,\pi)$ densely defined on the space of the representation $\pi$. 
The noncommutative analogue to the cotangent bundle 
$$
\Gh M:=\{(x,\pi) :\ x\in M, \ \pi\in \Gh_x M\},
$$
parametrises our symbols.
We define symbol classes, as well as a local  quantization associating an operator to one of these symbols. 
This local quantization uses the geometric exponentiation of vector fields and  allows for the identification of a point $y$ near $x\in M$ as $y=\exp_x^\bX v$, $v\in \bR^n$, once  a convenient frame~$\mathbb X$ has been chosen  and  $T_xM$ identified with $\bR^n$. This exponential map has to be considered in contrast with the exponential map $\Exp^\bX_x$ of the group $G_xM$. The upper index $\bX$ recalls that all these objects depend on a chosen frame, and the identifications it induces.

More technically,  we associate to a symbol $\sigma$ an operator defined via the following local quantization. Given a local frame~$\bX$ on an open set~$U$ adapted to the filtration of $M$ and a cut-off function~$\chi$,  the operator acts on $f\in C^\infty_c(M)$ by
\begin{equation}\label{eq_QX_intro}
	\Op^{\bX,  \chi}(\sigma)f(x):=
\int_{v\in \bR^n} \kappa_x^\bX (v) \ (\chi_x f)(\exp^\bX_x(- v )dv, \quad x\in U.
\end{equation}
In this definition, the distribution  $\kappa^\bX$ is called  the {\it convolution kernel} associated with $\sigma$ in the $\bX$-coordinates and corresponds to the group Fourier transform inverse of $\sigma$ in the following sense: 
$$
\sigma(x,\pi) 
= \int_{\bR^n} \kappa^\bX_x (v) \pi(\Exp_x^\bX v)^* d v.
$$
We define symbol classes $S^m(\widehat GM)$, $m\in \bR$, 
on which the  local quantization process~\eqref{eq_QX_intro} makes sense.
Moreover, the resulting operators enjoy symbolic properties for composition and adjoint amongst others. 
Ensuring that these objects are rigorously and intrinsically defined 
poses many regularity questions.
In particular, the convolution kernels $\kappa_x^\bX(v)$ display singularity in $v=0$ and the operator in~\eqref{eq_QX_intro} is of Calderon-Zygmund's type. This is the main difficulty of the work and it highly relies on previous works on nilpotent Lie groups of the second author~\cite{R+F_monograph}.

\smallskip

The passage from the local quantization in~\eqref{eq_QX_intro} to a global calculus follows classical lines \cite{Hintz,Zwobook}. We  define pseudodifferential operators as being locally described by the local quantization defined above. 
We show that they are described equivalently by a global quantization on~$M$ of our symbols via adapted  atlases up to smoothing operators.
To sum-up, 
we obtain classes of operators
\[
\Psi^\infty(M):=\cup_m \Psi^m(M),
\]
which forms a \emph{pseudodifferential calculus} in the sense that it
satisfies: 
\begin{enumerate}
\item 
\label{item_def_pseudodiff_calculus_inclusion}
The continuous inclusions
$\Psi^m (M)\subset \Psi^{m'}(M)$ hold for any $m\leq m'$.
\item
\label{item_def_pseudodiff_calculus_product}
$\Psi^\infty(M)$ is an algebra of operators.
Furthermore if $T_1\in \Psi^{m_1}(M)$, $T_2\in \Psi^{m_2}(M)$, 
then $T_1T_2\in \Psi^{m_1+m_2}(M)$,
and the composition is continuous as a map
$\Psi^{m_1}(M)\times \Psi^{m_2}(M)\to \Psi^{m_1+m_2}(M)$.
\item
\label{item_def_pseudodiff_calculus_adjoint}
$\Psi^\infty(M)$ is stable under taking the adjoint.
Furthermore if $T\in \Psi^{m}(M)$  then $T^*\in \Psi^{m}(M)$,
and taking the adjoint is continuous as a map
$\Psi^{m}(M)\to \Psi^{m}(M)$.
 \end{enumerate}
The calculus also makes it possible to construct parametrices and to introduce local Sobolev spaces $L^2_{s,loc}(M)$, $s\in \bR$, on which the pseudodifferential operators act in the following way:  if $T\in \Psi^{m}(M)$, then $T\in\mathcal L( L^2_{s+m,loc}(M),L^2_{s,loc}(M))$. 
We will also show that its polyhomogeneous subcalculus  coincides with the calculus developed by van Erp and Yuncken \cite{VeY1,VeY2}.
 
 \smallskip

The construction of this pseudodifferential calculus paves the way for new approaches to the study of PDEs on filtered manifolds. These include the analysis of the propagation of singularities and the investigation of compactness defects for square-integrable families on~$M$, among other applications.
A development in this direction is presented in~\cite{benedetto}, where the calculus introduced therein provides an alternative proof of the quantum ergodicity result established on three-dimensional contact manifolds in~\cite{CdVHT}, while also extending it to higher-dimensional contact manifolds endowed with a compatible complex structure.
The overarching motivation of our work lies in \emph{spectral geometry on filtered manifolds}, including contact, quasi-contact, and Engel manifolds, with particular emphasis on conditions ensuring the quantum ergodicity of sub-Laplacians or more general subelliptic operators, and on the geometric interpretation of the coefficients appearing in the Weyl trace formula.

 \subsection{Important technical ingredients of the proofs} 
 \label{subsec_ingredients}
 The strategies behind the proofs of  properties  such as $L^2$-boundedness, composition and adjoint follow established lines in  pseudodifferential theory, see e.g.~\cite{AG,R+F_monograph,Ste93}.
However, their implementations in the  case of filtered manifolds are technically highly challenging.

In the Riemannian case, pseudodifferential operators on manifolds are defined as being described through 
coordinates patches as  pseudodifferential operators  on $\bR^n$. 
This cannot be the case on a filtered manifold, as  natural coordinates do not exist - although privileged ones have been studied  \cite{ABB,Bellaiche,Choi_Ponge_1,Choi_Ponge_2} since the inception of subRiemannian geometry. 
This is a fundamental obstacle that we overcome with  the method of moving frames:  
the time-1 flow  of vector fields allows us to define 
 the exponential map $\exp^\bX$ with respect to a given local frame $\bX$.
This map can be compared with the bisubmersions of~\cite{AMY}.

A crucial milestone in this paper is the analysis of  the composition of the exponential map $\exp^\bX$  with respect to a local frame $\bX$ adapted to the filtration.
As expected from ideas dating back from the work of Nagel, Stein and Wainger in 
\cite{NagelSteinWainger1985}, we show that this composition is connected to the osculating group structure.
However, in order to describe precisely this connection, we introduce a new technical concept of algebraic nature, namely   terms that we call {\it higher order}. This concept is also useful in other situations, 
for instance to compare  the product laws of two osculating groups $G_xM$ and $G_yM$ above two points $x,y\in M$ close enough, or  to describe the effect of changing adapted frames in exponential maps.

As can be anticipated from the group case \cite{R+F_monograph}, 
our proofs of the calculus properties rely fundamentally on the analysis of the integral kernels of the operators through the 
convolution kernels $\kappa^\bX$ of their symbols.
Introducing the notion of amplitudes corresponding to our symbols separates steps in the proofs for composition and adjoint. 
It is  remarkable that our analysis of the  impact of the geometric setting then boils down to understanding the two following  situations: 
\begin{enumerate}
    \item  perturbing the variable of a convolution kernel $\kappa^\bX$ with  a term of higher order,
    \item  perturbing the group convolution of two convolution kernels with another function.
\end{enumerate}

\subsection{Organisation of the paper}\label{subsec:organisation}
The paper is structured linearly. In Section~\ref{sec:groups}, we recall results about symbol classes on graded Lie groups that will be useful for constructing the symbol classes $S^m(\widehat G M)$, $m\in M$.
 In particular, we revisit standard estimates in this setting with a special concern with the constants. Indeed, the symbol $\sigma(x,\cdot)$ will be related with the nilpotent graded group $\widehat G_x M$, which depends on $x\in M$. It is thus important to analyse how constants depend on the gradation of the group, the filtration of the Lie algebra $\mathfrak g_xM$ above each point $x\in M$, 
  and on its Lie structure, i.e. the norm of the Lie bracket $[\cdot,\cdot]_{\mathfrak g_xM}$, viewed as a bilinear form on $\mathfrak g_x M$. 
\smallskip 

Section~\ref{sec:manifold} recalls definitions about filtered manifolds and Section~\ref{sec:geom_maps} focuses on the second crucial tool of our analysis, namely the geometric exponential map $\exp^\bX$. It is here that  we use the notion of higher order   to  study the composition of exponential map.

\smallskip 

With all these tools in hands, we define the symbol classes $S^m(\widehat GM)$, $m\in\bR$, in Section~\ref{sec:symbols}, we study their properties and give examples. A special attention is devoted to kernel estimates in Section~\ref{sec:kernel_estmates} and to the analysis of convolution kernels obtained by the two perturbations mentioned at the end of Section \ref{subsec_ingredients}.
 The results in Section~\ref{sec:kernel_estmates} are used when quantizing the symbols for studying their actions on smooth functions on~$M$, or on elements of $L^2(M)$ (for symbols of non-positive order), and for proving properties of symbolic calculus.  

\smallskip 

The local quantization is presented in  
Section~\ref{sec:quantization}, where we  also study the boundedness on $L^2(M)$ of local pseudodifferential operators of order 0 and the dependence of the local quantization with respect to the frame~$\bX$ and the cut-off function~$\chi$. 
Theorems establishing the symbolic calculus  are proved in Section~\ref{sec:symb_cal}: determination of the symbol of the adjoint of a pseudodifferential operator, and of the composition of two pseudodifferential operators. These results are obtained by an approach based on amplitudes, in a classical manner for pseudodifferential calculus, adapted to filtered manifolds. 
\smallskip 

The concept of local quantization allows us to define a  pseudodifferential calculus on the manifold $M$ in Section~\ref{sec:global_sobolev}. It comprises  operators that are locally given as the local quantization of a symbol, or that are smoothing (i.e. have a smooth integral kernel). In this section, we also  analyse the characterisation of these pseudodifferential operators in terms of atlases, and introduce the adapted hierarchy of local Sobolev spaces. 
Finally, 
we  conclude our  study with the analysis of the links between the calculus we have constructed  and  van Erp and Yuncken's \cite{VeY1,VeY2}. 
%This requires the introduction of the notion of homogeneous symbols and of classical pseudodifferential operators in the context of filtered manifolds. 
%Two appendices\footnote{V should we swapp the two appendices?} devoted to technical points terminate this text. 

\medskip 

\noindent {\bf Acknowledgements.} 
The authors thank Robert Yuncken for several insightful discussions that played an important role in shaping their ideas, and Lino Benedetto for his helpful remarks and comments during the preparation of this text. They are also grateful to the anonymous referees for their thoughtful suggestions, which greatly contributed to improving the introduction.

The three authors acknowledge the support of The Leverhulme Trust via Research Project Grant RPG 2020-037.
The first author also benefits from the support of  the R\'egion Pays de la Loire via the Connect Talent Project HiFrAn 2022 07750,  the grant  ANR-23-CE40-0016 (OpART), and the France 2030 program, Centre Henri Lebesgue ANR-11-LABX-0020-01.

\section{Symbols classes on  graded Lie groups}\label{sec:groups}

In this section, 
we recall the definitions for the objects appearing in the analysis on  graded Lie groups 
and the natural  symbol classes in this context. This gives us the opportunities to introduce the notation used throughout the paper.
In addition, we  present their main properties, paying a special attention 
to how  estimates depend on  the group structure and the gradation.
 The material essentially follows \cite{R+F_monograph}, see also \cite{folland75,folland+stein_82} and \cite{FF0}.

 \subsection{Definitions for graded Lie groups and algebras}
 Here we recall some definitions and set some notation for the rest of the paper. References for this material include \cite{R+F_monograph}
 and \cite{folland+stein_82}.
 
 \subsubsection{Graded Lie algebra}
 
\begin{definition}
\label{def_gradedvs}
The (finite dimensional real) vector space $\cV$  \emph{admits the gradation}  by the sequence of vector subspaces $\cV_j$, $j\in \bN$, 
when $\cV = \cV_1 \oplus \ldots \oplus \cV_j \oplus \ldots$.  
\end{definition} 

In the above definition, as $\cV$ is finite dimensional,  all but a finite number of the $\cV_j$'s are~$\{0\}$.
Continuing with the setting of this definition, we define the linear map 
$$
A:\cV\to \cV \quad\mbox{via}\quad A(v) = j v, \ v\in \cV_j
\ \mbox{and}\ j\in \bN,
$$
and the  linear maps 
$$
\delta_r :=\Exp ((\ln r) A) = \sum_{k=0}^\infty \frac 1{k!} ((\ln r)A)^k, 
\quad r>0,
$$
or, in other words, the linear maps given via $\delta_r v = r^j v$, $v\in \cV_j$ and $r>0$. 

\begin{definition}
\label{def_dilation_gradedV}
Let $\cV = \oplus_{j=1}^\infty \cV_j $ be  a graded finite dimensional space.
\begin{enumerate}
	\item The maps $\delta_r$, $r>0$, defined above  are  called the \emph{dilations} of the graded vector space~$\cV$. 
\item 	The eigenvalues of the map  $A$ defined above are called the \emph{weights} of the dilations; they are the integers $j\in \bN$ such that $\cV_j\neq\{0\}$.
A basis of $\cV$ made up of eigenvectors for $A$ is called an \emph{adapted basis}.

By convention, if $\bV$ is an adapted basis, we will always assume that they are indexed so that 
the weights are increasing, that is, $\bV=(V_1,\ldots,V_n)$ satisfies
$$
A V_i = \upsilon_i V_i, 
\ i=1,\ldots, n=\dim \cV, 
\qquad \mbox{with}\quad 
0<\upsilon_1 \leq \ldots\leq \upsilon_n.
$$
\item The number of  distinct $\upsilon_j$, $j\geq 1$, is called the step of the gradation.
\item The \emph{homogeneous dimension} of the gradation is given by 
$$
Q:= \sum_{j=1}^\infty j \dim \cV_j = \upsilon_1+\ldots+\upsilon_n. 
$$

\item 	A function $f:\cV\to \bC$ is $m$-\emph{homogeneous} with respect to the $\delta_r$-dilations when $f(\delta_r v) =r^m f(v)$ for all $v\in \cV$ and $r>0$. This extends to measurable functions and distributions on $\cV$. A differential operator $T$ on $\cV$ is $m$-homogeneous with respect to the $\delta_r$-dilations when $T(f\circ \delta_r) = r^m (Tf)\circ \delta_r$ for all $f\in C_c^\infty(\cV)$ and $r>0$. This also extends to more general linear operators, e.g. $T:\cS(\cV)\to \cS'(\cV)$ etc.

\item A \emph{quasinorm} on $\cV$ is a continuous  map  $|\cdot|:\cV \to [0,+\infty)$ that is symmetric, definite and 
1-homogeneous with respect to the $\delta_r$-dilations, i.e. we have for any $v\in \cV$ and $r>0$:
$$
|-v|=|v|\geq0, \qquad |v|=0\Longleftrightarrow v=0, \qquad 
|\delta_r v|=r|v|.
$$ 
\end{enumerate}
\end{definition}

\begin{ex}
\label{ex_quasinorm}
	If $\bV=(V_1,\ldots,V_n)$ is an adapted basis of the graded vector space $\cV$, then fixing $\alpha>0$, the following defines a quasinorm. 
	$$
	|v|_{\alpha , \bV} = \left(\sum_{i=1}^n |v_i|^{\alpha / \upsilon_i}\right)^{1/\alpha},\qquad v=\sum_{i=1}^n v_i V_i.
	$$	
\end{ex}

 \begin{remark}
    \label{rem_eq_quasinorm} 
Any two quasinorms $|\cdot|$ and $|\cdot|'$ on the same graded space $\cV$ are equivalent in the sense that 
$$
\exists C>0, \quad 
\forall v\in \cV,\qquad
C^{-1} |v| \leq | v|' \leq C |v| .
$$
 \end{remark}

\begin{definition}
	A (real and finite dimensional) Lie algebra $\fg$ is \emph{graded} when the vector space $\fg$ admits a gradation 
	$\fg = \oplus_{j=1}^\infty  \fg_j $ satisfying 
$$
[\fg_i,\fg_j]\subset \fg_{i+j}
\quad\mbox{for all}\ i,j\in\bN_0,
$$
with the convention that $\fg_0=\{0\}$ is trivial.
	 \end{definition} 

We readily check that a Lie algebra admitting a gradation is graded if and only if  the dilations $\delta_r$, $r>0$, are morphisms of the Lie algebra $\fg$. 

\begin{remark}\label{rem:steps} If a Lie algebra is graded, then it is nilpotent
 and the step of the Lie algebra, i.e. its degree of nilpotency for the Lie bracket, is less or equal to the step of the gradation, (the number of distinct $\upsilon_j$). 
 \end{remark}

The dilations extend naturally to its universal enveloping Lie algebra $\sU(\fg)$.
Keeping the same notation for  the resulting maps
$$
\delta_r :  \sU(\fg)\to \sU(\fg),\quad r>0,
$$
these maps
are morphisms of the algebra $\sU(\fg)$. 
They allow us to define a notion of homogeneity in the universal enveloping Lie algebra:
an element $T\in \sU(\fg)$ is homogeneous of degree $\nu$ when $\delta_r T = r^\nu T $ for any $r>0$.
\begin{notation}
 The subspace of $\sU(\fg)$ comprised of elements of homogeneous degree equal to $\nu$
is denoted by 
$\sU_{\nu}(\fg)$.   
\end{notation}
If $\bV=(V_1,\ldots,V_n)$ is an adapted basis of $\fg$, then $\sU_{N}$ consists of all the linear combination of 
$$
\bV^\alpha:= V_1^{\alpha_1}\ldots V_n^{\alpha_n},\qquad \alpha=(\alpha_1,\ldots,\alpha_n)\in \bN_0^n,
$$
with $[\alpha]=N$, where
 $$
[\alpha]:= \alpha_1 \upsilon_1+\ldots +\alpha_n \upsilon_n.
$$
Above, we have adopted the convention $\bV^0=1$ for $\alpha=0$.

\subsubsection{Graded Lie groups}
Recall that there is a one-to-one correspondence between nilpotent Lie algebras and connected simply connected nilpotent Lie groups.
This leads to the following definition:

\begin{definition}
	A connected simply connected Lie group is \emph{graded} when its Lie algebra is graded. 
\end{definition}

Any graded Lie group is therefore nilpotent. 
By convention, in this paper, all the nilpotent Lie groups are assumed to be connected and simply connected. 
The exponential mapping 
$$
\Exp_G : \fg \to G
$$
of a nilpotent Lie group $G$ with Lie algebra $\fg$ is a bijection and furthermore a global diffeomorphism. 
This allows us to  define the dilations on a graded Lie group:

\begin{definition}
	The \emph{dilations} of a graded Lie group $G$ are the mappings  defined via
	$$
	\delta_r :G\to G, \qquad \delta_r \Exp_G V = \Exp_G (\delta_r V),\ V\in \fg,\ r>0.
	$$ 
\end{definition}

The dilations of a graded Lie group are morphisms of the Lie group. 
With the exponential mapping $\Exp_G$ and a choice of adapted basis on the graded Lie algebra $\fg$, it is possible to identify the spaces of functions on $G$ with spaces of functions on the vector space underlying~$\fg$. 
The dilations then lead naturally to the notion of homogeneity of a function on $G$ and operators acting on functional spaces on $G$ as in Definition \ref{def_dilation_gradedV}.

\begin{ex}\label{ex_right_left}
	If $\bV=(V_1,\ldots,V_n)$ is an adapted basis of the graded Lie algebra $\fg$, 
 then identifying elements of $\fg$ with left-invariant vector fields, 
 the differential operators $V_j$, $j=1,\ldots,n$, 
 and more generally the differential operators 
 $\bV^\alpha=V_1^{\alpha_1}\ldots V_n ^{\alpha_n}$
 are  homogeneous for the dilations $\delta_r$, $r>0$, with respective degrees $\upsilon_j$ 
and $[\alpha]$.
We can also consider the right-invariant vector fields $\widetilde V_j$, $j=1,\ldots,n$, associated with $\bV$, and more generally the corresponding right-invariant differential operator $\widetilde \bV^\alpha$, which will also be $[\alpha]$-homogeneous.
\end{ex}

\begin{notation}
\label{notation_right_left}
If $T\in \sU(\fg)$, we may use the same notation for the corresponding left-invariant differential operator, and we denote by  $\tilde T$  the corresponding right-invariant differential operator.
We will use this notation throughout this article.   
\end{notation}

In all this section, we are concerned with the dependence of constants with respect to the structure of the graded group, in particular its gradation and its Lie structure. We consider  the {\it structural constant} $\|[,]\|_{\bV}$ defined as follows:
If $T:\cV\times \cV \to \cV$ is a map and $\bV$ a basis of $\cV$, 
	we set
	$$
	\|T\|_\bV := \max_{1\leq i,j,k\leq n} |T_{i,j,k}|
	\qquad \mbox{where}\qquad 
	T(V_i,V_j) = \sum_{k=1}^n T_{i,j,k} V_k.
	$$
This defines a  norm on the space of bilinear maps $\cV\times \cV \to \cV$.
And $ \|[\cdot ,\cdot ]_\fg\|_{\bV}$ denotes taking this norm for the Lie bracket $ [\cdot ,\cdot ]_\fg$ of $\fg$ which is a bilinear map $\cV\times \cV \to \cV$. Most of the constants we will see below depends on the Lie structure only via  this structural constant $\|[,]\|_{\bV}$.

\subsubsection{Baker-Campbel-Hausdorff formula}
\label{subsubsec_BCH}

We will describe the law of the group in coordinates via the \emph{Baker-Campbel-Hausdorff formula},  abbreviated here to BCH, 
the exponential mapping $\Exp_G$ and a choice of a basis $\bV=(V_1,\ldots,V_n)$ of $\fg$ adapted to the gradation.

\smallskip 

The BCH formal writes:
if $a$ and $b$ are two non-commutative indeterminates,
we have in the sense of formal power series of $a$ and $b$ that $e^a e^b =e^c$ where 
\begin{equation}
	\label{eq_fBCH}
	c=\sum_{k=1}^\infty \frac{(-1)^{k+1} }k 
\sum_{\alpha_i+\beta_i\neq 0} 
\frac{1}{\alpha!\beta!|\alpha+\beta|}\left( \ad\,  a\right)^{\alpha_1}\left( \ad\,  b\right)^{\beta_1}
\ldots \left( \ad\,  a\right)^{\alpha_k}\left( \ad\,  b\right)^{\beta_k-1}b.
\end{equation}
When $\beta_k=0$, the  term should be modified so that 
the last factor is $\left( \ad\,  a\right)^{\alpha_k-1}a$.
Above,
$\ad\, a$ denotes the commutator with $a$, i.e. $\ad\, a (f) = af-fa,$ and similarly for $b$.
In particular,
BCH describes, on the one hand, the group product of a nilpotent Lie group
(aka Dynkin formula, via the group exponential after a choice of a basis of the Lie algebra)
and, on the other hand, the composition of flows of vector fields (see the appendix of  \cite{NagelSteinWainger1985}).
We will use the BCH formula in various ways later on for Theorem \ref{lem_Cq1BCH}.

\begin{notation}
\label{notation_BCH}
 The truncated BCH formula  with a sum with $|\alpha|+|\beta|\leq N$ is denoted by:
\begin{align*}
&{\rm BCH}_N(a,b)\\
\nonumber
&\;= \sum_{k=1}^\infty \frac{(-1)^{k+1} }k 
\sum_{\alpha_i+\beta_i\neq 0,\,|\alpha|+|\beta| \leq N} 
\frac{1}{\alpha!\beta!|\alpha+\beta|}\left( \ad\,  a\right)^{\alpha_1}\left( \ad\,  b\right)^{\beta_1}
\ldots \left( \ad\,  a\right)^{\alpha_k}\left( \ad\,  b\right)^{\beta_k-1} b.
\end{align*}
Various noncommutative setting may be used for the BCH formula, and we may write ${\rm BCH}_N(a,b;\ad)$ to specify which commutator structure $\ad$ is used.   
\end{notation}

\subsubsection{Quasinorms and Taylor formula}

Quasinorms were defined on graded vector space in Definition \ref{def_dilation_gradedV}, and in particular on graded Lie algebra.
Using the exponential mapping $\Exp_G$, we may identify the quasinorms on a graded group $G$ and on its Lie algebra $\fg$. 
Then a quasinorm $|\cdot|$ on $G$ satisfies the triangle inequality up to a constant:
\begin{equation}
    \label{eq_TriangleIneq}
    \exists C>0\qquad 
\forall x,y\in G\qquad 
|xy| \leq C(|x|+|y|).
\end{equation}

\begin{ex}
\label{ex_quasinormGX}
We define $|x|_\bV = |\Exp_G^{-1} \cdot |_{\alpha,\bV}$
with $\bV$ an adapted basis, $|\cdot|_{\alpha,\bV}$ defined in Example \ref{ex_quasinorm} and  $\alpha =2M_0$ where $M_0$ is a common multiple of the weights $\upsilon_1,\ldots,\upsilon_n$, the least one to fix the idea.  
Then $|\cdot|_\bV$ is a  quasinorm defined on $G$.
\end{ex}

\begin{notation}
  We call the quasinorm $|\cdot|_\bV$ of Example~\ref{ex_quasinormGX} the \emph{standard quasinorm associated to $\bV$}.
It is smooth away from 0 and $|\cdot|_\bV^{2M_0}$ is smooth, even at 0. We will use this notation all along the article.  
\end{notation}

\begin{remark}
\label{rem_Cst_trineq}
 In the triangle inequality  \eqref{eq_TriangleIneq} for the quasinorm $|\cdot|_\bV$ associated to $\bV$ of Example \ref{ex_quasinormGX}, we can choose the constant $C$ so that it  is an increasing function of $\|[,]_{\fg}\|_\bV$ depending on the gradation of $\fg$ and on the choice of adapted basis $\bV$, but otherwise not on the Lie structure, i.e. the full description of the Lie bracket $[\cdot, \cdot]_\fg: \fg\times \fg \to \fg$ (up to isomorphism of Lie algebras).
\end{remark}

Using the exponential mapping $\Exp_G$ and a fixed  basis $\bV=(V_1,\ldots,V_n)$ of $\fg$ adapted to the gradation,
 we will often identify a point $x$ in $G$ with the coordinates of $\Exp_G^{-1} x$ in $\fg$ with respect to $\bV$.
 As already mentioned, having fixed an adapted basis $\bV$, 
	we set 
$$
x= \Exp_G (x_1 V_1+\ldots+x_n V_n)\in G.
$$
A function on $G$ can then be realized as a function of the variables  $(x_1,\cdots ,x_n)\in\bR^n$. For example,   the monomials with respect to $\bV$ are defined as follows:

\begin{ex}
\label{ex_monomialX}
Having fixed an adapted basis $\bV$, 
	we set 
$$
x^\alpha := x_1^{\alpha_1}\ldots x_n^{\alpha_n},
\qquad x= \Exp_G (x_1 V_1+\ldots+x_n V_n)\in G,
\qquad \alpha=(\alpha_1,\ldots,\alpha_n)\in \bN_0^n.
$$
The functions $x^\alpha$, $\alpha\in \bN_0^n$
are called the monomials with respect to $\bX$.
The monomial $x^\alpha$ is $[\alpha]$-homogeneous for the dilations $\delta_r$, $r>0$.
\end{ex}

Our analysis will require Taylor estimates adapted to graded groups and due to Folland and Stein \cite{folland+stein_82}, see also   Theorem 3.1.51 in~\cite{R+F_monograph}.

\begin{theorem}
\label{thm_MV+TaylorG}
Let $G$ be a graded Lie group.
We fix a basis $\bV$ of $\fg$ adapted to the gradation and  a quasinorm $|\cdot|$ on $G$.
\begin{itemize}
    \item[(1)] {\rm Mean value theorem}. There exists $C_0>0$ and $\eta>1$ such that for any $f\in C^1(G)$ we have
    $$
    \forall x,y\in G,\qquad |f(xy) - f(x)| \leq C_0 \sum_{j=1}^n |y|^{\upsilon_j} \sup_{|z| \leq \eta |y|} |(V_j f)(xz)|.
    $$
    \item[(2)] {\rm Taylor estimate}. 
    More generally, 
    with the constant $\eta$ of point (1), 
    for any $N\in \bN_0$, 
    there exists $C_N>0$ 
    such that for any $f\in C^{\lceil N\rfloor}(G)$ we have
    \begin{equation*}\label{taylor_1}
    \forall x,y\in G,\qquad |f(xy) - \bP_{G,f,x,N} (y)| 
    \leq C_N \sum_{\substack{|\alpha|\leq \lceil N\rfloor+1\\ [\alpha]>N} }|y|^{[\alpha]}
    \sup_{|z| \leq \eta^{\lceil N \rfloor+1} |y|} |(\bV^\alpha f)(xz)|.
    \end{equation*}
Above,  $\lceil N\rfloor$ denotes  $ \max\{|\alpha| : \alpha\in \bN_0^n$ with $[\alpha]\leq N\}$
and
$\bP_{G,f,x,N}$ denotes the Taylor polynomial of $f$ at $x$ of order $N$ for the graded group $G$, i.e. the unique linear combination of monomials of homogeneous degree $\leq N$ satisfying 
$\bV^\beta\bP_{G,f,x,N} (0)=\bV^\beta f(x)$  for any $\beta\in \bN_0^n$ with $[\beta]\leq N$.
\end{itemize}
\end{theorem}

\begin{remark}
\label{remthm_MV+TaylorG}
If we choose the quasinorm $|\cdot| = |\cdot|_\bV$ associated with $\bV$ of Example \ref{ex_quasinormGX},
the constants $\eta$ and $C_N$ can be chosen so that they are increasing functions of $\|[,]_{\fg}\|_\bV$ depending on the gradation of~$\fg$ and on the choice of adapted basis $\bV$, but otherwise not on the Lie structure.
\end{remark}

In the sequel it will be useful to describe  the differentiation of the composition of functions. We denote by $*_G$ the  law of the group $G$ and, once given a basis $\bV$ of $\fg$ adapted to the gradation, 
we identify $G$ with $\bR^n$.

\begin{lemma}
\label{lem_comp_der_groupe}
Let $G, G'$ be two graded Lie groups.
We fix a basis $\bV$ (resp. $\bV'$) of $\fg$ (resp. of~$\fg'$) adapted to the gradation.
If $f: G'\to \bR$ and $g: G \to G'$ are smooth, then with the convention above,
$$
V_{j} (f\circ g)(v) = \sum_{\ell=1}^{\dim G'} h_\ell (v)
 \ V_\ell' f (g(v)),
 \quad\mbox{where}\ 
h_\ell(v):=\partial_{t=0} \left[(-g(v)) *_{G'} g (v*_G te_j)\right]_\ell.
$$
Above $(e_1,\ldots,e_{\dim G})$ denotes the canonical basis of $\bR^{\dim G}$, and $[v']_\ell$ the coefficient of $v'\in \bR^{\dim G'}$.
\end{lemma}

\begin{proof}[Proof of Lemma \ref{lem_comp_der_groupe}]
This follows readily from writing 
\begin{align*}
V_{j} (f\circ g)(v) &=  \partial_{t=0}\left(f (g(v*_G te_j)) \right)\\
& =  \partial_{t=0}f \left( g(v) *_{G'}\big ( (-g(v)) *_{G'} g (v*_G te_j)\big ) \right)\\
&= \sum_{\ell=1}^{\dim G'} \partial_{t=0} \left[(-g(v)) *_{G'} g (v*_G te_j)\right]_\ell
 \ V_\ell' f (g(v)),
 \end{align*}
\end{proof}

\subsection{Group Fourier transform on a nilpotent Lie group}
\label{subsec_cFG}

Let $G$ be a  nilpotent Lie group. 
Any representation $(\pi,\cH_\pi)$ of $G$ we consider in this paper is assumed to be unitary, strongly continuous and its Hilbert space $\cH_\pi$ separable. 
We follow \cite{Dixmier_C*,R+F_monograph}.

\subsubsection{Definition on $L^1(G)$}
\label{subsubsec_FL1}
We fix a Haar measure $dx$ on $G$. 
The \emph{group Fourier transform} of a function $\kappa\in L^1(G)$ at a representation $\pi$ of $G$ is  the operator 
$$
\int_G \kappa(x) \pi(x)^* dx \in \sL(\cH_\pi),
$$
denoted by 
$$
\cF_G \kappa (\pi) = \widehat \kappa(\pi) = \pi(\kappa) 
$$
Note that if $\pi_1$ and $\pi_2$ are equivalent, i.e. $U\pi_1 = \pi_2 U$ for some unitary map $U$ between their Hilbert spaces, then $U \cF_G \kappa (\pi_1) = \cF_G \kappa (\pi_2)U$.
Therefore, we may consider the group Fourier transform of  $\kappa\in L^1(G)$ at the equivalence class of the representation $\pi$ of $G$. 

The \emph{unitary dual} $\widehat G$ is the set of irreducible representations of $G$ modulo equivalence equipped with the hull kernel topology. 
It is a standard Borel set and the group Fourier transform of  $\kappa\in L^1(G)$ may be viewed as the measurable field of equivalence classes of  bounded operators 
$$
\cF_G \kappa = \{\pi(\kappa) \in \sL(\cH_\pi), \pi\in \Gh\}.
$$
We denote by $1_{\widehat G}$ (the class of) the trivial representation of $G$.  

Note that 
\begin{equation}\label{eq_convolution}
\mathcal  F_G (f)\circ \mathcal F_G(g)= 
\mathcal F(g\star_Gf),\;\; g\star_G f= \int_G
g(y )f(y^{-1} x) dy
=\int_G
g(xy^{-1} )f(y) dy.
\end{equation}

\subsubsection{Extension to $L^2(G)$}
\label{subsubsec_FL2}
There exists a measure $\mu_\Gh$ on $\widehat G$ such that 
$$
\forall f\in L^1(G)\cap L^2(G),\qquad
\int_G |f(x)|^2 dx = \int_{\widehat G} \|\pi(f) \|^2_{HS(\cH_\pi)} d\mu_\Gh(\pi), 
$$
where $ \|\cdot  \|_{HS(\cH_\pi)} $ denotes the Hilbert-Schmidt norm on $\cH_\pi$;
the measure $\mu_\Gh$ is unique once the Haar measure $dx = \mu_G$ has been fixed.
The measure $\mu_\Gh$ and the above formula are called the \emph{Plancherel measure} and the \emph{Plancherel formula} (see~\cite{Co_Gr,Kirilov} or \cite{Dixmier_C*,Dixmier_VN} for abstract version of Plancherel theorem in a general context).

We denote by $L^2(\widehat G)$ the Hilbert space of measurable fields of Hilbert-Schmidt operators $\{\sigma(\pi):\cH_\pi\to \cH_\pi$ Hilbert-Schmidt$\}$ 
(up to $\mu_\Gh$-a.e. equivalence and unitary equivalence)
such that the following quantity is finite
$$
\|\sigma\|_{L^2(\widehat G) } := \sqrt{\int_{\widehat G} \|\sigma \|  ^2_{HS(\cH_\pi)} d\mu_\Gh(\pi)}.
$$
This allows for a unitary extension of the group Fourier transform $\cF_G : L^2(G)\to L^2(\widehat G)$.

\subsubsection{Extension to distributions, the set of bounded invariant symbols}\label{subsec:distrib}
\label{subsubsec_FKLinfty}
We start with defining  the set of bounded invariant symbols $L^\infty(\widehat G)$:

\begin{definition}
	A \emph{bounded invariant symbol} of $G$ is a  measurable field of equivalence classes of bounded operators  $\sigma= \{\sigma (\pi)\in \sL(\cH_\pi) : \pi\in \widehat {G}\}$ on $\widehat G$ such that
	the following quantity is finite 
$$
\|\sigma\|_{L^\infty(\Gh)}:=\sup_{\pi\in \Gh} \|\sigma(\pi)\|_{\sL(\cH_\pi)}.
$$
The supremum above is understood as the essential supremum with respect to the Plancherel measure $\mu_\Gh$
and the notion of measurable fields of bounded operators 
is up to $\mu_\Gh$-a.e. equivalence and unitary equivalence.
The space of bounded invariant symbols of $G$ is  denoted by $L^\infty(\widehat G)$. 
\end{definition}

Note that differing from~\cite{R+F_monograph} and related papers, the symbols $\sigma$ in this section are invariant in the sense that they do not depend on~$x\in G$.
\smallskip

We check readily
$$
\forall \kappa\in L^1(G),\qquad
\| \pi( \kappa) \|_{L^\infty(\widehat G)}\leq \|\kappa\|_{L^1(G)},
\qquad\mbox{so}\qquad \cF_G(L^1(G))\subset L^\infty(\widehat G).
$$
However,  there are many other bounded invariant symbols.
In fact, it follows from Dixmier's full Plancherel theorem \cite[\S 18]{Dixmier_C*} that
$L^\infty(\widehat G)$ is  a von Neumann algebra (sometimes called the von Neumann algebra of the group) isomorphic to the subspace 
$$
\sL(L^2(G))^G
$$
of bounded operators on $L^2(G)$ that  commutes with left-translation of the group.
Indeed, if $\sigma\in L^\infty(\widehat G)$, then it is straightforward to check that the Fourier multiplier defined via $f \mapsto \cF_G^{-1} (\sigma \widehat f)$ is in $\sL(L^2(G))^G$ and its symbol is $\sigma$.
Conversely,  
 if $T\in \sL(L^2(G))^G$,
there exists a unique bounded invariant symbol $\widehat T\in L^\infty(\Gh)$ such that 
$\cF_G (Tf) = \widehat T \widehat f$ for any $f\in L^2(G)$.
Moreover  $\|\widehat T\|_{L^\infty(\widehat G)} = \|T\|_{\sL(L^2(G))}$.

From the Schwartz kernel theorem, 
it follows that if $T\in \sL(L^2(G))^G$, there exists a unique distribution $\kappa\in \cS'(G)$ such that $Tf = f\star_G\kappa $ for any $f\in \cS(G)$.
The distribution $\kappa$ is called the (right) convolution kernel of $T$. We set $\cF_G \kappa := \widehat T$ and we check easily that this extends the  group Fourier transform to 
\begin{align*}
	\cK(G) := \{\kappa\in \cS'(G) : (f\mapsto f\star_G\kappa) \in \sL(L^2(G))\},
\end{align*}
and that  $\cF_G : \cK \to L^\infty(\widehat G)$ is injective, hence bijective.  
We readily check the following property:
\begin{lemma}
\label{lem_compwmorph}
    Let $\theta$ be an automorphism of the nilpotent Lie group $G$.
    We keep the same notation for the corresponding automorphism of its Lie algebra $\fg$. 
    The Jacobian determinant of $\theta$ on $G$ is then constant and equal to $\det \theta$.
For any $T\in \sL(L^2(G))^G$, the operator $T_\theta$ defined by
$$
T_\theta f= (T (f\circ \theta^{-1})) \circ \theta 
$$
 is in $ \sL(L^2(G))^G$
with 
$$
\|T_\theta\|_{\sL(L^2(G))}=\|T\|_{\sL(L^2(G))}.
$$
Moreover, denoting by $\kappa$ and $\sigma$ the convolution kernel and symbol of $T$, the convolution kernel and symbol of $T_\theta$ are given by 
$$
\kappa_\theta:=(\det \theta )\, \kappa\circ \theta \quad\mbox{and} \quad
\theta_* \sigma(\pi)=\sigma (\pi\circ \theta^{-1}),\quad \pi\in \Gh,
$$
with 
$\|\theta_*\sigma\|_{L^\infty (\Gh)}
=\|\sigma\|_{L^\infty (\Gh)}.$
\end{lemma}

\subsubsection{Invariant symbols}

\begin{definition}
An \emph{invariant symbol} $\sigma$ on $G$ is a  measurable field of operators $\sigma= \{\sigma (\pi):\cH_\pi^{+\infty} \to \cH_\pi^{-\infty} : \pi\in \widehat G\}$ over $\widehat G$.
Here, $\cH_\pi^{+\infty}$ and $\cH_\pi^{-\infty}$ denote the spaces of smooth and distributional vectors in $\cH_\pi$.
\end{definition}

We identify the Lie algebra $\fg$  of $G$ with the space of real left-invariant vector fields on $G$ and the universal enveloping Lie algebra $\sU(\fg)$ with the space of left-invariant differential operators on $G$. 
If $\pi$ is a representation of $G$, we keep the same notation for the corresponding  representations of $\fg$ and $\sU(\fg)$. 

\begin{ex}
\label{ex_piTonH+-}
For any $T\in \sU(\fg)$, $\widehat T=\{\pi(T), \pi\in \widehat G\}$ is an invariant symbol.
In fact, 	each $\pi(T)$ acts on $\cH_\pi^{+\infty}$ and, by duality, on $\cH_\pi^{-\infty}$.
\end{ex}

\begin{remark}\label{rem_comp_symb}
    With this general definition, two invariant symbols may not be composable, 
although an invariant symbol can always be composed (on the right and on the left)  with $\widehat T$, $T\in \sU(\fg)$.
\end{remark}

In Section \ref{subsec_InvSymbClasses} below, we will restrict ourselves to smaller subspaces of invariant symbols which will contain $\{\pi(T), \pi\in \Gh\}$, $T\in \sU(\fg)$.  
To define these subspaces, 
we will consider the Sobolev spaces on $G$
whose definition is recalled in Section~\ref{subsubsec_L2sG}  below.
This  relies on the notion of Rockland operator recalled in Section \ref{subsec_R} and on the graded Lie structure of the group $G$
while the theory recalled above in Sections \ref{subsubsec_FL1}, \ref{subsubsec_FL2} and  \ref{subsubsec_FKLinfty}  only requires $G$ to be 
a locally compact unimodular group $G$ of type 1.

\subsection{Rockland operators and Sobolev spaces}
\label{subsec_R}

As in~\cite{R+F_monograph}, we adopt the following conventions:
\begin{definition}
A left-invariant differential operator $\cR\in \sU(\fg)$  on $G$ is \emph{Rockland} when it is homogeneous and for any $\pi\in \widehat G\setminus \{1_{\widehat G}\}$, 
$\pi(\cR)$ is injective on $\cH_\pi^\infty$. 

It is positive when 
$$
\forall f\in \cS(G),\qquad 
\int_G \cR f(x) \overline {f(x)} dx \geq 0.
$$
\end{definition}

\begin{ex}
\label{ex_RG}
	A positive Rockland operator always exists, for instance
\begin{equation}\label{def:R_introduction}
	\cR = \sum_j (-1)^{\frac {M_0}{\upsilon_j}}  V_{j}^{2\frac{M_0} {\upsilon_j}}, 
\end{equation}
having fixed a basis $\bV=(V_1,\ldots,V_n)$ adapted to the gradation and a multiple $M_0$ of the resulting weights $\upsilon_1, \ldots, \upsilon_n$. 
Its homogeneous  degree is $2M_0$.
For $M_0$ the least common multiple of the weights $\upsilon_1,\ldots,\upsilon_n$, 
	we denote by $\cR_\bV$ the resulting positive Rockland operator.
\end{ex}

The Helffer-Nourrigat theorem \cite{HN0} states that 
a Rockland operator is hypoelliptic (this result valid on graded groups was further investigated on manifolds by the same authors in \cite{HN} in the 80's, leading to the Helffer-Nourrigat conjecture proved recently in \cite{AMY}).
Moreover, if $\cR$ is  symmetric,  then $\cR$ and $\pi(\cR)$ for any $\pi\in \widehat G$ are essentially self-adjoint on $\cS(G)\subset L^2(G)$ and $\cH_\pi^\infty\subset \cH_\pi$ respectively;
we keep the same notation for their self-adjoint extensions. 
Denoting by $E_{\cR}$ and $E_{\pi(\cR)}$ the spectral resolutions of $\cR$ and $\pi(\cR)$, 
\begin{equation}
    \label{eq_spectralmeas}
    \mbox{i.e.}\ \cR = \int_\bR \lambda\, dE_{\cR}(\lambda)
\quad\mbox{and}\quad 
\pi(\cR) = \int_\bR \lambda\, dE_{\pi(\cR)}(\lambda),
\end{equation}
then for any interval $I\subset \bR$, we have
$$
\cF_G(E_{\cR}(I))=
\{E_{\pi(\cR)}(I), \pi\in \widehat G\},
$$
or in other words
\[
\forall f\in\mathcal S(G),\;\;
\mathcal F_G(E_{\cR}(I) f )= E_{\pi(\cR)}(I) \mathcal F_G f. 
\]
The spectral resolutions of $\cR$ and $\pi(\cR)$ allows us to consider spectral multipliers in these operators.

If in addition of being Rockland, $\cR$ is positive then $\pi(\cR)$ is also positive.
In fact, there is an equivalence between $\cR$ being positive and $\pi(\cR)$ positive for any $\pi\in \Gh$.
In this case, 
for each $\pi\in \widehat G\setminus\{1_{\widehat G}\}$, 
the operator
$\pi(\cR)$ has a discrete spectrum in $(0,+\infty)$ and 
each eigenvalue has finite multiplicity \cite{hul}.

 \subsubsection{Inhomogeneous Sobolev spaces $L^2_s(G)$  on $G$}
\label{subsubsec_L2sG}

For any $s\in \bR$ 
and  $\cR$ a positive Rockland operator of homogeneous degree $\nu$,
 the domain of the spectral multiplier $(\id+\cR)^{s}$  contains $\cS(G)$,
and 
$\pi(\id+\cR)^{s}$, $\pi\in \widehat G$, acts on $\cH_\pi^\infty$ and on $\cH_\pi^{-\infty}$.	
We denote by $L^2_{s,\cR}(G)$ on $G$ the closure of $\cS(G)$ in $\cS'(G)$ for the norm 
$$
\|f\|_{L^2_s(G),\cR}:=
\|(\id+\cR)^{s/\nu}f\|_{L^2(G)}.
$$
\begin{proposition}
\label{prop_sob}
The spaces $L^2_{s,\cR}(G)$ enjoy the following properties:
\begin{enumerate}
	\item Each $L^2_{s,\cR}(G)$ is a Hilbert space	and for $s=0$, $L^2_{0,\cR}(G)=L^2(G)$.
	\item For any $s_1 \geq s_2$, we have the continuous inclusions $L^2_{s_1,\cR}(G)\subseteq L^2_{s_2,\cR}(G)$.
 \item The following interpolation property holds
	$$
	\|f\|_{L^2_{t,\cR}(G)} \leq \|f\|_{L^2_{s,\cR}(G)}^{\theta} \|f\|_{L^2_{u,\cR}(G)}^{1-\theta}
	$$
for $	t\in (s,u)$ with $s<u$ and $t= \theta s +(1-\theta)u$, $\theta\in (0,1)$. 

	\item The spaces $L^2_{s,\cR}(G)$ and $L^2_{-s,\bar \cR}(G)$ are dual to each other.
	\item If $\bV$ is a basis of $\fg$ adapted to the gradation and $s\in \bN$ is a common multiple of the weights of the gradation, then $L^2_{s,\cR}(G)$ is the space of $f\in L^2(G)$ such that $\bV^\alpha f\in L^2(G)$ for any multi-index $\alpha$ with $[\alpha]=s$. Furthermore, a norm equivalent to the norm $\|\cdot\|_{L^2_{s}(G),\cR}$ is given by 
	$$
	f\longmapsto \|f\|_{L^2(G)} + \sum_{[\alpha]=s}\|\bV^\alpha f\|_{L^2(G)}.
	$$	
	\item\label{prop_sob_R1R2} 
	If $\cR_1$ and $\cR_2$ are two positive Rockland operators, then 
for any $s\in \bR$, 
$$
\exists C=C_{s,G,\cR_1,\cR_2}>0,\qquad
  \forall f\in \cS(G),\qquad
  \|f \|_{L^2_s(G),\cR_1}
  \leq C \|f \|_{L^2_s(G),\cR_2}.
	$$
 	\item\label{prop_sob_emb} {\rm Sobolev embedding.}
For $s>Q/2$ the inclusion   $L^2_{s,\cR} \subset C_b (G)$ is continuous; here $C_b(G)$ is the Banach space of continuous bounded function on $G$ equipped with the norm  $\|\cdot\|_{L^\infty(G)}$. 
\end{enumerate}
\end{proposition}

It follows from  Proposition \ref{prop_sob} that for each $s\in \bR$, the space $L^2_{s,\cR}(G)$ is independent of~$\cR$ or $\bV$,
and all the norms $\|f\|_{L^2_s(G),\cR}$ for various $\cR$ are equivalent.
Moreover, $L^2_{s,\cR}(G)$ depends only on $s\in \bR$ and on the connected simply connected nilpotent Lie group $G$.
We then omit the index $\cR$ to denote the space by $L^2_s(G)$.
This space is called the Sobolev space on $G$.

We can be more precise about the constant in Proposition \ref{prop_sob} \eqref{prop_sob_R1R2}:

\begin{remark}\label{rem_constant_RS}
Assume that  the gradation of the vector space underlying $\fg$ as well as an adapted  basis $\bV$ are fixed. 
Then, the constant $C$ may be chosen as
$$
C= \widetilde C (\max (s, \|[\cdot ,\cdot ]_\fg\|_{\bV}, C_\bV(\cR_1), C_\bV(\cR_2))),
$$
where  $\widetilde C:[0,\infty)\to (0,\infty)$ is an increasing function, and  the constants $C_\bV(\cR_i)$, $i=1,2$, are defined as follows:
let $\cR$ be a positive Rockland operator on~$G$, then it can be written as $\cR = \sum_{[\alpha]=\nu} c_\alpha \bV^\alpha$ (here $\nu$ is the homogeneous degree of $\cR$), and we set  
 $C_\bV (\cR):=
\max_{[\alpha]=\nu } |c_{\alpha}|.$
\end{remark}

\subsubsection{Further extensions of the group Fourier transform}

We set
$$
\cK_{a,b}(G) := \{\kappa\in \cS'(G), (f\mapsto f\star_G\kappa) \in \sL(L^2_a(G),L^2_b(G))\}.
$$
Fixing a positive Rockland operator $\cR$ of homogeneous degree $\nu$, if $\sigma$ is an invariant symbol, then 
so is the symbol
$$
(\id+\widehat \cR)^{b/\nu}  \sigma (\id+\widehat \cR)^{-a/\nu} 
=\{(\id+\pi (\cR))^{b/\nu}  \sigma(\pi) (\id+\pi (\cR))^{-a/\nu} : \pi\in \widehat G\}.
$$
We then define the following subspace of invariant symbols
$$
\sigma\in L^\infty_{a,b}(\widehat G)\;\;\mbox{if and only if}\;\;
(\id+\widehat \cR)^{b/\nu}  \sigma\, (\id+\widehat \cR)^{-a/\nu} \in L^\infty(\widehat G).
$$
Equipped with the norm 
\begin{equation}\label{def:norm_introduction}
\|\sigma\|_{L^\infty_{a,b}(\widehat G),\cR}
:= \|(\id+\widehat \cR)^{b/\nu}  \sigma \,(\id+\widehat \cR)^{-a/\nu}\|_{L^\infty(\widehat G)}, 
\end{equation}
it is a Banach space, isomorphic and isometric to the subspace of left-invariant operators in $\sL(L^2_a(G),L^2_b(G))$. 
The properties  of Rockland operators (see  Proposition~\ref{prop_sob}) imply that $L^\infty_{a,b}(\widehat G)$ is independent of the choice of a positive Rockland operator $\cR$. 
Moreover, 
the Schwartz kernel theorem allows us to extend the group Fourier transform into a bijection
$\cF_G :\cK_{a,b}(G) \to L^\infty_{a,b}(G)$ (see Proposition~5.1.29 in~\cite{R+F_monograph}). 
 
\subsubsection{Homogeneous Sobolev spaces}
\label{subsubsec_homL2sG}
For any $s\in \bR$ and any  positive Rockland operator $\cR$ of homogeneous degree $\nu$, 
the map 
$$
f\mapsto \|\cR^{s/\nu}f\|_{L^2(G)}:=\|f\|_{\dot L^2_s(G),\cR}
$$
defines a norm on $\cS(G)\cap \dom(\cR^{s/\nu})$. 
Moreover (see \cite[Section 4.4]{R+F_monograph}), if $\cR_1$ and $\cR_2$ are two positive Rockland operators, we obtain as in Proposition \ref{prop_sob} \eqref{prop_sob_R1R2},  
$$
\exists C=C_{s,G,\cR_1,\cR_2}>0,\qquad
  \forall f\in \cS(G),\qquad
  \|f \|_{\dot L^2_s(G),\cR_1}
  \leq C \|f \|_{\dot L^2_s(G),\cR_2}.
	$$
with a similar remark as in Remark \ref{rem_constant_RS} regarding the constant $C$ above.
Therefore, the subspace of tempered distributions obtained as the completion of  $\mathcal S(G)\cap \dom (\mathcal R^{s/\nu})$ for the norm $\|\cdot\|_{\dot L^2_s(G),\cR}$ is independent of a choice of a positive Rockland operator $\cR$.
It is called the homogeneous Sobolev subspace $\dot L^2_s(G)$.
Moreover  \cite[Section 4.4]{R+F_monograph}, the inclusion $\dot L^2_s(G)\subset \cS'(G)$ is continuous.
In the case when $s\in \bN$ is a common multiple of the weights of the gradation, fixing a basis $\bV$ of $\fg$ adapted to the gradation, 
 $\dot L^2_{s}(G)$ 
 contains $\cS(G)$ and 
 is the closure
  of $\cS(G)$
  in $\cS'(G)$ for $f\mapsto \sum_{[\alpha]=s}\|\bV^\alpha f\|_{L^2(G)}.$

\subsection{Invariant symbol classes}
\label{subsec_InvSymbClasses}

\subsubsection{Definitions}
We start with the definition of difference operators:

\begin{definition}\label{def:diff_op}
	Let $q\in C^\infty(G)$ and let $\sigma$ be an invariant symbol on $\widehat G$. 
	We say that $\sigma$ is $\Delta_q$-differentiable when $\sigma\in L^\infty_{a,b}(\widehat G)$ and $q \cF_G^{-1}\sigma\in \cK_{c,d}(G)$ for some $a,b,c,d\in \bR$. In this case, we set
	$$
	\Delta_q \sigma = \cF_G (q \cF_G^{-1}\sigma).
	$$
\end{definition}

We can now define our classes of symbols.
\begin{definition}
	An invariant symbol $\sigma$ on $\widehat G$ is of order $m$ if 
	$\sigma \in L^\infty_{0,-m}(\widehat G)$ and moreover, 	$\Delta_{x^\alpha} \sigma \in L^\infty _{0,-m +[\alpha]}(\widehat G)$ for any indices $\alpha\in \bN_0^n$. 
\end{definition}

In the above definition, 
 the $x^\alpha$'s are the monomials defined with respect to a basis $\bV$ adapted to the gradation (see Example \ref{ex_monomialX}). 
It is independent of the choice of adapted basis $\bV$.

\begin{notation}
  We denote by $S^m (\widehat G)$ the space of invariant symbols on $\widehat G$ of order $m$.  
\end{notation}

The space $S^m (\widehat G)$ is naturally equipped with a Fr\'echet structure given by the semi-norms
\begin{align}
\label{def_norm_symbol_1}
\|\sigma\|_{S^m (\widehat G), N,\bV,\cR}
&:=\|\sigma\|_{S^m (\widehat G), N,\cR}:=
\max_{[\alpha]\leq N}
\|\Delta_{x^\alpha} \sigma\|_{L^\infty_{0, -m +[\alpha]}(\widehat G),\cR}
\\
&=\max_{[\alpha]\leq N}
\|(\id+\cR)^{\frac{[\alpha]-m}\nu } 
\Delta_{x^\alpha} \sigma \|_{L^\infty(\Gh)}
\nonumber
\end{align}
with $N\in \bN_0$, 
having fixed a positive Rockland operator $\cR$ on $G$ (of homogeneous order~$\nu$) and the Haar measure induced by $\bV$.
With a Haar measure fixed, 
 the distribution $\kappa:=\cF_G^{-1} \sigma\in \cS'(G)$ is called the convolution kernel of $\sigma$.

\begin{ex}
\label{ex_XalphainSGh}
If $\bV$ is an adapted basis of $\fg$, 
then for any $\alpha\in \bN_0^n$,  the symbol $$
\widehat \bV^\alpha =\{\pi(\bV)^\alpha , \pi\in \widehat G\}
$$
is in $S^{[\alpha]}(\widehat G)$.
\end{ex}

We have the continuous inclusions:
\begin{equation}
\label{eq_inclusionSmGh}
   m_1 \leq m_2 \Longrightarrow
S^{m_1}(\widehat G)
\subset 
S^{m_2}(\widehat G). 
\end{equation} 
Moreover, the properties of interpolation of Sobolev spaces imply natural properties of interpolation for the space $S^m(\Gh).$ 

\subsubsection{Symbolic properties}
At the symbolic level, the symbol classes form a $*$-algebra:
\begin{theorem}
\label{thm_SmGh} Assume that an adapted basis $\bV$ of $\fg$ has been fixed. 
\begin{enumerate}
	\item If $\sigma\in S^m(\widehat G)$, then 
	$\Delta_{x^\alpha} \sigma \in L^\infty _{\gamma, -m +[\alpha]+\gamma}$ for any $ \alpha\in \bN_0^n$ and $\gamma\in \bR$. 
	Moreover, the following map is a continuous semi-norm on $S^m(\widehat G)$:
$$
\sigma\longmapsto\max_{[\alpha]\leq N}
\|\Delta_{x^\alpha} \sigma\|_{L^\infty_{\gamma, -m +[\alpha]+\gamma}(\widehat G),\cR}, \qquad N\in \bN_0, \, \gamma\in \bR.
$$
\item The composition and adjoint maps
$$
\left\{\begin{array}{rcl}
S^{m_1}(\widehat G) \times S^{m_2}(\widehat G)
&\longrightarrow & S^{m_1+m_2}(\widehat G)
\\
(\sigma_1,\sigma_2)& 	\longmapsto& \sigma_1\sigma_2
\end{array}
\right.
\quad\mbox{and}\quad
\left\{\begin{array}{rcl}
S^{m}(\widehat G) 
&\longrightarrow & S^{m}(\widehat G)
\\
\sigma& 	\longmapsto& \sigma^*
\end{array}
\right. ,
$$
are continuous. 
Moreover, denoting by $\kappa_1,\kappa_2$ 
the convolution kernels of $\sigma_1,\sigma_2$, their convolution product 
$\kappa_2 \star_G \kappa_1$ makes sense and is the convolution kernel of $\sigma_1\sigma_2$.
Denoting by $\kappa$ the convolution kernel of $\sigma$, 
then the convolution kernel of $\sigma^*$ is given by $v\mapsto \bar \kappa (v^{-1})$.
\end{enumerate}
\end{theorem}

The proof of Theorem \ref{thm_SmGh} may be found in~\cite{R+F_monograph}, see Theorems~5.2.20 and 5.5.22 therein (with the difference that the symbols in~\cite{R+F_monograph} also depend on $x\in G$).

\begin{remark}
\label{rem_thm_SmGh}
The proof of Theorem \ref{thm_SmGh} implies in particular that the constants describing the continuity of the linear map for the adjoint and the bilinear map for the composition depend on  the Lie structure only via $\| [\cdot,\cdot]\|_\bV$ after a choice of a basis $\bV$ adapted to the gradation.
Let us explain what this means for the continuity of the adjoint. We fix 
 a basis $\bV$ adapted to the gradation.
Then 
for any $N\in \bN_0$ and $m\in \bR$, there exist $C>0$ and $N'\in \bN$
 such that 
$$
\forall \sigma\in S^m(\widehat G),\qquad
\|\sigma^*\|_{S^m(\widehat G), N, \cR_\bV}
\leq
C \|\sigma\|_{S^m(\widehat G), N', \cR_\bV}.
$$
We could have taken any positive Rockland operator but we choose $\cR_\bV$ the one  associated with $\bV$ (see Example \ref{ex_RG}) to make the constant structural. 
The constants $C$ and $N'$ can be written as  increasing functions   of $\| [\cdot,\cdot]\|_\bV$. Naturally, these increasing functions depend on what has been fixed, that is, beside $m$ and $N$, on the gradation of the vector space underlying~$\fg$ and the adapted basis $\bV$.  
\end{remark}

\subsubsection{Asymptotic sums of symbol}

\begin{definition}
\label{def_asympexp}
    Let $\sigma\in S^m(
    \Gh)$ and $\sigma_j$, $j\in \bN_0$, be a sequence of symbols such that $\sigma_j\in S^{m_j}(\Gh)$ with 
    $m_0=m$ and $m_j$ strictly decreasing to $-\infty$. 
    We say that $\sum_j \sigma_j$ is an asymptotic expansion for $\sigma$ and we write
    $$
    \sigma\sim \sum_j \sigma_j,
    $$
    when we have for any $N\in \bN_0$,
    $$
    \sigma - \sum_{j=0}^N \sigma_j\in S^{m_{N+1}}(\Gh).
    $$
\end{definition}

Given a sequence $(\sigma_j)_{j\in\bN_0}$ as in Definition \ref{def_asympexp}, we can construct a symbol $\sigma\in S^{m_0}(\Gh)$  following the classical ideas due to Borel  (see e.g. \cite[Section 5.5.1]{R+F_monograph}), that will satisfy $\sigma\sim \sum_j \sigma_j$.
Moreover, $\sigma$ is unique up to $S^{-\infty}(\Gh)$.

\subsubsection{Kernel estimates}

The  convolution kernel associated with a symbol in some $S^m(\widehat G)$ will be Schwartz away from the origin 0, but may have a singularity at 0 \cite[Theorem 5.4.1]{R+F_monograph}
\begin{theorem}
\label{thm_kernelG}
	Let $\sigma\in S^m(\widehat G)$ and denote its convolution kernel by $\kappa =\cF_G^{-1} \sigma$.
	Then $\kappa$ is smooth away from the origin. Moreover,  fixing a quasinorm $|\cdot|$ on $G$, we have the following kernel estimates:
	\begin{enumerate}
		\item The convolution kernel $\kappa$ decays faster than any polynomial away from the origin:
	\begin{align*}
	    	\forall N\in \bN_0, \quad 
	\exists C=C_{\sigma, N}>0:\quad \forall y\in G, \\ |y|\geq 1\Longrightarrow
	|\kappa(y)|\leq C |y|^{-N}.
	\end{align*}
	\item 
	If $Q+m<0$ then $\kappa$ is continuous and bounded on $G$:
 $$
 \exists C=C_\sigma>0, \qquad 
 \sup_{y\in G} |\kappa(y)| \leq C.
 $$

 \item If $Q+m>0$, then 
	\begin{align*}
	\exists C=C_\sigma>0:
        \quad \forall y\in G , \qquad 0<|y|\leq 1\Longrightarrow
	|\kappa(y)|\leq  C|y|^{-(Q+m) }.
	\end{align*}

 \item If $Q+m=0$, then 
	\begin{align*}
	\exists C=C_\sigma>0:
        \quad \forall y\in G , \qquad 0<|y|\leq 1/2\Longrightarrow
	|\kappa(y)|\leq -  C \ln |y|.
	\end{align*}
	\end{enumerate}

In all the estimates above, 
 the constant $C=C_\sigma$ may be chosen of the form 
$$
C = C_1 \|\sigma\|_{S^m(\widehat G), N', \cR},
$$
where $\cR$ is a positive Rockland operator;  the constant $C_1$ and the semi-norm 
$\|\cdot \|_{S^m(\widehat G), N', \cR}$ are 
independent of $\sigma$. 
\end{theorem}

\begin{remark}
\label{rem_thm_kernelG}
\begin{enumerate}
    \item We can be even more precise. Having fixed  a basis $\bV$, consider  the associated quasinorm $|\cdot|_\bV$ of Example \ref{ex_quasinormGX} and the positive Rockland operator $\cR_\bV$ of Example \ref{ex_RG}.
The constant $C_1$ in each item of Theorem \ref{thm_kernelG} is an increasing function of $\|[\cdot,\cdot]\|_\bV$. Furthermore,
these increasing functions depend on  $m$ and,  in Part (1), also on $N$, as well as the gradation of the underlying subspace of the Lie algebra $\fg$ and on the adapted basis $\bV$, but not on the Lie structure.

\item Note that if $\kappa\in \cS'(G)$ is the convolution kernel of $\sigma\in S^m(\widehat G)$, then $x^\alpha \bV^{\beta_1} \widetilde\bV^{\beta_2}\kappa $ is the convolution kernel of the symbol $\Delta_{x^\alpha} \widehat \bV^{\beta_1} \sigma \widehat \bV^{\beta_2}$ 
which is in $S^{m-[\alpha]+[\beta_1]+[\beta_2]}(\widehat G)$ by Theorem~ \ref{thm_SmGh}. Recall that the left-invariant  differential operators $\bV^\alpha$ (resp. right-invariant $\widetilde\bV^\alpha$) have been introduced in Example~\ref{ex_right_left}.
\end{enumerate}
 \end{remark}

\begin{corollary}
\label{cor_thm_kernelG}
Assume that we have fixed a basis $\bV$ of $\fg$ adapted to the gradation and the corresponding Haar measure.
Let $\sigma\in S^m(\widehat G)$ with $m\in \bR$. 
\begin{enumerate}
    \item  If
$\alpha\in \bN_0^n$, $\beta\in \bN_0^n$, $p\in [1,\infty]$ with
 $m -[\alpha] +[\beta]<Q(\frac 1p -1)$ (with the convention $\frac 1\infty =0$), then the quantity 
$$
\|\sigma\|_{\alpha,\beta,p}
:=
\|  x^\alpha \bV^\beta  \kappa \|_{L^p(G)}, 
\qquad \sigma =\cF_G \kappa
$$
is finite. Moreover, the map $\sigma\mapsto \|\sigma\|_{\alpha,\beta,p}$ is a continuous semi-norm on $S^m(\widehat G)$.
\item If $N,M_1\in \bN_0$ such that 
 $-m+N \leq M_1$  with $M_1\in \bN_0$ a common multiple of the dilations' weights,
 then there exists a constant $C>0$ such that 
$$
   \|\sigma\|_{S^m (\widehat G), N,\cR}
   \leq C \max_{[\alpha']+[\beta']\leq N+M_1}
\|\sigma\|_{\alpha',\beta',1},
$$
for any $\sigma \in S^m(\widehat G)$ such that the right-hand side is finite.
\item Let $\sigma\in S^m(\widehat G)$.
Its convolution kernel $\kappa =\cF_G^{-1} \sigma$
is a distribution smooth  away from the
origin. At the origin, its order is  less than or equal  to the integer 
$$
p_\sigma:=\max\{|\alpha|: \alpha \in \bN_0^n , [\alpha] \leq s'\}
$$
where $s'$ is the smallest non-negative  integer multiple of the dilation weights such that $s'>Q/2+m$. 
Moreover, 
we have for any compact neighbourhood $\cC_0$ of $0$ in $\bR^n$, 
  $$
  \exists C>0 ,\;\;
  \forall \phi \in C_c^\infty (\cC_0 ), \quad 
   \left|\int_{\bR^n} \kappa (v) \phi(v) dv \right|
  \leq C \max_{|\alpha | \leq p_\sigma}  \| \partial^\alpha \phi\|_{L^\infty(\cC_0)}.
  $$
Above,   $C$ may be chosen of the form 
$$
C = C_1 \|\sigma\|_{S^m(\widehat G),N, \cR},
$$
with $C_1>0$, $N\in\mathbb N_0$.
\end{enumerate}
\end{corollary}
\begin{proof}
   (1) The first part  follows readily from the kernel estimates in Theorem \ref{thm_kernelG} and the second point in Remark \ref{rem_thm_kernelG}. 
    Indeed, only the singularity in $0$ has to be considered when  $m-[\alpha]+[\beta]\geq -Q$. Observing that $S^m(\widehat G)\subset S^{m+\epsilon}(\widehat G)$ for all $\epsilon>0$, we deduce from Theorem \ref{thm_kernelG} (3):
    $$
    0<|y|\leq 1, \;\;|y|^{[\alpha]} \bV^\beta \kappa(y)|\leq C |y|^{-(Q+m+\epsilon - [\alpha]+[\beta]}\in L^p(G) \;\; \mbox{for}\;\;p(Q+m+\epsilon-[\alpha]+[\beta])>Q.
    $$
    Note that it is possible to find such an $\epsilon$ as soon as $p(Q+m-[\alpha]+[\beta])>Q$, which concludes the proof.

   (2) For the second part, the properties of  Sobolev spaces (see \cite{FRSob} or  Proposition \ref{prop_sob})
 \begin{align*}
    \|\sigma\|_{S^m (\widehat G), N,\cR}
    &=
\max_{[\alpha]\leq N}
\|\Delta_{x^\alpha} \sigma\|_{L^\infty_{0,-m +[\alpha]}(\widehat G),\cR}
=\max_{[\alpha]\leq N}
 \| (I+\widehat \cR)^{\frac{-m+[\alpha]}\nu} \sigma\| _{L^\infty(\widehat G)}
\end{align*}
Then since
$-m+[\alpha]\leq M_1
<0$ and by boundedness of $(I+\widehat{\cR})^{-a/\nu}$ for $a>0$, we have
\begin{align*}
    \| (I+\widehat \cR)^{\frac{-m+[\alpha]}\nu} \sigma\| _{L^\infty(\widehat G)}\lesssim  \| (I+\widehat \cR)^{\frac{M_1}{\nu}}\sigma\|_{L^\infty(\widehat G)}\lesssim
  \sum_{[\beta]\leq M_1} 
\|\widehat \bV^\beta\Delta_{x^\alpha} \sigma\|_{L^\infty(\widehat G)}.
\end{align*}
Finally it is easy to see that
\begin{align*}
    \|\widehat \bV^\beta\Delta_{x^\alpha} \sigma\|_{L^\infty(\widehat G)}\lesssim \max_{[\alpha]\leq N}\|\sigma\|_{\alpha, \beta, 1}
\end{align*}
whence the result.

(3) If $m<-Q$, then by Theorem \ref{thm_kernelG}, $\kappa$ is a continuous bounded function on $G$, so it is a distribution of order 0.
Assume $m\geq -Q$.
For any $\phi\in C_c^\infty(G)$, we have by the Sobolev embeddings (see Proposition \ref{prop_sob} \eqref{prop_sob_emb}) with $s>Q/2$
$$
\|\phi * \kappa \|_{L^\infty(G)}
\lesssim \|(\id+\tilde \cR)^{s/\nu}(\phi * \kappa) \|_{L^2(G)};
$$
above $\tilde \cR$ denotes the right-invariant  operator corresponding to a fixed Rockland operator $\cR$ of homogeneous degree $\nu$.
The properties of the convolution yield:
$$
(\id+\tilde \cR)^{s/\nu}(\phi * \kappa) 
=((\id+\tilde \cR)^{\frac s\nu}\phi )* \kappa 
=((\id+\tilde \cR)^{\frac {s+m} \nu}\phi )* ((\id+\cR)^{-\frac {m} \nu} \kappa) ,
$$
so by the Plancherel formula, 
$$
\|(\id+\tilde \cR)^{s/\nu}(\phi * \kappa) \|_{L^2(G)}
\leq\|(\id+\widehat \cR)^{-\frac {m} \nu}\sigma \|_{L^\infty(\Gh)}\|(\id+\tilde \cR)^{\frac {s+m} \nu}\phi \|_{L^2(G)}.
$$
The properties of the Sobolev spaces (see Proposition \ref{prop_sob}) imply 
$$
\|(\id+\tilde \cR)^{\frac {s+m} \nu}\phi \|_{L^2(G)}\lesssim  \sum_{[\alpha]\leq s'} \|\tilde X^\alpha \phi\|_{L^2(G)}
\lesssim \sum_{|\alpha|\leq p_\sigma} \|\partial^\alpha \phi\|_{L^\infty(G)}, 
$$
having chosen $s$ very close to $Q/2$
and $s'$ and $p_\sigma$ as in the statement,  with the last implicit constant depending on the support of $\phi$.
This shows Part (3) for $m\geq -Q$.
\end{proof}

\begin{remark}
\label{rem_cor_thm_kernelG}
The proof of Corollary \ref{cor_thm_kernelG} implies in particular that the constant $C$ and the constants describing the continuity of $\sigma\mapsto \|\sigma\|_{\alpha,\beta,p}$ depend on  the Lie structure only via $\| [\cdot,\cdot]\|_\bV$ after a choice of a basis $\bV$ adapted to the gradation and the choice of the Rockland operator $\cR=\cR_\bV$, beside $m,\alpha,\beta,p$ for the first part and $m,N,M_1$ for the second.
\end{remark}

The semi-norms $\|\cdot\|_{ \alpha,\beta,p}$ defined in Corollary \ref{cor_thm_kernelG}
may not generate the topology on $S^m(\widehat G)$, but they have the advantage of being homogeneous in the following sense.
If $\sigma = \cF_G \kappa$ then $\sigma(r\cdot \pi) = \cF_G \kappa^{(r)}$ with $\kappa^{(r)}(x) = r^{-Q}\kappa(r^{-1}x)$ and we have
$$
\|\sigma(r \, \cdot ) \|_{\alpha,\beta,p}
= r^{Q(\frac{1}{p}-1)-[\beta]+[\alpha]}\|\sigma\|_{ \alpha,\beta,p}.
$$
Moreover, the inequality in the second part of Corollary \ref{cor_thm_kernelG} together with judicious expansions are often used to recover the membership to $S^m(\widehat G)$ as well as other properties.

\subsubsection{Spectral multipliers}
\label{subsubsec_specmultG}
Another important example of classes of symbols are given by the multipliers in the symbols of a positive Rockland operator. 

\begin{notation}\label{notation_Gm}
Let $\cG^m(\bR)$ be the space of smooth functions $\phi:\bR \to \bC$ growing at rate $m\in \bR$ in the sense that   
$$
\forall k\in \bN_0,\qquad \exists C=C_{k,\phi},\qquad
\forall \lambda\in \bR, \qquad |\partial_\lambda^k \phi(\lambda)|\leq C (1+|\lambda|)^{m-k}.
$$    
\end{notation}

The set $\cG^m(\bR)$ is a Fr\'echet space when equipped with the semi-norms given by
$$
\|\phi\|_{\cG^m, N} := \max_{k=0,\ldots, N} \sup_{\lambda\in \bR} 
(1+|\lambda|)^{-m+k}|\partial_\lambda^k \phi(\lambda)|, 
\qquad N\in \bN_0. 
$$
Many spectral multipliers in positive Rockland operators are in our symbol classes
\cite[Proposition 5.3.4]{R+F_monograph}:
\begin{theorem}
\label{thm_phi(R)}
Let $\phi\in \cG^m (\bR)$ and let $\cR$ be a positive Rockland operator on $G$.
	Then $\phi(\widehat \cR)\in S^{m\nu}(\widehat G)$ where $\nu$ is the  homogeneous degree of $\cR$. 
	Moreover, the map 
	$$ \cG^m (\bR)\ni \phi\mapsto \phi(\widehat \cR)\in S^{m\nu}(\widehat G)$$
	is continuous. 
\end{theorem}

\begin{remark}
\label{rem_thm_phi(R)}
	Again, the proof of the theorem shows that the constants describing the continuity of the linear map $ \cG^m (\bR)\ni \phi\mapsto \phi(\widehat \cR)\in S^{m\nu}(\widehat G)$ 
	depend on  the Lie structure only via $\| [\cdot,\cdot]\|_\bV$ after a choice of a basis $\bV$ adapted to the gradation.
\end{remark}

\subsection{Smoothing symbols}
\label{subsec_smoothingG}

\begin{definition}
The class of {\it smoothing symbols}
$$
S^{-\infty}(\widehat G):= \cap_{m\in \bR} S^m(\widehat G).
$$
is equipped  with the induced topology of projective limit.	
\end{definition}
The topology is well defined thanks to 
the continuous inclusions \eqref{eq_inclusionSmGh}.

Then, the results recalled above, or equivalently \cite[Lemme 5.4.13 and Theorem 5.4.9]{R+F_monograph}, imply the following:

\begin{proposition}
\label{prop_smoothing}
	 The group Fourier transform is an isomorphism between the topological vector spaces $\cS(G)$ and $S^{-\infty}(\widehat G)$.
\end{proposition}

\begin{remark}
	\label{rem_prop_smoothing}
Again, the proof of Proposition \ref{prop_smoothing} shows that the constants describing the continuity of the linear maps $\cF_G :\cS(G)\to S^{-\infty}(\widehat G)$ 
depend on  the Lie structure only via $\| [\cdot,\cdot]\|_\bV$ after a choice of a basis $\bV$ adapted to the gradation.
\end{remark}

\smallskip

In many circumstances,
it will be useful to handle the convolution kernels of symbols as if they were smooth functions and not distributions. 
The following property will allow us to do so. 
\smallskip

We fix a positive Rockland operator $\cR$ and a function  $\chi\in \cD(\bR)$ valued in $[0,1]$ and with $\chi=1$ on $[0,1]$.
We define the  symbols 
\begin{equation}\label{def:sigmaell}
\sigma_\ell := \sigma\, \chi (\ell^{-1} \widehat \cR), \quad \ell\in \bN.
\end{equation}
By Theorem \ref{thm_phi(R)}, the symbols 
$\chi (\ell^{-1} \widehat \cR)$ are smoothing, 
and so is $\sigma_\ell$ by  the algebraic symbolic properties of Theorem \ref{thm_SmGh} (2). 

\begin{proposition}
\label{prop_densitysmoothing}
Let $\sigma\in S^m(\Gh)$ with $m\in \bR$.
The sequence of smoothing symbols $(\sigma_\ell)_{\ell\in \bN}\subset S^{-\infty}(\Gh)$ defined in~\eqref{def:sigmaell} satisfies the following properties:

	\begin{enumerate}
		\item For any $m_1>m$, we have the convergence 
$$
\lim_{\ell\to \infty} \sigma_\ell	= \sigma \quad\mbox{in}\ S^{m_1}(\Gh).
$$		
In other words, 
the space 
$S^{-\infty}(\Gh)$ equipped with the $S^{m_1}(\widehat G)$-topology is dense in $S^{m}(\widehat G)$.
\item 
For any semi-norm $\|\cdot\|_{S^m(\Gh),N}$ of $S^m(\Gh)$, there exist a constant $C>0$ and $N'\geq N$ such that 
$$
\forall \ell\in \bN,\qquad  \|\sigma_\ell-\sigma\|_{S^m(\Gh),N}
\leq C \|\sigma\|_{S^m(\Gh),N'}.
$$	
Consequently, 
$$
\forall \ell\in \bN,\qquad  \|\sigma_\ell\|_{S^m(\Gh),N}
\leq (1+C) \|\sigma\|_{S^m(\Gh),N'}.
$$

		\item 
		Denoting by $\kappa$ and $\kappa_\ell$ the convolution kernels of $\sigma$ and $\sigma_\ell$ respectively,
	we
 have the convergences 
$
\lim_{\ell\to \infty}  \kappa_\ell = \kappa $ in $\cS'(G)$ and in  $C^\infty (U)$
for any open subset
$U\subset G\setminus \{0\}$.
\item We have the following relation between $S^m(\Gh)$-semi-norms: 
\[
\liminf_{\ell \to\infty}  \|\sigma_\ell\|_{S^m(\Gh),N} \geq \|\sigma\|_{S^m(\Gh),N} .
\]

	\end{enumerate}
\end{proposition}

\begin{proof}[Proof of Proposition \ref{prop_densitysmoothing}]

Parts 1 and 2.  By the symbolic properties of symbols in Theorem \ref{thm_SmGh} (2), given a semi-norm $\|\cdot \|_{S^{m_1}(\Gh),N}$, 
 there exist $C>0$ and $N'',N'\geq N$ such that 
	$$
\|	\sigma - \sigma_\ell\|_{S^{m_1}(\Gh),N}
=\left\| \sigma \ (1-\chi)(\ell^{-1} \widehat \cR)
\right\|_{S^{m_1}(\Gh),N}\leq C \| \sigma \|_{S^m(\Gh),N'}\|(1-\chi)(\ell^{-1} \widehat \cR)
\|_{S^{m_1-m}(\Gh),N''}.
$$
If $m_1\geq m$, 
by Theorem \ref{thm_phi(R)}, we have for some $N_1\in \bN$
 $$
\|(1-\chi)(\ell^{-1} \widehat \cR)
\|_{S^{m_1-m}(\Gh),N''} \lesssim \|(1-\chi)(\ell^{-1} \cdot )\|_{\cG^{\frac{m_1-m}\nu},N_1},
$$
and we check 
$$
\forall \ell\in \bN,\qquad 
\|(1-\chi)(\ell^{-1} \cdot )\|_{\cG^{\frac{m_1-m}\nu},N_1}
\lesssim  \ell ^{\frac{m-m_1}\nu}
.
$$
This shows Part (1) and the first estimate in Part (2), the second being obtained by the triangle inequality.

Part 3. The kernel estimates in Theorem \ref{thm_kernelG} and the estimate of the $S^{m_1}(\Gh)$-semi-norm of $\sigma-\sigma_\ell$ for $m_1>m$ gives the convergence in $C^\infty (U)$ for any open subset $U$ of $G\setminus\{0\}$.
Recall that $\cF^{-1}S^{m_1}(\Gh) \subset \cK_{m_1,0}$ and that the inclusion $\cK_{a,b}\subset \cS'(G)$ is continuous \cite[Proposition 5.1.17(3)]{R+F_monograph}.
This and Part (1) imply the convergence $\kappa_\ell \to \kappa$ in $\cS'(G)$.

Part 4. To prove Part (4), we  observe that by functional analysis, we have
\begin{align*}
\|(\id+\widehat \cR)^{-\frac m\nu} \sigma\|_{L^\infty(\Gh)}
&\geq 
\|(\id+\widehat \cR)^{-\frac m\nu} \sigma \, \chi(\ell^{-1}\widehat \cR)\|_{L^\infty(\Gh)}\\
&\qquad \qquad \geq \sup_{\pi\in \Gh}\sup_{\substack{v\in 1_{[0,\ell]}(\pi(\cR))\\ |v|_{\cH_\pi}=1}}
|(\id +\pi(\cR))^{-\frac m\nu} \sigma (\pi)v|_{\cH_\pi} ,
\end{align*}
so,
\begin{equation}
    \label{eq_FAchiel}
    \lim_{\ell\to \infty}
\|(\id+\widehat \cR)^{-\frac m\nu} \sigma \, \chi(\ell^{-1}\widehat \cR)\|_{L^\infty(\Gh)}
= \|(\id+\widehat \cR)^{-\frac m\nu} \sigma \|_{L^\infty(\Gh)}.
\end{equation}

For $|\alpha|=1$, the Leibniz property implies
$$
\Delta_\alpha \sigma_\ell = (\Delta_\alpha \sigma) \chi(\ell^{-1}\widehat \cR) +  \sigma (\Delta_\alpha \chi(\ell^{-1}\widehat \cR)),
$$
so 
$$
   \|(\id+\widehat \cR)^{-\frac {m-[\alpha]}\nu} \Delta_\alpha\sigma_\ell\|_{L^\infty(\Gh)}
  \geq   E_1 -E_2,
$$
where 
\begin{align*}
    E_1&:=\|(\id+\widehat \cR)^{-\frac {m-[\alpha]}\nu} (\Delta_\alpha \sigma) \chi(\ell^{-1}\widehat \cR) \ \|_{L^\infty(\Gh)} \\
    E_2 &:= \|(\id+\widehat \cR)^{-\frac {m-[\alpha]}\nu}  \sigma \ (\Delta_\alpha\chi(\ell^{-1}\widehat \cR))\|_{L^\infty(\Gh)}\\
    &\lesssim \|\sigma\|_{S^m,N_1} \|\chi(\ell^{-1}\widehat \cR)\|_{S^{[\alpha]}(\Gh),N_2}
  \lesssim_{\sigma,\chi} \ell^{-\frac {[\alpha]}\nu},
\end{align*}
 for some integers $N_i$'s,
having used the symbolic properties for composition  and for spectral multipliers (see Theorems \ref{thm_SmGh} and \ref{thm_phi(R)}).
Applying  \eqref{eq_FAchiel} to $E_1$, we obtain, 
$$
\liminf_{\ell \to \infty} \|(\id+\widehat \cR)^{-\frac {m-[\alpha]}\nu} \Delta_\alpha\sigma_\ell\|_{L^\infty(\Gh)}
  \geq 
\|(\id+\widehat \cR)^{-\frac {m-[\alpha]}\nu} (\Delta_\alpha \sigma)\|_{L^\infty(\Gh)} , 
$$
    for $|\alpha|=1$. The same ideas and method yield inductively the case of any index $\alpha$, proving Part (4).
\end{proof}

\begin{remark}
	\label{rem_prop_densitysmoothing}
Again, the proof of Proposition \ref{prop_densitysmoothing} shows that the constants involved in the statement 
depend on  the Lie structure only via $\| [\cdot,\cdot]\|_\bV$ after a choice of a basis $\bV$ adapted to the gradation.
\end{remark}

\subsection{Homogeneous and polyhomogeneous symbol classes}

\subsubsection{Homogeneous classes of symbols}
\label{subsubsec_invdotSm}
Recall that the dilations act on $\widehat G$ via 
\begin{equation}
\label{def:dil_pi}
r\cdot \pi (x)= \pi(rx), 
\quad \pi\in \widehat G, \ x\in G, \ r>0.	
\end{equation}
\begin{definition}
\label{def_hominvsymbolG}
	An invariant symbol $\sigma$ on $\widehat G$ is $m$-homogeneous when $\sigma(r\cdot \pi) = r^m \sigma(\pi)$ for almost all $\pi\in \widehat G$ and $r>0$. 
\end{definition}

\begin{ex}
    Let $\cR$ be a positive Rockland operator. We denote by $\nu$ its homogeneous degree.
    For any $s\in \bR$, the properties of homogeneous Sobolev spaces (Section \ref{subsubsec_homL2sG}) imply that the field of  spectrally defined operators 
    $\widehat \cR^{m/\nu}=\{\pi(\cR)^{m/\nu}: \pi\in \Gh\}$
    is a well-defined invariant symbol, 
    and we check readily that it 
      is  $m$-homogeneous.
\end{ex}

We define the following subspace of invariant symbols
$$
\sigma\in \dot L^\infty_{a,b}(\widehat G)
\quad \text{if and only if}\quad \sigma_{a,b}:=\widehat \cR^{b/\nu}  \sigma \, \widehat \cR^{-a/\nu}\in L^\infty(\widehat G);
$$
by this, we mean that there exists a unique element $\sigma_{a,b}$ of $L^\infty(\widehat G)$ that is composable on the left and the right by 
$\widehat \cR^{-b/\nu}$ and $\widehat \cR^{a/\nu}$  respectively such that $\sigma=\widehat \cR^{-b/\nu} \sigma_{a,b} \cR^{a/\nu}$.
Equipped with the norm 
$$
\|\sigma\|_{L^\infty_{a,b}(\widehat G),\cR}
:= \|\widehat \cR^{b/\nu}  \sigma \, \widehat \cR^{-a/\nu}\|_{L^\infty(\widehat G)}, 
$$
the space $\dot L^\infty_{a,b}(\widehat G)$ becomes a Banach space.

Since the inclusion $\dot L^2_s (G) \subset \cS'(G)$ is continuous for any $s\in \bR$, by the Schwartz kernel theorem,
 any symbol in $L^\infty_{0,s}(\widehat G)$ has a convolution kernel; 
more generally,  if $M_0\in \bN_0$ is a multiple of the dilation's weights,
any symbol in $L^\infty_{M_0,s}(\widehat G)$ has a convolution kernel. We can therefore apply at least formally a difference operator to these symbols. This allows us to consider the following definition:

\begin{definition}
\label{def_reghominvsymbolG}
An $m$-homogeneous symbol $\sigma$ is said to be \emph{regular} when $\sigma \in \dot L^\infty_{0, -m}(\widehat G)$ and furthermore
$\Delta_{x^\alpha} \sigma \in \dot L^\infty_{0, -m+[\alpha]}(\widehat G)$ for  all $\alpha \in \bN_0^n$.
We denote by $\dot S^m(\widehat G)$ the space of $m$-homogeneous regular symbols.
\end{definition}

The space $\dot S^m(\widehat G)$ is Fr\'echet when equipped with the semi-norms given by
\begin{equation}
\label{def:semi_norm_hom_group}
\|\sigma\|_{\dot S^m(\widehat G),N,\cR}:=
\max_{[\alpha]\leq N}
\|\Delta_{x^\alpha} \sigma\|_{\dot L^\infty_{0,-m+[\alpha]}(\widehat G),\cR},
\end{equation}
having fixed a positive Rockland operator $\cR$.

In \cite{FF0}, we proved the following property. We shall say 
 that 
a function $\psi:\bR\to \bR$ satisfies $\psi\equiv 1$ on a neighbourhood of $+\infty$ when there exists $\Lambda>0$ such that $\psi\equiv 1$ on $(\Lambda,+\infty)$. 

\begin{proposition}
\label{prop_dotS0Gh}
	Let $\sigma$ be a $m$-homogeneous invariant symbol on $\widehat G$. 
	The following properties are equivalent:
	\begin{enumerate}
		\item For all $\alpha \in \bN_0^n$, $\Delta_{x^\alpha} \sigma \in \dot L^\infty_{0, -m+[\alpha]}(\widehat G)$.
		\item For all $\gamma\in \bR$ and  $\alpha \in \bN_0^n$, $\Delta_{x^\alpha} \sigma \in \dot L^\infty_{\gamma, -m+[\alpha]+\gamma}(\widehat G)$.
		\item For one (and then any) $\psi\in C^\infty(\bR)$ with 
		$\psi\equiv 0$ on a neighbourhood of 0 and $\psi\equiv 1$ on a neighbourhood of $+\infty$ and for one (and then any)	positive Rockland operator $\cR$, $\psi(\widehat \cR)\, \sigma \in S^m(\widehat G)$ and $ \sigma\, \psi(\widehat \cR)\in S^m(\widehat G)$.
	\end{enumerate}
In this case, the following map
$$
 \sigma \longmapsto  \Delta_{x^\alpha} \sigma, 
 \qquad \dot S^0(\widehat G)\longrightarrow \dot L^\infty_{\gamma, -m+[\alpha]+\gamma,\gamma}(\widehat G),
$$
is continuous for each $\gamma\in \bR$ and $\alpha\in \bN_0^n$.
And in the case $m=0$, the maps
 $$
\sigma \longmapsto  \psi(\widehat \cR) \sigma, 
\quad\mbox{and}\quad
\sigma \longmapsto   \sigma \psi(\widehat \cR), 
$$
are also continuous $\dot S^0(\widehat G)\rightarrow S^0(\widehat G)$ and injective. 
\end{proposition}

\begin{remark}
\begin{enumerate}
\item An improved presentation of the proof showing that the property in Proposition \ref{prop_dotS0Gh} (3) is independent of a choice of a positive Rockland operator $\cR$ and a function $\psi$ is obtained by considering the proof in Section \ref{subsubsec_hom+inhomsymb}. 
    \item  The constants describing the continuity properties in Proposition \ref{prop_dotS0Gh} depend on  the Lie structure only via $\| [\cdot,\cdot]\|_\bV$ after a choice of a basis $\bV$ adapted to the gradation.	  
\end{enumerate}
\end{remark}

In \cite{FF0},
we also showed that
the classes of regular homogeneous symbols enjoy the algebraic properties for composition and adjoint as described for the inhomogeneous case in Theorem~ \ref{thm_SmGh}. Furthermore, a remark regarding constants similar to Remark \ref{rem_thm_SmGh} is true.

\begin{ex}
\label{ex_XalphainSGhh}
If $\bV$ is an adapted basis of $\fg$, 
then for any $\alpha\in \bN_0^n$,  the symbol $$
\widehat \bV^\alpha =\{\pi(\bV)^\alpha , \pi\in \widehat G\}
$$
is in $\dot S^{[\alpha]}(\widehat G)$.
\end{ex}

\subsubsection{Polyhomogeneous expansion and symbol classes}\label{subsubsec_poly_exp}

\begin{definition}
\label{def_asympexph}
A symbol $\sigma \in S^m(\widehat G)$ admits the polyhomogeneous expansion 
$$
\sigma \sim_h \sum_{j=0} ^\infty \sigma_{m-j}
$$
when 
	each $\sigma_{m-j}$ is homogeneous and in $\dot S^m(\widehat G)$,  and $\sigma$ admits the inhomogeneous expansion (in the sense of Definition \ref{def_asympexp})
	$$
 \sigma \sim  \sum_{j=0}^\infty  \sigma_{m-j}\psi(\widehat \cR),
 $$
 for some positive Rockland operator $\cR$ and  $\psi\in C^\infty(\bR)$ with 
		$\psi\equiv 0$ on a neighbourhood of 0 and $\psi\equiv 1$ on a neighbourhood of $+\infty$.
\end{definition}

By the symbolic calculus on groups, this definition is independent of the choice of the Rockland operator $\cR$  (see \cite[Proposition~4.6 ]{FF0}
or Proposition \ref{prop_dotS0Gh} above)
and we have an equivalent definition considering $\sum_{j=0}^\infty \psi(\widehat \cR) \sigma_{m-j}$.

Unlike
the expansion of a general  symbol of order $m$ (in the sense of Definition~\ref{def_asympexp}) that  may be perturbed by any term of lower order, 
if $\sigma\in S^m(\Gh)$ admits a polyhomogeneous expansion $
 \sigma \sim  \sum_{j=0}^\infty  \sigma_{m-j}\psi(\widehat \cR),
 $ then the homogenenous symbols $\sigma_{m-j}$, $j=0,1,2,\ldots$ are unique. 
 This follows readily from  the injectivity of the maps mentioned in Proposition \ref{prop_dotS0Gh}.

\begin{definition}
\label{def:SmphGh}
    The polyhomogeneous symbol classes of order $m\in \bR$ is the space $S^m_{ph}(\Gh)$ of symbols $\sigma\in S^m(\Gh)$ admitting a polyhomogeneous expansion.
\end{definition}

The properties of the homogeneous and inhomogeneous symbol classes imply that 
the product of a symbol in $S^{m_1}_{ph}(\Gh)$
with a symbol in $S^{m_2}_{ph}(\Gh)$ is in $S^{m_1+m_2}_{ph}(\Gh)$.
Similarly, the adjoint of a symbol in $S^{m}_{ph}(\Gh)$ is in $S^{m}_{ph}(\Gh)$.

\begin{ex}
\label{ex_I+Rmnuph}
For any  positive Rockland operator  $\cR$, 
we have 
$(\id+\widehat \cR)^{m/\nu} \in S^m_{ph}(\Gh)$
where $\nu$ is the homogeneous degree of $\cR$.
Indeed, if 
$\psi\in C^\infty(\bR)$ satisfies 
		$\psi\equiv 0$ on a neighbourhood of 0 and $\psi\equiv 1$ on a neighbourhood of $+\infty$, 
		by spectral calculus
		\begin{align*}
		(\id+\widehat \cR)^{m/\nu} \psi(\widehat \cR) 
		&= 
		  (\id+\widehat \cR^{-1})^{m/\nu}\psi(\widehat \cR)  \widehat \cR ^{m/\nu}\\
		&\sim   \widehat \cR ^{m/\nu} \psi(\widehat \cR) + \frac m \nu \widehat \cR^{-1 + m/\nu}\psi(\widehat \cR)  + \ldots ,
		\end{align*}
		having used the Taylor series of $(1+x)^{m/\nu}$.
\end{ex}

\subsection{Difference operators and Leibniz property}
An important property of the difference operator $\Delta_{x^\alpha}$ introduced in Definition~\ref{def:diff_op} is that they satisfy a form of Leibniz property which is a direct consequence of properties regarding monomials:
\begin{lemma}
\label{lem_LeibnizG}
Let $G$ be a graded Lie group and $\bV$ an adapted basis.
\begin{enumerate}
    \item 
For any $\alpha\in \bN_0^n$, we have the following relations between $\bV$-monomials:
   $$
(xy)^\alpha = \sum_{[\alpha_1]+[\alpha_2]=[\alpha]} 
c^{(\alpha)}_{\alpha_1,\alpha_2} x^{\alpha_1}y^{\alpha_2}
$$ 
for unique coefficients $c^{(\alpha)}_{\alpha_1,\alpha_2}\in \bR$ 
satisfying 
$$
c^{(\alpha)}_{\alpha_1,0} =\left\{ \begin{array}{cc}
1&\mbox{if} \ \alpha_1 =\alpha,\\
0&\mbox{otherwise}
\end{array}\right.
\qquad\mbox{and}\qquad
c^{(\alpha)}_{0,\alpha_2} =\left\{ \begin{array}{cc}
1&\mbox{if} \ \alpha_2 =\alpha,\\
0&\mbox{otherwise}
\end{array}\right.
$$
Moreover, the coefficients $c^{(\alpha)}_{\alpha_1,\alpha_2}$ are universal polynomial expressions in the structural constants of $G$ for $\bV$, i.e. in the constant $c_{i,j,k}$ defined via $[V_i,V_k]=\sum_k c_{i,j,k} V_k$.
\item For any $\sigma_1,\sigma_2\in S^{-\infty}(\widehat G)$, 
$$
\Delta_{x^\alpha} (\sigma_1\sigma_2) =
\sum_{[\alpha_1]+[\alpha_2]=[\alpha]} 
c^{(\alpha)}_{\alpha_1,\alpha_2} \
\Delta_{x^{\alpha_1}} \sigma_1\
\Delta_{x^{\alpha_2}} \sigma_2. 
$$
\end{enumerate}
\end{lemma}
\begin{proof}
   Part (1) follows from  \cite[Proposition 5.2.3 (4)]{R+F_monograph}. 
   For Part (2), the convolution kernel of $\sigma_1\sigma_2$ is $\kappa_2 * \kappa_1$ where $\kappa_i$ is the convolution kernel of $\sigma_i$, $i=1,2$.
   We have
\begin{align*}
   x^\alpha \kappa_2 * \kappa_1 (x) 
&=\int_G (y (y^{-1} x) )^\alpha  \kappa_2 (y) \kappa_1(y^{-1}x)  dy
\\&= \sum_{[\alpha_1]+[\alpha_2]=[\alpha]} 
c^{(\alpha)}_{\alpha,\alpha_2} 
\int_G y^{\alpha_1}   \kappa_2 (y)\ (y^{-1} x) ^{\alpha_2}\kappa_1(y^{-1}x)  dy 
\end{align*}
The conclusion follows by taking the group Fourier transform.
\end{proof}

From the definition of the symbol classes, 
it follows that $\Delta_{x^\alpha}$ is a continuous operator $S^m(\widehat G)\to S^{m-[\alpha]}(\widehat G)$ for any $m\in \bR\cup\{-\infty\}$ and any $\alpha \in \bN_0^n$. 
But we can also show the continuity of any difference operator as defined in  Definition~\ref{def:diff_op}:
\begin{proposition}
\label{prop_DeltaqG}
	Let $q\in C_c^\infty(G)$. 
	Then $\Delta_q$ is a continuous operator $S^m(\widehat G)\to S^{m}(\widehat G)$ for any $m\in \bR\cup\{-\infty\}$.
 Moreover, 
 writing the Taylor series of $q$ at 0 as $\bT_0 q (x)\sim \sum_\alpha c_\alpha x^\alpha$, 
the map 
	$$
\sigma \longmapsto	\Delta_q \sigma  - \sum_{[\alpha]\leq N} c_\alpha \Delta_{x^\alpha} \sigma
$$
is continuous $S^m(\widehat G)\to S^{m-(N+1)}(\widehat G)$  for any $N\in \bN_0$.
\end{proposition}

\begin{proof}[Proof of Proposition \ref{prop_DeltaqG}]
By Corollary \ref{cor_thm_kernelG}, 
it suffices to prove that for any $m\in \bR$, for any multi-indices $\alpha_0,\beta_0\in \bN_0^n$, 
 there exists $N_0\in \bN_0$ 
 such that for any $N\in \bN_0$ with $N\geq N_0$,    we have
$$
\forall \sigma\in S^{-\infty}(\widehat G),
\qquad
\|\Delta_q \sigma  - \sum_{[\alpha]\leq N} c_\alpha \Delta_{x^\alpha} \sigma\|_{\alpha_0,\beta_0,1}
\leq C_{q,N,\alpha_0,\beta_0}
\|\sigma\|_{S^m(\widehat G), N',\cR_\bV},
$$
for some (large) integer $N'\in \bN_0$ depending on $N$ and other constants but not $\sigma$.
Moreover, by modifying $\sigma$ and $m$, it suffices to show the case $\alpha_0=\beta_0=0$.
Hence, we are led to analyse 
$$
\|\Delta_q \sigma  - \sum_{[\alpha]\leq N} c_\alpha \Delta_{x^\alpha} \sigma\|_{0,0,1} = \|q\kappa - \sum_{[\alpha]\leq N}  c_\alpha x^\alpha \kappa \|_{L^1(G)}.
$$

We fix a function $\chi\in C_c^\infty(G)$  such that $\chi=1$ on a neighbourhood of $\supp \, q$ and  
of 0.
Let $\kappa\in \cS(G)$.
We may write
$$
q\kappa - \sum_{[\alpha]\leq N}  c_\alpha x^\alpha \kappa  
= \rho_1 +\rho_2,
\quad\mbox{where}\quad	
\rho_1 := \kappa \chi (q- \bP_{G,q,0,N}), 
  \quad
	\rho_2 := (\chi-1) \kappa\bP_{G,q,0,N}   .
$$
As $\kappa$ is Schwartz away from 0 (see Theorem \ref{thm_kernelG}), the $L^1$-norm of  $\rho_2$ is finite.
When choosing $N_0$ large enough, this is also the case for $\rho_1$
from  the kernel estimates near 0 
(see Theorem~\ref{thm_kernelG}) 
together with  Taylor's estimate (see Theorem \ref{thm_MV+TaylorG}).
Moreover, quantitatively,  each $\|\rho_i\|_{L^1(G)}$, $i=1,2$, is bounded up to a constant by an $S^m(\widehat G)$-semi-norm in $\sigma$.
The conclusion follows. 
\end{proof}

\begin{remark}
	\label{rem_prop_DeltaqG}
	Again, the proof of Proposition \ref{prop_DeltaqG} shows that the constants describing the continuity of the linear map above depend on  the Lie structure only via $\| [\cdot,\cdot]\|_\bV$ after a choice of a basis $\bV$ adapted to the gradation.
	\end{remark}

A direct consequence  of Proposition \ref{prop_DeltaqG} is the following asymptotic expansion: 
$$
\Delta_q  \sigma  \sim \sum_{[\alpha]\leq N} c_\alpha \Delta_{x^\alpha} \sigma.
$$

\subsection{Transformation of symbols by morphisms}

 By a {\it morphism    $\theta$ of the graded nilpotent Lie group}, 
 we mean a  morphism of the nilpotent Lie group whose associated Lie algebra morphism (for which we keep the same notation) respects the gradation of the Lie algebra in the sense that it maps each $\fg_j$ on itself.
Using the notation of Lemma \ref{lem_compwmorph}, we have:

\begin{proposition}
\label{prop_compwmorph}
    Let $\theta$ be an automorphism of the graded nilpotent Lie group $G$. 
If $\sigma\in S^m(\Gh)$, then  $\theta_* \sigma \in S^m(\Gh)$
while if $\sigma\in \dot S^m(\Gh)$, then  $\theta_* \sigma \in \dot S^m(\Gh)$
 Moreover, the map 
$\sigma\mapsto \theta_* \sigma $  is an automorphism  of the topological vector space  $S^m(\Gh)$
and also of the topological vector space  $\dot S^m(\Gh)$.
\end{proposition}

As a consequence of the continuity of $\sigma\mapsto \theta_* \sigma $ ,  once we have chosen a Rockland operator~$\mathcal R$ and $N\in\bN$, there exists a constant $C>0$ such that for all $\sigma \in S^m(\Gh)$,
\[
\| \theta_*\sigma\|_{S^m(\widehat G), N, \mathcal R} \leq C \| \sigma\|_{S^m(\widehat G), N, \mathcal R} .
\]
The proof below will show that 
the constant $C$ depends on various quantities of the group. 

\begin{remark}\label{rem_constant_theta}
Assume as in Remark~\ref{rem_constant_RS} that  the gradation of the vector space  underlying~$\fg$ as well as an adapted  basis $\bV$ are fixed. 
Then, the constant $C$ may be chosen as
$$
C= \widetilde C (\max (s, \|[\cdot ,\cdot ]_\fg\|_{\bV}, C_\bV(\cR), \| \theta\|_{\mathcal L(\mathfrak g)}))
$$
where  $\widetilde C:[0,\infty)\to (0,\infty)$ is an increasing function  of the structural constant $\|[,]\|_{\bV}$, of the constant $C_\bV(\cR)$ introduced in Remark~\ref{rem_constant_RS}, and of the norm of $\theta$ as a linear map of $\mathfrak g$.
\end{remark}

\begin{proof}
We keep the same notation $\theta$ for the automorphism of the graded nilpotent Lie group~$G$, and the corresponding automorphisms of the Lie algebra $\fg$ and of its enveloping algebra $\sU(\fg)$. 
We observe that if $\cR$ is a Rockland operator viewed as an element of $\sU(\fg)$,
then $\cR_\theta:=\theta(\cR)$ is also a Rockland operator since $\pi(\cR_\theta) = \pi_1(\cR)$ with $\pi_1=\pi\circ \theta$ for any $\pi\in \Gh$.
For instance, if $\cR$ is as in Example \ref{ex_RG},
$$	
	\cR = \sum_j (-1)^{\frac {M_0}{\upsilon_j}}  V_{j}^{2\frac{M_0} {\upsilon_j}}, 
 \quad\mbox{then}\quad
\cR_\theta = \sum_j  (-1)^{\frac {M_0}{\upsilon_j}} (\theta( V_{j}))^{2\frac{M_0} {\upsilon_j}}.
$$
If in addition of being Rockland, $\cR$ is also positive, then $\pi(\cR_\theta) = \pi_1(\cR)$ is positive for any $\pi\in \Gh$, so $\cR_\theta$ is also positive.
We check readily 
$$
\|\theta_* \sigma\|_{L_{a,b}^\infty(\Gh),\cR}
= \|\sigma\|_{L_{a,b}^\infty(\Gh),\cR_\theta}
\quad\mbox{and}\quad
\Delta_q (\theta_*\sigma)
= \theta_*( \Delta_{q\circ \theta^{-1}}\sigma ).
$$
As 
$x_j \circ \theta^{-1}$ is a linear combination of $x_k$'s of the same degree as $x_j$, we have
$$
\|\Delta^\alpha \theta_* \sigma \|_{L_{a,b}^\infty(\Gh),\cR}
\asymp
\sum_{[\beta] =[\alpha]}
\|\Delta^\alpha \sigma \|_{L_{a,b}^\infty(\Gh),\cR_\theta}.
$$
We deduce the relation about the semi-norms in the inhomogeneous case. 
The homogeneous case is similar. 
\end{proof}

%%%%%%%%%%%%%%%%%%%%%%%%%%%%%%%%%%%%%%%%%%%%%%%%%%%%%%%%%

\section{Filtered manifolds} \label{sec:manifold}

In this section, we recall properties of filtered manifolds, and discuss their   frames, geometric exponential maps,  differential operators, and group exponentiation.

\subsection{Filtered manifolds and their osculating bundles}\label{sec:def_filt_manifold}

\begin{definition}
\label{def_filteredM}
	A {\it filtered manifold} is  a smooth  manifold $M$ whose tangent bundles admits a nested sequence of sub-bundles 
$\{0\}=H^0\subseteq H^1\subseteq\ldots\subseteq H^r=TM$
 that respects the Lie bracket of vector fields:
$$
[\Gamma(H^i), \Gamma(H^j)]\subseteq \Gamma(H^{i+j}), \quad 0\leq i,j \leq r,
$$
with the convention $H^\ell=H^r$ whenever $\ell\geq r$.
\end{definition}

By convention, in this paper, a smooth manifold is also assumed to be
second countable, paracompact Hausdorff.

\begin{ex}\label{ex:subrie}
A fundamental example consists in   a manifold $M$ 
equipped with a distribution $\cD$ generating $TM$ in the sense that any vector field may be obtained by iterated commutator brackets of elements of $\Gamma (H_1)$ and that  the filtration is obtained as follows. We set $H_1:=\cD$, and for each $j>1$,  $H_j$ the subspace of $TM$ such that 
 $\Gamma (H_j)$ is the $C^\infty(M)$-module generated by  
$[V_1,[V_2,\ldots,[V_\ell, \ldots]\ldots]] $, with $V_1,\ldots, V_\ell\in \Gamma (H_1)$, $\ell=2,\ldots,j$.
When the rank $\dim H_j$ of the vector bundle $H_j$ is constant for each $j$, 
$M$ is a filtered manifold. 
When $H_1$ generates $TM$,    $r=2$ and $\dim H_2=1$, the manifold is {\it contact.} 
When  $H_1$ generates $TM$ and  is equipped with a metric $g$, $M$ is an  {\it equiregular subRiemannian manifold
  $(M, H_1, g)$} (references about subRiemannian manifold include \cite{ABB} for example). 
\end{ex}

From now on, we fix a filtered manifold $M$ as in Definition \ref{def_filteredM}. We set 
$$
n=\dim M.
$$
For each $x\in M$, the vector space
$$\gr(T_x M) := \oplus_{i\geq 1} \left(H^i_x / H^{i-1}_x\right)
$$
is naturally 
graded and the gradation of the underlying vector space is independent of $x\in M$. 
Its {\it weights} (see Definition~\ref{def_dilation_gradedV})  are denoted by 
$$
0<\upsilon_1\leq \ldots \leq \upsilon_n,
$$ 
 counted with multiplicity, and, when counted without multiplicity,
$$
0<\upsilon_1 = w_1<\ldots < w_s = \upsilon_n,
$$
where we denote by $s$ the {\it step of the gradation} ($s\leq r$).
These weights, counted with multiplicity or not, are constant functions on $ M$.
We denote the dimension of each of these subspaces  in the decomposition $
\gr(T_x M) = \oplus_{i=1}^s H^{w_i} / H^{w_i-1},
$
by 
$$
d_i := \dim H^{w_i} / H^{w_i-1}, \qquad i=1,\ldots, s;
$$
these are independent of $x\in M$.

For each $x\in M$, the direct sum of vector spaces
$\gr(T_x M) = \oplus_i \left(H^i_x / H^{i-1}_x\right)$
is naturally 
equipped with a {\it Lie bracket}
that we denote $[\cdot,\cdot ]_{\fg_x M}$. Indeed, take  $U, V\in \gr(T_xM)$ and consider  $X\in \Gamma(H^{w_i})$ such that $U=X(x) \mod H_x^{w_i}$ and $Y\in \Gamma(H^{w_j})$ such that $V= Y(x) \mod H^{w_j}_x$. Then 
$$[U, V]_{\fg_xM}:=[X, Y]_x \mod H^{w_i+w_j}_x$$
is independent of the choice of representatives $X$ and $Y$ of $U, V\in \gr(T_xM)M$.

When $\gr(T_x M)$ is equipped with this Lie bracket, 
we denote the resulting {\it Lie algebra} 
$$
\fg_x M = \left (\gr(T_x M), [\cdot,\cdot ]_{\fg_x M}\right).
$$
 It is naturally graded, and the gradation of the underlying vector space is independent of $x\in M$.

\smallskip

For each $x\in M$, 
we denote by $G_x M$ the corresponding  {\it nilpotent Lie group}, by\ $\sU(\fg_x M)$ the corresponding {\it universal enveloping Lie algebra} to  $\fg_x M$. 
The unions
$$
G M :=\cup_{x\in M} G_x M
  \qquad\mbox{and}\qquad 
\fg M :=\cup_{x\in M}\fg_x M
$$
are naturally equipped with a structure of smooth vector bundles; 
they are called  the {\it bundles of osculating groups and Lie algebras} over $M$.
Moreover,
the \textit{group bundle exponential map
$$
\Exp: \fg M\to GM, \qquad 
\Exp_x :=\Exp_{G_x M} : \fg_x M \to G_x M,
$$
is smooth and bijective.
We also denote by
$$
\sU(\fg M) := \cup_{x\in M} \sU(\fg_x M)
$$
the bundle of universal enveloping Lie algebras over $M$, 
and by
$$
\sU_{N}(\fg M) := \cup_{x\in M} \sU_{N}(\fg_x M),
\qquad \mbox{and}  \qquad
\sU_{\leq N}(\fg M) := \cup_{x\in M} \sU_{\leq N}(\fg_x M),
$$
the linear sub-bundles of $\sU(\fg)$ of elements of homogeneous degree equal to $N$
and of  linear combination of elements of homogeneous degree at most equal to $N$ respectively.
They can be described in terms of adapted local frames, as we will see in the next section.
This yields a natural smooth structure for the  vector bundles  $\sU_{N}(\fg M)$ and 
$\sU_{\leq N}(\fg M)$ of finite rank. 
Moreover, the natural inclusions 
$$
\sU_{\leq N_1}(\fg M) \hookrightarrow  \sU_{\leq N_2}(\fg M),\qquad  N_1 \leq N_2,
$$
are smooth bundle morphisms. Hence, their direct limit 
$$
\sU(\fg M) = \varinjlim_N \sU_{\leq N}(\fg M)
$$
inherits a structure of smooth vector  bundle (of infinite rank).}

\subsection{Local frames adapted to the filtration}
\label{subsec_X}

\subsubsection{Generalities}
\label{subsubsec_generality_frame}
Recall that a frame $\mathbb{X}:=(X_1, \ldots, X_n)$ on an open subset $U$ of a manifold~$M$ yields a  local trivialization of the tangent space on $U$ set via the linear map
$$
\mathbb{R}^n\longrightarrow TM, \quad v=(v_1,\cdots,v_n)\longmapsto 
\sum_{i=1}^n v_iX_i.
$$
Local frames exist on open neighbourhoods of any points of the manifold $M$.
Although $M$ can always be covered by such open sets equipped with frames,  
$M$
may not admit  global frames.

If  $\bX$ and $\bY$ are two frames on the same open subset $U\subset M$, then we may write $$	X_i(x)=\sum _{1\leq k\leq n}T_{k,i}(x)Y_k(x),\qquad 1\leq i\leq n,$$
for uniquely defined smooth maps $x\mapsto T_{k,i}(x)$ on $U$, $1\leq k,i\leq n$.
These define the 
change of frame matrix $T=
T^{\bY,\bX} = (T_{i,j})_{1\leq i,j\leq n} $
as a smooth map $U\to \textrm{GL}(n,\bR)$.
Viewing $\bX$ and $\bY$   as column matrices,
we have
$\bX=T^{{\rm t}}\bY.$
If $\bZ$ is another  adapted frame on $U$, then $T^{\bZ, \bX}_x=T^{\bZ, \bY}_xT^{\bY, \bX}_x$.

\subsubsection{Frames adapted to the filtration}
\label{subsubsec_adaptedFrame}
When the manifold $M$ is filtered, 
a  frame $\mathbb{X}:=(X_1, \ldots, X_n)$ of an open subset $U\subset M$ is said to be \emph{adapted to the filtration} $H$ when for every $x\in U$, 
$X_{1,x},\ldots, X_{d_1,x}$ is a basis of $H^{w_1}_x $, 
$X_{1,x},\ldots, X_{d_1+d_2,x}$ is a basis of $H^{w_2}_x$, and so on. 
In this case, we  say that $\mathbb{X}=(X_1, \ldots, X_n)$  is a  \emph{local adapted  frame} or an \emph{adapted frame} on~$U$.
 Such local frames exist in neighbourhoods of any points of the manifold $M$. Thus, every smooth manifold $M$ has a covering by open sets equipped with   adapted frames. 
\smallskip

 We also associate to an adapted frame a basis of the osculating Lie algebras. More precisely, if $\bX$ is an adapted frame on $U$, 
we set for each $x\in U$, 
\begin{align*}
\langle X_i\rangle_x 
&:= X_i(x) \ \mbox{mod} \ H^{w_1-1}_x, \quad i=1,\ldots, d_1, 
\\
\langle X_i\rangle_x 
&:= X_i(x) \ \mbox{mod} \ H^{w_2-1}_x, \quad i=d_1+1,\ldots, d_1+d_2,
\quad \mbox{etc.}  
\end{align*}
Note that each $\langle X_i\rangle_x$ has weight $\upsilon_i$ in the graded Lie algbera $\fg_x M$, i.e. $\delta_r \langle X_i\rangle_x = r^{\upsilon_i} \langle X_i\rangle_x$.
In other words, the basis $\langle X_1\rangle_x, \ldots,\langle X_n\rangle_x $ is adapted to the gradation of $\mathfrak g_x M$, as defined in Definition~\ref{def_dilation_gradedV}. 
Moreover, each $\langle X_i \rangle \in \Gamma (\sU_{\upsilon_i}(\fg M|_U))$ is a smooth section over the bundle of osculating universal Lie algebras.
More generally, 
$$
\langle\bX\rangle^\alpha := \langle X_1\rangle^{\alpha_1} \ldots \langle X_n\rangle^{\alpha_n}, \quad\alpha=(\alpha_1,\ldots, \alpha_n)\in \bN_0^n,
$$
defines an element of $\Gamma(\sU_{[\alpha]}(\fg M|U))$
where 
$$
[\alpha] := \upsilon_1 \alpha_1 + \ldots + \upsilon_n \alpha_n.
$$
The element of $\Gamma(\sU_{0}(\fg M|U))$ corresponding to  $\alpha=0$ is
$\langle\bX\rangle^\alpha = 1 $.
Any element in $\Gamma(\sU(\fg M|U))$ may be written in a unique way as a finite linear combination over $C^\infty(U)$ of $\langle\bX\rangle^\alpha$. In other words, 
 $\langle\bX\rangle^\alpha$, $\alpha\in \bN_0^n$, is a free basis of the $C^\infty(U)$-module $\Gamma(\sU(\fg M|U))$.
Furthermore, $\Gamma(\sU_{N}(\fg M|U))$ is a submodule with basis $\langle\bX\rangle^\alpha$, $[\alpha]=N$.

\subsubsection{Transition between two adapted frames}
Let $BU(d_1, \dots, d_{s})$ be the group of invertible real block-upper-triangular matrices for which the $j$th diagonal block has size $d_j\times d_j$.

\begin{lemma}\label{old_lem_XY}
Let  $\bX$ and $\bY$ be two adapted frames on an open subset $U\subset M$.

\begin{enumerate}
    \item The transition matrix $T=T^{\bY,\bX}$ (see Section \ref{subsubsec_generality_frame}) is valued in $BU(d_1, \ldots d_{s})$.  
\item 
We  have
$$
\langle \bX\rangle_x=
{\rm diag}(T_x)^{\rm t}\langle \bY\rangle_x,
$$
	where ${\rm{diag}}(T_x)$ is the block-diagonal part of $T_x\in BU(d_1, \ldots d_{s})$.
\item
We may write $\langle \bX \rangle^\alpha $ as:
$$
\langle \bX \rangle^\alpha =\sum_{[\beta]= [\alpha] }\tilde T_{\alpha,\beta}\, \langle \bY \rangle^\beta,
\qquad \tilde T_{\alpha,\beta}\in C^\infty(U).
$$
Then the coefficients  $\tilde T_{\alpha,\beta}$ are determined via 
$$
({\rm diag}(T_x) v)^\alpha =
\sum_{[\beta]= [\alpha]} \tilde T_{\alpha,\beta}(x) v^\beta
\quad x\in U, \ v\in \bR^n,
$$
where $v^\alpha= v_1^{\alpha_1} \ldots v_n^{\alpha_n}$.
\end{enumerate}
\end{lemma}

\begin{proof}[Proof of Lemma~\ref{old_lem_XY}]
For each $i=1, \ldots, n$, 
let $j\in \{1, \ldots, s\}$ be the level of the gradation for which $d_{j-1}<i\leq d_j$, i.e. for which $w_j=\upsilon_i$, then 
	\begin{align}\label{ironman}
	X_i(x)=\sum _{1\leq k\leq d_{j}}T_{k,i}(x)Y_k(x), 
	\quad\mbox{or in other words}\ 
	X_i(x)=\sum _{\upsilon_k \leq \upsilon_i}T_{k,i}(x)Y_k(x).
	\end{align}
 We deduce the structure of the matrix $T$, and Part (1) follows.
 
	Quotienting~\eqref{ironman} by $H^{w_{j-1}}$  gives 
$$
	\langle X_i\rangle_x=\sum_{d_{j-1}<k\leq d_j }T_{k,i}(x)\langle Y_k\rangle_x
	= \sum_{\upsilon_k = \upsilon_i}T_{k,i}(x)\langle Y_k\rangle_x
	\quad \mbox{in} \ \mathfrak{g}M|_U,
	$$
	so 
$	\langle \bX\rangle_x={\rm diag}(T_x)^{\rm t}\langle \bY\rangle_x.$
This is Part (2). 
Part (3) is an easy consequence of Part (2).
\end{proof}

\subsection{Geometric exponentiation}
\label{subsec_expX}
 In this section, we recall the definition of geometric exponentiation for any local frame on any smooth manifold $M$ as well as some of its properties. This notion does not require the manifold $M$ to be filtered or the local frame to be adapted to the filtration, and the results hold in this general context.

\subsubsection{Definitions of $\exp$ and $\exp^\bX$}
The exponential map is defined as follows. 
If $X$ is a vector field on a smooth manifold $M$ and $x\in M$, then $\exp_x X$ is the time-one flow, when defined, of the vector field $X$, that is, the solution  at time 1 of the Ordinary Differential Equation (ODE) defined by $X$ and starting from $x$: 
$$
\dot \gamma(t) =X(\gamma(t)), \quad \gamma(0)=x,
\qquad \exp_x X:=\gamma(1).
$$
Given a local  frame  $(\mathbb X, U)$ on a manifold, 
we set
$$
\exp^\bX( x,v):=\exp^\bX_x(v):= \exp_x \sum_j v_j X_j, 
\qquad x\in U, \ v\in \bR^n,
$$
when defined. 
This  yields a {\it geometric exponential map}
$\exp^\bX$ associated with the frame $(\bX,U)$. It is defined on a subset of $ U\times\mathbb{R}^n$.
In fact, 
there exists an open neighbourhood   ${U}_{\mathbb X} \subset U \times \mathbb{R}^n$ of the zero section $ U\times \{0\}$  on which the map
\begin{equation}
\label{eq_mapxexpxv}
\exp^{\mathbb X}:\left\{
\begin{array}{rcl}
		{U}_{\mathbb X} & \longrightarrow &	M\times M\\
		 (x,v) &\longmapsto&  \left(x, \exp^\bX_x(v)\right)\end{array}\right.
\end{equation}
is a $C^\infty$-diffeomorphism onto its image in $U\times U$. 

Without
 restriction, we may assume that $U_\bX$ is such that for any $(x,v)\in U_\bX$, there exists a neighbourhood $I$ of $[0,1]$ such that  $(x,tv)\in U_\bX$ 
   for $t\in I$. 
We will always assume so for technical reasons that will become apparent below. 

\subsubsection{Some fundamental properties}
The theory of ODEs has the following two  implications. 
Firstly,  
for any $x,y\in M$ such that $\ln^\bX_x y$ and $\ln^\bX_y x$ exist, we have 
\begin{equation}
    \label{eq_ODEln}
    \ln^\bX_xy = - \ln^\bX_y x,
\end{equation}
and also that for any $x\in U$, we have:
$$
\partial_{v_j} \exp_x^\bX v\Big|_{v=0} = X_j, \quad j=1,\ldots,n.
$$
In particular $\jac_0 \exp_x^\bX =1 = \jac_x \ln_x^\bX$ when we take the volume form obtained by the pullback of $dv$ from $\bR^n$ to $U\subset M$.

Secondly,
the definition of the geometric exponential implies 
\begin{equation}
\label{eq_ddtfexpXxtv}
\frac{d}{dt} f(\exp_x^\bX tv)
= 
\bigl(\sum_i v_i  X_i f\bigr) (\exp_x^\bX tv),
\qquad 
f\in C^\infty(U), \  (x,v)\in U_\bX, 
\end{equation}
for any $t$ in a neighbourhood of $[0,1]$.

\subsubsection{Some Taylor estimates}
Using inductively the derivation in \eqref{eq_ddtfexpXxtv} will allow us to obtain the following (Euclidean) Taylor estimates: 

\begin{lemma}
\label{lem_Taylorexpxu}
Let $\bX$ be a frame on an open set $U\subset M$.
Let $f\in C^\infty (U)$ and  $N\in \bN$. For any  $(x,v)\in U_\bX$, we have:
$$
\Bigl|f(\exp^\bX_x v)  -\sum_{k\leq N} \frac 1{k!} \bigl(\sum_i v_i  X_i\bigr)^k f(x) \Bigr| \leq 
\frac 1{(N+1)!} \sup_{t\in [0,1] } \Bigl| 
\Big(\big(\sum_i v_i  X_i\big  )^{N+1} f\Big) 
(\exp^\bX_{x} t v)\Bigr|.
$$
\end{lemma}

\begin{proof}[Proof of Lemma \ref{lem_Taylorexpxu}]
Set $F(t) : = f(\exp_x^\bX tv)$ for $t$ in a neighbourhood of $[0,1]$.
We have
$$
F^{(k)}(t) 
= \Bigl(\bigl(\sum_i  v_i  X_i\bigr)^k f\Bigr) (\exp_x^\bX tv), 
$$
for $k=1$ by \eqref{eq_ddtfexpXxtv}, and inductively for any $k\in \bN$. 
We conclude with 
Taylor's estimate for $F$ at order $N$ from $t=0$ to $t=1$.
\end{proof}

Using again \eqref{eq_ddtfexpXxtv}, we can also obtain   Taylor expansions related to  changes of frames:
\begin{lemma}
\label{lem_uYX}
Let  $\bX$ and $\bY$ be two  frames on an open subset $U\subset M$.
Let $U_{\bX,\bY}$ be an open neighbourhood of $U\times \{0\}$ in $U_{\bX}\cap U_{\bY}\subset U\times \bR^n$ 
such that for any $(x,v)\in U_{\bX,\bY}$, the points $(x,tv)$   lie in $U_{\bX,\bY}$ for $t$ in some neighbourhood of $[0,1]$. 
Consider 
 the function $u := u^{\bY,\bX}$ defined via
$$
(x,u(x,v)) = \ln^\bY (\exp^\bX (x,v)), 
\qquad 
(x,v)\in U_{\bX,\bY}.
$$
Then the map $u = u^{\bY,\bX}:U_{\bX,\bY}\to \bR^n$ is smooth and given via
\begin{equation}\label{def:uxvk}
 u(x,v):= u_x(v) = \ln^\bY_x\circ \exp^\bX_x(v), 
 \qquad (x,v) \in U_{\bX,\bY}.
\end{equation}
It satisfies 
$$
u_x(0)=0
\quad\mbox{and}\quad 
D_0u_x(v)=\frac{d}{dt}\bigg|_{t=0}u_x(tv)=T_xv, 
$$
where $T$ is their change of frame matrix  as in Section \ref{subsubsec_generality_frame}.
Moreover, the Taylor series at $v\sim 0$ of $u_x(v)$ is given by 
$$
 u_x(v)\sim  T_xv + \sum_{k\geq 2} \frac{1}{k!}\Big (\sum_j v_j X_j\Big )^{k-1}T_xv.
  $$
\end{lemma}
\begin{proof}[Proof of Lemma \ref{lem_uYX}]
We check readily that the map $u$ is smooth, 
given by \eqref{def:uxvk} and that it satisfies $u_x(0)=0$. 
Below, we will show that for any $(x,v)\in U_{\bX,\bY}$ and $t$ in a neighbourhood of $[0,1]$, we have:
\begin{equation}
	\label{eq_ddtkuxtv}
		\frac{d^k}{dt^k}u_x(tv)
    =
    \left(\sum_{j=1}^n v_j X_{j, z}\right)^{k-1}T_zv \bigg|_{z=\exp^{\bX}_x(tv)} \qquad  k=1,2,3 \ldots
\end{equation}
From this, we  can obtain the Taylor series of $u_x(tv)$ at $t\sim 0$, and deduce the Taylor series of $u_x(v)$ at $v\sim 0$; hence, showing \eqref{eq_ddtkuxtv} will  conclude the proof. 

On the one hand, we have the equality in \eqref{eq_ddtfexpXxtv}	
	and on the other, since $\exp_x^\bX tv = \exp_x^\bY u(x,tv)$, the left-hand side in \eqref{eq_ddtfexpXxtv} is also equal to 
	$$
	\frac{d}{dt} f(\exp_x^\bY u(x,tv)) =
	\left(\sum_k \left(\frac{d}{dt} u(x,tv)\right)_k  Y_k \right) f(\exp_x^\bY u(x,tv)).
	$$
	We have  $\sum_j v_j X_j =\sum_k (Tv)_k Y_k$ in \eqref{eq_ddtfexpXxtv}, and 
	these equalities are valid for any $f\in  C^\infty(U)$. It allows us to identify the first $t$-derivative of $u_x(tv)$ and we obtain  \eqref{eq_ddtkuxtv} for $k=1$.
We then obtain \eqref{eq_ddtkuxtv} 
inductively on $k=1,2,\ldots$
\end{proof}
 
\subsection{Homogeneous order and  principal part of a differential operator}
\label{subsec:diff_op}
A differential
operator on a manifold $M$ is a uniformly
 finite linear combination over $C^\infty(M)$ of products of vector fields and the identity.
If $\bX = (X_1,\ldots,X_n)$ is a frame over an open subset $U$ of a manifold~$M$, then any differential operator $P$ on $U$ may be written as a  non-commutative polynomial in the $X_j$'s, 
that is, as a linear combination over $C^\infty (U)$ of $X_{i_1} X_{i_2}\ldots X_{i_\ell}$, $1\leq i_1,\ldots,i_\ell \leq n$. 
This writing is not unique.  
We say that $P$ is of degree at most $d$, if $P$ coincides on $U$ with a linear combination of  $X_{i_1} X_{i_2}\ldots X_{i_\ell}$
with $1\leq i_1,\ldots,i_\ell \leq n$ and $\ell\leq d$.
The smallest of these $d$ over all the local frames $(\bX,U)$ of $M$ defines   the degree of $P$.

\smallskip 

We now go back to the case of a filtered manifold $M$ and an adapted frame $(\bX,U)$. 
    We can express the differential operators 
    in terms of the ordered products
$$
\bX^\alpha := X_1^{\alpha_1} \ldots X_n^{\alpha_n},\qquad
\alpha=(\alpha_1,\ldots,\alpha_n)\in \bN_0^n,
$$
with the convention $\bX^0 =\id$ for $\alpha=0$ as in the preceding section. Moreover, this writing will have the advantage of being unique:
    
\begin{lemma}
\label{lem_PulcXalpha}
Let $\bX$ be an adapted frame on the open subset $U\subset M$.
Then any   differential operator on $M$
 may be uniquely  written on $U$ as 
  a finite linear combination over $C^\infty (U)$ of $\bX^\alpha$, $\alpha\in \bN_0^n$, on $U$.
\end{lemma}

\begin{proof}[Proof of Lemma \ref{lem_PulcXalpha}]
We observe that 
given $i_1,\ldots,i_\ell\in \{1,\ldots,n\}$, 
we can define the associated multi-index
$\alpha =(\alpha_1,\ldots,\alpha_n)$ 
with $\alpha_k = |\{j : i_j = k\}|$
and 
we have
$$
X_{i_1} X_{i_2}\ldots X_{i_\ell}
= 
\bX^\alpha + 
\sum_{|\beta|<|\alpha|}
c_\beta \bX^\beta, \qquad c_\beta\in C^\infty(U),
$$    
This is trivially true for the length $\ell=1$.
For the length $\ell=2$, 
this is a consequence of the frame $\bX$ being adapted to the filtration since 
    $X_i X_j = X_j X_i + [X_i,X_j]$
    and $[X_i,X_j] \in \Gamma( H^{\upsilon_i +\upsilon_j}|_U)$
    may be written as a linear combination of $X_k$, $k\leq   \dim H^{\upsilon_i +\upsilon_j}$.
    We then obtain the property inductively over the length $\ell$. 
This implies that  any differential operator 
may be written as 
  a finite linear combination over $C^\infty (U)$ of $\bX^\alpha$, $\alpha\in \bN_0^n$.
  
 The uniqueness of the writing  is equivalent to  the property that if such a finite linear combination coincides with the zero operator, then all the coefficients are equal to zero. For this, it suffices to construct for each $x_0\in U$ and each $\alpha\in \bN_0^n$ a smooth function $f_{\alpha,x_0}$ in a neighbourhood of $x_0$ satisfying
  $$
  \bX^\beta f_{\alpha,x_0} (x_0)=
  \left\{\begin{array}{ll}
  0 & \mbox{if} \ [\beta]\leq [\alpha] \  \mbox{with}\ \beta\neq \alpha,\\
  1 & \mbox{if} \ \beta = \alpha.
  \end{array}\right.
  $$
  The natural choices are 
 $f_{0,x_0}=1$ for $\alpha=0$ and 
 $$
f_{j_0,x_0} := \left[\ln_{x_0}^\bX x\right]_{j_0}
\quad\mbox{for}\ |\alpha|=1.
$$
We check inductively on $[\alpha]$ that the functions
$$
f_{\alpha,x_0}
:= f_{1,x_0}^{\alpha_1}\ldots f_{n,x_0}^{\alpha_n},
\quad \alpha=(\alpha_1,\ldots,\alpha_n)\in \bN_0,
$$
satisfy the properties described above.
\end{proof}

Lemma \ref{lem_PulcXalpha}
 allows us to define the notion of {\it homogeneous order at most $N$} for differential operator ($N\in\bN$).
 
\begin{definition}
\label{def_homorder}
    A differential operator on $M$ is of {\it homogeneous order at most $N$} when 
    it can be written as a linear combination over $C^\infty(U)$ of $\bX^\alpha$, $[\alpha]\leq N$, with respect to any adapted frame $(\bX,U)$.
\end{definition}

The following statement shows that 
if a differential operator is written 
as a $C^\infty(U)$-linear combination of $\bX^\alpha$, $[\alpha]\leq N$, with respect to one  frame $(\bX,U)$, then it will be so for any frame $\bY$ on $U$.

\begin{lemma}
\label{lem_DOXY}
Let $\bX$ and $\bY$ be two adapted frames on the same  open subset $U\subset M$.
\begin{enumerate}
    \item 
For any non-zero multi-index  $\alpha\in \bN_0^n$, 
we have for any $x\in U$
$$
\bX^\alpha_x 
=\sum_{[\beta]=[\alpha]}
\tilde T_{\alpha,\beta} \bY^\beta
+
\sum_{[\beta]<[\alpha]} r_\beta(x) \bY^\beta,
$$
where the coefficients $\tilde T_{\alpha,\beta}$, $[\beta]=N$, were defined in Lemma \ref{old_lem_XY}, 
and $r_\beta\in C^\infty(U)$ are some functions also depending on $\alpha$.
\item Consider a differential operator on $P$ written as $P=\sum_{[\alpha]\leq N} c_\alpha ^\bX (x)\bX^\alpha$ with respect to the frame $\bX$.
Then $P$ is written as $P=\sum_{[\beta]\leq N} c_\beta^\bY(x) \bY^\beta$ with respect to the frame $\bY$ for some unique coefficients $c_\beta^\bY$, and we have above any $x\in U$
 $$
 \sum_{[\alpha]=N} c_\alpha^\bX(x) \langle \bX\rangle_x^\alpha = 
 \sum_{[\beta]=N} c_\beta^\bY(x) \langle \bY\rangle_x^\alpha \quad \mbox{in}\ 
  \sU_{N}(\fg_x M).
 $$
\end{enumerate}
\end{lemma}

\begin{proof}
Let $T$ be the transition matrix as in Lemma \ref{old_lem_XY}. We have $\bX=T^{{\rm t}}\bY$, so
$$
\bX^\alpha 
= 
[T^{{\rm t}}\bY]_1^{\alpha_1}
\ldots
[T^{{\rm t}}\bY]_n^{\alpha_n}.
$$
Part (1) follows by the Leibniz property of vector fields and Lemma \ref{old_lem_XY} (3). Part (2) follows by linearity.
\end{proof}

We denote by $\DO(M)$
    the space of differential operators on $M$, 
    and by $\DO^{\leq N}(M)$ the subspace of differential operators of homogeneous order at most $N$.
    Lemma \ref{lem_PulcXalpha} and \ref{lem_DOXY} show that
    if $(\bX,U)$ is an adapted frame, then the set of monomials $\bX^\alpha$, $\alpha\in \bN_0^n$, is a  basis of 
    the $C^\infty(U)$-module   
$\DO(U)$ while the set of monomials $\bX^\alpha$, $[\alpha]\leq N$, is a basis of 
    the submodule   
$\DO^{\leq N}(U)$.
\smallskip 

Let us open a brief parenthesis for the case of $H_1$ generating $TM$ as in  Example \ref{ex:subrie}:

\begin{corollary} \label{corlem_PulcXalphasubRie}
Assume that $H_1$ generates $TM$.
Let $\bX$ be an adapted frame on the open subset $U\subset M$.
Then $\DO(U)$ is the $C^\infty(U)$-module generated by 
$$
X_{i_1}\ldots X_{i_\ell}, \qquad   1\leq i_1,\ldots,i_\ell\leq d_1=\dim H_1, 
\quad \ell\in \bN_0.
$$
Moreover, $\DO^{\leq N}(U)$ is the $C^\infty(U)$-module generated by $X_{i_1}\ldots X_{i_\ell}$,  $1\leq i_1,\ldots,i_\ell\leq d_1$, $\ell=0,\ldots,N$.
\end{corollary}

In other words, 
any differential operator of degree $\leq N$ may be written on $U$ as a combination over $C^\infty(U)$ of 
$X_{i_1}\ldots X_{i_\ell}$,  $1\leq i_1,\ldots,i_{\ell}\leq d_1$, $\ell=0,\ldots,N$; 
however, this writing may not be unique.
\begin{proof}[Proof of Corollary \ref{corlem_PulcXalphasubRie}]
The first claim of the statement follows readily from
Lemma \ref{lem_PulcXalpha} and
any $X_j$, $j>d_1$ being obtained as a $C^\infty(U)$-combination of nested bracket of $X_i$, $i=1,\ldots, d_1$.

As $H_1$ generates $TM$ (see Example \ref{ex:subrie}),  each 
$X_j$, $j=1,\ldots, n$ may be written as a $C^\infty$-combination of 
$X_{i_1}\ldots X_{i_\ell}$ with $\ell \leq \upsilon_j$.
This allows us to prove the rest of the statement inductively on $N$ by writing for every $\alpha\in \bN_0\setminus\{0\}$, $\bX^\alpha = X_j \bX^{\alpha'}$ for some $j=1,\ldots, n$ and $\alpha'\in \bN_0^n$ uniquely determined by $\alpha$ and satisfying $[\alpha] = \upsilon_j + [\alpha']$.
\end{proof}

\smallskip 

If $P\in \DO^{\leq N}(M)$ and $(\bX,U)$ is an adapted frame, then
for each $x\in U$, we define  the element of $\sU_{N}(\fg_x M)$
\begin{equation}\label{def:princ_symbol}
\princ_{N,x} (P):=\sum_{[\alpha]=N} c_\alpha(x) \langle \bX\rangle_x^\alpha,
\end{equation}
where the coefficients $c_\alpha$ come from the unique expression $P=\sum_{[\alpha]\leq N} c_\alpha  \bX^\alpha$ on $U$.
By Lemma~\ref{lem_DOXY} (2), 
this defines an element of $\sU_{N}(\fg_x M)$ that does not depend on  the choice of the adapted frame $\bX$ and that is smooth in $x$. 
This allows us to define the  principal part of a differential operator of homogeneous order at most $N$ as the morphism of $C^\infty(M)$-module:
$$
\princ_N : \DO^{\leq N}(M) \to \Gamma(\sU_{N}(\fg M)).
$$

We may say that $P\in \DO^{\leq N}(M)$ is {\it of homogeneous order $N$} when its principal part $\princ_N(P)$ is non-zero in $\Gamma(\sU_{N}(\fg M))$.
Note that the degree of $P$ as a differential operator is less or equal to its homogeneous order $N$; it is usually much smaller than $N$.

\subsection{Group exponentiation and BCH formula}
\label{sec:exp_map}

If an adapted frame $(\bX,U)$ is fixed, we  use the adapted basis $\langle \bX\rangle_x$ to  write the \emph{group exponentiation}:
\begin{equation}
\Exp^\bX_x (v):=
	\Exp_x\left(\sum_{i}v_i\langle X_i\rangle_x\right) \in G_x M.
  \label{expCoordonates}
\end{equation}
This yields the smooth mapping 
$$
\Exp^\bX:\left\{\begin{array}{rcl}
 U \times \mathbb{R}^n
&\longrightarrow& {G}M|_U\\
(x,v)&\longmapsto& \Exp_x\left(\sum_{i=1}^n v_i\langle X_i\rangle_x\right)
\end{array}\right. 
$$
with inverse 
$$
\Ln^\bX := (\Exp^\bX)^{-1} \ : \  \Exp^\bX(U\times \bR^n) \longrightarrow U\times \bR^n.
$$
This  allows us  to extend the definition of the dilations $\delta_r$ to $U\times \bR^n$. 
Moreover, the Lebesgue density $|dv|$ on $\mathbb{R}^n$ gives us a frame-dependent
choice of volume density on $TM|_U$ and a  Haar system for ${G}M|_U$ by pushing forward through the trivialization.

\smallskip

The next lemma explains what happens 
when the adapted frame is changed:

\begin{lemma}\label{lem_XYExp}
Let $\bX$ and $\bY$ be  two adapted frames  $\bX$ and $\bY$ on an open subset $U\subset M$, and~$T$ the matrix of change of frame as in Lemma \ref{old_lem_XY}.
Then, the filtration preserving map ${\rm{diag}}(T_x)$ satisfies
$$
\forall (x, v)\in U\times \bR^n\qquad
	\Ln^{\bY}_x\circ \Exp^{\bX}_x(v)={\rm{diag}}(T_x).
$$
\end{lemma}

\begin{proof}[Proof of Lemma~\ref{lem_XYExp}]
Since $\langle \bX\rangle_x ={\rm diag}(T_x)^{\rm t}\langle \bY\rangle_x $, we have
$$
\sum_i v_i \langle X_i\rangle_x  = 
\sum_i v_i \sum _{\upsilon_i = \upsilon_k}T_{ki}(x)\langle Y_k(x)\rangle_x=
\sum_k ({\rm diag}(T_x) v)_k \langle Y_k\rangle_x ,
$$
so
$
\Exp_x^\bX v 
= \Exp_x^\bY ( {\rm diag}(T_x) v).$ 
\end{proof}

\smallskip

The basis $\langle X_1\rangle_x, \ldots,\langle X_n\rangle_x $ together with the  group exponential map ${\rm Exp}_x^{\bX}$
allows for a precise group isomorphism between  the group fiber ${G}_xM$ with  $\bR^n$.
 
\begin{remark}\label{rem_quasinormbX}
We have already given natural examples of  quasinorms on graded groups associated with an adapted basis, see Example \ref{ex_quasinormGX}.
 With the above isomorphism in mind, it is natural to consider the    quasinorm on $\bR^n$
 given by 
 $$
 |v|_\bX = 
  \left(\sum_{i=1}^n |v_i|^{\alpha / \upsilon_i}\right)^{1/\alpha},
  \quad v=(v_j)_{i=1}^n\in \bR^n,
  $$
   $\alpha =2M_0$ where $M_0$ is a common multiple of the weights $\upsilon_1,\ldots,\upsilon_n$, the least one to fix the idea.
  We will use this quasinorm  throughout the rest of the work.
 \end{remark}

 To describe more precisely the group law inherited by $\bR^n$, 
 we introduce the map given at every point $x$ by the adjoint map $\ad_{\fg_x M}$ in the $\langle \bX\rangle$-coordinates:
\begin{definition}
\label{def_adxbX}
   Let $\bX$ be an adapted frame on an open set $U$ of the filtered manifold $M$.  
   For each $x\in U$, 
and for any $v,w\in \bR^n$, we denote by 
   $\ad_x^\bX (v) (w)$  the vector in $\bR^n$ whose coordinates $(u_1,\ldots, u_n)$ satisfy 
 $$
  \sum_{k=1}^n u_k \langle X_k\rangle_x := \left[\sum_{i=1}^n v_i \langle X_i\rangle_x\,,\, \sum_{j=1}^n w_j \langle X_j\rangle_x\right]_{\fg_x M}.  
$$  
This defines a skew-symmetric map $\ad_x^\bX:\bR^n\times \bR^n \to \bR^n$ we call  the {\it adjoint map in the $\langle \bX\rangle$-coordinates}.
\end{definition}
This equips $\bR^n$ with a graded Lie algebra structure  isomorphic to the one of $\fg_x M$, the isomorphism being given by the choice of basis~$\langle \bX\rangle_x$.
Moreover, the map 
$$
\ad^\bX : x\longmapsto \ad_x^\bX,
\quad U\longrightarrow C^\infty (U , \sL (\bR^n)),
$$
is smooth.
It describes  the  group law of $G_x M$
in exponential coordinates \eqref{expCoordonates}:
\begin{lemma}
\label{lem_BCHGxM}
Let $\bX$ be an adapted frame on an open set $U$ of the filtered manifold $M$.  For any $x\in U$ and $v,w\in \bR^n$, we have:
 $$
 \Exp^\bX_x (v)\  \Exp^\bX_x (w)
=
\Exp^\bX_x (z(x;v,w))
=
\Exp_x \left(\sum_i z_i(x;v,w)\langle X_i\rangle_x\right)
$$  
where $z(x;v,w)= (z_k(x;v,w))_{k=1}^n$ is given by 
$$
z(x;v,w)= {\rm BCH}_{s_0}(v,w;\ad_x^\bX),   \quad x\in U, \ v,w\in \bR^n,
$$
 where  $s_0\geq s $,  $s$ being the step of 
 every $\fg_x M$.
 For each $k=1,\ldots, n$, 
 the function $z_k(x;v,w)$
 is polynomial in $v,w$ with coefficients depending smoothly on $x\in U$. 
 Moreover, $z_k(x;v,w)$ is 
 $\upsilon_k$-homogeneous, in the sense that 
$z_k(\delta_r v,\delta_r w;x)= r^{\upsilon_k} z_k(x;v,w)$.
\end{lemma}
Recall that that the step of the graded spaces $\fg_x M$ was defined in  Section~\ref{sec:def_filt_manifold}
and that  ${\rm BCH}_{s_0}$ has been introduced in Notation~\ref{notation_BCH}.

We denote by 
\begin{equation}
\label{eqdef_v*xw}
 v*_x^\bX w :=
\Ln_x^\bX \left(\Exp^\bX_x (v)\  \Exp^\bX_x (w)\right)\qquad v,w\in \bR^n
\end{equation}
the  {\it law of the group} $\bR^n$ 
isomorphic to $G_x M$ via the group exponential map $\Exp_x$ and the choice of the basis $\langle \bX\rangle_x$.
By Lemma \ref{lem_BCHGxM} (and with its notation), we have
$$
v*_x^\bX w = z(x;v,w) ={\rm BCH}_{s_0}(v,w;\ad_x^\bX), 
$$
so that
$$
\Exp^\bX_x (v*_x^\bX w) =
\Exp^\bX_x (v)\  \Exp^\bX_x (w).
$$
If the adapted frame $\bX$ is fixed, we will allow ourselves to use the shorthand $*_x^\bX =*_x$.
The reader should keep in mind that star-notation is used in various ways in this paper: $*_x$ gives the product law of the group $G_xM$ after suitable identifications
while the convolution product on a group $G$  will be denoted by $\star_{G}$, in accordance
with~\eqref{eq_convolution}.

\section{Action of  geometric maps on the osculating group}\label{sec:geom_maps}

In this section we analyse
the composition of geometric exponential maps.
After defining  the notion of higher order maps between graded vector spaces in Section~\ref{sec:higher_order}, we use this notion on filtered manifolds to describe 
various algebraic properties in Section~\ref{sec:higher_order_M}, 
and eventually the composition of geometric exponential mappings  in Section~\ref{sec:comp_exp}. 

The notion of terms of higher order is crucial in this analysis. 
It is also present in various forms 
in many papers studying the composition of geometric exponential on filtered or subRiemannian manifolds, for instance in~\cite{Choi_Ponge_1,Choi_Ponge_2}.

\subsection{Higher order terms on graded  spaces}\label{sec:higher_order}
Here, 
we define 
what higher order means for maps between graded vector spaces.
This notion will be at the core of  the comparison of the commutator brackets of vector fields at $x$ with the structure of the Lie algebra $\fg_x$,
and in various other places.

In this section, we consider  two graded  vector spaces $\cV_1$ and $\cV_2$ and  denote by $\delta_{1,r}$ and $\delta_{2,r}$ for $r>0$ the associated dilations (see Definition \ref{def_dilation_gradedV}).

\begin{definition}
	A function $f:\cV_1\to \cV_2 $ is \emph{homogeneous} of degree $m$ (for the dilations $\delta_{1,r},\delta_{2,r},r>0$) when $f(\delta_{1,r} v) =\delta_{2,r^m} f(v)$ for any $r>0$ and $v\in \cV_1$. 
	In the case $m=1$, we say that $f$ is homogeneous. 
\end{definition}

\begin{ex}
\label{ex_homopolymap_grlaw}
    If $G$ is a graded Lie group with a fixed adapted basis $\bV$ for its Lie algebra, then the group law expressed as a map $(v,w)\mapsto v *_G w$, $\bR^n\times \bR^n \to \bR^n$ is a homogeneous map.  
\end{ex}

We fix bases $ (V_1,\ldots,V_{\dim \cV_1})$ and $(W_1,\ldots,W_{\dim \cV_2})$ for $\cV_1$ and $\cV_2$ adapted to the gradation, 
and we denote by $\upsilon_j$ the weight of $W_j$, $j=1,\ldots, \dim \cV_2$.
It is easy to check that a function $f:\cV_1\to \cV_2 $ is homogeneous  if and only if, for each $j=1,\ldots, \dim \cV_2$, the $j$-components $f_j:\cV_1\to \bR$ of $f$ is an $\upsilon_j$-homogeneous map.
This motivates the following definition:

\begin{definition}
\label{def_hot}
\begin{enumerate}
\item
We say that a polynomial map $r:\cV_1 \to \cV_2$ is \emph{higher order} (for the gradations) when 
the degrees of homogeneity (for the dilations on $\cV_1$) of 
the monomials occurring in its $j$-component $r_j$  
are $>\upsilon_j$. 
The notion of higher order   extends to formal power series. 
 \item  We say that a smooth function $r$ defined on a neighbourood of 0 in $\cV_1$ and valued in~$\cV_2$ is higher order when its Taylor series at 0 is a higher order  power series.
 \end{enumerate}
\end{definition}

It is not difficult to construct higher order polynomial maps from homogeneous ones:
\begin{ex}
\label{ex_hotpoly}
Let  $q:\cV_1\to \cV_2$ and 
$r:\cV_1\to \cV_2$ be polynomial maps that are  respectively homogeneous and higher order.   
\begin{enumerate}
    \item If $p:\cV_1\to \bC$ is a polynomial satisfying $p(0)=0$, then $pq$ and $pr$ are higher order polynomial maps.
    \item This first example can be generalised as follow. We fix a basis $(W_1,\ldots,W_{\dim \cV_2})$  adapted to the gradation of $\cV_1$, so that we may write $q=(q_k)_{k=1}^{\dim \cV_2}$
    and $r=(r_k)_{k=1}^{\dim \cV_2}$ in the corresponding coordinates.
    Given any family of polynomials $p_k:\cV_1\to \bC$, $k=1,\ldots,\dim \cV_2$,
    then the polynomial maps $(p_k r_k)_{k=1}^{\dim \cV_2}$  are higher order. 
    If in addition $p_k(0)=0$ for every $k=1,\ldots, \dim \cV_2$, then  the polynomial maps $(p_k q_k)_{k=1}^{\dim \cV_2}$  are higher order. 
\end{enumerate}    
\end{ex}

The following properties characterise 
the higher order maps. In particular, the characterisations (2) or (3) imply that the notion of higher order does not depend on the chosen graded bases.

\begin{proposition}
\label{prop_char_higherorder}
Let $r:\cV_1 \to \cV_2$ be a smooth map. 
The following properties are equivalent:
\begin{enumerate}
	\item The function $r$ is higher order.
	\item For one (and then any) neighbourhood $V_0$ of 0, we have for any $v\in V_0$ 
	$$
\lim_{\eps\to 0}	\delta_{\eps^{-1}} r(\delta_\eps v) =0.
	$$
\item There exists a unique smooth map 
$R:\bR\times \cV_1\to \cV_2 $ satisfying 
$\delta_{\eps^{-1}} r(\delta_\eps v)= \eps R(\eps,v)$.
\end{enumerate}	
If an adapted basis of $\cV_2$ has been chosen, then it is also equivalent to:
\begin{itemize}
	\item[(4)]  For $j=1, \ldots, \dim \cV_2$, we have $\partial_v^\alpha r_j(0)=0$, $[\alpha]\leq \upsilon_j$; 
 equivalently, this means that the Taylor series of $r_j$ about~$0$ is 
$$
r_j(v)\sim_{v\sim 0} \sum_{[\alpha]>\upsilon_j }\partial^\alpha r_j(0) v^\alpha.
$$
\end{itemize}
\end{proposition}

\begin{proof}
	As above, 
we fix bases $ (V_1,\ldots,V_{\dim \cV_1})$ and $(W_1,\ldots,W_{\dim \cV_2})$ for $\cV_1$ and $\cV_2$ adapted to the gradation, with $\upsilon_j$ the weight of $W_j$.
We denote   by $r_j$ the $j$th component of $r$.
By definition, 
$r$ is higher order if and only if Property (4) holds.

Whether $r$ is higher order or not, we may write, 
$$
r_j(v) = \sum_{[\alpha]\leq \upsilon_j} \partial^\alpha r_j(0) v^\alpha + R_j(v),
$$
with $R_j\in C^\infty(\cV_1)$.
Fixing a quasinorm $|\cdot|_1$ on $\cV_1$, by the Taylor estimate  due to Folland and Stein at order $\upsilon_j$
(see \cite{folland+stein_82} or  Theorem~\ref{thm_MV+TaylorG}), 
for any compact neighbourhood $W$ of 0 in $\cV_1$,
there exists  a constant $C_j>0$ such that
$$
\forall v\in W,
\qquad  
|R_j(v)|\leq C_j |v|_1^{\upsilon_j+1}.
$$

If Property (1) holds, that is, if the Taylor series of $r$ at 0 is higher order, then for any $j=1,\ldots, \dim \cV_2$ 
$$
|\eps^{-\upsilon_j} r_j(\delta_\eps v)|
= |\eps^{-\upsilon_j} R_j(\delta_\eps v) |
\leq\eps^{-\upsilon_j}
C |\delta_\eps v|_1^{\upsilon_j+1}
= C \eps |v|_1^{\upsilon_j+1}
\longrightarrow_{\eps\to 0 } 0, 
$$
so Property  (2) hold. 

If Property (2) holds, or equivalently if for any $j=1,\cdots,\dim \cV_2$ and $v\in \cV_1$
$$
0=
\lim_{\eps\to 0}\eps^{-\upsilon_j} r_j(\eps v)  
	=
\lim_{\eps\to 0} 	
	\sum_{[\alpha]\leq \upsilon_j} \eps^{-\upsilon_j+[\alpha]} \partial^\alpha r_j(0) v^\alpha + \eps^{-\upsilon_j} R_j(\delta_\eps v), 
	$$
then $\partial^\alpha r_j(0)=0$, $[\alpha]\leq \upsilon_j$ and Property (1) holds.  
Moreover, the smoothness of the $r_j$'s implies
$$
r_j(v) = \sum_{|\alpha|\leq N_0} \partial^\alpha r_j (0) v^\alpha + \sum_{|\alpha|=N_0+1} v^\alpha \rho_{j,\alpha}(v), \quad\mbox{with}\ \rho_{j,\alpha}\in C^\infty (\cV_1).
$$
Choosing  $N_0 = \max\{|\alpha| : \alpha\in \bN_0, [\alpha]\leq \upsilon_j\}$, 
since $\partial^\alpha r_j(0)=0$, $[\alpha]\leq \upsilon_j$, we obtain
$$
\eps^{-\upsilon_j} r_j (\delta_\eps v) =\sum_{|\alpha|\leq N_0, [\alpha]\leq \upsilon_j}\eps^{-\upsilon_j +[\alpha]} \partial^\alpha r_j (0) v^\alpha + \sum_{|\alpha|=N_0+1}\eps^{-\upsilon_j +[\alpha]} v^\alpha \rho_{j,\alpha}(\delta_\eps v),
$$
implying Property (3).
We have obtained  (1) $\Longleftrightarrow$  (2)
$\Longrightarrow$ (3).
Clearly, (3) $\Longrightarrow$ (2).
 The proof is complete.
\end{proof}

Let us present two corollaries of Proposition \ref{prop_char_higherorder} and its proof.
The first one gives estimates of a higher order function:

\begin{corollary}
\label{cor1lem_char_higherorder}
Let $r:\cV_1 \to \cV_2$ be a smooth map of higher order.
Having fixed quasinorms $|\cdot|_1$ and $|\cdot|_2$ on $\cV_1$ and $\cV_2$, for any bounded neighbourhood $W$ of 0 in $\cV_1$, there exists $C>0$ such that 
$$
\forall v\in W,\;\; 
|r(v)|_2 \leq C |v|_1^{1+\frac 1{\upsilon_n} }.
$$
\end{corollary}

Note
that in the isotropic case, i.e. when $\upsilon_1=\ldots=\upsilon_n=1$, we recover the Euclidean or Riemannian exponent $1+\frac 1{\upsilon_n} = 1+1 =2$. However, this is not generally the case; for instance, it
won't be the case for a regular subRiemannian manifold as in Example~\ref{ex:subrie}.

\begin{proof}[Proof of Corollary \ref{cor1lem_char_higherorder}] 
Keeping the notation in the proof of Proposition \ref{prop_char_higherorder},
the estimates  therein
show that when $r$ is higher order, we have for any $v\in W$
$$
|r_j (v) |=|R_j(v)| \leq C_j |v|_1^{\upsilon_j +1},
$$
so
$$
\max_{j=1,\ldots, \dim \cV_2}
|r_j(v)|^{\frac 1\upsilon_j}
\leq \max_{j=1,\ldots, \dim \cV_2} C_j^{\frac 1\upsilon_j} |v|_1^{1+ \frac 1{\upsilon_j}}
\lesssim  |v|_1^{1+\frac 1{\upsilon_n}}.
$$
As $u\mapsto \max_{j=1,\ldots, \dim \cV_2} |u|^{1/\upsilon_j}$ is a quasinorm on $\cV_2$ and all the quasinorms are equivalent, the statement is proved.
\end{proof}

In the second corollary, we consider  higher order functions depending  continuously and  smoothly on a parameter:
\begin{corollary}
\label{cor_derivatives_cV}
\begin{enumerate}
\item  
Let $A$ be a topological vector space and $(a,v)\mapsto r(a,v )$ be a continuous map $A\times \cV_1\to \cV_2$ such that for every $a\in A$, $r(a,\, \cdot)$ is a smooth function on $\cV_1$ that is higher order. 
Then there exists a unique continuous function $R:A\times \bR\times \cV_1\to \cV_2$ such that for every $a\in A$, $R(a,\,\cdot,\,\cdot)$ satisfies Property (3) of Proposition \ref{prop_char_higherorder}, that is, 
$R(a,\, \cdot,\ \, \cdot)$ is a smooth function 
$\bR\times \cV_1\to \cV_2 $ satisfying 
$$
\forall \eps,v\in \bR\setminus\{0\}\times \cV\qquad 
\delta_{\eps^{-1}} r(a,\delta_\eps v)= \eps R(a,\eps,v).
$$
    \item 	Let $(t,v)\mapsto r(t,v)$ be a smooth map  $I \times \cV_1\to \cV_2$, $I$ an interval of~$\bR$.
	Assume that for every $t\in I$, $r$  is higher order in $v$. 
	Then $\partial_{t=0} r(t,v)$  is higher order.
\end{enumerate}
 
\end{corollary}

\begin{proof}We use the notation in the proof of Proposition~\ref{prop_char_higherorder} and consider bases for $\cV_1$ and $\cV_2$ that are adapted  to the gradation where $r$ has coordinates $r_j$.
Part (1) follows from Property~(4). 
For Part (2), by Proposition~\ref{prop_char_higherorder},
 $\partial^\alpha_v r_j(t,0)=0$ for any $t\in I$ and for any $\alpha$ with $[\alpha]\leq \upsilon_j$, so it is also the case for $\partial_{t=0} r(t,v)$ by smoothness 
 and $\partial_{t=0} r(t,v)$  is higher order.
\end{proof}

\subsection{Higher order remainders on filtered manifolds}
\label{sec:higher_order_M}

The notion of higher order term defined above allows us to now investigate the algebraic structure  of a  filtered manifold: up to remainder terms of higher order, it will be given by homogeneous polynomials or series. 
Following Corollary~\ref{cor_derivatives_cV} (1), we extend the notion of higher order to maps depending smoothly on a point $x$ of the manifold and we will 
say that such a $x$-dependent polynomial  or a $x$-dependent remainder are homogeneous or higher order above some open subset $U\subset M$ of interest, when   they are thus  for all $x\in U$. 

\subsubsection{Nested commutator brackets  on filtered manifolds}
\label{subsubsec_structuralhot}
The notion of higher order map defined above together with 
the map $\ad_x^\bX$ (see Definition \ref{def_adxbX}) allow for the description of the commutator bracket on $M$ modulo higher order terms:

\begin{proposition}
\label{prop_adxbX}
Let $\bX$ be an adapted frame on an open set $U$ of the filtered manifold $M$. 
The commutator brackets of vector fields may be written as:
    $$
    \ad \left(\sum_{i=1}^n v_i X_i\right) 
    \left( \sum_{j=1}^n w_j X_j\right) (x)
    = \sum_{j=1}^n (q_j (x;v,w) + r_j (x;v,w)) X_j,
    $$
    where the $q_j(x;v,w)$'s and $r_j(x;v,w)$'s are  polynomials  in $v,w$ with smooth coefficents in $x\in U$;  the $q_j$'s are $\upsilon_j$-homogeneous and described by
    $$
    (q_j(x;v,w))_{j=1}^n = \ad_x^\bX (v)(w),
    $$
and  $r(x;v,w)=(r_j(x;v,w))_{j=1}^n$ is  higher order in $v,w$.
\end{proposition}

Before entering into the proof of Proposition \ref{prop_adxbX}, let us state some structural properties of  the commutator brackets for $\bX$.
We denote by  $c_{i,j,k}(x)$  the structural constants of the commutator brackets for $\bX$, that is, 
$$
\ad (X_i)(X_j)(x)=
[X_i,X_j](x) = \sum_{k=1}^n c_{i,j,k}(x) X_k(x);
$$
the $c_{i,j,k}(x)$ are smooth functions of $x\in U.$
As $\bX$ is adapted to the filtration, we see:
\begin{itemize}
	\item[(a)] $c_{i,j,k}=0$ when $\upsilon_k>\upsilon_i+\upsilon_j$,
	\item[(b)] when $\upsilon_k=\upsilon_i+\upsilon_j$, the $c_{i,j,k}(x)$'s are the structural constants of $[,]_{\fg_x}$ for $\langle \bX\rangle_x$:
	$$
\left[\langle X_i\rangle_x ,\langle X_j\rangle_x\right]_{\fg_x} = 
\sum_{\upsilon_k=\upsilon_i+\upsilon_j} c_{i,j,k}(x) \langle X_k\rangle_x,
$$
or in other words, using the map $\ad_x^\bX:\bR^n\times \bR^n \to \bR^n$ (see Definition \ref{def_adxbX})  and  denoting the canonical basis of $\bR^n$ by $(e_1,\ldots,e_n)$,
$$ 
\ad_x^\bX (e_i,e_j)=\sum_{\upsilon_k=\upsilon_i+\upsilon_j} c_{i,j,k} e_k
$$
\item[(c)]  the polynomial $(c_{i,j,k}(x)v_i w_j)_{k=1}^n$ from the components of 
$$
 \sum_{\upsilon_k<\upsilon_i+\upsilon_j}  c_{i,j,k}(x) v_i w_j X_k(x) 
 = [v_i X_i,w_j X_j](x) - \sum_{\upsilon_k=\upsilon_i+\upsilon_j}  c_{i,j,k}(x) v_i w_j X_k(x),
$$
are higher order. 
\end{itemize}  

\begin{proof}
Using the structural constants $c_{i,j,k}$, we can  write
$$
\left[\sum_{i=1}^n v_i X_i \, ,\, \sum_{j=1}^n w_j X_j\right](x) =
\sum_{\upsilon_i+\upsilon_j\leq\upsilon_k} v_iw_j c_{i,j,k}(x) X_k(x)
=
\sum_{k=1}^n (q_k(x;v,w)+r_k(x;v,w))
 X_k(x),
$$
where
$$
q_k(x;v,w) := \sum_{\upsilon_k = \upsilon_i + \upsilon_j} c_{i,j,k} (x) v_i w_j.
$$
Hence the maps $q(x;v,w):= (q_k(x;v,w))_{k=1}^n$ and $r(x;v,w):= (r_k(x;v,w))_{k=1}^n$ are polynomials in $v,w$ with smooth coefficients in $x\in U$.
From the observation in (b) above, 
$q(x;v,w)=\ad_x^\bX (v) (w)$ is homogeneous in $v,w$, while from the observation in (c), 
$r(x;v,w)$ is higher order in $v,w$.
\end{proof}

Proposition \ref{prop_adxbX}  implies inductively the following property of nested commutator brackets:
\begin{corollary}
 \label{cor_prop_adxbX} 
 Let $\bX$ be an adapted frame on an open set $U$ of the filtered manifold $M$. 
Let $\ell\in\bN_0$, $\alpha=(\alpha_1,\ldots,\alpha_\ell)\in \bN_0^\ell$
and $\beta=(\beta_1,\ldots,\beta_\ell)\in \bN_0^\ell$.
We can write 
\begin{align*}
&(\ad \sum_{i=1}^n v_i X_i  )^{\alpha_1}(\ad \sum_{i=1}^n w_i X_i  )^{\beta_1}
\ldots (\ad \sum_{i=1}^n v_i X_i  )^{\alpha_\ell}(\ad \sum_{i=1}^n w_i X_i  )^{\beta_\ell} \ (x)
\\&\qquad= 
\sum_{k=1}^n (q_k^{\alpha,\beta}(x;v,w)+r^{\alpha,\beta}_k(x;v,w))
 X_k(x),
\end{align*}
with 
$q^{\alpha,\beta}_k(x;v,w)$ 
and $r^{\alpha,\beta}_k(x;v,w)$ polynomials in $v,w$ with smooth coefficients in $x\in U$;
 the polynomial map 
$q (x;v,w) = (q_k^{\alpha,\beta}(x;v,w))_{k=1}^n$ is 
$([\alpha]+[\beta])$-homogeneous and is
defined via 
$$
q^{\alpha,\beta}(x;v,w):=
  (\ad_x^\bX v)^{\alpha_1}(\ad_x^\bX w)
)^{\beta_1}
\ldots (\ad_x^\bX v)^{\alpha_\ell}(\ad_x^\bX w)^{\beta_\ell} \ (x)
$$
while the polynomial map 
$r^{\alpha,\beta}(x;v,w)=(r^{\alpha,\beta}_k(x;v,w))_{k=1}^n$ is higher order. 
\end{corollary}

\begin{proof}
The proof follows by induction from Proposition \ref{prop_adxbX}, together with the observations from Examples \ref{ex_hotpoly}
and the fact that if $q(x;v,w)$  and $r(x;v,w)$ are polynomial maps that are respectively homogeneous and  higher order, then so are their $x$-derivatives $X_\ell q (x;v,w)$ and 
$X_\ell r (x;v,w)$,
$\ell =1,\cdots n$ (for the latter, see Corollary \ref{cor_derivatives_cV} (2)).    
\end{proof}

\subsubsection{Examples and further properties of higher order remainders}

In this section, we give two further examples  and some  general properties of  higher order remainders.

\smallskip

The first example concerns the laws  of the osculating groups, or rather the law $*_x^\bX$ of the isomoprhic group $\bR^n$ (see \eqref{eqdef_v*xw}). 
Later on, we will need the following property:

\begin{lemma}
\label{lem_group_laws}
Let  $(\bX,U)$ be an adapted frame on a filtered manifold $M$. 
\begin{enumerate}
    \item The map $(x;v,w)\mapsto v*_x^\bX w$ is  a homogeneous polynomial map in $v,w$ with coefficients depending smoothly in $x\in U$, and so are the maps $(x;v,w)\mapsto\bX^\beta_x (v*_x^\bX w)$.
    
    \item The maps defined by
\begin{align*}
r_1(x;v,w)
&:=w *^\bX_{\exp^\bX_x v}v -w*^\bX_x v	\\
r_2(x;v,w)
&:=\left(-(w*^\bX_x v) \right)*_x^\bX \left( w  *^\bX_{\exp^\bX_x v}v \right)	
\end{align*}
are smooth on an open neighbourhood of the zero section of $U\times \bR^n\times \bR^n$. They are higher order in $v,w$.
\end{enumerate}
\end{lemma}

\begin{proof}
Part (1) follows from 
Lemma \ref{lem_BCHGxM} and the homogeneity of the group law (see Example~\ref{ex_homopolymap_grlaw}).
For Part (2) and $r_1$,  we have
by the homogeneity of the group law
\begin{align*}
 \delta_{\eps}^{-1}r_1(x;\delta_{\eps} v,\delta_{\eps}w)
 &=w *^\bX_{\exp^\bX_x \delta_{\eps} v}v -w*_x v    \\
 &={\rm BCH}_{s_0}(v,w;\ad_{\exp^\bX_x \delta_{\eps} v}^\bX)
 -{\rm BCH}_{s_0}(v,w;\ad_x^\bX);
\end{align*}
this last expression tends to 0 as $\eps\to 0$ since
 ${\rm BCH}_{s_0}(v,w;\ad_x^\bX)$ is a polynomial expression in $\ad_x^\bX v$ and $\ad_x^\bX w$ and $\lim_{\eps\to 0}\ad_{\exp^\bX_x \delta_{\eps} v}^\bX = \ad_x^\bX$. 
 For $r_2$, we see
 \begin{align*}
 \delta_{\eps}^{-1}r_2(x;\delta_{\eps} v,\delta_{\eps}w)
 &=\left(-(w*^\bX_x v) \right)*_x^\bX \left( w  *^\bX_{\exp^\bX_x \delta_\eps v}v \right)\longrightarrow_{\eps\to 0}0.
\end{align*} 
 We conclude with Proposition \ref{prop_char_higherorder}.
\end{proof}

Before giving the second example, let us show some properties of higher order remainders: 

\begin{lemma}
\label{lem:exp_derivative_bis}
 Let $(\bX, U)$ be an adapted frame on a filtered manifold $M$. Let $r:  U\times\bR^n\times \bR^n\to \bR^n$ be a   smooth  higher order map.
Then, the maps 
 \begin{align*}
        (v,w)  \mapsto r_x(v*_x^\bX w,v),\qquad     (v,w)  \mapsto r_{\exp_x^\bX v} (v,w),\qquad 
      (v,w)  \mapsto (-v)*_x^{\bX}(v+{r}_x(v,w)),
 \end{align*}
are  also smooth and higher order.  
 \end{lemma}
 
 Note that this lemma extends to functions $r_x$ depending only on the variable~$v$. 

The proof relies on the characterizations of higher order functions in  Proposition~\ref{prop_char_higherorder}, and its extensions to functions depending on a parameter in Corollary \ref{cor_derivatives_cV}.

\begin{proof}
For analyzing the function, $(v,w)\mapsto r_x(v*_x^\bX w,v)$, we observe that by the properties of dilation with respect to the group law,  for $v,w\in\bR^n$,
\[
\delta_\eps^{-1} r_x(\delta_\eps v*_x^\bX \delta_\eps  w,\delta_\eps v)=\delta_\eps^{-1} r_x(\delta_\eps( v*_x^\bX  w),\delta_\eps v).
\]
Therefore
\[
\lim_{\eps\to 0} \delta_\eps^{-1} r_x(\delta_\eps v*_x^\bX \delta_\eps  w,\delta_\eps v)=0,
\]
and $(v,w)\mapsto r_x(v*_x^\bX w,v)$ is higher order by  Proposition \ref{prop_char_higherorder}.

Regarding 
 the function  $(v,w)\mapsto r_{\exp_x^\bX v} (v,w)$, 
 consider the function $R=R(x,v,w,\eps)$ such that 
 $\delta_\eps^{-1} r(x,\delta_\eps v,\delta_\eps w)= \eps R(x,v,w,\eps)$. 
  By Corollary \ref{cor_derivatives_cV} (1), $R$ is continuous in $x,v,w,\eps$ and we have 
 \[
 \lim_{\eps\to 0} \delta_\eps^{-1} r(\exp_x^\bX \delta_\eps v,\delta_\eps v,\delta_\eps w) =  
  \lim_{\eps\to 0} \,\eps R( \exp_x^\bX \delta_\eps v,v,w)=0,
 \]
 showing that $(v,w)\mapsto r_{\exp_x^\bX v} (v,w)$ is higher order by Proposition \ref{prop_char_higherorder}.

Consider now 
$(v,w)\mapsto (-v)*_x^{\bX}(v+{r}_x(v,w))=:\widetilde r_x(v,w)$.
By homogeneity of the group laws, we have:
$$
\delta_\eps ^{-1} \widetilde r_x (\delta_\eps v,\delta_\eps w)	
=
(-v)*_x  ( v+\delta_\eps ^{-1} r_x(\delta_\eps v,\delta_\eps w))
=
{\rm BCH}_{s_0}( -v, v+\delta_\eps ^{-1} r_x(\delta_\eps v,\delta_\eps w), \ad_x^\bX)
$$
As  $r$  is higher order, by Proposition \ref{prop_char_higherorder}, 
$$
\lim_{\eps\to 0} \delta_\eps ^{-1} r_x(\delta_\eps v,\delta_\eps w)=0
\quad\mbox{so}\quad
\lim_{\eps\to 0} 	\delta_\eps ^{-1} \widetilde r_x (\delta_\eps v,\delta_\eps w)	
=
(-v)*_x   v
=0,
$$
and  $\widetilde r_x$  is higher order. 
\end{proof}

The second example of higher order remainders we will use later on occurs when considering  changes of adapted frames:

\begin{proposition}
\label{prop_T_Taylor}
We consider two local frames $\bX$ and $\bY$ as in  Lemma \ref{lem_uYX} (whose notation we keep), with the additional assumption that they are adapted to the filtration of $M$. 
The map $r=r(x,v)=r_x(v)$ defined via 
\begin{equation}\label{def:rx(v)}
-{\rm diag}(T_x)^{-1}u_x(-v) = v *_x^\bX 
r(x,v), 
\qquad 
(x, v)\in U_{\bX,\bY}
\end{equation}
is smooth on $U_{\bX,\bY}$ and higher order. 
\end{proposition}

\begin{proof}
By Lemma \ref{lem_uYX}, the Taylor series of $u_x(v)-\textrm{diag}(T_x)v$ about $v\sim 0$ is given by
$$
u_x(v)-\textrm{diag}(T_x)v \sim    \left(T_x-\textrm{diag}(T_x)\right)v
    +\sum_{k\geq 2}\frac{1}{k!}\Big (\sum_j v_j X_j\Big )^{k-1}T_xv.
$$
Since   $\left(T_x-\textrm{diag}(T_x)\right)v$ and each term of the series $\sum_{k>2}$ are higher order in $v$,  $u_x(v)-{\rm diag}(T_x)v$  is higher order. Hence, the function 
$$
\widetilde{r}_x(v):=-{\rm diag}(T_x)^{-1}u_x(-v)-v
$$
 is smooth on $U_{\bX, \bY}$ and higher order.
By Lemma \ref{lem:exp_derivative_bis}, 
$r_x(v)=(-v)*_x(v+\widetilde{r}_x(v))$
 is higher order.  
\end{proof}

\subsection{Composition of geometric exponential mappings}\label{sec:comp_exp}

In this section, we study the composition of geometric exponential mappings, comparing it with the law of the osculating group. Then, we derive technical properties that we shall use later.

\subsubsection{Relation with the law group}\label{section_3.5.1}

The core of this section is the next result which is expressed using the notion of higher order  developed in Section \ref{sec:higher_order}
and the structure of the osculating groups via the group law $*_x^\bX$ defined in \eqref{eqdef_v*xw}:
\begin{theorem}
\label{lem_Cq1BCH}
 Let $\bX$ be an adapted frame on an open set $U\subset M$.
Then 
$\ln^\bX_x(\exp^\bX_{\exp^\bX_x w} (v))$ may be written in the form:
 \begin{equation}
    \label{eq_Cq1BCH_loc}
    \ln^\bX_x(\exp^\bX_{\exp^\bX_x w} (v))  =   w*^\bX_x v \ + \ r (x;v,w);
\end{equation}
 the map  $r(x;v,w)$  is smooth on an open neighbourhood of the 0-section of $U\times \bR^n\times \bR^n$,  valued in $\bR^n$ and    higher order in $v,w$. 
\end{theorem}

\begin{proof}[Proof of Theorem \ref{lem_Cq1BCH}]
By \cite[Appendix]{NagelSteinWainger1985}, for any $f\in C_c^\infty(U)$, 
we have the following equality between Taylor series for $v,w\sim 0$:
\begin{equation}
\label{eq_BCHcq}
	f(\exp^\bX_{\exp^\bX_x w} (v))
\sim f(\exp_x Z(v,w)),
\end{equation}
where $\exp_x Z $ denotes\ the 1-time flow of a vector field $Z$ on $U$ with initial data $x$ (see Section~\ref{subsec_expX})  and  $Z(v,w)$  formally given by the BCH formula in \eqref{eq_fBCH} at $x$;
in the BCH, 
the adjoint operation $\ad$ refers to the commutator bracket between vector fields, $a=\sum_i w_iX_i$ and $b=\sum_i v_i X_i$.
The meaning of \eqref{eq_BCHcq} is that for any $N\in \bN_0$, we have:
\begin{align}
	f(\exp^\bX_{\exp_x w} v)
	&= f( \exp_x(Z^{(N)}(v,w) ) +O(|v|^{N+1} +|w|^{N+1}),\nonumber
	\\
		&= \sum_{k=0}^N \frac 1{k!} 
		\left((Z^{(N)}(v,w))^k f\right )( x ) +O(|v|^{N+1} +|w|^{N+1}),\label{eq_BCHcq1}
\end{align}
where $Z^{(N)}(v,w)$ denotes the truncation in $(v,w)$ at order $N$ of the formal series defining $Z(v,w)$, 
i.e. with Notation \ref{notation_BCH}:
$$
Z^{(N)}(v,w)={\rm BCH}_N (\sum_i w_iX_i,\sum_i v_i X_i,\ad).
$$
Note that $Z^{(N)}(v,w)$ is indeed a vector field on $U$, 
and its $\bX$-coordinates, i.e. 
$$
z^{(N)}(x;v,w) = (z_j^{(N)}(x;v,w))_{j=1}^n,\qquad
Z^{(N)}(v,w)(x) =\sum_{j=1}^n z_j^{(N)}(x;v,w) X_j(x),
$$
are polynomial in $v,w$ and smooth in $x\in U$. 
Because of this dependence in $x$, in general, we expect that $\exp_x^\bX (z^{(N)}(x;v,w))$ be different from $\exp_x (Z^{(N)}(v,w))$.
\medskip 

We will need the following two observations. 
\begin{enumerate}
\item [(i)] Firstly, Corollary \ref{cor_prop_adxbX} implies that 
$$
z^{(N)}(x;v,w) = \widetilde  z^{(N)}(x;v,w)  \, + \, r^{(N)}(x;v,w), 
$$
 where $r^{(N)}$ and $ \widetilde z^{(N)}$ are $N$-truncated power series   in $v,w$ with smooth coefficients in $x\in U$ with  $r^{(N)}$ of higher order and $\widetilde z^{(N)}$ 
 given by BCH for $\ad_x^\bX$.
 By  Proposition \ref{prop_adxbX},
 if $N>s$ where $s$ is the step of the gradation, we have
 $$
 \widetilde z^{(N)} := {\rm BCH}_N (w,v;\ad_x^\bX)=w*^\bX_x v.
 $$

\item[(ii)] Secondly, omitting the writing of the dependence in $v,w$, 
we have
\begin{align*}
    (Z^{(N)})^2 &=\sum_{j_1,\, j_2}\left(z_{j_1}^{(N)}X_{j_1}z_{j_2}^{(N)}X_{j_2}\right)\\
    &=\sum_{j_1,\, j_2}z_{j_1}^{(N)}z_{j_2}^{(N)} X_{j_1} X_{j_2} \ + \ \sum_{j_1}z_{j_1}^{(N)}\, \sum_j (X_{j_1}z_{j}^{(N)})\, X_{j}
    .
\end{align*}
We write
$\bX_I:=X_{i_1}\ldots X_{i_k}$
for any multi-index
$I=(i_1,\ldots,i_k)\in \{1,\ldots,n\}^k$
of length  $|I|=k$. Then, inductively, 
  we can write  for $k=2,3,\ldots$
\begin{equation}\label{eq:combinatoric}
(Z^{(N)})^k  =
\sum_{(j_1,\ldots,j_k)=[I,I_1,\ldots,I_k]}
(\bX_{I_1}z_{j_1}^{(N)})\ldots (\bX_{I_k}z_{j_k}^{(N)}) \bX_I,
 \end{equation}
where the sum is over the multi-indices $I, I_1,\ldots,I_k$ with 
at least one of the $I_1,\ldots,I_k$  empty but $I\neq \emptyset$, and such that their concatenation $[I, I_1, \ldots , I_k] = (j_1,\ldots, j_k)$ is in  $\{1,\ldots,n\}^k$ with $j_i\not \in I_i$ and $(j_1,\ldots, j_k)$ runs over $\{1,\ldots,n\}^k$. 
This relation is proved in the Appendix~\ref{app:combi}. 
 \end{enumerate}

\smallskip

We can apply the formula in \eqref{eq_BCHcq1} to functions $f$ that are valued in a finite dimensional vector space. We 
choose $f$ as $\ln^\bX_{x}$ locally.
Proceeding for instance as in the proof of Lemma~\ref{lem_Taylorexpxu}, 
we see that for $f=\ln^\bX_{x}$, we have $f(x)=0$
and for any $u\in \bR^n$
\begin{equation}
    \label{eq_sumvXf=v}
    \sum_j u_j X_j f(x) = u = \sum_j u_j e_j,
\end{equation}
with $(e_1,\ldots,e_n)$ denoting the canonical basis of $\bR^n$.
Moreover,
 for $k>1$, 
\begin{equation}
\label{eq:vanishing}
	\left(\sum_j u_j X_j\right)^k f(x) = \sum_{j_1,\ldots,j_k} u_{j_1}\ldots u_{j_k} (X_{j_1}\ldots X_{j_k} f)(x)=
 0. 
\end{equation}
In other words, the (non-commutative) polynomial expression
\begin{equation}\label{def:product_form}
\sum_{I=(j_1,\ldots,j_\ell)} u_{j_1}\ldots u_{j_\ell}\bX_{(j_1,\ldots,j_\ell)}
\end{equation}
vanishes when  applied to $f=\ln^\bX_{x}$ at $x$ if $\ell>1$. Hence, we focus in identifying such expressions in \eqref{eq:combinatoric}. For this we will use the third following observation:
\begin{enumerate}
    \item[(iii)] 
    Reorganising the terms of $(Z^{(N)})^k$ in terms of the cardinal  of the indices $I$ in equation~\eqref{eq:combinatoric}, and writing 
    \begin{equation}\label{eq:ZNk}
(Z^{(N)})^k=\sum_{\ell=1}^k\sum_{I=(j_1,\cdots , j_\ell)} a_I \bX_I,    
    \end{equation}
the terms $\sum_{I=(j_1,\cdots , j_\ell)} a_I \bX_I$ is the sum of terms that have the property~\eqref{def:product_form} for $\ell\geq 1$. 
This is proved in  the Appendix~\ref{app:combi}. 
\end{enumerate}

In view of (iii) and  of~\eqref{eq:vanishing}, when applying  $(Z^{(N)})^k$ written as in \eqref{eq:ZNk} to the function $f=\ln^\bX_{x}$ at $x$,  the $\ell$-th term  vanishes when $\ell >1$.
Hence,  we are left with the sum over the $I$'s with only   1 element:
\begin{align*}
(Z^{(N)})^k  f(x)
&=\sum_{j=1}^n\sum _{(j_1,\ldots,j_k)\in [\{j\},I_1,\cdots, I_k]}
(\bX_{I_1}z_{j_1})\ldots (\bX_{I_k} z_{j_k})X_jf(x)
\\
&=\sum_{j=1}^n\sum _{(j_1,\ldots,j_k)\in [\{j\},I_1,\cdots, I_k]}
(\bX_{I_1}z_{j_1})\ldots (\bX_{I_k} z_{j_k})e_j.
\end{align*}
If $k>1$, 
each term in the sum above is a product of $k$ $x$-derivatives of the higher order polynomials 
$z_{j_i}(x;v,w)$'s   in $(v,w)$;
so $(Z^{(N)})^k  f(x)$ is of higher order.

\smallskip

Therefore, using $f(x)=0$ and equation~\eqref{eq_sumvXf=v},
  we obtain  for any $N>s$
  $$
 \ln^\bX_x(\exp^\bX_{\exp^\bX_x w} (v))
  = z^{(N)}(x;v,w) +
 \sum_{k=2}^N \frac 1{k!} 
		\left(Z^{(N)}(v,w)^k f\right )( x ) +O(|v|^{N+1} +|w|^{N+1}),
 $$
 with $z^{(N)}(x;v,w)=\widetilde z(x;v,w)+r^{(N)}(x;v,w)$, 
 the polynomial
  $\widetilde z(x;v,w) =w*^\bX_x v$
  being homogeneous while  $r^{(N)}(x;v,w)$ and  every 
 $\left((Z^{(N)}(v,w))^k f\right )( x )$, $k>1$, are of higher order. The conclusion follows. 
\end{proof}

\subsubsection{Consequences of Theorem~\ref{lem_Cq1BCH}}

We\ will use Theorem \ref{lem_Cq1BCH} via the following  statement which will bring  key arguments while proving the main properties of the pseudodifferential calculus. This section can be skipped at first reading. We recall that when a frame $(\bX,U)$ has been fixed, for $x\in U$,  $*_x^\bX$ describes the law of the group $G_xM$. We use the shorthand $*_x^\bX=*_x$ in this section.

 \begin{corollary}
\label{cor_lem_Cq1BCH}
We continue with the setting of Theorem \ref{lem_Cq1BCH}. 
\begin{enumerate}
	\item We may write 
$$
-\ln^\bX_{\exp^\bX_{x} v} \exp^\bX_{x}( v*_{x}(-w) )
= w *_x \widetilde r(x;v,w);
$$
above, $\widetilde r$ is a smooth function on a 0-section  of $M\times \bR^n\times \bR^n$   where $(x,v,w)\mapsto -\ln^\bX_{\exp^\bX_{x} v} \exp^\bX_{x} (v*_{x}(-w))$ is defined and smooth. 
Moreover, $\widetilde r(x;v,w)$ is higher order in $(v,w)$ and vanishes at $v=0$ and at $w=0$,
so we  have locally
$$
 \left|\widetilde r(x;v,w) \right|_{\bX}
 =O(|v|_{\bX}), \ O(|w|_{\bX}), \ O \left(|v|_{\bX}+|w|_{\bX}\right)^{1+\frac 1{\upsilon_n}}.
 $$
\item We may write 
$$
-\ln^\bX_{x} \exp^\bX_{\exp^\bX_x -v}( v*_x (-w)  )
= w *_x \check r(x;v,w);
$$
above $\check r$ is a smooth function on a 0-section  of $M\times \bR^n\times \bR^n$   where $(x,v,w)\mapsto -\ln^\bX_{x} \exp^\bX_{\exp^\bX_x -v} ((-w) *_x v) $ is defined and smooth. 
Moreover, $\check  r(x;v,w)$ is higher order in $(v,w)$ and vanishes at $v=0$ and at $w=0$,
so we  have locally
$$
 \left|\check r(x;v,w) \right|_{\bX}
 =O(|v|_{\bX}), \  \ O(|w|_{\bX}), \ \  O \left(|v|_{\bX}+|w|_{\bX}\right)^{1+\frac 1{\upsilon_n}}.
 $$
\end{enumerate}
\end{corollary}
Above, we have used the quasinorm $|\cdot|_\bX$ defined in Remark \ref{rem_quasinormbX}.
 
Before entering into the proof of Corollary~\ref{cor_lem_Cq1BCH}, let us explain what is meant by `locally' in the statement.
Actually, the function $\widetilde r$ above is simply defined as 
\begin{equation}\label{def:tilde_v}
\widetilde r(x;v,w) = (-w)*_x 
\left(-\ln^\bX_{\exp^\bX_{x} v} \exp^\bX_{x}( v*_{x}(-w))\right), 
\end{equation}
and the fact that the inequality above holds locally  means that for any compact subset $\cK$ of~$U_\bX$ satisfying
$$
 \forall x\in U, \ w,v\in\bR^n,\quad
(x,w), (x,v) \in \mathcal K\Longrightarrow
 (\exp^\bX_{x} v , \exp^\bX_{x} (v*_{x}(-w) ) )\in \exp^\bX(U_\bX),
 $$
 there exists $C>0$ such that for any $x\in U$, $v,w\in \mathbb R^n$ with $(x,w)$ and $(x,v)$ lying in $\mathcal K$ we have 
\begin{align*}
\left|\widetilde r(x;v,w)\right|_{\bX}
\leq C |v|_{\bX}, \quad \left|r_1(x;v,w)\right|_{\bX}
\leq C |w|_{\bX},
\\
\left|\widetilde r(x;v,w)\right|_{\bX}
\leq C \left(|v|_{\bX}+|w|_{\bX}\right)^{1+\frac 1{\upsilon_n}},	
\end{align*}
the constant $C$ depends on $N,\mathcal K, \bX,M$, while the hypotheses on $\mathcal K$, besides its compactness, ensures that the left-hand side of the inequality above makes sense. 
The meaning of the local estimates in Part 2 of Corollary \ref{cor_lem_Cq1BCH} is similar to the one of Part 1.

\begin{proof}[Proof of Corollary \ref{cor_lem_Cq1BCH}
] 
By Theorem \ref{lem_Cq1BCH} (and using its notation),  we may write:
$$
\ln^\bX_{\exp^\bX_{x} v} \exp^\bX_{x}( v*_{x}(-w) ) 
=  r (\exp_x v; v *_x (-w),-v) \ - \ w,
$$
with $r$ higher order. 
The function $\tilde r$  defined via~\eqref{def:tilde_v}  is smooth
and also higher order 
by  
Lemmata \ref{lem_group_laws} and \ref{lem:exp_derivative_bis}.
We also have 
\begin{align*}
	\widetilde r(x;v,0)&=\ln^\bX_{\exp^\bX_{x} v} \exp^\bX_{x} v =0,
\\
\widetilde r(x;0,w)&=(-w)*_x (\ln^\bX_x \exp^\bX_{x} w ) = (-w)*_x w=0.
\end{align*}
The first two local estimates follow from applying the Mean Value Theorem of~\cite{folland+stein_82} (see Theorem~\ref{thm_MV+TaylorG} (1) and Remark \ref{remthm_MV+TaylorG}). 
For the third local estimate, we apply 
 Corollary \ref{cor1lem_char_higherorder}
 to $\cV_1=\bR^n\times \bR^n$ 
 and $\cV_2=\bR^n$, 
 with $|\cdot|_\bX$ the quasinorm on $\cV_2$, 
 and $(v,w)\mapsto |v|+|w|$ the quasinorm on $\cV_1$; the estimate is indeed  locally uniform by Corollary \ref{cor_derivatives_cV} (1) and the proof of  Corollary \ref{cor1lem_char_higherorder}.
  This shows Part 1.

For Part 2, we proceed as for Part 1 having observed the two  following formulae:
\begin{align*}
  \check r(x;v,w) &= (-w) *_x
\left( -\ln^\bX_{x} \exp^\bX_{\exp^\bX_x -v} v*_x(-w)   \right),\\
\ln^\bX_{x} \exp^\bX_{\exp^\bX_x -v} v*_x(-w) 
&= -w + r(x;v*_x (-w) ,-v).
\end{align*}
\end{proof}

%%%%%%%%%%%%%%%%%%%%%%%%%%%%%%
\section{Symbols classes on filtered manifolds}\label{sec:symbols}

In this section, we introduce the notion of symbols on filtered manifolds in Section~\ref{subsec_symbol+k} and define  symbol classes  in Section~\ref{sec:def_symbol}. We study their algebraic properties  in Section~\ref{sec:op_symbols} and focus on two important types of symbols: on the one hand, spectral multipliers in Section~\ref{sec:spectral_multipliers} and, on the other hand, polyhomogeneous symbols in Section~\ref{subsec_polyhom}.
 The quantization process  of these symbols will be described  in Section~\ref{sec:quantization}, after studying their convolution kernels in Section~\ref{sec:kernel_estmates}.
 
\subsection{The symbols and their  convolution kernels}
\label{subsec_symbol+k}

\subsubsection{Definitions}
\label{subsubsec_defsymbol+k}

\begin{definition}
A \emph{symbol} $\sigma$ on $M$ is a collection of invariant symbols 
	$$
	\sigma (x)= \{\sigma (x,\pi):\cH_\pi^{+\infty} \to \cH_\pi^{-\infty} : \pi\in \widehat {G}_x M\}
	$$
	 on $\widehat {G}_x M$ parametrised by $x\in M$. 
\end{definition}

Alternatively, we will see a symbol as a collection of fields of operators
$$
\sigma = \{\sigma (x,\pi):\cH_\pi^{+\infty} \to \cH_\pi^{-\infty} : x \in M, \ \pi\in \Gh_x M\}. 
$$
The operators are defined up to the unitary equivalence of representation (see Section \ref{subsec_cFG}) and are
parametrised by 
$$
\Gh M := \{(x,\pi) : x \in M, \ \pi\in \Gh_x M\}.
$$

\smallskip

Our definition of convolution kernels relies on smooth Haar systems whose definition we now recall:

\begin{definition}
 A smooth Haar system  $\mu=\{
	\mu_x\}_{x\in M}$ for $GM$ is a choice of Haar measure $\mu_x$ for each of the group fibers $G_x M$ such that the map 
 $x\mapsto \int_{G_x M} f(x,v) d\mu_x(v)$ is smooth on $M$ for any $f\in C_c^\infty(GM)$.   
\end{definition}

 \begin{remark}
 \label{rem_Haarsys1dens}
Alternatively, a smooth Haar system $\mu$ is the image under the group bundle exponential mapping $\Exp:\fg M\to GM$ of a positive 1-density of the vector bundle $\fg M$, i.e. $\mu= \Exp_* w$, $w\in |\Omega|^+(\fg M)$ (see Appendix \ref{subsec_vertdensityE} for these notions). 
 \end{remark}

\begin{definition}
    Let $\{
	\mu_x\}_{x\in M}$ be a smooth Haar system. Let $\sigma$ be a symbol on $M$.
    If  for every  $x\in M$, 
  the following makes sense
    \begin{equation}\label{eq:conv_ker}
\kappa^{\mu}_{\sigma,x} = \cF_{G_x M,\mu_x}^{-1} \sigma(x),
\end{equation}
where $\cF_{G_x M,\mu_x}$ denotes
the Fourier transform  on $G_x M$  with respect to the Haar measure $\mu_x$, 
 then we say that $\sigma$ \emph{admits a convolution kernel}.
The distribution $\kappa^{\mu}_{\sigma,x} \in \cS'(G_x M)$ is called the convolution kernel of $\sigma$ at $x\in M$ with respect to $\mu$. 
\end{definition}

\begin{ex}
\label{ex_symbolwithkappa}
    If $\sigma$ is a symbol such that 
for any $x\in M$, we have $\sigma (x)\in L^\infty_{a,b}(\widehat {G}_x M)$ for some $a,b\in \bR$ (possibly depending on $x$), 
then $\sigma$ admits a convolution kernel.
For instance, this is the case for a symbol such that $\sigma(x)$ is in $S^m(\Gh_x M)$ for every $x\in M.$
\end{ex}

  \smallskip 

Let $\mu'=\{\mu'_x\}_{x\in M}$ be another smooth Haar system. Then $\mu'_x$ is another Haar measure for $G_x M$, 
which is therefore equal to $\mu_x$ up to a constant $c(x)>0$. This defines a positive smooth function $c\in C^\infty(M)$ satisfying 
\begin{equation}
    \label{eq_mumu'}
    \mu'_x = c(x)\mu_x.
\end{equation}
Hence, if $\sigma$ admits a convolution kernel with respect to one smooth Haar system $\mu$, then it admits a convolution kernel with respect to any smooth Haar system $\mu'$. 
 Moreover, even though $\kappa_{\sigma,x}^{\mu}$ depends on the choice of such a Haar system,  its existence is independent of which Haar system is chosen and we have:
\begin{equation}
\label{eq_kappamumu'}
 \kappa^{\mu}_{\sigma,x} = c(x)\kappa^{\mu'}_{\sigma,x}, 
 \quad x\in M.    
\end{equation}

\subsubsection{Convolution kernels in an adapted frame.}
If   $\bX$ is an adapted frame on an open set $U\subset M$, this induces a smooth Haar system $\mu^\bX = \{\mu^\bX_x\}_{x\in U}$  on $U$  
given by 
$$
\mu^{\bX} = (\Exp^\bX)_* dv, \quad\mbox{i.e.}\quad
\int_{G_x M} f\, d\mu_x^\bX
= \int_{\bR^n} f(\Exp_x^\bX v)\, dv,
$$
where $dv$ is the Lebesgue measure on $\bR^n$.

Let $\sigma$ be a symbol  on $M$ admitting a convolution kernel.
We consider 
the  convolution kernel 
$\kappa^{\mu^\bX}_{\sigma,x}=\kappa^{\mu^\bX}_{x} \in \cS'(G_x M)$ at $x\in U$
corresponding to the induced Haar measure $\mu^\bX$,
or, equivalently, its  composition with the group exponential mapping:
\begin{equation}
	\label{eq_kappaX}
\kappa_x^\bX (v):=  \kappa_x^{\mu^\bX} (\Exp^\bX_x v).
\end{equation}
Note that
$$
\kappa_x^\bX(v)
dv = 
(\Exp^\bX)^* (\kappa^{\mu^\bX}_x d\mu^\bX_x )
,
$$
and at least formally:
$$
\sigma(x,\pi)=\int_{\bR^n} \kappa_x^\bX (v) \pi(\Exp_x^\bX -v) dv,
$$
in the sense that $\sigma(x,\cdot)$ is the group Fourier transform of $\kappa^{\mu^\bX}_{x} \in \cS'(G_x M)$ for the Haar measure $\mu^\bX$. 

\begin{definition}
We call the map $\kappa_x^{\mu^\bX}\in \cS'(G_x M)$
or
 $\kappa_x^\bX \in \cS'(\bR^n)$ (as in~\eqref{eq_kappaX})
 the {\it convolution kernel in the $\bX$-coordinates} at $x$, 
 depending on whether we see this object as a distribution on $G_xM$ or on $\bR^n$.
\end{definition}

The relation between the convolution kernels in two different frames is explained in the following statement.
\begin{lemma}
\label{lem_XYkappa}
Let $\bX$ and $\bY$ be two adapted frames on an open subset $U$ of $M$.
Let $\sigma$ be a symbol admitting a convolution kernel.
Then, using the notation of Lemma~\ref{old_lem_XY},
 we have:
$$
\kappa_{\sigma,x}^\bX (v) = \det( \theta_x)\, \kappa_{\sigma,x}^\bY (\theta_x\,  v) , 
\quad\mbox{where} \ 
\theta_x :={\rm diag}\, T_x.
$$
\end{lemma}
In other words, considering the notation of Lemma \ref{lem_compwmorph}
and the isomorphism $ \theta_x ={\rm diag}\, T_x$
of the Lie group $G$, 
we have
$$
(\theta_x)_* \cF_{G_x M} (\kappa_{\sigma,x}^\bX \circ \Ln_x^\bX)
= \cF_{G_x M} (\kappa_{\sigma,x}^\bY \circ \Ln_x^\bY).
$$
\begin{proof}
The formula follows from 
    $$
\sigma(x,\pi)=\int_{\bR^n} \kappa_{\sigma,x}^\bX (v) \pi(\Exp_x^\bX v) dv=\int_{\bR^n} \kappa_{\sigma,x}^\bY (v_1) \pi(\Exp_x^\bY v_1) dv_1,
$$
and Lemma  \ref{lem_XYExp}.
\end{proof}

\subsubsection{First examples}\label{sec_first_examples}
Let us describe some examples that we will be develop along this section.

\begin{ex}
\label{ex_fId}
A function  $f\in C^\infty(M)$ defines 
    the symbol 
    $$
    f \id = \{f(x) \id_{\cH_\pi}, (x,\pi) \in \Gh M\},
    $$
    with convolution kernel $\kappa_f^\bX:=f(x)\delta_0$ in any $\bX$-coordinates.
\end{ex}

The following example  follows readily from  Example \ref{ex_XalphainSGh}:
\begin{ex}
\label{ex_widehatlangleXrangle}
The symbol 
$$
\widehat {\langle X_j\rangle} := \{\pi(\langle X_j\rangle_x ) : x\in U, \pi\in \widehat G_x M\}
$$
admits $\kappa_{\widehat {\langle X_j\rangle}}^\bX:=\partial_{v_j}\delta_0$
as convolution kernel in the $\bX$-coordinates.
\end{ex}

In order to generalise this example to higher order, we introduce the following notation:
\begin{notation}\label{notation_LX}
We denote by 
 $L_{\langle X_j\rangle_{x}}$ and $R_{\langle X_j\rangle_{x}}$ 
the vector fields of $\bR^n$ corresponding to 
 the left and right-invariant vector fields associated with $\langle X_j\rangle_{x}\in \fg_{x} M$ on the  Lie group $G_{x}M$: for $\varphi\in\mathcal S(\bR^n)$ 
 $$
L_{\langle \bX \rangle,j,x} \varphi (v) = (\langle X_j\rangle (\varphi \circ \Ln^\bX) )(\Exp^\bX_x v)
\quad\mbox{and}\quad
R_{\langle \bX \rangle,j,x} \varphi (v) = (\widetilde{\langle X_j\rangle} (\varphi \circ \Ln^\bX) )(\Exp^\bX_x v).
$$ 
We shall use the short-hand:
$$
L_{\langle \bX\rangle_x}^{\beta}
:=
L_{\langle X_1\rangle_{x}}^{\beta_{1}}
\ldots L_{\langle X_n\rangle_{x}}^{\beta_{n}},\qquad 
R_{\langle \bX\rangle}^{\gamma}
:= 
L_{\langle X_1\rangle_{x}}^{\gamma_1}
\ldots L_{\langle X_n\rangle_{x}}^{\gamma_{n}}.
$$
\end{notation}

\begin{ex}
\label{ex_widehatlangleXrangle_bis}
Revisiting Example~\ref{ex_widehatlangleXrangle}, we observe that 
the symbol 
$$
\widehat {\langle \bX\rangle}^\alpha := \{\pi(\langle X\rangle_x^\alpha ) : x\in U, \pi\in \widehat G_x M\}
$$
admits $\kappa^\bX_{\widehat {\langle \bX\rangle}^\alpha }:=L_{\langle \bX\rangle_x}^{\alpha}\delta_0$
as convolution kernel in the $\bX$-coordinates.
\end{ex}

\subsubsection{Smooth symbols and their derivatives}

\begin{definition}
Let $\sigma$ be a symbol on $M$ admitting a convolution kernel. 
The symbol $\sigma$ or its convolution kernel is  \emph{smooth in the variable} $x$ (or just smooth) when for any adapted frame $(\bX,U)$, the convolution kernel in the $\bX$-coordinates is smooth in $x$ in the sense that  
$$
\kappa^\bX  \in C^\infty(U, \cS'(\bR^n)).
$$    
\end{definition}

A symbol admitting a convolution kernel is smooth in the variable $x$ on a  open subset $U\subset M$ small enough (to admit adapted frames) if there is one adapted frame $\bX$ on $U$
where $\kappa^\bX  \in C^\infty(U, \cS'(\bR^n))$. This notion is invariantly defined since Lemma \ref{lem_XYkappa} implies that
 $\kappa^\bY  \in C^\infty(U, \cS'(\bR^n))$ for any other adapted frame $\bY$ on $U$. We will give it an invariant meaning in Section \ref{subsubsec_invariantmeaning}  below.
 
 We observe that the space of smooth symbols is a module over $C^\infty(M)$.

\begin{notation}\label{notation_DXbeta}
   For any smooth symbol $\sigma$ and any multi-index $\beta\in \bN_0^n$, we adopt the notation 
$$
D_\bX^\beta \sigma 
$$
 for the symbol (when it exists)
whose convolution kernel in the $\bX$-coordinates is $\bX_x ^\beta \kappa_x^\bX$ at any $x\in U$.
If $|\beta|=1$, that is, if $\beta=e_j$ is the index which is 0 everywhere except at the $j$th entry where it is 1,  we may write $D_{X_j}$  instead of $D_\bX^\beta$. As we shall see later (see Remark~\ref{rem:LeibnizD}), the derivation~$D_\bX$ satisfies Leibniz rules.
\end{notation}

\subsubsection{Invariant meaning of convolution kernels}
\label{subsubsec_invariantmeaning} 
Here, we discuss the invariant definition of convolution kernels as distributional densities.
This viewpoint is  not required in practice in our subsequent analysis. However, it sheds a geometric light on the above considerations. We will use  
the language of vector bundles and of their functional spaces recalled in Appendix~\ref{secApp_VecBundleS}. 
We will apply it to the vector bundle $\fg M$ comprised of the osculating Lie algebras.

We   observe that an adapted frame $(\bX, U)$ yields a local trivialisation of the vector bundle $\fg M$, that is, 
\begin{equation}
\label{eq_phibX}
\phi_\bX:U\times \bR^n \to \fg M, 
\quad
\phi_\bX(x,v) = (x,\sum_j v_j \langle X_j\rangle_x ).
\end{equation}
Conversely, given a local trivialisation of the vector bundle $\fg M$, we may operate a linear change of coordinates and assume that it comes from an adapted frame. 
The properties of the change of frames in 
Lemma \ref{lem_XYkappa} together with  
 \eqref{eq_kappamumu'} may be summarised in the following statement:
\begin{lemma}
\label{lem_convkerD}
 If $\sigma$ is a smooth symbol on $M$, 
then  there exists a unique tempered vertical density $\kappa$ in $\cS'(\fg M, |\Omega|(\fg M))$ such that 
$\phi_\bX^* \kappa  (x,v)= \kappa^\bX_x (v) |dv|$ for any  trivialisation
$\phi_\bX$ induced by an adapted frame $(\bX,U)$ as in  \eqref{eq_phibX}.
\end{lemma}

It is  more natural for our future analysis to consider objects on the group bundle $GM$ rather than on the algebra bundle $\fg M$.
This was already the case for  smooth Haar systems. Indeed, we preferred them over  1-densities on $\fg M$, although 
the two notions are in one-to-one correspondence via  the group bundle exponential mapping $\Exp : \fg M \to GM$
(see  Remark~\ref{rem_Haarsys1dens}):
$$
\{\mu \ \mbox{smooth Haar system}\} = \Exp^* |\Omega|^+(\fg M).
$$
For this reason, we define the convolution kernel of a smooth symbol in the following way:

\begin{definition}
    The \emph{convolution kernel density} (or convolution kernel for short) of a  smooth symbol $\sigma$  on $M$ is 
  $\tilde \kappa:=\Exp^*\kappa$ with $\kappa$ as in Lemma \ref{lem_convkerD}.  
\end{definition}

We have
$$
(\Exp_x^\bX)_* \tilde\kappa_{x}\, (v)  = 
(\Ln^\bX_x)^* \tilde\kappa_{x} (v) = \kappa_x^\bX (v) |dv|,
$$   
and at least formally, 
$$
\sigma(x,\pi)=\int_{G_x M} \tilde\kappa_{x} (w) \,\pi(w)^*.
$$

We  define the space
$$
\cS'(GM;|\Omega|(GM)):=
\Exp^* \cS'(\fg M; |\Omega|(\fg M)) = 
\{\Exp^*\kappa, \kappa \in \cS'(\fg M; |\Omega|(\fg M))\},
$$
and call it the space of convolution kernel densities.

\begin{definition}
\label{def_mhomS'GMvert}
Let $m\in \bR.$
    An element $\tilde \kappa\in \cS'(GM;|\Omega|(GM))$ is $m$-homogeneous when $\kappa = \Ln^* \tilde \kappa \in \cS'(\fg M;|\Omega|(\fg M))$
    is $m$-homogeneous (as defined via Lemma \ref{lem_mhomcS'Evert}). 
\end{definition}

Given a smooth Haar system $\mu $, we denote by $\kappa^\mu_x\in \cS'(G_x M)$ the distribution corresponding to $\kappa = \Ln^* \tilde \kappa$ and  $\Ln^* \mu \in |\Omega|(\fg M)$ via 
Lemma \ref{lem_kappaomegax}. 
In the case of a convolution kernel densities associated with a smooth kernel, this coincides with the notation defined in Section~\ref{subsubsec_defsymbol+k}.

 We check that $\tilde \kappa\in \cS'(GM;|\Omega|(GM))$ is $m$-homogeneous when the distribution $\kappa^\mu_x\in \cS'(G_x M)$ is $m$-homogeneous.

\smallskip

It will be handy to consider Schwartz convolution kernels densities:
\begin{equation}
    \label{eq_SconvkerD}
    \cS(GM; |\Omega|(GM)):=\Exp^* \cS(\fg M; |\Omega|(\fg M)) = \{\kappa\circ \Exp, \kappa \in \cS(\fg M; |\Omega|(\fg M))\}.
\end{equation}
Both $ \cS'(GM; |\Omega|(GM))$ and $ \cS(GM; |\Omega|(GM)))$ are naturally equipped with structures of topological vector spaces, with $ \cS(GM; |\Omega|(GM))$  a dense subspace of $ \cS'(GM; |\Omega|(GM))$.

\subsubsection{Semi-norms and difference operators}\label{subsubsec:diffop}
In Example \ref{ex_symbolwithkappa}, we have considered symbols belonging in  $S^m(\Gh_x M)$ associated with each $x\in M$.
By~\eqref{eq_kappamumu'}, the membership of $\sigma(x,\cdot)$ in $S^m(\Gh_x M)$ is a property that holds independently of the choice of a smooth Haar system $\mu$.
However,  its $S^m(\Gh_x M)$-semi-norms require fixing  a graded basis of the Lie algebra and a Rockland operator on $G_x M$.
We can now consider semi-norms  of $S^m(\Gh_x M)$ defined in~\eqref{def_norm_symbol_1} using adapted frames: 
\begin{lemma}
\label{lem_SmnormsigmaxRx}
Let $\bX$ and $\bY$ be two adapted frames on an open subset $U$ of $M$. 
For any $x\in U$, any $N\in \bN$ and any Rockland operator $\cR$ on $G_x M$, 
      there exists a constant $C$ such that for any symbol $\sigma$ admitting a convolution kernel, we have 
      $$
      \|\sigma(x,\cdot)\|_{S^m (\Gh_x M),N, \langle \bY\rangle_x,\cR} 
      \leq C \|\sigma(x,\cdot)\|_{S^m (\Gh_x M),N, \langle \bX\rangle_x,\cR}.
      $$
\end{lemma}
In the inequalities above, the semi-norms may not exist in the sense that they are infinite. The meaning of the estimates above is that if the right-hand side is finite, then so is the left-hand side and the inequality holds. This will be so for all the lemmata in this section.
\begin{proof}
The result follows from Lemmata \ref{lem_XYExp} and \ref{lem_XYkappa} and  Proposition \ref{prop_compwmorph}.  
\end{proof}

In Lemma~\ref{lem_SmnormsigmaxRx}, we have used implicitly the difference operators defined in each group $G_xM$. We can 
  extend the definition of the difference operators in the following way.

\begin{definition}\label{def:rock_on_M}
Let  $\bX$ be an  adapted frame  on an open subset $U\subset M$ and let $\sigma$ be a symbol on $U$ admitting a convolution kernel.
If $\sigma (x,\cdot)\in S^m (\Gh_x M)$ for some $m\in \bR\cup\{-\infty\}$ and any $x\in U$, 
then for any $\alpha\in \bN_0^n$, 
$\Delta^\alpha_\bX \sigma$
is the symbol on $U$ whose convolution kernel in the $\bX$-coordinates is 
$u^\alpha \kappa_{\sigma,x}^\bX (u)$.
\end{definition}

\subsubsection{Rockland operators over $GM$}\label{sec:rockland}
The following definition will allow us to choose a positive Rockland operator above each group fiber $G_x M$ depending smoothly in $x\in M$:
\begin{definition} \label{def:Rockland}
An element $\cR\in \Gamma (\sU (\fg M))$ is \emph{Rockland} on $G M$,
resp. \emph{positive Rockland} on $G M$, when each $\cR_x$ is so on the fiber $\sU(\fg_x M)$.
Denoting by  $\nu$ its degree of homogeneity, 
we then call {\it Rockland symbol} and denote by $\pi(\mathcal R)$ the collection 
$$
\pi(\mathcal R_x)\in S^\nu(\widehat G_x M),\;x\in M,\; \pi\in \widehat G_xM.
$$
\end{definition}

It follows from Section \ref{subsec_X} that any Rockland operator on $GM$ may be written with respect to any adapted frame $(\bX,U)$ as $\cR = \sum_{[\alpha]=\nu} c_\alpha \langle\bX\rangle^\alpha$.
The homogeneous degree $\nu$ is independent of the frame $\bX$.
\smallskip 

We will often use in this text the 
 positive Rockland element  on $G M|_U$  given by 
\begin{equation}\label{ex_RU}
\cR_{\bX} = \sum_j (-1)^{\frac {M_0}{\upsilon_j}} \langle X_{j}\rangle^{2\frac{M_0} {\upsilon_j}},
\end{equation}
where  $M_0$ is a (fixed) multiple of $\upsilon_1,\ldots,\upsilon_n$ that we will generally choose as 
the least common multiple of $\upsilon_1,\ldots,\upsilon_n$.
The resulting Rockland symbol is 
$$
\widehat \cR_\bX
= 
\{\pi(\cR_{\bX,x}), (x,\pi)\in \Gh M|_U\}.
$$
From this, we can readily construct a positive Rockland element on $GM$, see Remark \ref{rem_existenceR}.

  \begin{remark}
 \label{rem_subLaplacianSymbol}  
 Although we will not consider the particular case  of an equiregular subRiemannian manifold in detail in this paper, let us mention that
the more natural and global positive Rockland operator is given by the subLaplacian symbol in that setting. Indeed, 
 recall that
 a regular subRiemannian manifold  $(M, \cD, g)$ is 
 naturally a filtered manifold with, furthermore, $H_1=\cD$ equipped with a metric $g$ (see Example~\ref{ex:subrie}).
 The subLaplacian symbol $\widehat \sL$ is then globally defined as follows: on any open set $U\subset M$ equipped with an adapted frame $\bX = (X_1,\ldots, X_n)$, it is given by
    $$
    \widehat {\sL}|_U = -\sum_{j=1}^{d_1} \widehat X_j^2,
    $$
    with  $X_{1,x},\ldots,X_{d_1,x}$ a $g_x$-orthonormal basis of $H_1(x)$  (if this is not so, a Graham-Schmidt procedure allows us to assume so). We check readily that it defines a global object that is the symbol of a positive Rockland operator on $GM$.
  \end{remark} 

\smallskip

The next lemma shows that  the semi-norms involved in the definition of  the symbol classes $S^m(G_xM)$ are uniform in $x$, once chosen a Rockland operator on~$M$, in the sense of  Definition~\ref{def:rock_on_M}.
This point will be important below (see Section~\ref{sec:def_symbol}) for defining global symbol classes on $M$.

\begin{lemma}
\label{lem_SmnormsigmaxRS}
Let $\bX$ and $\bY$ be two adapted frames on an open subset $U$ of $M$.
 Let $\cR$ and~$\cS$ be two positive Rockland elements on $GM|_U$.
    For any $x\in U$ and any $N\in \bN$, 
      there exists a constant $C$ such that for any symbol $\sigma$ admitting a convolution kernel, we have 
      $$
      \|\sigma(x,\cdot)\|_{S^m (\Gh_x M),N, \langle \bY\rangle_x,\cR_x} 
      \leq C \|\sigma(x,\cdot)\|_{S^m (\Gh_x M),N, \langle \bX\rangle_x,\cS_x}.
      $$  
      Moreover, the constant $C$ depends  locally uniformly on $x\in U$, that is, it is uniform when $x$ remains in a compact subset $\cC$ of $U$. 
\end{lemma}

\begin{proof}
This follows from    the properties of Rockland operators, 
    see Proposition \ref{prop_sob} (6).
    The fact that the constant depends only locally uniformly in $x$ is a consequence of  Remarks~\ref{rem_constant_RS} and~\ref{rem_constant_theta}
    following Propositions \ref{prop_sob}
    and \ref{prop_compwmorph} respectively. 
\end{proof}

\subsubsection{Derivatives and $S^m(G_x M)$-membership}
The existence of the $x$-derivatives of a symbol in $S^m(G_x M)$  is independent of a choice of frame in the following sense:
\begin{lemma}
\label{lem_DXYsigma}
Let $\bX$ and $\bY$ be two adapted frames on an open subset $U$ of $M$. 
If  $\sigma$ is a smooth symbol such that $D_{X_j}\sigma $ makes sense for any $j=1,\ldots,n$ with 
$D_{X_j}\sigma (x,\cdot)\in S^m(\Gh_x M)$
for any $x\in U$, then 
$D_{Y_\ell}\sigma$ also makes sense for any $\ell=1,\ldots, n$ with 
$D_{Y_\ell}\sigma(x,\cdot)\in S^m(\Gh_x M)$ for any $x\in U$.
Moreover, fixing a Rockland element $\cR$ on $GM|_U$,  we have 
\begin{align*}
 &\|D_{Y_\ell}\sigma(x,\cdot)\|_{S^m (\Gh_x M),N, \langle \bX\rangle_x,\cR_x} 
    \\&\qquad   \leq C\Bigl( \|\sigma(x,\cdot)\|_{S^m (\Gh_x M),N', \langle \bX\rangle_x,\cR_x}+ 
\sum_j\|D_{X_j}\sigma(x,\cdot)\|_{S^m (\Gh_x M),N, \langle \bX\rangle_x,\cR_x} \Bigr),    
\end{align*}
where $N'\in \bN_0$ and  the constant $C>0$ are independent of $\sigma$, $N'$ is independent of $x\in U$ while $C$  depends on $x\in U$ locally uniformly.
\end{lemma}

\begin{proof}
As in the proof of Lemma~\ref{old_lem_XY} (see~\eqref{ironman}),  we may write 
$$
Y_\ell = \sum_{j: \upsilon_j \leq \upsilon_\ell } c_j X_j, \qquad c_j \in C^\infty(U).
$$
Therefore, by Lemma \ref{lem_XYkappa},
$$
Y_{\ell,x} \kappa_{\sigma,x}^\bY 
= 
\sum_j c_j (x) X_{j,x} \kappa_{\sigma,x}^\bY
= 
\sum_j c_j (x) X_{j,x} 
(\det \theta_x^{-1} \  \kappa_{\sigma,x}^\bX \circ \theta_x^{-1}),
$$
and we compute
\begin{align*}
& X_{j,x} 
(\det \theta_x^{-1} \  \kappa_{\sigma,x}^\bX \circ \theta_x^{-1})
\\&\qquad =   
 (X_{j,x} \det \theta_x^{-1})
\kappa_{\sigma,x}^\bX \circ \theta_x^{-1}
+\det \theta_x^{-1} \  (X_{j,x}\kappa_{\sigma,x}^\bX) \circ \theta_x^{-1}
+
\det \theta_x^{-1} \  X_{j,x_1=x}\kappa_{\sigma,x}^\bX \circ \theta_{x_1}^{-1}.
\end{align*}
For the last expression, 
vector calculus in $\bR^n$ yields
$$
X_{j,x_1=x}\,\kappa_{\sigma,x}^\bX (\theta_{x_1}^{-1}(v))
= 
\sum_{i=1}^n (X_{j,x} \theta_x^{-1}(v))_{i}
(\partial_i \kappa_{\sigma,x}^\bX) (\theta_{x_1}^{-1}(v)).
$$
Now, recall that $\theta_x:\bR^n\to \bR^n$ is linear and block diagonal for every $x\in U$, so $(X_{j,x} \theta_x^{-1}(v))_{i}$ is a polynomial in $v$ that is homogeneous of degree $\upsilon_i$ and  smooth in $x\in U$.
Furthermore,  using \cite[Proposition 1.26]{folland+stein_82} or \cite[Section 3.1.5]{R+F_monograph}, we have
$$
\partial_i = L_{\langle X_i\rangle_x}
+\sum_{i': \upsilon_{i'}<\upsilon_i} p_{i,i'}(x,v) L_{\langle X_{i'}\rangle_x},
$$
where $p_{i,i'}(x,v)$ is a  polynomial in $v$ that is $(\upsilon_{i}-\upsilon_{i'})$ homogeneous for the graded structure, and smooth in $x$. We recall that the vector fields $L_{\langle X_i\rangle_x}$ have been introduced in Notation~\ref{notation_LX}.
Hence, we obtain
\begin{align*}
 X_{j,x_1=x}\,\kappa_{\sigma,x}^\bX (\theta_{x_1}^{-1}(v))
&=\sum_i (X_{j,x} \theta_x^{-1}(v))_{i}
\Bigl(L_{\langle X_i\rangle_x}\kappa_{\sigma,x}^\bX
+\sum_{i': \upsilon_{i'}<\upsilon_i} p_{i,i'}(x,v) L_{\langle X_{i'}\rangle_x}\kappa_{\sigma,x}^\bX\Bigr) (\theta_{x_1}^{-1}(v))
\\&=\sum_{k=1}^n 
q_{j,k}(x,v) (L_{\langle X_k\rangle_x}
\kappa_{\sigma,x}^\bX) (\theta_{x_1}^{-1}(v)),
\end{align*}
with $q_{j,k}(x,v)$ being a  polynomial in $v$ that is homogeneous for the graded structure of degree at least $\upsilon_k$, smooth in $x$.
We set $\tilde q_{j,k}(x,v)=q_{j,k}(x,\theta_x^{-1}v)$; 
it is a polynomial in $v$ smooth in $x$ with  the same properties as $q_{j,k}(x,v)$.
Denoting by 
$\rho_{j,k}$  the symbol with convolution kernel in the $\bX$-coordinates given by $
\tilde q_{j,k}(x, \cdot ) (L_{\langle X_k\rangle_x}
\kappa_{\sigma,x}^\bX) $, 
 the properties of $S^m(\Gh_x M)$ 
imply that for any $N\in \bN_0$, 
$$
\|\rho_{j,k}(x,\cdot)\|_{S^m (\Gh_x M),N, \langle \bX\rangle_x,\cR} 
      \leq C \|\sigma(x,\cdot)\|_{S^m (\Gh_x M),N', \langle \bX\rangle_x,\cR},
$$
where $N'\in \bN_0$ and  the constant $C$ are independent of $\sigma$
while $N'$ is independent of $x\in U$ and $C$  depends on it locally uniformly. This remains so even if we replace a fixed Rockland operator $\cR_x$ on $G_x M$ with a Rockland element on $GM|_U$ as in Lemma \ref{lem_SmnormsigmaxRS}.
Since we have 
    \begin{align*}
    D_{Y_\ell} \sigma(x,\cdot)& = \sum_j c_j(x) \frac{X_{j,x} \det \theta_x^{-1}}{\det \, \theta_x^{-1}}
    (\theta_x)_*\sigma(x,\cdot) + c_j(x)  (\theta_x)_* (D_{X_j}\sigma(x,\cdot))\\
    &\qquad+ 
    c_j(x) \sum_k  (\theta_x)_* \rho_{j,k} (x,\cdot),
    \end {align*}
    we conclude  as in the proofs of Lemmata \ref{lem_SmnormsigmaxRx} and \ref{lem_SmnormsigmaxRS}.
\end{proof}

\subsection{The classes of symbols $S^m(\Gh M)$}\label{sec:def_symbol}
We now have set the notation required for defining symbol classes on the filtered manifold~$M$.

\subsubsection{Definition of the class $S^m(\Gh M)$}
\label{subsec_SmGhM}

\begin{definition}\label{def:Sm}
		A symbol $\sigma$ is of order $m$ on $M$ when 
		\begin{enumerate}
			\item 
   $\sigma(x,\cdot)\in S^m(\widehat {G}_x M)$ for each $x\in M$ and $\sigma$ is smooth in the variable $x\in M$,
			\item  for any  adapted frame $U$ on an open subset $U\subset M$ and for any multi-index $\beta\in\bN_0^n$, the symbol  $D_\bX^\beta\sigma$ exists with 
   $D_\bX^\beta\sigma (x,\cdot)\in S^m(\widehat {G}_x M)$ for any $x\in U$,
			\item 
			 for any local adapted frame $\bX$ on $U\subset M$,
for any compact subset $\cC\subset U$, 
  for any positive Rockland element $\cR$ on $G M|_U$, and
		for any $N\in \bN_0$ and $\beta\in \bN_0^n$, the following quantity is finite:
		$$
\sup_{x\in \cC}		\|D_\bX^\beta\sigma(x,\cdot)\|_{S^m(\widehat {G}_x M), N,\langle \bX\rangle_x,\cR_x}<\infty.
$$
\end{enumerate}			
We denote by $S^m(\widehat G M)$ the space of symbols of order $m$ on $M$. 
\end{definition}

In Property (3), we use the semi-norms defined in~\eqref{def_norm_symbol_1} on the group $G_xM$ for the choice of graded basis $\langle \bX\rangle_x$, 
and we do so for each $x\in U$.
Property (1) implies that $\sigma$ admits a convolution kernel and is smooth.
By Lemmata \ref{lem_SmnormsigmaxRx} and \ref{lem_SmnormsigmaxRS}, and, inductively, by Lemma~\ref{lem_DXYsigma}, if Property (2) holds for one  adapted frame $\bX$ on an open subset $U\subset M$, it will hold for any other  adapted frame~$\bY$ on~$U$. 
This will be also true for Property (3)
for any other choice of positive Rockland operator $\cS$ on $GM|_U$
and we have
$$
\sup_{x\in \cC}		\|D_\bY^{\beta_1}\sigma(x,\cdot)\|_{S^m(\widehat {G}_x M), N,\langle \bY\rangle_x,\cS_x}
\leq C 
\sup_{x\in \cC}	
\sum_{|\beta|\leq |\beta_1|} \|D_\bX^\beta\sigma(x,\cdot)\|_{S^m(\widehat {G}_x M), N,\langle \bX\rangle_x,\cR_x},
$$
with a constant $C$ dependent on $M,H, N,\bX,\bY,\cR,\cS,\cC,\beta_1$ but not on $\sigma$.
\medskip

We equip $S^m(\widehat G M)$ with the  topology generated by the semi-norms given in Property (3) of  Definition~\ref{def:Sm}.
This topology is in fact Fr\'echet as we assume that the manifold $M$ is second countable. 
Indeed, 
if   $\bX$ is an adapted frame on an open set $U\subset M$ and if $\cC$ is a compact of $U$, we set 
\begin{equation}\label{def:semi_norm}
\|\sigma\|_{S^m(\widehat G M),(\bX,U), \cC, N}
:=
\max_{|\beta|\leq N}
\sup_{x\in \cC}		\|D_\bX^\beta\sigma(x,\cdot)\|_{S^m(\widehat {G}_x M), N,\langle \bX\rangle_x, \cR_{x}}, \qquad N\in \bN_0, 
\end{equation}
where $\cR_x=\cR_{\bX,x}$ 
 is the positive Rockland operator given in $x\in U$ by~\eqref{ex_RU}. 
We now fix a countable sequence of compact subsets $\cC_k$, $k=0,1,2,\ldots$
covering $M=\cup_{k} \cC_k$ and such that 
each compact is included in a bounded open subset $U_k$ on which an adapted frame $\bX_k$ exists.  
The continuous semi-norms given by 
\begin{equation}
	\label{eq_seminormSmGhMN}
	\|\sigma\|_{S^m(\widehat G M),N}
:=\max_{k\leq N}\|\sigma\|_{S^m(\widehat G M),(\bX_k,U_k), \cC_k, N}, \qquad N\in \bN_0,
\end{equation}
form a countable fundamental basis for the topology of $S^m(\widehat G M)$,  yielding the Fr\'echet topology. 

The properties of symbol classes on groups imply readily  the continuous inclusions:
\begin{equation}
\label{eq_inclusionSmGhM}
   m_1 \leq m_2 \Longrightarrow
S^{m_1}(\widehat G M)
\subset 
S^{m_2}(\widehat G M),
\end{equation} 
and the natural properties of interpolation for the spaces $S^m(\Gh)$, $m\in\bR$.

\subsubsection{Examples from Section~\ref{sec_first_examples}}\label{sec:example_bis}

If $f\in C^\infty(M)$ as in Example~\ref{ex_fId}, then 
    the symbol 
    $f \id = \{f(x) \ id_{\cH_\pi}, (x,\pi) \in \Gh M\}$ satisfies 
  $D_\bX^\beta (f \id) =(\bX^\beta f) \id$ 
  for all $\beta
  $ and $f\id \in S^0(\Gh M)$.
  
\smallskip

Similarly, the  symbol 
$
\widehat {\langle X_j\rangle} := \{\pi(\langle X_j\rangle_x ) : x\in U, \pi\in \widehat G_x M\}
$
introduced in Example~\ref{ex_widehatlangleXrangle}
satisfies $D_\bX^\beta \widehat {\langle X_j\rangle}=0$ for all $\beta
$, and 
$\{\pi(\langle X_j\rangle_x ) ,\pi\in \Gh\}\in S^{\upsilon_j}(\Gh_x M)$ for any $x\in U$.
Consequently, 
$\widehat {\langle X_j\rangle} 
\in S^{\upsilon_j}(\widehat G M|_U)$.
\smallskip 

\subsubsection{Smoothing symbols}
\begin{definition}
The class of smoothing symbols 
$$
S^{-\infty}(\widehat G M):= \cap_{m\in \bR} S^m(\widehat G M)
$$
is equipped  with the induced topology of projective limit.	
\end{definition}
This topology makes sense thanks to the continuous inclusions in \eqref{eq_inclusionSmGhM}.

The advantage of smoothing symbols is that their convolution kernels are given by Schwartz densities.
Indeed,
the convolution kernel $\kappa_{\sigma,x}^\bX (v)$ of a smoothing symbol $\sigma\in S^{-\infty}(\Gh M)$ in the coordinates of an adapted frame $(\bX,U)$ is Schwartz in $v$ and depends smoothly in $x\in U$.
Therefore, using the notation in \eqref{eq_SconvkerD},
$$
\kappa_\sigma\in \cS(G M, |\Omega|(GM)).
$$

\subsection{Operations on symbols}\label{sec:op_symbols}

We analyse the different operations that we can perform on symbols on filtered manifolds and that preserve their membership to the symbol classes defined above. This section relies on the analogue results for symbols and kernels on groups stated in Section~\ref{sec:groups}.
Indeed, by definition of our symbol classes on $M$, for any $\sigma\in S^m(\Gh M) $, any smooth system of Haar measure $\mu=\{\mu_x, x\in G\}$ and any $x\in G$, the symbol $\sigma(x)$ is in $S^m(G_x M)$ with convolution kernel $\kappa^\mu_{\sigma,x}$.

\subsubsection{Symbolic algebraic properties}

From the Leibniz property of vector fields and Theorem~\ref{thm_SmGh} together with Remark~\ref{rem_thm_SmGh}, it follows readily:
\begin{proposition}
\label{prop_comp+adj}
For any 
$ m_1 , m_2 , m \in \bR \cup \{-\infty\} $,
	the composition and adjoint operations
$$
\left\{\begin{array}{rcl}
S^{m_1}(\widehat G M) \times S^{m_2}(\widehat G M)
&\longrightarrow & S^{m_1+m_2}(\widehat G M)
\\
(\sigma_1,\sigma_2)& 	\longmapsto& \sigma_1\sigma_2
\end{array}
\right.
\;\; \mbox{and}\;\;
\left\{\begin{array}{rcl}
S^{m}(\widehat G M) 
&\longrightarrow & S^{m}(\widehat G M)
\\
\sigma& 	\longmapsto& \sigma^*
\end{array}
\right. 
$$
are continuous maps.
Moreover, for any smooth Haar system $\mu=\{\mu_x\}_{x\in M}$, 
the  convolution kernels of $\sigma_1\sigma_2$
and $\sigma^*$
are formally given by 
\begin{align}
\label{conv_kernel_product}
\kappa_{\sigma_1\sigma_2,x}^{\mu}(v)&=
\kappa_{\sigma_2,x}^{\mu} \star_{G_x M,\mu_x}
\kappa_{\sigma_1,x}^{\mu} (v) = 
\int_{G_x M}\kappa_{\sigma_2,x}^{\mu} (v w^{-1}) \ \kappa_{\sigma_1,x}^\mu(w) d\mu_x(w),\\
\label{conv_kernel_adjoint}
 \kappa_{\sigma^*,x}^{\mu}(v)&=\bar \kappa_{\sigma,x}^{\mu} (v^{-1}),
\end{align}
where $\star_{G_xM,\mu_x}$ denotes the convolution product on $G_x M$ with respect to the Haar measure $\mu_x$.
\end{proposition}

Note that the convolution kernel $\kappa_{\sigma_1\sigma_2,x}^{\mu}
$ in~\eqref{conv_kernel_product} is the convolution product   of  two tempered distributions 
on $G_x M$ for the Haar measure $\mu_x$. Moreover, it makes sense on the group $G_x M$ as each distribution can be written as the sum of a Schwartz function with a distribution with compact support,
see Theorem \ref{thm_kernelG}.

\begin{proof}[Proof of Proposition \ref{prop_comp+adj}]
Consider a smooth Haar system $\mu=\{\mu_x\}_{x\in M}$ and $x\in M.$
The properties of composition and taking the adjoint of symbols on groups imply that the symbol $\sigma_1(x)\sigma_2(x)$ on the group $G_x M$
is in $S^{m_1+m_2}(G_x M)$ with convolution kernel given by the convolution product 
$\kappa_{\sigma_2,x}^\mu\star_{G_x M,\mu_x} \kappa^\mu_{\sigma_1,x}$, 
see  Theorem \ref{thm_SmGh}.

If $\mu'=\{\mu'_x\}_{x\in M}$ is another smooth Haar system, it differs from $\mu=\{\mu_x\}_{x\in M}$ by a positive constant $c(x)$, see \eqref{eq_mumu'}, and so do the convolution kernels of symbols, see \eqref{eq_kappamumu'}. As we check readily
\begin{align*}
 \kappa_{\sigma_2,x}^\mu\star_{G_x M,\mu_x} \kappa^\mu_{\sigma_1,x}
&=
c(x) \ 
\kappa_{2,x}^{\mu'}\star_{G_x M,\mu'_x} \kappa^{\mu'}_{1,x},  \\
 \kappa_{\sigma^*,x}^{\mu}
 &=
c(x) \  \kappa_{\sigma^*,x}^{\mu'},
\end{align*}
the convolution kernels of the symbols $\sigma_1\sigma_2$ and $\sigma^*$ are well defined as densities independently of the choice of a smooth Haar system. This shows that these symbols satisfy Part (1) of Definition \ref{def:Sm}.

Once a frame~$\bX$ is fixed on an open subset $U\subset M$, then equation~\eqref{conv_kernel_product} reads for $v\in \bR^n$ at $x\in U$,
\begin{equation}\label{key_for_Leibniz}
 \kappa_{\sigma_1\sigma_2,x}^{\bX}(v)=
\int_{\bR^n}\kappa_{\sigma_2,x}^{\bX} (v w^{-1}) \ \kappa_{\sigma_1,x}^\bX(w) dw.
\end{equation}
Routine checks using the Leibniz properties of vector fields and Theorem \ref{thm_SmGh} imply that $\sigma^*$ and $\sigma_1\sigma_2$ satisfy 
Parts (2) and (3) of Definition \ref{def:Sm}.
\end{proof}

\begin{remark}
    \label{rem:LeibnizD}
    Note that the proof of Proposition~\ref{prop_comp+adj} (namely equation~\eqref{key_for_Leibniz}) shows that the derivation $D_\bX$ satisfies Leibniz rules in the sense that for $(\sigma_1,\sigma_2)\in S^{m_1}(\widehat GM)\times S^{m_2}(\widehat GM)$ and $\beta\in \bN^d$ with $|\beta|=1$, 
    \[ 
    D_\bX^\beta (\sigma_1\sigma_2)=(D_\bX^\beta \sigma_1)\sigma_2+\sigma_1(D_\bX^\beta \sigma_2).
    \]
\end{remark}

Recall that a smooth function $f\in C^\infty(M)$ gives rise to a symbol $f\id \in S^0(M)$, see Example \ref{ex_fId}.
For any $\sigma\in S^m(\Gh M)$, the symbol $f\sigma$ may be viewed as the composition of $f\id \in S^0(M)$ with $\sigma\in S^m(\Gh M)$.
Hence it is also in $S^m(\Gh M)$ by Proposition \ref{prop_comp+adj}.
In other words, $S^m(\Gh M)$ is stable under multiplication by $C^\infty(M)$. 
It is clearly stable under addition, hence $S^m(\Gh M)$ is a $C^\infty(M)$-module. By Proposition \ref{prop_comp+adj}, $S^m(\Gh M)$
is also stable under taking the adjoint and  the space of symbols $\cup_{m\in \bR}S^m (\Gh M)$ is an 
$*$-algebra over the ring $C^\infty(M)$.

\medskip

In Example \ref{ex_widehatlangleXrangle}, we showed that given an adapted frame $\bX$ on an open subset $U\subset M$, the symbol  $\widehat {\langle X_j\rangle}$  is in $S^{\upsilon_j}(\Gh M|_U)$.
We can describe more precisely the composition with $\widehat {\langle X_j\rangle}$.
\begin{lemma}
Let $\bX$ be an  adapted frame  on an open subset $U\subset M$. Let $m\in \bR\cup \{-\infty\}$.
If $\sigma \in S^m (\widehat G M|_U)$ and $ j=d_{i-1}+1,\ldots, d_i$,
then $\widehat {\langle X_j\rangle}\sigma $ 
and $\sigma \widehat {\langle X_j\rangle}$ are 
in $S^{m +w_j} (\widehat G M|_U)$
with respective convolution kernels in the $\bX$-coordinates:
$$
L_{\langle X_j\rangle_{x}}\kappa^\bX_{\sigma,x}
\quad\mbox{and}\quad
R_{\langle X_j\rangle_{x}}\kappa^\bX_{\sigma,x}.
$$
\end{lemma}
\begin{proof}
    This follows readily from the group case, see Remark \ref{rem_thm_kernelG} (2).
\end{proof}

The composition property allows us to generalise Examples \ref{ex_fId} and \ref{ex_widehatlangleXrangle_bis} with the following two further sets of examples:

\begin{ex}
 Any element in $\Gamma (H^{w_i}/H^{w_{i-1}})$ is in $S^{w_i}(\widehat G M)$, $i=1,2,\ldots$.
 Indeed, it can be written in some $\bX$-coordinates as a $C^\infty(M)$-combination of $\widehat {\langle X_j\rangle}$, $j=d_{i-1}+1,\ldots, d_i$, or equivalently, as a sum of products of  symbols $f \id \in S^0(\Gh M)$ (as in Example \ref{ex_fId}) with symbols $\widehat {\langle X_j\rangle} \in S^{\upsilon_j}(\Gh M)$ (as in Examples \ref{ex_widehatlangleXrangle} and \ref{ex_widehatlangleXrangle_bis}).
\end{ex}

\begin{ex}
\label{ex_widehatlangleXranglealpha}
If $\bX$ is an adapted frame on an open subset $U\subset M$,
for any $\alpha\in \bN_0$, 
the symbol 
$\widehat {\langle \bX\rangle}^\alpha$ 
is in $S^{[\alpha]}(\Gh M|_U)$ and
admits 
$$
\kappa_{\widehat {\langle \bX\rangle}^\alpha,x}^{\bX} =L_{\langle \bX_1\rangle_x}^{\alpha_1} \ldots L_{\langle \bX_n\rangle_x}^{\alpha_n}\delta_0=L_{\langle \bX\rangle_x}^{\alpha}\delta_0,
$$
as convolution kernel in the $\bX$-coordinates. 
In particular, $\id = \widehat {\langle \bX\rangle}^\alpha$ with $\alpha =0$ is an element of $\Gamma(\sU(\fg M))$  of zero-th order, with convolution kernel in the $\bX$-coordinates given by $\delta_0$.
\end{ex}

This may be generalised in the following observation:
\begin{lemma}
\label{lem_LTx}
Let $T\in \Gamma_{\leq m}(\sU(\fg M))$.
If $\bX$ is an adapted frame on an open subset $U\subset M$, 
then any element $T$ of $\Gamma_{\leq m}(\sU(\fg M))$ can be written uniquely as
$$
T = \sum_{[\alpha]\leq m} c_\alpha \ {\langle \bX\rangle}^\alpha, \qquad c_\alpha \in C^\infty (U).
$$
The symbol associated to $T$ then is 
$$
\widehat T = \sum_{[\alpha]\leq m} c_\alpha \widehat {\langle \bX\rangle}^\alpha
\in S^m(\widehat G M|_U),$$
and
   the convolution kernel of $\widehat T$ 
   is given  in the $\bX$-coordinates 
   by 
   $\kappa_{T,x}^\bX = L_{T_x} \delta_0$, where 
   $L_{T_x} = \sum_{[\alpha]\leq m} c_\alpha (x) L_{\langle \bX\rangle}^\alpha $ is the differential operator on $\bR^n$ identified with the left-invariant operator corresponding to the element $T_x\in \sU(\fg_x M)$, once $G_x M$ is identified with $\bR^n$ via $\Exp^\bX_x$.
\end{lemma}

\begin{proof}
The writing of $T$ is explained in Section \ref{subsubsec_adaptedFrame}. The rest of the statement  follows readily from 	
	the property of composition of symbols 
(Proposition \ref{prop_comp+adj}) and Examples \ref{ex_fId} and \ref{ex_widehatlangleXranglealpha}.
\end{proof}

\subsubsection{Difference operators}

The definition of the difference operators performed in Section~\ref{subsubsec:diffop} extends naturally in the following way.
For any   adapted frame $\bX$ on an open subset $U\subset M$ and any $\sigma \in S^m(\widehat GM|_U)$ where $m\in \bR\cup\{-\infty\}$, 
 for any $\alpha\in \bN_0^n$, 
we denote by $\Delta^\alpha_\bX \sigma$
the symbol on $U$ whose convolution kernel in the $\bX$-coordinates is 
$u^\alpha \kappa_{\sigma,x}^\bX (u)$.
We check readily:

\begin{lemma}\label{lem:prop_diff}
Let $\bX$ be an  adapted frame  on an open subset $U\subset M$. 
For any $\alpha\in \bN_0^n$ and $m\in \bR\cup\{-\infty\}$, the following map
is continuous
$$
\left\{
\begin{array}{rcl}
 S^m(\widehat G M) & \longrightarrow &S^{m-[\alpha]}(\widehat G M|_U)\\
\sigma & \longmapsto & \Delta_\bX^\alpha \sigma
\end{array}
\right.
 $$
\end{lemma}

If follows from the group case (see Lemma \ref{lem_LeibnizG}) that the difference operators $\Delta_\bX^{\alpha}$ satisfy a Leibniz property:
\begin{corollary}
\label{cor_GM_Leib}
Let $\bX$ be an adapted frame on an open subset $U$ of $M$.
       There exist  smooth functions $c^{(\alpha)}_{\alpha_1,\alpha_2} \in C^\infty (U)$, $\alpha,\alpha_1,\alpha_2\in \bN_0^n$ satisfying 
$$
\forall \sigma_1,\sigma_2\in S^{-\infty}(\widehat G M),\qquad 
\Delta_\bX^\alpha (\sigma_1\sigma_2) 
=
\sum_{[\alpha_1]+[\alpha_2]=[\alpha]} 
c^{(\alpha)}_{\alpha_1,\alpha_2} \
\Delta_\bX^{\alpha_1} \sigma_1\
\Delta_\bX^{\alpha_2}\sigma_2. 
$$
\end{corollary}

If also follows  from the group case 
(see Proposition \ref{prop_DeltaqG} and Remark \ref{rem_prop_DeltaqG}) that locally compactly supported functions on $GM$ (in the following sense) yields continuous generalised difference operators:

\begin{definition}
\label{def:loccompsupp}
A function  $\chi_0\in C^\infty (M\times \bR^{n'})$  is $M$-\emph{locally compactly supported}
 in $\bR^{n'}$ when for any compact subset $\cC\subset M$, there exists a compact neighbourhood $V\subset \bR^{n'}$ of $v=0$ such that 
$\chi_0(x,v)=0$ for $(x,v)\notin \cC \times V$.
\end{definition}

\begin{corollary}
\label{cor_GM_Deltaq}
Let $\bX$ be an adapted frame on an open subset $U$ of $M$.
 Let $q\in C^\infty(U\times \bR^n)$ be $U$-locally compactly supported. 
Then the operation $\Delta_q$ defined via
$$
\Delta_q \sigma (x,\pi)  = \cF_{G_x M, \mu^\bX} (q(x,\Ln_x^\bX \cdot) \kappa_x^{\mu_\bX}), \qquad \sigma = \cF_{G_x M,\mu^\bX}\kappa_x^{\mu_\bX}
$$
is a continuous operator $S^m(\widehat G M)\to S^{m}(\widehat G M)$ for any $m\in \bR\cup\{-\infty\}$.
Moreover, fixing an  adapted frame  $\bX$ on an open subset $U$ of $M$ and
 writing the Taylor series of $q(x,\cdot)$ at 0 as $\bT_0 q (x,v)\sim \sum_\alpha c_\alpha(x) v^\alpha$, produces smooth functions $c_\alpha\in C^\infty (U)$, $\alpha\in \bN_0^n$, and
the map 
	$$
\sigma \longmapsto	\Delta_q \sigma  - \sum_{[\alpha]\leq N} c_\alpha \Delta_\bX^\alpha  \sigma
$$
is continuous $S^m(\widehat G M|_U)\to S^{m-(N+1)}(\widehat G M|_U)$  for any $N\in \bN_0$.
\end{corollary}

We may summarise parts of Corollary \ref{cor_GM_Deltaq} as 
$$
\Delta_q  \sigma  \sim \sum_{\alpha\in \bN^n_0} c_\alpha \Delta_\bX^\alpha \sigma \quad\mbox{in}\ S^m(\Gh M|_U),
$$
having used the following natural notion of expansion of a symbol:

\begin{definition}\label{def:asymptotics}
 A symbol $\sigma \in S^m(\Gh M)$ admits the (inhomogeneous)  expansion 
\[
\sigma \sim \sum_{j=0} ^\infty \sigma_{m-j} \quad \mbox{in} \ S^m(\Gh M)
\]
when for all $N\in \bN_0$,
$\displaystyle{
\sigma -\sum_{j=0} ^N \sigma_{m-j}\in  S^{m-N-1}(\Gh M).
}$
\end{definition}

Given a sequence $(\sigma_j)$ as in Definition \ref{def:asymptotics}, we can construct a symbol $\sigma\in S^{m}(\widehat GM)$  following the classical ideas due to Borel  (see e.g. \cite[Section 5.5.1]{R+F_monograph}), that will satisfy $\sigma\sim \sum_j \sigma_j$.
Moreover, $\sigma$ is unique up to $S^{-\infty}(\widehat GM)$.
\smallskip

We will often apply Corollary \ref{cor_GM_Deltaq} with functions $q$ whose first terms in their Taylor series is zero. 
To describe this precisely, we  adapt the notion of functions vanishing at a certain order at 0 to the anisotropic setting:

\begin{definition}
\label{def_vanishingHOMorder}
Let $f$ be a (scalar-valued) function defined on an open subset of $\bR^n$ containing 0.
Here, $\bR^n$ is equipped with the dilations $\delta_r$.
We say that  $f$ {\it vanishes at homogeneous order} $N\in \bN_0$ when $f(v) = O(|v|^{N+1})$ for one (and then any) homogeneous quasinorm $|\cdot|$ on $\bR^n$. 

We abuse the vocabulary for the case $N=-1$ where $f(v)= O(1)$. 
\end{definition}

We readily adapt the usual properties of functions vanishing at a certain order by using  the Taylor expansion due to Folland and Stein~\cite{folland+stein_82} (see also Theorem~\ref{thm_MV+TaylorG}).
\begin{lemma}\label{lem_van_homo_char}
	Let $f$ be a (scalar-valued) function defined and smooth on an open subset $V$ of~$\bR^n$ containing 0. Let $N\in \bN_0$.
	The following properties are equivalent:
	\begin{enumerate}
		\item $f$ vanishes at homogeneous order $N$.
		\item $\partial^\alpha f(0)=0$ 
		for any $\alpha\in \bN_0$, $[\alpha]\leq N$. \item for any fixed $v\in V$, we have  $ f(\delta_\eps v) = O(\eps^{N+1})$ for $\delta\in \bR$.
	\end{enumerate}
	If it is the case, then for any $\alpha_0\in \bN_0^n$, $\partial^{\alpha_0} f$ vanishes at homogeneous order $\max( N-[\alpha_0],-1)$.
\end{lemma}

We observe that in our context, Property (2) is equivalent to 
\begin{itemize}
    \item[(2)'] $L_{\langle \bX \rangle_x} ^\alpha f(0)=0$ 
		for any $\alpha\in \bN_0$, $[\alpha]\leq N$. 
\end{itemize}
In any case, Property (2) allows to detect when the first coefficients $c_\alpha$ in the asymptotics given by Corollary \ref{cor_GM_Deltaq} vanish. 

\subsubsection{Differentiation of symbols}
The differentiation of symbols introduced in Notation~\ref{notation_DXbeta} preserves the symbol classes.

\begin{lemma}
Let $\bX$ be an  adapted frame  on an open subset $U\subset M$. 
For any $\beta\in \bN_0^n$ and $m\in \bR\cup\{-\infty\}$, the following map
is continuous
$$
\left\{
\begin{array}{rcl}
 S^m(\widehat G M) & \longrightarrow &S^{m}(\widehat G M|_U)\\
\sigma & \longmapsto & D_{\bX}^\beta \sigma
\end{array}
\right.
 $$
\end{lemma}

\begin{proof}
    The convolution kernel of $D_\bX^\beta \sigma$ in $\bX$-coordinates is $\bX^\beta\kappa_{\sigma, x}^\bX(u)$. For every $\beta'\in \bN^n$ there are functions $c^{\beta\beta'}_\gamma\in C^\infty(U)$ such that $\bX^{\beta'}\bX^{\beta}=\sum_{[\gamma]\leq[\beta]+[\beta']}c^{\beta\beta'}_\gamma \bX^\gamma$. Then we have $$\|D_{\bX}^\beta\sigma\|_{S^m(\widehat{G}M), \cC, N, \cR}\leq C\|\sigma\|_{S^m(\widehat{G}M), \cC, N+[\beta], \cR}$$ where $C\leq \sum_{[\gamma]\leq[\beta]+[\beta']}\sup_{x\in \cC}|c_\gamma^{\beta\beta'}(x)|$.
\end{proof}

To sum-up the elements of this section and the previous one, we can generalise easily this  property. 

\begin{corollary}
\label{cor_contSmTXDelta}
Let $\bX$ be an  adapted frame  on an open subset $U\subset M$. 
For any $\alpha,\beta, \beta',\gamma\in \bN_0^n$
and $m\in \bR\cup\{-\infty\}$, 
if $\sigma\in S^m(\widehat G M)$, 
then the symbol $\Delta^\alpha_\bX D_\bX^{\beta'} \widehat{\langle \bX\rangle}^{\beta} \sigma  \widehat{\langle \bX\rangle}^{\gamma}$ belongs to the symbol class  $S^{m-[\alpha]+[\beta]+[\gamma]}(\widehat G M|_U)$ with convolution kernel 
in the $\bX$-coordinates 
$$
(-u)^\alpha  \bX^{\beta'} L_{\langle \bX\rangle}^{\beta}R_{\langle \bX\rangle}^{\gamma} \kappa^\bX_{\sigma,x}  (u).
$$
Moreover, 
the following map is continuous:
$$
\left\{
\begin{array}{rcl}
S^m(\widehat G M) &\longrightarrow &S^{m-[\alpha]+[\beta]+[\gamma]}(\widehat G M|_U)\\
\sigma & \longmapsto & \Delta^\alpha_\bX  D_\bX^{\beta'}\,\widehat{ \langle \bX\rangle}^{\beta} \sigma  \widehat{\langle \bX\rangle}^{\gamma}  
\end{array}
\right.
$$
\end{corollary}

\subsection{Spectral multipliers of Rockland symbols}\label{sec:spectral_multipliers}

A fundamental example of symbols consists in the {\it spectral multipliers} of positive Rockland symbols. 
Here, we will generalise the properties of spectral multipliers obtained in Theorem \ref{thm_phi(R)} on groups and show that it holds on filtered manifolds:
\begin{proposition}\label{prop:fct_de_R}
	Let $\phi\in \cG^m (\bR)$ with $m\in \bR$ (see Notation~\ref{notation_Gm}), and let $\cR$ be a positive Rockland element on $G M$.
	Then the symbol 
    $$
    \phi(\widehat \cR) = \{\phi (\pi(\cR_x)):(x,\pi)\in \widehat G M\}
    $$
    is in $S^{m\nu}(\widehat G M)$ where $\nu$ is the  homogeneous degree of $\cR$. 
	Moreover, the map
	$$
	 \phi\longmapsto \phi(\widehat \cR), \qquad 
	 \cG^m (\bR)\longrightarrow S^{m\nu}(\widehat G M),
	 $$
	  is continuous. 
\end{proposition}

This section will be mainly devoted to 
the proof of Proposition \ref{prop:fct_de_R},
which follows the arguments in
\cite[Sections 4.7 \& 4.8]{fischerMikkelsen} and
\cite[Section 5.3]{R+F_monograph}.
The strategy  is  to decompose the multiplier function $\phi$ in dyadic pieces $\psi_j (2^{-j}\cdot)$, having shown  beforehand enough properties for 
spectral multipliers of the form $\psi (t\widehat \cR)$ with $t\in (0,1]$ and $\psi\in C_c^\infty(\bR)$.
The analysis of $\psi (t\widehat \cR)$ may be obtained
with two different methods:
in  \cite[Section 5.3]{R+F_monograph}
via heat kernel properties,
and in \cite[Sections 4.7 \& 4.8]{fischerMikkelsen} via the Helffer-Sj\"ostrand formula.
Only the latter analysis lends itself readily to a generalisation on manifolds. 
Hence, we  will start with recalling  the Helffer-Sj\"ostrand formula \cite{Davies,HS}.

After presenting the proof 
we will observe that essentially the same arguments will allow for the multiplier function to depend on $x$ 
and also that Proposition \ref{prop:fct_de_R}  implies the density of smoothing symbols.

\subsubsection{The Helffer-Sj\"ostrand formula}\label{subsubsec:HSformula}
If   $T$ is a self-adjoint operator densely defined on a separable Hilbert space $\cH$, and if  $\psi\in C_c^\infty(\bR)$, then 
the spectrally defined operator $\psi(T)\in \sL(\cH)$    coincides with  
\begin{equation}
\label{HSformula}
\psi(T) = \frac 1{\pi} \int_\bC \bar \partial \tilde \psi(z)\
(T-z)^{-1} L(dz).
\end{equation}
Here,  $\tilde \psi$ is any almost analytic extension of $\psi$ such that $\int_{\bC} |\bar \partial \tilde \psi(z)| |\IM\, z|^{-1} L(dz)$ is finite, 
and  $L(dz)=dxdy$, $z=x+iy$, is the Lebesgue measure on $\bC$. 
Recall that an almost analytic extension of a function 
 $\psi\in C^\infty(\bR)$ is any function  $\tilde\psi\in C^\infty(\bC)$ satisfying 
    $$
    \tilde \psi\big|_{\bR} = \psi \qquad\text{and}\qquad \bar{\partial} \tilde\psi\big|_{\bR}=0, 
    \qquad
    \mbox{where}\quad \bar \partial =\frac12( \partial_x +i\partial_y).
    $$
    This can be extended to smooth functions $\psi$ not necessarily with compact support but with enough decay at infinity \cite[Appendix]{fischerMikkelsen}:
\begin{lemma}
\label{lem_tildepsicGm'}
Let $\psi\in \cG^{m}(\bR)$ with $m<-1$. Then, there exists an almost analytic extension $\tilde {\psi} \in C^\infty(\bC)$ to $\psi$ 
for which the Helffer-Sj\"ostrand formula in \eqref{HSformula} holds for any self-adjoint operator $T$.
Moreover,   we have for all $N\in\bN_0$,
$$
\int_{\bC} \big| \bar\partial \tilde \psi( z) \big| \left(\frac{1+|z| }{|\IM \, z|}\right)^N L(dz) \leq C_{N}  \|\psi\|_{\cG^{m},N+3},
$$
with  the constant  $C_{N}>0$ depending on~$N$ and~$m$, but not on~$\psi$.
\end{lemma}

\subsubsection{Properties of the resolvent}
In order to  use  the Helffer-Sj\"ostrand formula, 
we  show some properties for the  resolvent of  $\widehat \cR$:
\begin{lemma}
\label{lem_I+RNGM}
    Let $\cR$ be a positive Rockland element on $G M$.
    \begin{enumerate}
        \item For any $N\in \bZ$, the symbol $(\id +\widehat \cR)^N = \{(\id + \pi(\cR_x))^N:(x,\pi)\in \widehat G M\}$ is in $S^{N\nu}(\widehat G M)$ where $\nu$ is the  homogeneous degree of $\cR$. 
        \item For any $z\in \bC \setminus \bR$, 
        the symbol 
        $$
        (\widehat \cR-z)^{-1} = \{( \pi(\cR_x)-z)^{-1}:(x,\pi)\in \widehat G M\}
        $$
        is in $S^{-\nu}(\widehat G M)$. Moreover,  for any seminorm $\|\cdot\|_{S^{-\nu}(\Gh M), (\bX,U),\cC,N}$, there exist $C>0$ and $p\in \bN_0$ such that
        $$
        \forall z\in \bC\setminus \bR,\qquad 
        \|(\widehat \cR -z)^{-1}\|_{S^{-\nu}(\Gh M), (\bX,U),\cC,N} \leq C \left(1+ \frac {1+|z|}{|\IM z|}\right)^p.
        $$
    \end{enumerate}    
\end{lemma}
\begin{proof}[Proof of Lemma \ref{lem_I+RNGM}]
For Part (1),  the case of $N=1$ by
Lemma \ref{lem_LTx} and the properties of composition of symbols (Proposition \ref{prop_comp+adj})
    imply that it suffices to show the case $N=-1$.
By Theorem \ref{thm_phi(R)},  $(\id +\widehat \cR_x)^{-1} \in S^{-\nu}(\Gh_x M)$
for every $x\in M$, 
and from Remark~\ref{rem_thm_phi(R)}, we obtain 
that the 0-seminorm defined in \eqref{def:semi_norm}, i.e.
    $$
    \max_{x\in \cC} \|(\id +\widehat \cR_x)^{-1}\|_{S^{-\nu}(\Gh_x M), (\bX, U),\cC, 0}<\infty,
    $$
    is finite; we could also obtain this fact 
    from the analysis of the Bessel kernels in  \cite[Section 4.3]{R+F_monograph} without resorting to  Theorem \ref{thm_phi(R)} and Remark \ref{rem_thm_phi(R)}.
   
Differentiating with respect to $x$, we obtain 
\begin{equation}\label{diff:(R-z)}
D_{X_j}(\id +\widehat \cR_x)^{-1}  = 
(\id +\widehat \cR_x)^{-1} (D_{X_j} \widehat \cR)_x (\id +\widehat \cR_x)^{-1} 
\end{equation}
in a local adapted frame
$(\bX,U)$ (at first formally in terms of convolution kernels).
Inductively, we obtain an expression for $D_\bX^\beta (\id +\widehat \cR)^{-1},$ $\beta\in \bN_0^n,$ 
as a linear combination of products of $(\id+\widehat \cR_x)^{-1}$
and  
$D_\bX^{\beta'}\widehat \cR$, $\beta'\in \bN_0^n$.
By Lemma \ref{lem_LTx}, $\widehat \cR \in S^{\nu}(\Gh M)$, so $D_\bX^{\beta'}\widehat \cR \in S^{\nu}(\Gh M)$.
This yields 
suitable seminorm bounds for $D_\bX^\beta (\id +\widehat \cR)^{-1},$ showing that $(\id+\cR)^{-1}\in S^{-\nu}(\Gh M).$
This concludes the proof of Part (1).

Let us show Part (2). 
By \cite[Proposition 4.33]{fischerMikkelsen} and its proof, 
$(\widehat \cR_x -z)^{-1}\in S^{-\nu}(\Gh_x M)$ for any $x\in M$ with for any $\pi\in \Gh_x M$
\begin{align*}
 \|(\id+\pi(\cR_x))^{\frac{\nu +\gamma}\nu }
(\pi( \cR_x) -z)^{-1}
(\id+\pi(\cR_x))^{-\frac{\gamma}\nu }\|_{\sL(\cH_\pi)}
&=
\|(\pi(\cR_x) -z)^{-1}(\id+\pi(\cR_x))\|_{\sL(\cH_\pi)}
\\& \leq \sup_{\lambda\geq 0}\left|\frac{1+\lambda}{\lambda-z}\right| \leq 1+ \frac {1+|z|}{|\IM z|},  
\end{align*}
by functional analysis.
Using~\eqref{diff:(R-z)}
in a local adapted frame
$(\bX,U)$ and arguing
inductively, we obtain an expression for $D_\bX^\beta \Delta_\bX^\alpha(\widehat \cR-z)^{-1},$ $\beta,\alpha\in \bN_0^n,$ with  
suitable seminorm bounds by proceeding as above (see  the proof of \cite[Proposition 4.33]{fischerMikkelsen}).
This concludes the proof of Part (2).  
\end{proof}

\subsubsection{Proof of Proposition \ref{prop:fct_de_R}}\label{subsubsec:proof_spec_mult}
The main step in the proof of Proposition \ref{prop:fct_de_R} is the following statement:
\begin{lemma}
\label{lem_prop:fct_de_R}
Let $\cR$ be a positive Rockland operator on $G M$.
We denote by $\nu$ its  homogeneous degree.
\begin{enumerate}
    \item For any $\psi\in \cG^m (\bR)$ with $m<-1$, 
	the symbol $\psi(\widehat \cR)$ is in $S^{-\nu}(\widehat G M)$. 
	Moreover, for any seminorm $\|\cdot\|_{S^{-\nu}(\widehat G M),(\bX,U), \cC, N} $, there exist $C>0$ and 
	a seminorm $\|\cdot\|_{\cG^{m},N'}$ such that 
	$$
	\forall \psi\in \cG^{m}(\bR), \ \forall t\in (0,1],\;\; 
	\|\psi(t\widehat \cR)\|_{S^{-\nu}(\widehat G M),(\bX,U), \cC, N} \leq C t^{-1} \|\psi\|_{\cG^{m},N'}.
	$$
    \item 
    If $\psi \in \cS(\bR) $, then 
    $\psi (\widehat \cR) \in S^{-\infty}(\Gh M)$. 
    Moreover, the map $\psi\mapsto \psi (\widehat \cR)$ is continuous $\cS(\bR) \to S^{-\infty}(\Gh M)$.
    \item For any $m\in \bR$ and any seminorm  $\|\cdot\|_{S^{m\nu}(\widehat G M),(\bX,U), \cC, N} $,  there exist $C>0$ and $k\in \bN$
	 such that for any  $\psi \in  C^\infty_c(\frac 12,2) $, 
     $$
     \forall t\in (0,1],\;\;
	\|\psi(t\widehat \cR)\|_{S^{m\nu}(\widehat G M),(\bX,U), \cC, N} \leq C t^{m} \max_{k'=0,\ldots,k}\sup_{\lambda\in  \bR} |\psi^{(k')}(\lambda)|.
     $$  
\end{enumerate}
\end{lemma}

\begin{proof} [Proof of Lemma \ref{lem_prop:fct_de_R}]
Let us prove (1). Let $\psi\in \cG^m(\bR)$, $m<-1$.
We consider the almost analytical extension $\tilde \psi$ constructed as in Lemma \ref{lem_tildepsicGm'}.
By Lemma \ref{lem_I+RNGM} (2) and the formula in \eqref{HSformula}, we have
\begin{align*}
     \| \psi(\widehat \cR)\|_{S^{-\nu}(\Gh M), (\bX,U),\cC,N} 
   & \leq
\frac 1{\pi} \int_\bC |\bar \partial \tilde \psi(z)|
\|(\widehat \cR-z)^{-1}\|_{S^{-\nu}(\Gh M), (\bX,U),\cC,N}  L(dz)
\\&\lesssim 
\int_\bC |\bar \partial \tilde \psi(z)|
\left(1+ \frac {1+|z|}{|\IM z|}\right)^p
 L(dz)  
 \lesssim \|\psi\|_{\cG^m, p+3}.
\end{align*}
This provides the estimate for $t=1$. To obtain the estimate for any  $t\in (0,1]$, it suffices to  observe 
$$
\psi(t\widehat \cR) = \frac 1{\pi} \int_\bC \bar \partial \tilde \psi(z)\
(t\widehat \cR-z)^{-1} L(dz)
=
\frac {t^{-1}}{\pi}\int_\bC \bar \partial \tilde \psi(z)\
(\widehat \cR-t^{-1}z)^{-1} L(dz),
$$
and to proceed as above (see the proof of \cite[Corollary 4.37]{fischerMikkelsen}). This shows Part (1).

Let us sketch the proof of Parts (2) and (3) which follows the argument of the proof of \cite[Lemma 4.38]{fischerMikkelsen}. 
For 
(2), we consider $\psi \in \cS( \bR).$
As $\widehat \cR\geq 0$, we may assume that $\psi(\lambda)=0$ for $\lambda\leq -1/2$.
Then for any $N\in \bN_0$, $\psi_N (\lambda) := (1+\lambda)^N \psi(\lambda)$ defines a  function $\psi_N \in \cS(\bR)$.
Applying Part (1) to $\psi_N$ and using the properties of composition of symbols (Proposition~\ref{prop_comp+adj}) to 
$$
\psi(\widehat \cR) = (\id +\widehat \cR)^{-N} \psi_N(\widehat \cR),
$$
 thanks to Lemma \ref{lem_I+RNGM} (1),  imply Part (2).

Finally, let  $\psi \in C^\infty_c(\frac 12,2) $. 
The case $m=-1$ follows from Part (1). Then, for $m=N-1$ with  
 $N\in \bN$, one 
applies Part (1) to $\phi_N(\lambda):=\lambda^{-N} \psi(\lambda)$ and uses the properties of composition of symbols (Proposition \ref{prop_comp+adj})
on 
$$
\psi(\widehat \cR) 
= {\widehat \cR}^N \phi_N(\widehat \cR),
$$
together with $\widehat \cR^N\in S^{N\nu}(\Gh M)$. Writing $\psi = \psi \chi$ with $\chi\in C_c^\infty (\frac 14 ,4)$ such that $\chi=1$ on $(1/2,2)$ 
and  using the properties of composition of symbols (Proposition \ref{prop_comp+adj}) on 
$$
\psi(\widehat \cR)
=
\psi(\widehat \cR)
\chi(\widehat \cR),
$$ 
yields the case of $m=-2$. 
Inductively, we obtain the case $m=-2,-3,\ldots$
Part (3) follows by interpolation.
\end{proof}

We can now show Proposition \ref{prop:fct_de_R}. 
Our proof will follow closely 
\cite[Section 4.8]{fischerMikkelsen}.

\begin{proof}[Proof of Proposition \ref{prop:fct_de_R}] 
We only require to prove the result for $m<-1$. Indeed,  for  $\phi\in \cG^m (\bR)$, with $m\geq -1$,  choosing $N\in \bN_0$ such that $ N >  m+1$, one can find $\psi\in \cG^{m-N} (\bR)$ such that  for all $\lambda\in [0,+\infty)$, $\phi(\lambda)=(1+\lambda)^N \psi(\lambda)$. The result follows by the properties of symbols with respect to product (see Proposition~\ref{prop_comp+adj}) and Lemma \ref{lem_I+RNGM} (1). Thus we focus on $m<-1$. 
By Lemma \ref{lem_prop:fct_de_R} (2), we may assume that $\phi$ is supported on $[1,\infty)$.
	Let $(\eta_j)$ be a dyadic decomposition of $[1,+\infty)$, that is,  $\eta_0\in C_c^\infty (\frac 12,2) $ with 
\begin{equation}\label{def:dyadic}
	\sum_{j=0}^\infty \eta_j(\lambda)=1 \ \mbox{for all}\ \lambda\geq 1, \quad\mbox{where}\quad \eta_j(\lambda):=\eta_0(2^{-j}\lambda).
	\end{equation}
	We may write for any $\lambda\geq 0$
	$$
	\phi(\lambda) = 
    \sum_{j=0}^\infty \eta_j(\lambda) \phi(\lambda)=\sum_{j=0}^\infty 2^{jm} \psi_j (2^{-j}\lambda), 
\quad\mbox{where}\quad \psi_j (\mu):= 2^{-jm} \phi(2^j \mu) \eta_0(\mu).
	$$
	We observe that 
	$$
	\psi_j \in C_c^\infty (\frac 12,2)
	\qquad \mbox{and} \qquad
	\sup_{\lambda\geq 0} |\psi_j ^{(k)}(\lambda)| \lesssim_k \|\phi\|_{\cG^{m},k}
		$$
		for any $k\in \bN_0$
with an implicit constant independent of $j$.
Let $(\bX,U)$ be an adapted frame, 
let $\alpha,\beta\in \bN_0^n$, $x\in M$ and $\pi\in \Gh_x$.
For each $j\in \bN_0$, 
let us consider the operator
$$
T_j (\alpha,\beta,\gamma; x,\pi):=
T_j := 2^{jm} (\id+\pi(\cR_x))^{\frac{-m\nu +[\alpha]+\gamma}\nu}
\Delta_\bX^\alpha D_\bX^\beta
\psi_j (2^{-j}\pi(\cR_x) )
(\id+\pi(\cR_x))^{-\frac{\gamma}\nu}.
$$
The properties  of composition and adjoint for symbols  in Proposition~\ref{prop_comp+adj} imply
\begin{align*}
    \|T_i^* T_j\|_{\sL(\cH_\pi)}
    &= 2^{(i+j)m }
    \|(\id+\pi(\cR_x))^{-\frac{\gamma}\nu}
    \Delta_\bX^\alpha D_\bX^\beta
\psi_i (2^{-i}\pi(\cR_x) )
    (\id+\pi(\cR_x))^{\frac{[\alpha]+\gamma}\nu} \times \\
    & \qquad\qquad\qquad  \times  (\id+\pi(\cR_x))^{\frac{-2m\nu +[\alpha]+\gamma}\nu}
\Delta_\bX^\alpha D_\bX^\beta
\psi_j (2^{-j}\pi(\cR_x) )
(\id+\pi(\cR_x))^{-\frac{\gamma}\nu}\|_{\sL(\cH_\pi)}\\
&\lesssim 2^{(i+j)m} \|
\psi_i (2^{-i}\pi(\cR_x) )\|_{S^0(\Gh_x M),[\alpha],[\beta],c}
\| \psi_j (2^{-j}\pi(\cR_x) )\|
_{S^{2m\nu} (\Gh_x M),[\alpha],[\beta],c}, 
\end{align*}
for some $c>0.$
Lemma \ref{lem_prop:fct_de_R} (3) yields that for some 
 $k\in \bN_0$, we have
 \begin{align*}
    \|T_i^* T_j\|_{\sL(\cH_\pi)}
    &\lesssim 2^{(i+j)m} 2^{-j (2m)} \max_{k'=0,\ldots,k} 
	\sup_{\lambda \geq 0} |\psi_i^{(k')}(\lambda)|
	\max_{k'=0,\ldots,k} 
	\sup_{\lambda \geq 0} |\psi_j^{(k')}(\lambda)|\\
    &\lesssim 2^{(i-j)m} \|\phi\|_{\cG^{m},k}^2,
\end{align*}
and similarly for $\|T_iT_j^*\|_{\sL(\cH_\pi)}$.
By the Cotlar-Stein Lemma \cite[\S VII.2]{Ste93},  $\sum_j T_j$ converges in the strong operator topology to a bounded operator on $\cH_\pi$ with norm $\lesssim \|\phi\|_{\cG^{m},k}$,
the implicit constant being independent of $(x,\pi)\in G\times \Gh$.
This implies that 
$$
\sup_{(x,\pi)\in \Gh M}
\|(\id+\pi(\cR_x))^{\frac{-m\nu +[\alpha]+\gamma}\nu}
\Delta_\bX^\alpha D_\bX^\beta\phi(\pi(\cR_x))
(\id+\pi(\cR_x))^{-\frac{\gamma}\nu}\|_{\sL(\cH_\pi)} 
\lesssim  \|\phi\|_{\cG^{m},k}
$$
for some $k\in \bN$. 
As this holds for any $\alpha,\beta, (\bX,U)$, the conclusion follows.
\end{proof}

\subsubsection{Multipliers depending on $x$}
\label{subsubsec_multdeponx}
We observe that the proof above is robust enough to allow a dependence of the multiplier function $\phi$ in $x$. More precisely,  we can consider 
 multiplier functions  in the space
$$
\cG^m(M\times \bR) = C^\infty(M)\otimes \cG^m(\bR),
$$
equipped with the tensor topology; 
equivalently, this is  
the Fr\'echet space of smooth functions $\phi:M\times \bR \to \bC$ generated by the  semi-norms given by
$$
\|\phi\|_{\cG^m, N} := \max_{k,j,[\alpha]\leq N}\sup_{(x,\lambda)\in \cC_k\times \bR} 
(1+|\lambda|)^{-m+j}|\bX_k^{\alpha} \partial_\lambda^j \phi(x,\lambda)|,
$$
where   a countable sequence of compact subsets $\cC_k$, $k=0,1,2,\ldots$
covering $M=\cup_{k} \cC_k$ has been fixed, 
with each compact $\cC_k$  included in a bounded open subset $U_k$ on which an adapted frame $\bX_k$ exists.

Adapting the proof above, we readily obtain:
\begin{corollary}
\label{corprop:fct_de_R}
	Let $\phi\in \cG^m (M\times \bR)$ with $m\in \bR$, and let $\cR$ be a positive Rockland element on $G M$.
	Then the symbol 
    $$
    \phi(\widehat \cR) := \{\phi (x,\pi(\cR_x)):(x,\pi)\in \widehat G M\}
    $$
    is in $S^{m\nu}(\widehat G M)$ where $\nu$ is the  homogeneous degree of $\cR$. 
	Moreover, the map
	$$
	 \phi\longmapsto \phi(\widehat \cR), \qquad 
	 \cG^m (M\times \bR)\longrightarrow S^{m\nu}(\widehat G M),
	 $$
	  is continuous. 
\end{corollary}

\subsubsection{Density of smoothing symbols}
\label{subsubsec_densitysmoothing}
The proofs of Propositions \ref{prop_smoothing} and \ref{prop_densitysmoothing}, with Remarks~\ref{rem_prop_smoothing} and~\ref{rem_prop_densitysmoothing},
can be readily adapted to the manifold case by using Proposition \ref{prop:fct_de_R}.
We obtain the following result:

\begin{proposition}
\label{prop_smoothingM}
	\begin{enumerate}
		\item The fiberwise group Fourier transform is an isomorphism 
		from the topological vector space of the vertical Schwartz densities $ \cS(GM; |\Omega|(GM))$ onto $S^{-\infty}(\widehat G M)$.

\item Let $\sigma\in S^m(\Gh M)$ with $m\in \bR$.
We can construct a sequence of smoothing symbols $(\sigma_\ell)_{\ell\in \bN}\subset S^{-\infty}(\Gh M)$ satisfying the following properties:
	\begin{enumerate}
		\item For any $m_1>m$, we have the convergence 
$$
\lim_{\ell\to \infty} \sigma_\ell	= \sigma \quad\mbox{in}\ S^{m_1}(\Gh M).
$$		
In other words, 
the space 
$S^{-\infty}(\Gh M)$ equipped with the $S^{m_1}(\Gh M)$-topology is dense in $S^{m}(\Gh M)$.
\item 
For any semi-norm $\|\cdot\|_{S^m(\Gh),(\bX,U),\cC,N}$ of $S^m(\Gh)$, there exist a constant $C>0$ and $N'\geq N$ such that 
$$
\forall \ell\in \bN,\qquad  \|\sigma_\ell-\sigma\|_{S^m(\Gh),(\bX,U),\cC,N}
\leq C \|\sigma\|_{S^m(\Gh),(\bX,U),\cC,N'}.
$$	
Consequently, 
$$
\forall \ell\in \bN,\qquad  \|\sigma_\ell\|_{S^m(\Gh),(\bX,U),\cC,N}
\leq (1+C) \|\sigma\|_{S^m(\Gh),(\bX,U),\cC,N'}.
$$	

		\item 
		Denoting respectively by $\kappa^\bX$ and $\kappa^\bX_\ell$ the convolution kernels of $\sigma$ and $\sigma_\ell$ in the $\bX$-coordinates,
	we
 have the convergences 
$$
\lim_{\ell\to \infty}  \kappa^\bX_\ell = \kappa^\bX 
\quad\mbox{in}\  \Gamma(\cS'(\bR^n))
\quad \mbox{and in} \ \Gamma(C^\infty (U_1)),
$$
for any open subset of $U_1\subset G\setminus \{0\}$.
\item Moreover, $\liminf_{\ell \to\infty}  \|\sigma_\ell\|_{S^m(\Gh M),(\bX,U),\cC,N} \geq \|\sigma\|_{S^m(\Gh M),(\bX,U),\cC,N} $.
	\end{enumerate}
	\end{enumerate}
\end{proposition}

\begin{remark}
\label{rem_altpf_prop_comp+adj}
The above result  will be  useful in many proofs. Indeed, it will allow us to assume that symbols are smoothing, and hence have Schwartz convolution kernels. 

For instance, 
   instead of manipulating distributions 
   in the composition of symbols 
   in  Proposition \ref{prop_comp+adj}, 
   we can instead consider Schwartz functions 
by considering smoothing symbols and applying the result of density in Proposition \ref{prop_smoothingM}. Indeed, 
if $\sigma_1,\sigma_2$ and $\sigma$ are smoothing, 
 given a smooth Haar measure $\mu$ and fixing $x\in M$, 
the formulae for $\kappa_{\sigma_1\sigma_2,x}^{\mu}$ and $ \kappa_{\sigma^*,x}^{\mu}$
in Proposition \ref{prop_comp+adj}
make sense as Schwartz functions on $G_x M$.
 We check as above that these define Schwartz densities independently of a choice of a smooth Haar measure. 
  The estimates with respect to suitable seminorms of symbol classes are  checked using the results on $G_x M$, $x\in M$.
  And we can conclude by  density (Proposition \ref{prop_smoothingM}). 
\end{remark}

\subsection{Homogeneous and polyhomogeneous symbols} \label{subsec_polyhom}

\subsubsection{Homogeneous symbols} \label{sebsec:poluhom}

We define the class $\dot S^m(\Gh M)$ of homogeneous symbol of order~$m$ by replacing $S^m(\widehat {G}_x M)$ with $\dot S^m(\widehat {G}_x M)$ in Definition \ref{def:Sm} (see Section~\ref{subsubsec_invdotSm} for the definition of homogeneous symbols on a group).
As Lemmata \ref{lem_SmnormsigmaxRx},  
\ref{lem_SmnormsigmaxRS} and 
\ref{lem_DXYsigma} can easily be adapted to hold for $\dot S^m(\widehat {G}_x M)$-semi-norms, 
this notion is intrinsically defined. 
We denote by $\dot S^m(\widehat G M)$ the space of homogeneous symbols of order $m$ on $M$.  
If $\bX$ is an adapted frame on an open set $U\subset M$ and $\cC$ a compact subset of $U$, we  define 
the associated seminorm $\|\cdot \|_{\dot S^m(\widehat G M),(\bX,U), \cC, N}$ as in \eqref{def:semi_norm}
by replacing the $S^m(\widehat {G}_x M)$-semi-norm with the corresponding one on $\dot S^m(\widehat {G}_x M)$ (see~\eqref{def:semi_norm_hom_group}).
It is a Fr\'echet space for the countable family of semi-norms analogous to \eqref{eq_seminormSmGhMN}.

\begin{ex}
  Adapting Example~\ref{ex_widehatlangleXrangle_bis}, 
the symbol 
$\widehat {\langle \bX\rangle}^\alpha := \{\pi(\langle X\rangle_x^\alpha ) : x\in U, \pi\in \widehat G_x M\}$ is in $\dot S^{[\alpha]}(\Gh M)$.  
\end{ex}

\begin{ex}
\label{ex_I+RmnuphM}
Extending Example \ref{ex_I+Rmnuph} to the manifold setting, 
for any  positive Rockland operator  $\cR$, 
we have 
$(\id+\widehat \cR)^{m/\nu} \in S^m_{ph}(\Gh M)$
where $\nu$ is the homogeneous degree of $\cR$, 
and 
$$
		(\id+\widehat \cR)^{m/\nu}  
		\sim_h   \widehat \cR ^{m/\nu}  + \frac m \nu \widehat \cR^{-1 + m/\nu}  + \ldots 
$$
\end{ex}

 By construction, every symbol $\sigma \in \dot S^m(\widehat G M)$ admits a convolution kernel density that is $m$-homogeneous in the sense of Definition \ref{def_mhomS'GMvert}.

Adapting the proof for the inhomogeneous symbol classes, we obtain that 
the homogeneous classes $\dot S^m (\Gh M)$ enjoy the algebraic properties for composition and adjoint as described in Proposition \ref{prop_comp+adj}  for the inhomogeneous case.

\subsubsection{Links between  homogeneous and inhomogeneous  symbol classes}
\label{subsubsec_hom+inhomsymb}

\begin{proposition}
\label{prop_homsigmapsi}
	Let $\cR$ be a positive Rockland element of $\sU(\fg M)$.
Let $\psi_0 \in C^\infty(M\times \bR)$ be such that $\psi_0(x,\lambda)=0$ for $\lambda$ in a neighbourhood of $(-\infty,0]$
and $\psi_0(x,\lambda)=1$ on a neighbourhood of $+\infty$, these two $\lambda$-neighbourhoods depending locally uniformly on $x\in M$.
Then for any $\sigma\in \dot S^{m_1}(\Gh M)$ and $\psi\in \cG^{m_2}(\bR)$,	 the symbols 
	 $$
\sigma \, (\psi\psi_0) (\widehat\cR)=\{\sigma(x,\pi)\, (\psi\psi_0)(x,\pi(\cR_x)) :
x\in M, \ \pi\in \Gh_x M\}
$$
and 
$$
 (\psi\psi_0) (\widehat\cR) \,\sigma =\{(\psi\psi_0)(x,\pi(\cR_x))\sigma(x,\pi) :
x\in M, \ \pi\in \Gh_x M\}
$$
	 are in $ S^{m_1+m_2 \nu}(\Gh M)$
	  where $\nu$ is the homogeneous degree of $\cR$.
	Moreover, the maps
	$$
	(\sigma,\psi)\longmapsto (\psi\psi_0) (\widehat \cR) \, \sigma
	\quad\mbox{and}\quad 
	(\sigma,\psi)\longmapsto \sigma \, (\psi\psi_0) (\widehat \cR) ,
	$$
	are continuous $\dot S^{m_1}(\Gh M)\times \cG^{m_2}(\bR)\to S^{m_1+m_2\nu}(\Gh M)$.
\end{proposition}

Below, we will use Proposition \ref{prop_homsigmapsi} with $\psi\equiv 1$, that is, $\sigma \psi_0(\widehat \cR) \in S^{m_1}(\Gh M)$ if $\sigma\in \dot S^{m_1} (\Gh M)$. 
In the proof of Proposition \ref{prop_homsigmapsi} and in the remainder of this section, we will use the notation 
\[
\Gh M^*=\bigcup_{x\in M} \left(\Gh_xM\setminus \{1\}\right)
\]
and write $\Gh_xM^*$ for the fibers of $\Gh M^*$.

\begin{proof}[Proof of Proposition \ref{prop_homsigmapsi}]
From the group case (see Section \ref{subsubsec_invdotSm} and Proposition~4.6 in~\cite{FF0}), 
we see that for any $x\in M$, 
the symbols
$\sigma(x,\cdot) (\psi\psi_0) (x,\widehat \cR_x)$ 
and 
$(\psi\psi_0) (x,\widehat \cR_x) \sigma(x,\cdot) $ 
are in $S^{m}(\Gh_x M)$.
Moreover, setting $m:=m_1+m_2\nu$, the $S^{m}(\Gh_x M)$-semi-norms of these symbols
are bounded by some $S^{m_1}(\Gh_x M)$-semi-norms of $\sigma$ and 
$\cG^{m_2}(\bR)$-semi-norms in $\psi$
up to a constant locally uniformly in $x\in M$. 
The main issue is the $x$-derivatives.
Let $\bX$ be an adapted frame on an open subset $U$ of $M$.
By the Leibniz rule in $x$, we have at least formally with $|\beta|=1$
$$
D_\bX^\beta (\sigma (\psi\psi_0)(\widehat \cR))
= D_\bX^\beta  \sigma \, (\psi\psi_0)(\widehat \cR)
+ \sigma\, D_\bX^\beta (\psi\psi_0)(\widehat \cR).
$$
By construction, $D_\bX^\beta \sigma\in \dot S^m(\Gh M|_U)$ and therefore the first term may be analysed fiberwise as above.
For the second term, we can not in general claim that $D_\bX^\beta (\psi\psi_0)(\widehat \cR)$ vanishes identically for low frequencies. To deal with this issue, we modify the 
 dyadic decomposition~\eqref{def:dyadic} into a dyadic decomposition on $[0,+\infty)$: we fix $\eta_0\in C^\infty([-1,1], [0,1])$  
and $\eta_1\in C_c^\infty ([\frac 12,\frac 32],[0,1])$ such that 
$$
1 = \sum_{j=0}^\infty \eta_j(\lambda)
\quad \mbox{for any}\ \lambda\geq 0, \quad\mbox{with}\ 
\eta_j(\lambda):= \eta_1 (2^{j-1}\lambda).
$$
We then write
$$
\sup_{(x,\pi)\in \Gh M^*}
\|\sigma(x,\pi)\, D_\bX^\beta (\psi\psi_0)(\widehat \cR_x) \, (\id + \pi(\cR_x))^{-\frac{m}\nu}\|_{\sL(\cH_\pi)}
\leq C_0+C_{\psi,\beta}
\sum_{j=1}^\infty E_{j}, 
$$
where
\begin{align*}
C_0&:=
\sup_{(x,\pi)\in \Gh M^*}\|\sigma(x,\pi)\eta_0 (\pi( \cR_x))\, D_\bX^\beta (\psi\psi_0)(\widehat \cR_x) \, (\id + \pi(\cR_x))^{-\frac{m}\nu}\|_{\sL(\cH_\pi)},\\
  E_{j}&:=\sup_{(x,\pi)\in \Gh M^*}\|\sigma(x,\pi)\eta_j (\pi( \cR_x))\|_{\sL(\cH_\pi)},\\
C_{\psi,\beta}&:=\sup_{(x,\pi)\in \Gh M^*}
\|  D_\bX^\beta (\psi_0\psi)(\widehat \cR_x) \, (\id + \pi(\cR_x))^{-\frac{m}\nu}\|_{\sL(\cH_\pi)}.
\end{align*}
The constant $C_0$ is finite since we have
$C_0 \leq C'_0\|\sigma\|_{\dot S^{m_1}(\Gh M), 0}$ with 
the constant 
\begin{align*}
C'_0 := 
&\sup_{(x,\pi)\in \Gh M^*}\|\widehat \cR_x^{\frac {m_1}\nu } \eta_0 (\pi( \cR_x))(\id + \pi(\cR_x))^{-\frac{m_1}\nu}\|_{\sL(\cH_\pi)}\\
&\qquad \times \, 
\|(\id + \pi(\cR_x))^{\frac{m_1}\nu} D_\bX^\beta (\psi\psi_0)(\widehat \cR_x) \, (\id + \pi(\cR_x))^{-\frac{m}\nu}\|_{\sL(\cH_\pi)},
\end{align*}
being finite and in fact bounded by the product of a $S^0(\Gh M)$-semi-norm in $\eta_0(\widehat \cR)$ with a 
 $S^{m_2\nu}(\Gh M)$-semi-norm in $(\psi\psi_0)(\widehat \cR)$ up to a constant.

By Corollary~\ref{corprop:fct_de_R}, $C_{\psi,\beta}$ is a finite constant since $D_\bX^\beta (\psi_0\psi) = -D_\bX^\beta (1-\psi_0\psi)$ and $1-\psi_0\psi\in C^\infty(M)\otimes C_c^\infty(\bR)$.
Regarding the $E_j$-s, we write 
\begin{align*}
  E_{j}&=\sup_{(x,\pi)\in \Gh M^*}\|\sigma(x,\pi)\eta_1 (2^{j-1} \pi( \cR_x))\|_{\sL(\cH_\pi)}
 \\& =\sup_{(x,\pi)\in \Gh M^*}\|\sigma(x,\pi)\eta_1 (2^{\frac{j-1}\nu}\cdot \pi( \cR_x))\|_{\sL(\cH_\pi)}
  \\
&=\sup_{(x,\pi')\in \Gh M^*}\|\sigma(x,2^{-\frac{j-1}\nu}\cdot \pi')\eta_1 (\pi'(\cR_x))\|_{\sL(\cH_{\pi'})}
\\&= 2^{-m_1\frac{j-1}\nu}\sup_{(x,\pi')\in \Gh M^*}\|\sigma(x, \pi')\eta_1 (\pi'(\cR_x))\|_{\sL(\cH_{\pi'})}
\end{align*}
where we have used the homogeneity of $\sigma$. We obtain
\begin{align*}
E_j &\leq
2^{-m_1\frac{j-1}\nu} \|\sigma\|_{\dot S^{m_1}(\Gh M), 0} \sup_{\lambda>0 } (1+ \lambda)^{\frac {m_1}\nu} |\eta_1(\lambda)|
\end{align*}
by functional analysis; alternatively, we can estimate $S^{-m_1}(\widehat GM)$-seminorms of $\eta_1 (\pi'(\cR_x))$ by $\sup_{\lambda>0 } \lambda^{\frac {m_1}\nu} |\eta_1(\lambda)|$ using  Lemma \ref{lem_prop:fct_de_R}.
Hence for $m_1>0$, $\sum_j E_j$ is finite.

If $m_1\leq 0$, we modify the analysis of the second term by  considering $\sigma_1:=\sigma \widehat \cR^{N_1}$ and $\psi_1(\lambda):=\lambda^{-N_1}\psi$ for $N_1 $ large enough. Indeed, as $\sigma_1$ will be in $\dot S^{m_1+N_1\nu }$, we will choose $N_1 > -m_1/\nu$ and we will still have $\sigma \psi (\widehat \cR)=\sigma_1 \psi_1(\widehat \cR)$.

Adapting the reasoning above to any derivative $\bX^\beta$, $\beta\in \bN_0^n$, and $\Delta_\bX^\alpha$, $\alpha\in \bN_0^n$, we  obtain that the quantity $\|\sigma (\psi_0\psi) (\widehat \cR)\|_{S^m(\Gh M), N}$ is finite for any $N\in \bN_0$, 
and therefore $\sigma (\psi_0\psi) (\widehat \cR) \in S^m(\Gh M) $. 
The continuity follows from tracking the constant in  the above estimates. 
The result for the symbol $(\psi_0\psi) (\widehat \cR) \sigma$ is proved in a similar manner.
\end{proof}

The following statement shows that modulo smoothing symbols, the precise choices of the positive Rockland element $\cR$ or the  cut-off $\psi_0$ for the low frequencies in Proposition \ref{prop_homsigmapsi} are irrelevant:
\begin{lemma}
\label{lem_indeppsi0R}
Let $\psi_1,\psi_2$ be two smooth functions on $M\times \bR$ satisfying for $i=1,2$, 
 $\psi_i(x,\lambda)=0$ for $\lambda$ in a neighbourhood of $(-\infty,0]$
and $\psi_i(x,\lambda)=1$ on a neighbourhood of $+\infty$, these two $\lambda$-neighbourhoods depending locally uniformly on $x\in M$.
Let $\cR$ and $\cS$ be two positive Rockland elements on $GM$.
\begin{enumerate}
    \item The symbol $\psi_1(\cR) - \psi_2(\cS)$ is smoothing:
$$
\psi_1(\cR) - \psi_2(\cS)
\in S^{-\infty}(\Gh M).
$$
\item More generally, for any $\sigma\in \dot S^m (\Gh M)$, the symbols
$$
\sigma \psi_1(\cR) - \sigma \psi_2(\cS)
\quad\mbox{and}\quad
 \psi_1(\cR)\sigma -  \psi_2(\cS)\sigma
$$
are smoothing, that is, in $S^{-\infty} (\Gh M).$
\end{enumerate}
\end{lemma}

\begin{proof} [Proof of Lemma \ref{lem_indeppsi0R}]
By Corollary~\ref{corprop:fct_de_R} and Proposition \ref{prop_homsigmapsi},
we may assume $\psi_1=\psi_2=\psi_0$.
Part (1) is a consequence of Part (2) with $\sigma = \id$. Let us prove Part (2). 
We consider a dyadic decomposition  on $[0,+\infty)$ as introduced in the proof of Proposition~\ref{prop_homsigmapsi}:
we fix $\eta_0\in C^\infty([-1,1], [0,1])$  
and $\eta_1\in C_c^\infty ([\frac 12,\frac 32],[0,1])$ such that 
$$
1 = \sum_{j=0}^\infty \eta_j(\lambda)
\quad \mbox{for any}\ \lambda\geq 0, \quad\mbox{with}\ 
\eta_j(\lambda):= \eta_1 (2^{j-1}\lambda).
$$
For each $j\in \bN_0$, 
the symbol 
$\sigma\, \eta_j(\widehat \cR) \left(
\psi_0 (\widehat \cR)-\psi_0 (\widehat \cS)\right)$
is  smoothing by Proposition \ref{prop_homsigmapsi} and the properties of composition of symbols stated in Proposition~\ref{prop_comp+adj}.
The properties of the supports of $\psi_0$ and $\eta_0,\eta_1$ imply the existence of  $j_0\in \bN$ such that for all $j\geq j_0$, $\eta_j \psi_0=\eta_j$.
Hence, for $j\geq j_0$, 
$$
\eta_j(\widehat \cR) \left(
\psi_0 (\widehat \cR)-\psi_0 (\widehat \cS)\right)
=
\eta_j(\widehat \cR) - \eta_j(\widehat \cR)\psi_0 (\widehat \cS)
= \eta_j(\widehat \cR)\, (1-\psi_0)(\widehat \cS) .
$$
For any semi-norm $\|\cdot\|_{S^{-N_1}(\Gh M), N_2}$, we have
$$
\|\sigma \eta_j(\widehat \cR) (1-\psi_0)(\widehat \cS)\|_{S^{-N_1}(\Gh M), N_2}
\lesssim \|\sigma \eta_j(\widehat \cR) \|_{S^{m-\nu}(\Gh M), N'_2} 
\|(1-\psi_0)(\widehat \cS)\|_{S^{-N_1-m+\nu}(\Gh M), N''_2},
$$
for some $N'_2,N''_2\in \bN_0$,
where $\nu$ is the homogeneous degree of $\cR$.
By Proposition~\ref{prop:fct_de_R}, 
$\|(1-\psi_0)(\widehat \cS)\|_{S^{-N_1+\nu}(\Gh M), N''_2}$ is finite since $1-\psi_0\in C_c^\infty(\bR)$, while we have by Lemma~\ref{lem_prop:fct_de_R} (3):
$$
\|\sigma \eta_j(\widehat \cR) \|_{S^{m-\nu}(\Gh M), N'_2} 
\lesssim 
\|\sigma\|_{S^{m}(\Gh M), \tilde N_2}
\| \eta_j(\widehat \cR) \|_{S^{-\nu}(\Gh M), N''_2}
\lesssim \|\eta_j\|_{\cG^{-1},N_3} \lesssim 2^{-j} ,
$$
 for some $\tilde N_2,N_2'', N_3\in \bN_0$..
Hence, the sum
$$
\sum_{j\geq j_0}
\|\eta_j(\widehat \cR)\left(
\psi_0 (\widehat \cR)-\psi_0 (\widehat \cS)\right)
\|_{S^{-N_1}(\Gh M), N_2} 
$$
is finite for any semi-norm $\|\cdot\|_{S^{-N_1}(\Gh M), N_2}$.
This concludes the proof for $\sigma \psi_1(\cR) - \sigma \psi_2(\cS)$. The case of 
$ \psi_1(\cR)\sigma -  \psi_2(\cS)\sigma$ may be dealt with in a similar manner.
\end{proof}

The following statement shows that modulo smoothing symbols, cutting the low frequencies on the left or on the right makes no difference:
\begin{lemma}
\label{lem_homsigmapsiLR}
    Let $\psi_0$ be a smooth function on $M\times \bR$ satisfying $\psi_0(x,\lambda)=0$ for $\lambda$ in a neighbourhood of $(-\infty,0]$
and $\psi_0(x,\lambda)=1$ on a neighbourhood of $+\infty$, 
these two $\lambda$-neighbourhoods depending locally uniformly on $x\in M$.
Let $\cR$  be a positive Rockland element on $GM$.
For any $\sigma\in \dot S^m (\Gh M)$, the symbol
$\sigma \psi_0(\widehat \cR) - \psi_0(\widehat \cR)\sigma $
is smoothing, that is, in $S^{-\infty} (\Gh M).$
\end{lemma}
\begin{proof}
Let $\psi_1$ be a smooth function on $M\times \bR$ satisfying 
$\psi_1(x,\lambda)=0$ for $\lambda$ in a  neighbourhood of $(-\infty,0]$ (locally uniformly in $x\in M$)
and $\psi_1=1$ on the support of $\psi_0$.
We may compute modulo smoothing symbols by Corollary~\ref{corprop:fct_de_R} and Proposition \ref{prop_homsigmapsi}
\begin{align*}
 \sigma \psi_0(\widehat\cR) - \psi_0(\widehat\cR)\sigma 
& = \psi_1(\widehat \cR)
 \left(\sigma \psi_0(\widehat\cR) - \psi_0(\widehat\cR)\sigma\right) + S^{-\infty}(\Gh M)\\
 &=\psi_1(\widehat \cR)
 \sigma\,  (\psi_0-1)(\widehat\cR) - \psi_1(\psi_0-1)(\widehat\cR) \, \sigma 
 + S^{-\infty}(\Gh M)\\
 & 
 \in S^{-\infty}(\Gh M),
\end{align*}
    concluding the proof.
\end{proof}

\subsubsection{Polyhomogeneous expansion  and symbols}
\label{subsubsec_SmphGhM}

We extend to filtered manifolds the notion of polyhomogeneous expansion of symbols (see Section~\ref{subsubsec_poly_exp} for the groups case). 

\begin{definition}
\label{def_polyhomexp}
	A symbol $\sigma\in S^m(\Gh M)$ admits the {\it polyhomogeneous expansion}
	$$
\sigma \sim_h \sum_{j=0} ^\infty \sigma_{m-j}, 
\qquad \sigma_{m-j}\in \dot S^{m-j}(\Gh M),
$$
if for  a function $\psi\in C^\infty(M\times \bR)$ 
 satisfying $\psi(x,\lambda)=0$ for $\lambda$ in a neighbourhood of $(-\infty,0]$
and $\psi_0(\lambda)=1$ on a neighbourhood of $+\infty$, 
these two $\lambda$-neighbourhoods depending locally uniformly on $x\in M$, 
 and for a positive Rockland element $\cR$, 
the symbol $\sigma$ admits the following expansion in $S^m(\Gh M)$ (in the sense of Definition~\ref{def:asymptotics}):
$$
\sigma \sim  \sum_{j=0} ^N  \psi(\widehat \cR)\sigma_{m-j}  .
$$
\end{definition}
By Lemma \ref{lem_indeppsi0R}, if the property in Definition \ref{def_polyhomexp} holds for one such function $\psi$ and one positive Rockland element $\cR$ then it holds for any of them.
By Lemma \ref{lem_homsigmapsiLR}, the property above is also equivalent to 
$$
\sigma - \sum_{j=0} ^N \psi(\widehat \cR) \sigma_{m-j} \in S^{m-N-1}(\Gh M) .
$$

\begin{definition}
\label{def:SmphGhM}
A symbol  is \emph{polyhomogeneous} of order $m\in \bR$ when it is a symbol in  $S^m(\Gh M)$ admitting a  polyhomogeneous expansion.
We denote by $S^m_{ph}(\Gh M)$ the classes of  polyhomogeneous symbols of order $m\in \bR$.
\end{definition}

By construction, if $\sigma\in S^m_{ph}(\Gh M)$, then $\sigma(x,\cdot)\in S^m_{ph}(\Gh_x M)$ for every $x\in M$. 
Consequently, 
the polyhomogeneous expansion on $\Gh M$, when it exists, is unique as it is unique  in the group case (Section \ref{subsubsec_poly_exp}).
\smallskip 

The properties of the homogeneous and inhomogeneous symbol classes imply that 
the polyhomogeneous symbol classes enjoy the following algebraic properties of composition and adjoint:
\begin{itemize}
    \item[(a)] the product of a symbol in $S^{m_1}_{ph}(\Gh M)$
with a symbol in $S^{m_2}_{ph}(\Gh M)$ is in $S^{m_1+m_2}_{ph}(\Gh M)$,
\item[(b)] the adjoint of a symbol in $S^{m}_{ph}(\Gh M)$ is in $S^{m}_{ph}(\Gh M)$.
\end{itemize}

%%%%%%%%%%%%%%%%%%%%%%%%%%%%%%%%%%%%%%%%%%%%%%%%%%%%%%%%%%%%%%%%%%%%%%%%%%%%%%%%%%%%%%%%%%%%%%%%%%%%%%%%%%%

\section{Kernel estimates and related results}\label{sec:kernel_estmates}

As for operators of Calderon-Zygmund's type, a deep analysis of the kernel will be at the roots of the analysis of the operators obtained by quantization of the symbols in $\cup_{m\in \bR}S^m(\widehat GM)$. In this section, we first present in Section~\ref{subsec:kernel_est} kernel estimates that are really similar to those of the group case. 
Then, in the next two sections, we develop consequences of kernels estimates regarding kernel functions that will appear when developing a pseudodifferential calculus in the rest of the article. Such terms do not appear in the group case. They are mainly due to the fact that the coordinates provided by the exponential map on the manifold do not coincide exactly with the variables of the osculating group. This difficulty will be overcome by identifying terms of higher order, that we will be able to treat. 
In Section~\ref{subsec:kernel_lemma}, we  derive technical lemma that we will use in Sections~\ref{sec:quantization} to analyse the invariance for the change of frame. In Section~\ref{subsec_kernel_convolution}, we focus on identifying convolution structure up to higher order terms. 
 This  section can be skipped at first reading.

 \begin{notation}\label{notation:q_alpha}
 If an adapted frame $\bX$ has been fixed on an open subset $U\subset M$, then 
  we consider the  family $(q_\alpha (x,v))_{\alpha\in \bN_0^n}$ consisting in the homogeneous polynomials $q_\alpha (x,v)=q_{x,\alpha}(v)$ on $\bR^n$ that form a basis dual to $(L_{\langle \bX\rangle_x}^\beta)_{\beta\in \bN_0^n}$.
  We also consider the quasinorm $|\cdot|_\bX$
  defined in Remark~\ref{rem_quasinormbX}.
 \end{notation}

\subsection{Kernel estimates}\label{subsec:kernel_est}
As a consequence of Theorem \ref{thm_kernelG} and Remark \ref{rem_thm_kernelG}, we have estimates on the convolution kernel of a symbol.

\begin{theorem}
\label{thm_kernelM}
Let $\bX$ be an adapted frame on an open set $U\subset M$. 
 	Let $\sigma\in S^m(\widehat G M)$ and $m\in \bR$.
 	Then for each $x\in U$, the distribution $\kappa_x^\bX$ is Schwartz  away from the origin. 
  Moreover, if $\cC$ is a fixed compact subset of $U$, we have the following kernel estimates:
	
	\begin{enumerate}
	    \item For any $N\in \bN_0$ , there exists $C>0$ such that for all $x\in \cC$ and $v\in \bR^n$, we have:
 	$$
|v|_\bX\geq 1 \;\;\Longrightarrow\;\;
	|\kappa_x^\bX (v)|\leq C |v|_{\bX}^{-N}.
	$$
\item	If $Q+m<0$ then $\kappa_x^\bX$ is continuous on $G$ and bounded: 
there exists $C>0$ such that  
\[
\sup_{(x,v)\in \cC\times \bR^n} | \kappa_x^\bX (v)|<C.
\]
\item	If $Q+m>0$, then there exists $C>0$ such that for all $x\in \cC$ and $v\in \bR^n$, we have
	$$
	0<|v|_\bX \leq 1 \ \Longrightarrow \ 
	|\kappa^\bX_x(v)|\leq  C|v|_\bX^{-(Q+m)}.
	$$
 \item	If $m=-Q$, then there exists $C>0$ such that for all $x\in \cC$ and $v\in \bR^n$, we have
	$$
	0<|v|_\bX \leq 1/2 \ \Longrightarrow \ 
	|\kappa^\bX_x(v)|\leq - C\ln |v|_\bX.
	$$
		\end{enumerate}

 In all the estimates above, 
 the constant $C$ may be chosen of the form 
$$
C = C_1 \|\sigma\|_{S^m(\widehat G M),(\bX,U), \cC, N'}
$$
with $C_1>0$, $N'\in\mathbb N_0$. 
\end{theorem}

Corollaries \ref{cor_contSmTXDelta} and Theorem \ref{thm_kernelM} yield also  estimates for 
$(-u)^\alpha L_{\langle \bX\rangle}^{\beta}R_{\langle \bX\rangle}^{\gamma}  \bX_x^{\beta'}\kappa^\bX_{ \sigma,x}( u)$. 
They,  together with Corollary \ref{cor_thm_kernelG}, and its subsequent remarks, also yield the following notion of homogeneous semi-norms.
Recall that the definition of the semi-norms is in~\eqref{def:semi_norm}.

\begin{corollary}
\label{cor2_kernelM}
    Let $\sigma\in S^m(\widehat G M)$ with $m\in \bR$, $\alpha, \beta,\beta'\in \bN_0^n$, $p\in [1,\infty]$.
If $m -[\alpha] +[\beta]<Q(\frac 1p -1)$  then the quantity 
\begin{equation}\label{def:norm(X,U)}
\|\sigma\|_{(\bX,U), \cC, \alpha,\beta,\beta',p}
:=
\Bigl\| \sup_{x\in \cC}| v^\alpha L_{\langle \bX \rangle}^\beta \bX^{\beta'}\kappa^\bX_{\sigma, x}| \Bigr\|_{L^p(\bR^n)}
\end{equation}
is finite. Moreover, the map $\sigma\mapsto \|\sigma\|_{(\bX,U),\cC, \alpha,\beta,\beta',p}$ is a continuous semi-norm on $S^m(\widehat G M)$
while if $m \leq -M_0$ with $M_0\in \bN_0$ a common multiple of the dilations' weight, we have for all $N\in\bN$
$$
 \|\sigma\|_{S^m(\widehat G M),(\bX,U), \cC, N}
\leq 
C
\max_{\substack{[\alpha]+[\beta]\leq N+M_0\\|\beta'|\leq N}}
\|\sigma\|_{(\bX,U),\cC,\alpha,\beta,\beta',1},
$$
\end{corollary}

From the kernel estimate above via Corollary \ref{cor_thm_kernelG} (3), we also obtain  an estimate for the order of the distribution~$\kappa_x^\bX$. We will need the following quantification of this property:

\begin{corollary}
\label{cor_OrderDistrib}
    	Let $\sigma\in S^m(\widehat G M)$.
	Let $\bX$ be an adapted frame on an open set $U\subset M$. 
 For any $x\in U$,
the distribution $\kappa_x^\bX$ is smooth away from the
origin. At the origin, 
its order is the integer 
$$
p_\sigma:=\max\{|\alpha|: \alpha \in \bN_0 , [\alpha] \leq s'\}
$$
where $s'$ is the smallest non-negative  integer such that $s'>Q/2+m$. 
Moreover, for any compact $\cC \subset U$, 
we have for any compact neighbourhood $\cC_0$ of $0$ in $\bR^n$, 
  $$
  \exists C>0 ,\;\;
  \forall \phi \in C_c^\infty (\cC_0 ), \;\; \forall x\in \cC, \;\; 
   \left|\int_{\bR^n} \kappa_x^\bX (v) \phi(v) dv \right|
  \leq C \max_{|\alpha | \leq p_\sigma}  \| \partial^\alpha \phi\|_{L^\infty(\cC_0)}.
  $$
Above,   $C$ may be chosen of the form 
$$
C = C_1 \|\sigma\|_{S^m(\widehat G M),(\bX,U), \cC_0, N'},
$$
with $C_1>0$, $N'\in\mathbb N_0$.
\end{corollary}

Finally,
the next corollary identifies the neighbourhood of $0$ as the singular part of the kernel of a symbol, in terms of smoothing symbols. 
 We obtain it here as a consequence of the kernel estimates, although we can also deduce it from Proposition \ref{prop_smoothingM} (1).

\begin{corollary}\label{cor:ker_smoothing}
Let $\sigma\in S^m(\Gh M)$ for some $m\in \bR$. 
\begin{enumerate}
	\item If its convolution kernel $\kappa_{\sigma}^\bX$ in an adapted frame $(\bX,U)$ is a smooth function on $U\times \bR^n$ then the restriction of $\sigma$ to $U$ is smoothing, i.e. $\sigma|_U \in S^{-\infty}(\Gh M|_U)$.
	\item If $\sigma$ is smoothing, i.e. $\sigma\in S^{-\infty}(\Gh M)$, then its convolution kernel in any adapted frame $(\bX,U)$ is a smooth function on $U\times \bR^n$.
	\item If $q\in C^\infty (GM)$ is such that 
	$q\equiv 0$ in a neighbourhood of $M\times \{0\}$, 
	then $\Delta_q \sigma\in S^{-\infty} (\Gh M)$. 	
\end{enumerate}
\end{corollary}
\begin{proof}
Part (1) follows from Corollary \ref{cor2_kernelM}.
Part (3) follows from Part (1). 
	Part (2) follows from the kernel estimates in Theorem~\ref{thm_kernelM} together with Corollary \ref{cor_contSmTXDelta}.
\end{proof}

\subsection{Kernel estimates and higher order terms}\label{subsec:kernel_lemma}

We first analyse convolution kernels of the form $(x,v)\mapsto \kappa_{x}^\bX (v *_x^\bX r(x,v))$ with $r$ of higher order. 
We introduce the notation:
\begin{definition}
\label{def_min_index}
    For $m\in \bN_0$, we define 
    $$
    |m|_{min}:=\min\{|\alpha|: \alpha\in \bN^n, \, [\alpha]=m\}.
    $$
    For  $\alpha\in \bN_0^n$ with $[\alpha]=m$, we write  $| \alpha|_{min} :=| m|_{min}$.
\end{definition}

 \begin{ex}
 \label{ex_min_index}
     \begin{enumerate}
         \item In the case of a stratification, 
which implies $\{\upsilon_1,\ldots,\upsilon_n\}=\{1,\ldots,s\}$ with $s$ the gradation step, we compute readily:
   $|1|_{min}=1$, and more generally  
   $$
   |j|_{min} = 1\quad\mbox{for}\ j\leq s, \qquad
   |j|_{min} = 1+{\rm Int}[j/(s+1)]
   \quad\mbox{for}\ j\in \bN,
   $$
above, ${\rm Int}[x]$
denotes the largest integer less or equal to $x>0$.
\item Let us consider $\bR^3$ with the gradation with weights $(2, 3, 5)$.
We compute:
 $$
 |1|_{min}=0, \quad |2|_{min} = |3|_{min}=|5|_{min}=1 \quad\mbox{but}\quad |4|_{min} = 2.
 $$ 
Hence, in the graded but not stratified case, the map $M\mapsto |M|_{min}$ may not be increasing.
\end{enumerate}
 \end{ex}

 Definition \ref{def_min_index} is motivated by the following observations:
\begin{lemma}
\label{lem_min_index}
Let $q$ be a homogeneous polynomial on $\bR^n$ that is $m$-homogeneous for $\delta_r$.
\begin{enumerate}
    \item Then for any fixed $v\in \bR^n$, we have for the isotropic dilations:
    $$
    q(\eps v)  = O(\eps^{|m|_{min}}).
    $$
    \item If  $r$ is a smooth higher order map on $\bR^n$, then $p=q\circ r$ vanishes at homogeneous order $m + |m|_{min}-1$ (in the sense of Definition \ref{def_vanishingHOMorder}).
\end{enumerate}
\end{lemma}
\begin{proof}
Part (1). Writing $q(v) = \sum_{[\alpha]=m} c_\alpha v^\alpha$, we have
$$
q(\eps v) = \sum_{[\alpha]=m} c_\alpha (\eps v)^\alpha = \sum_{[\alpha]=m} c_\alpha \eps^{|\alpha|} v^\alpha.
$$
This proves Part (1).
      
Part (2).  By Proposition \ref{prop_char_higherorder}, we may write $\delta_{\eps}^{-1}r(\delta_\eps v) = \eps R(v,\eps)$, with $R$ smooth. 
Hence, we have
   $$
   p(\delta_\eps v) =  q(\delta_\eps\delta_\eps^{-1} r(\delta_\eps v) )
   = \eps^m q (\delta_\eps^{-1} r(\delta_\eps v) )
   = \eps^m q(\eps R(v,\eps))
   = O(\eps^{m + |m|_{min}}),
   $$
   by Part (1). We conclude the proof of Part (2) with Lemma \ref{lem_van_homo_char}.
\end{proof}

In our analysis of the local calculus on $M$, 
we will often consider a symbol $\sigma$
modified into a symbol $\tau$ formally given by its convolution kernel  $
\kappa_{\tau,x}^\bX:v\mapsto
\kappa_{\sigma,x}^\bX (v *_x^\bX r(x,v))$
with $r$ higher order.
The following statement gives an approximation of $\tau$
by an asymptotic sum of symbols related to $\sigma$:

\begin{proposition}
\label{prop:kernel_higher_order}
    	Let $\sigma\in S^m(\Gh M)$.
	Let $\bX$ be an adapted frame on an open set $U\subset M$. 
Let $r\in C^\infty (U\times \bR^n)$ 
be  higher order. 
Let $\chi\in C^\infty (U\times \bR^n)$  be a function $U$-locally compactly supported in $\bR^n$.
We also assume that $\chi=1$ on a neighbourhood of $\{(x,0), x\in U\}$.
For every $\alpha\in \bN_0$, with Notation~\ref{notation:q_alpha} for $q_\alpha$,
we set 
$$
p_\alpha(x, v) := q_\alpha (x,  r(x,v)), 
\quad x\in U, v\in \bR^n, 
$$
The following holds:
\begin{enumerate} 
\item 
For each $\alpha\in \bN_0$, 
the symbol $\Delta_{\chi p_\alpha} \widehat {\langle\bX\rangle}^\alpha\sigma$ 
is in $S^{m-|\alpha|_{min}}(\Gh M|_U)$,
and the map 
$$
\sigma\longmapsto \Delta_{\chi p_\alpha} \widehat {\langle\bX\rangle}^\alpha\sigma, 
\qquad S^m(\Gh M)\longrightarrow S^{m-| \alpha|_{min}}(\Gh M|_U),
$$ 
is continuous. 

\item 
The symbol
$$
\tau = \sigma^{(r,\chi)}
$$
 given by its convolution kernel in the $\bX$-coordinates
\begin{equation}\label{def:symbol_*r}
\kappa_{\tau,x}^\bX (v) := \chi(x,v) \kappa_{\sigma,x}^\bX (v *_x^\bX r(x,v)), \;\; v\in\bR^n,
\end{equation}
is in $S^m(\Gh M|_U)$. Moreover, it admits the asymptotic expansion in $S^m(\Gh M)$
\begin{equation}\label{def:tauj}
\tau \sim \sum_j \tau_j, 
\qquad \tau_j:= \sum_{j=|\alpha|_{min}} \Delta_{\chi p_\alpha} \widehat {\langle\bX\rangle}^\alpha\sigma 
\ \in S^{m-j}(\Gh M|_U),
\end{equation}
and, for each  $N\in \bN_0$, the map 
$$
\sigma \longmapsto \tau - \sum_{j=0}^N \tau_j, 
\qquad S^m(\Gh M)\to S^{m-N-1}(\Gh M|_U),$$ 
is continuous.
\end{enumerate}
\end{proposition}

When the filtered manifold $M$ is equiregular subRiemannian as in Example~\ref{ex:subrie}, 
the gradation is in fact a stratification so 
 the indices $|\alpha|_{min}$ are computed in  Example \ref{ex_min_index} (1).
 However, if the gradation is not a stratification (for instance as in Example \ref{ex_min_index} (2)), there may be terms with very large $[\alpha]$ contributing to the $S^{m-1}$ term  in the asymptotic expansion~\eqref{def:tauj} in that case. 
\smallskip 

Note however that the sum defining $\tau_j$ in~\eqref{def:tauj} is finite. Indeed, for $j\in\bN$ fixed, the set $\cA_j=\{\alpha\in\bN^n;\, j=|\alpha|_{\rm min}\}$ cannot contain a sequence of multi-indices $(\alpha^{(\ell)})_{\ell\in\bN}$ such that $|\alpha^{(\ell)}|\rightarrow_{\ell\rightarrow+\infty}+\infty$. 
Indeed, if it would, then we could find $N$ such that for $\ell\geq N$, we have  $[\alpha^{(\ell)}]\geq|\alpha^{(\ell)}|>(j+1)\upsilon_n$.  Choosing then $\beta$ such that $[\beta]=[\alpha^{(\ell)}]$ and  $|\beta|=|\alpha|_{\rm min}=j$, we have simultaneously  $[\beta]\leq \upsilon_n |\beta|=\upsilon_n j$, and $[\beta]=[\alpha^{(\ell)}]>(j+1)\upsilon_n$, which is a contradiction.
\smallskip

The proof of Proposition \ref{prop:kernel_higher_order} relies on a  structural property that we now state. For various choices of symbols $\sigma$ and cut-off $\chi$, we use the  notation $\tau(\sigma,\chi)$ to denote the symbol 
determined  by its kernel in the $\bX$-coordinates given formally by~\eqref{def:symbol_*r}.

\begin{lemma}
\label{lem_der_kappatau}
	   	Let $\sigma\in S^m(\Gh M)$.
	Let $\bX$ be an adapted frame on an open set $U\subset M$. 
Let $r\in C^\infty (U\times \bR^n)$ 
be higher order. 
Let $\chi\in C^\infty (U\times \bR^n)$  be a function $U$-locally compactly supported in $\bR^n$.
We assume that $\chi(x,v)$ vanishes at order~$N_0$ in~$v$ for any~$x\in U$, and we
consider the symbol $\tau=\tau(\sigma,\chi)$ determined  by its kernel in the $\bX$-coordinates given formally by~\eqref{def:symbol_*r}.
We can construct functions 
$$
\chi_{\beta,\beta_1},
\chi'_{\beta',\beta'_1,\beta'_2}\in C^\infty(U\times \bR^n),
\quad \beta,\beta',\beta_1,\beta'_1,\beta'_2\in \bN_0^n, \quad 
  |\beta_1|\leq |\beta|,\ |\beta'_1|+|\beta'_2|\leq |\beta'|, 
$$
 vanishing at homogeneous order $\max(N_0+[\beta_1]-[\beta], -1)  $
 and 
 $\max(N_0+[\beta'_2], -1)  $  respectively,
and such that
 \begin{align*}
 L_{\langle \bX\rangle_x}^\beta 
 \kappa_{\tau(\sigma,\chi),x}
 &= \sum_{|\beta_1|\leq |\beta|}
 \kappa_{\tau_{\beta,\beta_1},x}
 \qquad \mbox{where}\quad \tau_{\beta,\beta_1}:=
\tau(\widehat {\langle \bX\rangle_x}^{\beta_1} \sigma,\chi_{\beta,\beta_1}),\\
\bX^{\beta'}_x \kappa_{\tau(\sigma,\chi),x}& = 
\sum_{|\beta'_1|+|\beta'_2|\leq |\beta'|}
\kappa_{\tau'_{\beta',\beta'_1,\beta'_2},x}
\qquad\mbox{where}\quad 
\tau'_{\beta',\beta'_1,\beta'_2}
:=
\tau(D_\bX^{\beta'_1} \widehat {\langle \bX\rangle_x}^{\beta'_2}\sigma,\chi'_{\beta',\beta'_1,\beta'_2}).
 \end{align*}
\end{lemma}

\begin{remark}\label{rem:lemma6.10}
 The numerology above is explained by the following observation:  if there was no terms in $r$, that is if $r=0$, it is then straightforward that $\tau\in S^{m}(\Gh M|U)$. Moreover, the function 
$L_{\langle \bX\rangle_x}^\beta 
 \kappa_{\tau(\sigma,\chi),x}$ is the kernel associated with a symbol of order $m-N_0+[\beta]$ (by Corollary~ \ref{cor_GM_Deltaq}) and so is any $\tau_{\beta,\beta_1}$.
 Similarly,  $\bX^{\beta'}
 \kappa_{\tau(\sigma,\chi),x}$ is a kernel associated with a symbol of order $m-N_0$ and so are $\tau'_{\beta',\beta'_1,\beta'_2}$.
The result of Lemma~\ref{lem_der_kappatau} consists in proving that these properties still holds with a non zero term in $r$, when $r$ is higher order.  Indeed, the order of the vanishing of the functions $\chi_{\beta,\beta_1}$ and $\chi_{\beta',\beta'_1,\beta'_2}$ is exactly compensating the loss of regularity of the kernel when it is differentiated in $v$. 
\end{remark}

The proof of Lemma \ref{lem_der_kappatau} relies on Lemma~\ref{lem_comp_der_groupe} reformulated in the context of the groups $G_xM
$ for $x\in M$ as follows:
if $f:\bR^n\to \bR$ and $g:\bR^n\to \bR^n$ are smooth, then 
\begin{equation}
\label{eq_comp_der_M}
L_{\langle X_{j}\rangle_x} f\circ g(v) = \sum_{\ell=1}^n h_\ell (v)
 \ (L_{\langle X_\ell\rangle_x} f) (g(v)),
\end{equation}
where $h_\ell(v):=\partial_{t=0} \left[(-g(v)) *_x^\bX g (v*_x ^\bX te_j)\right]_\ell$.

\begin{proof}[Proof of Lemma \ref{lem_der_kappatau}]
	Thanks to the  Leibniz property of vector fields, 
inductively, it suffices to show the result 
for  $|\beta|=|\beta'|=1$,
 that is for $L_{\langle X_{j}\rangle_x}$ and for $X_{j,x}$.
 
 \smallskip 

Let us analyse the $L_{\langle X_{j}\rangle_x}$-derivative, which means that we choose $\beta$ such that $[\beta]=\upsilon_j$. 
We have
\begin{align*}
L_{\langle X_{j}\rangle_x}	\kappa_{\tau,x}^\bX (v) 
&= L_{\langle X_{j}\rangle_x}\chi(x,v) \, \kappa_{\sigma,x}^\bX (v *_x^\bX r(x,v))
+ \chi(x,v)
 L_{\langle X_{j}\rangle_x} \left(\kappa_{\sigma,x}^\bX (v *_x^\bX r(x,v))\right)
\end{align*}
As $L_{\langle X_{j}\rangle_x}\chi(x,v)$ vanishes at homogeneous order $\max(N_0 -\upsilon_j,-1)$, 
the first term is of the correct form.
 By~\eqref{eq_comp_der_M}, the second writes
\begin{align*}
 L_{\langle X_{j}\rangle_x} \left(\kappa_{\sigma,x}^\bX (v *_x^\bX r(x,v))\right)=
\sum_{\ell=1}^n
  f_\ell(x,v)
 \left(L_{\langle X_\ell\rangle_x} \kappa_{\sigma,x}^\bX\right) (v *_x^\bX r(x,v)),
 \end{align*}
where 
 \begin{align*}
 f_\ell(x,v)&:=\partial_{t=0} 
 \left[ 
 (-(v*_x^\bX  r(x,v))) *_x^\bX v*_x^\bX  t e_j *_x^\bX r(x,v*_x^\bX te_j)\right]_\ell.
 \end{align*}
 In the sum above, the terms for which $ \upsilon_\ell<\upsilon_j$ are already of the correct form. Indeed, the derivatives $L_{\langle X_\ell\rangle_x}$ induces loss of regularities smaller or equal to $\upsilon_j=[\beta]$. It is not the case for the other terms. However, we can conclude for them once we prove that for $ \upsilon_\ell\geq \upsilon_j$, the function $f_\ell$ vanishes at homogeneous order $\upsilon_\ell-\upsilon_j$.
 Let us do so. By homogeneity, we have:
  \begin{align*}
 \eps^{-\upsilon_\ell+ \upsilon_j} f_\ell(x,v)&=\partial_{t=0} 
 \left[ \delta_\eps^{-1}\left(
 (-(v*_x^\bX r(x,v))) *_x^\bX v*_x^\bX \eps^{\upsilon_j} t e_j *_x^\bX r(x,v*_x ^\bX\eps^{\upsilon_j} te_j)\right) \right]_\ell ,
 \end{align*}
 so 
   \begin{align}\label{eq:der57}
 \eps^{-\upsilon_\ell+ \upsilon_j} f_\ell(x,\delta_\eps v)&=\partial_{t=0} 
 \left[ 
 (-(v*_x^\bX s_0(\eps)) *_x^\bX v*_x^\bX  t e_j *_x^\bX s_t(\eps) ) \right]_\ell ,
 \end{align}
where $s_t(\eps ):= \delta_\eps^{-1} r(x,\delta_\eps (v *_x^\bX te_j))$; here, we have fixed $(x,v)$.
As $r$ is smooth and higher order,
by Proposition \ref{prop_char_higherorder}, $(t,\eps)\mapsto s_t(\eps)$ is smooth on $\bR\times \bR$ and vanishes at $\eps=0$ for any $t$. 
The RHS in~\eqref{eq:der57} is smooth in $\eps$ and taken as $\eps=0$, it is equal to 
\[
\partial_{t=0} [(-v)*_x ^\bX v*_x ^\bX te_j ]_\ell 
=\partial_{t=0} [ te_j ]_\ell = [e_j]_\ell = \delta_{j=\ell}.
\] 
We have obtained 
$$
\forall j \neq \ell, \qquad 
\lim_{\eps\to 0 }  \eps^{-\upsilon_\ell+ \upsilon_j} f_\ell(x,\delta_\eps v)=0.
$$
Consequently, 
  $f_\ell$ vanishes at 
homogeneous order $\upsilon_j-\upsilon_\ell$ when 
 $\upsilon_j\leq \upsilon_\ell$.
We can now conclude that $L_{\langle X_{j}\rangle_x}	\kappa_{\tau,x}^\bX$ is of the correct form. 
\smallskip

We now analyse the $X_{j,x}$-derivatives. 
We have
\begin{align*}
X_{j,x} \kappa_{\tau,x}^\bX (v) 
&= X_{j,x}\chi(x,v) \, \kappa_{\sigma,x}^\bX (v *_x^\bX r(x,v))
+ \chi(x,v) 
(X_{j}\kappa_{\sigma,x}^\bX) (v *_x^\bX r(x,v))\\
&\qquad 	+ \chi(x,v) 
X_{j,x_1=x}\left(\kappa_{\sigma,x}^\bX (v *_{x_1}^\bX r(x,v))\right)
+ \chi(x,v) 
X_{j,x_1=x}\left(\kappa_{\sigma,x}^\bX (v *_{x}^\bX r(x_1,v))\right).
\end{align*}
 The first two terms are already in the correct form. 
  Let us analyse the third term. We first transform it in order to highlight higher order terms:
 \begin{align*}
X_{j,x_1=x}\left(\kappa_{\sigma,x}^\bX (v *_{x_1}^\bX r(x,v))\right)
 	&= 
 X_{j,x_1=x} \kappa_{\sigma,x}^\bX  (v *_x^\bX r(x,v)*_x^\bX s(x,x_1,v))\\
 &=\sum_{\ell=1}^n
  [X_{j,x_1=x} s(x,x_1,v)]_\ell
 (L_{\langle X_\ell\rangle_x} \kappa_{\sigma,x}^\bX) (v *_x^\bX r(x,v)),
 \end{align*}
 where 
 $$
 s(x,x_1,v) := (-(v *_{x}^\bX r(x,v)))*_x (v *_{x_1}^\bX r(x,v) ).
 $$
  We observe that 
 $$
 \delta_\eps^{-1} s(x,x_1,\delta_\eps v)
 =
 (-(v *_{x}^\bX \delta_\eps^{-1} r(x,\delta_\eps v)))*_x (v *_{x_1}^\bX \delta_\eps^{-1} r(x,\delta_\eps v) )\longrightarrow_{\eps\to0} 0,
 $$
since $\lim_{\eps\to 0} \delta_\eps^{-1} r(x,\delta_\eps v))) =0$ by Proposition \ref{prop_char_higherorder}, $r$ being higher order. Using again Proposition \ref{prop_char_higherorder} implies that $s$ is higher order in $v$.
By Corollary \ref{cor_derivatives_cV}, so is $X_{j,x_1=x} s(x,x_1,v)$.
Hence, $[X_{j,x_1=x} s(x,x_1,v)]_\ell$ vanishes at homogeneous order $\upsilon_\ell$, implying that the third term is of the correct form.  

 Let us analyse the fourth term. As before, we start by exhibiting higher order terms:
 \begin{align*}
X_{j,x_1=x}\left(\kappa_{\sigma,x}^\bX (v *_x^\bX r(x_1,v))\right)
 	&= 
 X_{j,x_1=x}\kappa_{\sigma,x}^\bX  (v *_x^\bX r(x,v)*_x^\bX (-r(x,v))*_x^\bX r(x_1,v))\\
 &=\sum_{\ell=1}^n
  \left[ X_{j,x_1=x}  \tilde s(x,x_1,v)\right] _\ell
 L_{\langle X_\ell\rangle_x} \kappa_{\sigma,x}^\bX (v *_x^\bX r(x,v)),
 \end{align*}
 where
 $$
 \tilde s(x,x_1,v):=(-r(x,v))*_x^\bX r(x_1,v).
 $$
  We observe that 
   $$
 \delta_\eps^{-1}\tilde s(x,x_1,\delta_\eps v)
 =
  (-\delta_\eps^{-1} r(x,\delta_\eps v))*_x^\bX \delta_\eps ^{-1} r(x_1,\delta_\eps v) \longrightarrow_{\eps \to 0} 0,
$$
since $\lim_{\eps\to 0} \delta_\eps^{-1} r(x,\delta_\eps v)) =0$ by Proposition \ref{prop_char_higherorder}, $r$ being higher order. Using again Proposition \ref{prop_char_higherorder} implies that $\tilde s$ is higher order in $v$.
By Corollary \ref{cor_derivatives_cV}, 
 so is $X_{j,x_1=x} \tilde s(x,x_1,v)$.
Hence, $[X_{j,x_1=x} \tilde s(x,x_1,v)]_\ell$ vanishes at homogeneous order $\upsilon_\ell$, implying that the fourth term is of the correct form.  
We conclude that $X_{j,x}	\kappa_{\tau,x}^\bX$ is of the correct form. 
\end{proof}

We can now present the proof of Proposition \ref{prop:kernel_higher_order}.

\begin{proof}[Proof of Proposition \ref{prop:kernel_higher_order}]
By Lemma \ref{lem_min_index} (2),  $v\mapsto p_\alpha (x,v)$ vanishes at order $[\alpha] + |\alpha|_{min}-1 $.
Part (1) then follows from   Corollaries \ref{cor_GM_Deltaq} and \ref{cor_contSmTXDelta}.
Let us now prove Point (2).
We may assume that on the support of $\chi$,
there exists a constant $C_0>0$ such that
$$
 \forall (x,v)\in U\times B_0, \quad 
\forall w\in \bR^n,\quad |w|_\bX \leq |r(x,v)|_\bX 
\Longrightarrow
|v*_x^\bX w|_\bX \geq C_0 |v|_\bX.
$$

We now apply  the Taylor estimates due to Folland and Stein~\cite{folland+stein_82} on the group $G_x M$
and recalled in Theorem~\ref{thm_MV+TaylorG} (2) 
(see also  Remark \ref{remthm_MV+TaylorG} for the local uniformity in $x$). Using  the notation  $\lceil N \rfloor$ and $\eta$ of Theorem~\ref{thm_MV+TaylorG},
 we have for any $r\in \bR^n$
 \begin{align}
 \label{eq_pf_kappa(vr)}
	&| \kappa_{\sigma,x}^\bX\left(v*_x^\bX r\right)
- \sum_{[\alpha]\leq N} q_\alpha (r) 
L_{\langle \bX\rangle_x}^\alpha \kappa_x^\bX (v)|
\\&\qquad\qquad \leq C \sum_{\substack{|\alpha|\leq \lceil N \rfloor +1\\ [\alpha]> N}}|r|_\bX^{[\alpha]}\sup_{|w|_\bX \leq \eta^{\lceil N \rfloor+1}|r|_\bX}\bigg| (L_{\langle\bX\rangle_x}^\alpha \kappa_x^{\bX})(v*_x^\bX w)\bigg|, \nonumber   
 \end{align}
with a constant $C=C_{N,x}$ depending locally uniformly in $x\in U$; we may assume $\eta\geq 1$. 
We set
\begin{equation}\label{notation_noyaux31}
\kappa_{N,\sigma,\chi, x}(v):=
\kappa_{N,x}(v):=
\sum_{[\alpha]\leq N} \chi (x,v) q_\alpha (x,  r(x,v))
L_{\langle\bX\rangle_x}^\alpha  \kappa_{\sigma,x}^\bX (v),
\end{equation}
and, by Part 1, the kernels $(x,v)\mapsto \chi (x,v) q_\alpha (x,  r(x,v))
L_{\langle\bX\rangle_x}^\alpha  \kappa_{\sigma,x}^\bX (v)$ are convolution kernels in the $\bX$-coordinates of symbols in $S^{m- |\alpha|_{min}}(\widehat GM|_U)$.

Let us now analyse $\kappa_{N,x}$.
We choose $N$ large enough so that $$
m+\epsilon_0 N>-Q
\qquad\mbox{with} \qquad \epsilon_0:=1/\upsilon_n.
$$ 
For $|v|_\bX\leq 1$, we have by Corollary \ref{cor1lem_char_higherorder} and the kernel estimates from  Theorem \ref{thm_kernelM}:
\begin{align*}
|\kappa_{\tau,x}(v)-\kappa_{N,x} (v) | 
&\leq C \sum_{\substack{|\alpha|\leq \lceil N \rfloor +1\\ [\alpha]> N}}|r(x,v)|_\bX ^{[\alpha]}\sup_{|w|_\bX\leq \eta^{\lceil N \rfloor+1}|r(x,v)|_\bX}| (L_{\langle\bX\rangle_x}^\alpha \kappa_x^{\bX})(v*_x^\bX w)|\\
&\lesssim  \sum_{\substack{|\alpha|\leq \lceil N \rfloor +1\\ [\alpha]> N}}|v|_\bX^{[\alpha](1+\eps_0)}\sup_{|w|_\bX\leq \eta^{\lceil N \rfloor+1}|r(x,v)|_\bX}| v*_x^\bX w|_\bX^{-Q-m -[\alpha]}  \|\sigma\|_{S^m(\widehat{G}M), (\bX, U), \cC, N'} \\
&\lesssim  |v|_\bX^{N\eps_0-m-Q}  \|\sigma\|_{S^m(\widehat{G}M), (\bX, U), \cC, N'} .
\end{align*}
More precisely, we have obtained the existence of a constant $C>0$ depending locally uniformly in $x\in U$ such that for $|v|_\bX\leq 1$, 
\begin{equation}
\label{eqpfprop:kernel_higher_orderEST}
    |\kappa_{\tau,x}(v)-\kappa_{N,x} (v) | 
\leq C
|v|_\bX^{\epsilon_0 N -Q-m} 
\max_{\substack{|\alpha|\leq \lceil N \rfloor +1\\ [\alpha]> N}}
\sup_{v'\in \bR^n\setminus \{0\}}
||v'|_\bX^{-Q-m-[\alpha]}L_{\langle\bX\rangle_x}^\alpha \kappa_x^{\bX}(v')|.
\end{equation}
As the supports of both $\kappa_{\tau,x}(v)$ and $\kappa_{N,x} (v)$
are included in the support of the function $\chi(x,v)$, the kernel estimates from  Theorem \ref{thm_kernelM} imply that there exist $C_1>0$ and $N'\in \bN_0$  such that $$\|\tau- \sum_{[\alpha]\leq N }
\Delta_{\chi p_\alpha }\widehat{\langle\bX\rangle}^\alpha \sigma \|_{(\bX, U),\cC ,0, 0, 0, \infty}\leq C_1\|\sigma\|_{S^m(\widehat G M),(\bX,U), \cC, N'}$$
for the semi-norm defined in \eqref{def:norm(X,U)}, as soon as
$N>(Q+m)/\epsilon_0$.
\smallskip

We now introduce derivatives in $x$ and $v$. By Lemma~\ref{lem_der_kappatau},
$$
	L_{\langle \bX\rangle_x}^\beta 
\bX^{\beta'}_x \kappa_{\tau(\sigma,\chi),x}(v)
 = \sum_{|\beta_1|\leq |\beta|}
\sum_{|\beta'_1|\leq |\beta'|}
\kappa_{\tau_{\beta_1,\beta'_1},x}(v),
\quad\mbox{where}\quad
\tau_{\beta_1,\beta'_1}:=
\tau(\widehat {\langle \bX\rangle_x}^{\beta_1} 
D_\bX^{\beta'_1} \sigma,\chi_{\beta_1,\beta'_1}),
$$
and we can perform the same analysis as above on each $\kappa_{\tau_{\beta_1,\beta'_1},x}$. Extending the notation in~\eqref{notation_noyaux31}, we obtain:
\begin{align*}
&	|\kappa_{\tau_{\beta_1,\beta'_1},x}(v)-\kappa_{N,\widehat {\langle \bX\rangle_x}^{\beta_1} 
D_\bX^{\beta'_1} \sigma,\chi_{\beta_1,\beta'_1},x} (v) | 
\\&\quad \leq C
|v|_\bX^{\epsilon_0 N -Q-m} 
\max_{\substack{|\alpha|\leq \lceil N \rfloor +1\\ [\alpha]> N}}
\sup_{v'\in \bR^n\setminus \{0\}}
||v'|_\bX^{-Q-m-[\alpha]}L_{\langle\bX\rangle_x}^\alpha L_{\langle \bX\rangle_x}^{\beta_1} 
\bX^{\beta'_1} \kappa_x^{\bX}(v')|.
\end{align*}
In particular, we have 
\begin{equation}
\label{eq_der_pfprop:kernel_higher_order}
|L_{\langle \bX\rangle_x}^\beta
\bX^{\beta'}_x\left (\kappa_\tau (v)
- 
\kappa_{N,\sigma,\chi, x}(v)\right )
| 
\lesssim |v|_\bX^{\epsilon_0 N -Q-m - [\beta]} ,
\end{equation}
with an implicit constant depending on a seminorm of $\sigma\in S^m (\Gh M)$,
thereby finishing the proof,  once we show 
\begin{equation}
\label{eq_endpfprop:kernel_higher_order}
|L_{\langle \bX\rangle_x}^\beta 
\bX^{\beta'}_x\kappa_{N,\sigma,\chi, x}(v)
-\sum_{|\beta_1|\leq |\beta|}
\sum_{|\beta'_1|\leq |\beta'|}\kappa_{N,\widehat {\langle \bX\rangle_x}^{\beta_1} 
D_\bX^{\beta'_1} \sigma,\chi_{\beta_1,\beta'_1}, x}(v)| 
\lesssim |v|_\bX^{\epsilon_0 N -Q-m - [\beta]} ,
\end{equation}
again
with an implicit constant depending on seminorms of $\sigma\in S^m (\Gh M)$. 
This will follow  from the fact that the Taylor polynomials at order $N_1\leq \epsilon_0 N-Q-m-[\beta]$ of 
$L_{\langle \bX\rangle_x}^\beta 
\bX^{\beta'}_x\kappa_{N,\sigma,\chi, x}(v)$
and $\sum_{|\beta_1|\leq |\beta|}
\sum_{|\beta'_1|\leq |\beta'|}
\kappa_{N,\widehat {\langle \bX\rangle_x}^{\beta_1} 
D_\bX^{\beta'_1} \sigma,\chi_{\beta_1,\beta'_1}, x}(v)$ coincide:
\begin{equation}
	\label{eq_pf_bPcoincide}
\bP_{N_1}^{L_{\langle \bX\rangle_x}^\beta\bX_x^{\beta'}\kappa_{N,x}}(v)
= \sum_{|\beta_1|\leq |\beta|}
\sum_{|\beta'_1|\leq |\beta'|}
\bP_{N_1}^{\kappa_{N, \widehat {\langle \bX\rangle_x}^{\beta_1} 
D_\bX^{\beta'_1} \sigma,\chi_{\beta_1,\beta'_1},x}}(v).
\end{equation}
Indeed, in this case, \eqref{eq_endpfprop:kernel_higher_order} 
follows readily from the Taylor estimates due to Folland-Stein 
at order $N_1$
for 
$L_{\langle \bX\rangle_x}^\beta 
\bX^{\beta'}_x\kappa_{N,\sigma,\chi, x}(v)$
and $\sum_{|\beta_1|\leq |\beta|}
\sum_{|\beta'_1|\leq |\beta'|}\kappa_{N,\widehat {\langle \bX\rangle_x}^{\beta_1} 
D_\bX^{\beta'_1} \sigma,\chi_{\beta_1,\beta'_1}, x}(v)$.

It remains to show \eqref{eq_pf_bPcoincide}.
We start by observing that  the analysis above implies that the (Folland-Stein) Taylor polynomial 
for $\kappa_{\tau,x}(v)$ and for 
$\kappa_{ N,x}(x)$
coincide up to order $N_1 <\epsilon_0 N-Q-m$:
$\bP_{N_1}^{\kappa_{\tau,x}}(v)
=
\bP_{N_1}^{\kappa_{N,x}}(v).$
This implies that the $(x,v)$-derivatives of these polynomials also coincide:
$$
L_{\langle \bX\rangle_x}^\beta 
\bX^{\beta'}_x 
\bP_{N_1}^{\kappa_{\tau,x}}(v)
=
L_{\langle \bX\rangle_x}^\beta 
\bX^{\beta'}_x 
\bP_{N_1}^{\kappa_{N,x}}(v) ,
$$
for any $\beta,\beta'$.
From the general properties of Taylor polynomials,
we have for $N_2\geq [\beta]$
$$
L_{\langle \bX\rangle_x}^\beta 
\bX^{\beta'}_x 
\bP_{N_2}^{\kappa_{\tau,x}}(v)
=
\bP_{N_2-[\beta]}^{L_{\langle \bX\rangle_x}^\beta\bX_x^{\beta'}\kappa_{\tau,x}}(v), 
\qquad 
L_{\langle \bX\rangle_x}^\beta 
\bX^{\beta'}_x 
\bP_{N_2}^{\kappa_{N,x}}(v)
=
\bP_{N_2-[\beta]}^{L_{\langle \bX\rangle_x}^\beta\bX_x^{\beta'}\kappa_{N,x}}(v).
$$
Hence, we have for  $N_1<\epsilon_0 N- Q-m-[\beta]$
$$
\bP_{N_1}^{L_{\langle \bX\rangle_x}^\beta\bX_x^{\beta'}\kappa_{\tau,x}}(v)
=\bP_{N_1}^{L_{\langle \bX\rangle_x}^\beta\bX_x^{\beta'}\kappa_{N,x}}(v).
$$
By Lemma \ref{lem_der_kappatau}, the LHS is equal to
$$
\sum_{|\beta_1|\leq |\beta|}
\sum_{|\beta'_1|\leq |\beta'|}
\bP_{N_1}^{\kappa_{\tau_{\beta_1,\beta'_1},x}}(v)
=\sum_{|\beta_1|\leq |\beta|}
\sum_{|\beta'_1|\leq |\beta'|}
\bP_{N_1}^{\kappa_{
\sigma_{\beta_1,\beta'_1},\chi_{\beta_1,\beta'_1},x}}(v),
$$
having applied the above result to $\sigma_{\beta_1,\beta'_1}= \widehat{\langle \bX\rangle_x}^{\beta_1} 
D_{\bX}^{\beta'_1} \sigma $ with $\chi_{\beta_1,\beta'_1}$.
This shows 
\eqref{eq_pf_bPcoincide} and concludes the proof. 
\end{proof}

\subsection{Kernel estimates and perturbed convolution} \label{subsec_kernel_convolution}
The following technical result analyses 
an integral expression given by the perturbation of  the convolution of a kernel with a function.

\begin{lemma}
\label{lem:symbol_perturbconvolution}
   Let $\bX$ be an adapted frame on an open set $U\subset M$. 
 Let  $J\in C^\infty (U\times \bR^n\times \bR^n)$ be $U$-locally compactly supported in $\bR^n\times \bR^n$.
 Let $\sigma\in S^m(\Gh M)$	and denote its integral kernel in the $\bX$ coordinates by $\kappa^\bX_\sigma$.
For any $f\in \cS(\bR^n)$, we define in the distributional sense:
\begin{equation}\label{int_eq_convol}
\kappa_{\sigma,f,J,x}^\bX (v)
  := 
\int_{\bR^n} \kappa_{\sigma ,x}^\bX (w ) \ f(v*^\bX_x (-w) ) \ J(x,v,w) dw, 
\qquad x\in U, \ v\in \bR^n.
\end{equation}
We define the  coefficients $c_{\alpha}(x,v)$  described via
the Taylor expansion about $w=0$
$$
	J(x,v,w)
\sim \sum_{\alpha} c_{\alpha}(x,v) w^{\alpha} .
$$
The following convolution makes sense in the sense of distribution as the convolution of a distribution with a Schwartz function $f$:
$$
f\star^\bX_{G_x M} (\cdot ^\alpha  \kappa_{\sigma,x}^\bX )\, (v)=
\int_{\bR^n} w^{\alpha}  \kappa_{\sigma,x}^\bX (w ) \ f(v*_x^\bX (-w)) \  dw, 
\qquad x\in U, \ v\in \bR^n.
$$
For any $p\in [1,\infty]$ and any $N\in \bN_0$ with $N\geq  m+Q/p$, 
for any compact subset $\cC\subset U$, 
there exist $C>0$ and $N_1\in \bN_0$ (independent of $f,\sigma$) such that 
the following inequality holds uniformly in $x\in \cC$:
\begin{equation}\label{ineq_conv}
\|\kappa_{\sigma,f,J,x}^\bX -
\sum_{[\alpha]\leq N} c_\alpha (x,  \ \cdot \ )
\ f \star_{G_x M}^\bX (\, \cdot\, ^\alpha   \kappa_{\sigma,x}^\bX )  \|_{L^p(\bR^n)}\leq C
\|\sigma\|_{S^m (\Gh M), (\bX, U), \cC, N_1}
 \| f\|_{L^p(\bR^n)} .
\end{equation}
\end{lemma} 

Recall that the notation $\star_{G_xM}$ has been introduced in~\eqref{eq_convolution}. The mention of the frame~$\bX$ is added in order to settle the choice of coordinates. 

\begin{proof} 
We fix $\chi_1,\chi_2\in C^\infty (U\times \bR^n)$ valued in $[0,1]$ and $U$-locally compactly supported such that $\chi_1(x,v)=1=\chi_2(x,w)$
on the support of $J(x;v,w)$.
We can insert $\chi_1(x;v)\chi_2(x,w)$   in the integral expression~\eqref{int_eq_convol} of $\kappa_{\sigma,f,x}^\bX (v)$.
Moreover, the kernel estimates and the properties of the convolution on groups allow us to  replace $\kappa_{\sigma,x}^\bX(w)$ with $\chi_1(x,v)\chi_2(x,w)\kappa_{\sigma,x}^\bX(w)$ in the desired inequality~\eqref{ineq_conv}.
Hence, we are now  analysing:
\begin{align*}
\kappa_{N,x}^\bX(v)
&:=\kappa_{\sigma,f,J,x}^\bX (v) -  
\sum_{[\alpha]\leq N_1} (\chi_1 c_\alpha) (x,v) \ f \star_{G_x M}^\bX
(\cdot ^\alpha  \chi_2 (x,\cdot)\kappa_{\sigma,x}^\bX )  (v)\\
&=
\int_{\bR^n} \kappa_{\sigma ,x}^\bX (w ) \ f(v*^\bX_x (-w) ) R_{N_1}(x,v,w) dw, 
\end{align*}
where
$$
R_{N_1}(x,v,w):= 
J(x,v,w) -
\sum_{[\alpha]\leq N_1} (\chi_1 c_\alpha) (x,v)
w^\alpha   \chi_2 (x,w).
$$
By the Taylor estimates due to Folland and Stein (see Theorem~\ref{thm_MV+TaylorG} and Remark \ref{remthm_MV+TaylorG}), 
we have
$$
|R_{N_1}(x,v,w)|
\lesssim \chi_1(x,v) |w|_\bX^{N_1+1},
$$
with an implicit constant depending locally uniformly on $x\in U$.
Hence, by H\"older's inequality, with $1=\frac 1p +\frac 1{p'}$, 
\begin{align*}
|\kappa_{N,x}^\bX(v)|
&\leq 
\|f(v *_x^\bX (-\ \cdot \ )\|_{L^p(\bR^n)}
\|\kappa_{\sigma ,x}^\bX  R_{N_1}(x,v,\cdot )\|_{L^{p'}(\bR^n)}\\
	&\lesssim  
\|f\|_{L^p(\bR^n)}
\chi_1(x,v) \|\kappa_{\sigma ,x}^\bX |w|_\bX^{N_1+1} \|_{L^{p'}(\bR^n)}.
\end{align*}
We can now conclude with the kernel estimates of Theorem~\ref{thm_kernelM}.
\end{proof}

\section{The local quantization}\label{sec:quantization}

This section is devoted to the local quantization of the classes of symbols $S^m(\widehat GM)$ defined in  Section~\ref{sec:symbols}. 
After its definition in Section~\ref{sec:local_quantization}, 
 we discuss
its $L^2$-boundedness for symbols of order 0 and  
 its dependence on local items  in Sections~\ref{sec:boundedness} and~\ref{sec:ind_cut-off} respectively.

\subsection{Local quantization}\label{sec:local_quantization}
 
\subsubsection{Definition}\label{sec:local_quant}
We consider a frame $\bX$ on an open set $U\subset M$ and  a $\bX$-cut-off~$\chi$
in the following sense (with the notation of Section \ref{subsec_expX}):

\begin{definition}
Let $(\mathbb X, U)$ be a local adapted frame. 
A \textit{cut-off for} $\exp^\bX$, or a {\it cut-off for the frame $(\bX,U)$}, is a smooth function $\chi\in C^\infty(U\times U)$ valued in $[0,1]$, properly supported in $\exp^\bX({U}_{\mathbb X})\subset U\times U$ for some $U_\bX$, and identically equal to $1$ on a neighbourhood of the diagonal of $U\times U$.
\end{definition}

Equivalently, 
a function $\chi\in C^\infty(U\times U)$ is a cut-off for $(\bX,U)$ when  
$$
\chi_0(x,v)= \chi(x,\exp_x v),
$$
defines a  function $\chi_0\in C^\infty(U\times \bR^n)$ that 
satisfies $\chi_0=1$ on a neighbourhood of $U\times \{0\}$
and that is  
$U$-locally compactly supported in the sense of Definition \ref{def:loccompsupp}.
Given a cut-off $\chi$ for $\exp^\bX$, 
	we will often use the shorthand:
	$$
	\chi_x (y):=\chi(x,y).
	$$
    
For a symbol $\sigma\in S^m(\widehat{G  }M)$ with associated convolution kernel density $\kappa\in \cS'(GM;|\Omega|(GM))$,
we define the operator $\Op^{\bX,  \chi}(\sigma)$  at any function $f\in C_c^\infty(U)$ via
\begin{equation}\label{eq_QX}
	\Op^{\bX,  \chi}(\sigma)f(x):=
\int_{w\in G_x M} \kappa_x (w) \ (\chi_x f)(\exp^\bX_x(- \Ln^\bX_x w )), \quad x\in U,
\end{equation}
with an integration in the density sense.
When we consider a smooth Haar system $\mu=\{\mu_x\}_{x\in U}$ on $U$, this formula becomes
$$
\Op^{\bX,  \chi}(\sigma)f(x)=
\int_{w\in G_x M} \kappa^\mu_x (w) \ (\chi_x f)(\exp^\bX_x(- \Ln^\bX_x w ))\, d\mu_x(w).
$$
In particular, for  the Haar system induced by $\bX$, 
we obtain 
$$
\Op^{\bX,  \chi}(\sigma)f(x)
=
\int_{v \in \bR^n}  \kappa_x^{\bX} ( v) \ (\chi_x f)(\exp^\bX_x - v) \ dv,
$$
in terms of  the convolution kernel $\kappa^\bX$  of $\sigma$ in the $\bX$-coordinates.

From now on, we assume that we have fixed a positive 1-density, that is, a smooth positive
measure on $M$ denoted by $dx$ whenever $x$ is the integrated variable on $M$; this means that $dx$
is a positive Borel measure such that the image of $dx$ in any local coordinates has a smooth
Radon-Nykodym derivative with respect to the Lebesgue measure.

\smallskip 

Performing the change of variable $y=\exp^\bX_x(-v)$,
or more accurately, considering the pullback of $dv$ via $\exp^\bX_x - v$, 
the quantization formula is also given by
\begin{equation}\label{eq_int_kernel}
	\Op^{\bX,  \chi}(\sigma)f(x)
=
\int_{M}  \kappa_x^{\bX} (- \ln^\bX_x y) \ (\chi_x f)(y)\ \jac_{y} (\ln^\bX_x) \ dy;
\end{equation}
here
$\jac_{y} (\ln^\bX_x) $
is the (smooth) Radon-Nykodym derivative $\frac{d \ln^\bX_x}{dy} $  of $(\exp^\bX_x - v)^*dv$ against $dy$.
The latter formula shows that   the integral kernel of $\Op^{\bX,  \chi}(\sigma)$ is formally given by the map 
\[
(x,y)\longmapsto \kappa_x^{\bX} (- \ln^\bX_x y) \ \chi_x (y)\ \jac_{y} (\ln^\bX_x).
\]

\subsubsection{Action on smooth compactly supported functions}
\begin{lemma}
\label{lem_actonDistrib_locQ_1}
Let $(\bX,U)$ be an adapted local frame  and $\chi$ a cut-off function  for $(\bX,U)$. Let $m\in\bR$ and 
    $\sigma\in S^{m}(\Gh M|_U)$, then 
  $\Op^{\bX, \chi} (\sigma) $  acts continuously on $C_c^\infty(U)$.
  Moreover, $\sigma \mapsto \Op^{\bX, \chi} (\sigma)$ is a continuous morphism of  topological vector spaces 
 $S^{m}(\Gh M|_U)\to \sL(C_c^\infty(U)) $. 
\end{lemma}

\begin{proof}
The quantization formula may be rephrased for $f\in C_c^\infty (U)$, $x\in U$, as
\begin{equation}
	\label{eq_pfOpcontCcinfty}
\Op^{\bX, \chi} (\sigma) f (x)
= (\kappa_x^\bX , g)_{\cD'(\bR^n),C_c^\infty(\bR^n)}, 
\ \mbox{with}  \  g(v) := (\chi_x f) (\exp_x^\bX -v).
\end{equation}
Note that $g=g_{x,f}\in C_c^\infty(\bR^n)$ and that  the maps
 $x\mapsto \kappa_x^\bX$ and $x\mapsto g_{x,f}$ are smooth $U\to \cD'(\bR^n)$ 
 and $U\to C_c^\infty(\bR^n)$ respectively.
 Hence, 
 the function $x\mapsto \Op^{\bX, \chi} (\sigma) f (x)$ is smooth in $x\in U$.

Let $\cC$ be a compact subset of $U$ with non-empty interior. 
Let $\theta \in C_c^\infty (U)$ with $\theta=1$ on a neighbourhood of~$\cC$ and $\supp\, \chi (x,\cdot)$, $x\in \cC$. 
Then for any $f\in C_c^\infty(U)$ supported in $\cC$, we have
  $$
\Op^{\bX, \chi} (\sigma) f 
=
\theta \,\Op^{\bX, \chi} (\sigma) \theta f
=\Op^{\bX, \chi} (\theta \sigma)  \theta f.
$$
Hence, $\Op^{\bX, \chi} (\sigma) f \in C_c^\infty (U)$.

It follows readily from \eqref{eq_pfOpcontCcinfty} and Corollary \ref{cor_OrderDistrib} that the map
$f\mapsto \Op^{\bX, \chi} (\sigma) f$ is continuous 
$C_c^\infty (U)\to C_c^\infty (U)$. Moreover, the estimates obtained from the application of Corollary \ref{cor_OrderDistrib} also imply that the map 
$\sigma\mapsto \Op^{\bX, \chi} (\sigma)$ is continuous 
$S^{m}(\Gh M|_U)\to \sL(C_c^\infty(U)) $. 
\end{proof}

We will extend the action of the operators $\Op^{\bX, \chi} (\sigma)$, $\sigma\in S^m(\Gh M)$, to distributions in Corollary  \ref{cor_actonDistrib_loc_Q}.

\subsubsection{Differential operators as pseudodifferential operators}
With the preceding definitions, we obtain the differential calculus as the quantization of symbols given by elements of $\Gamma (\sU (\fg M))$:

\begin{proposition}
\label{ex_OpT}
Let $(\bX,U)$ be an adapted frame. 
Let $T\in \Gamma (\sU (\fg M)) $.
With the notation of Lemma \ref{lem_LTx},
writing  $T = \sum_{[\alpha]\leq m} c_\alpha \ {\langle \bX\rangle}^\alpha$ in $\Gamma (\sU (\fg M|_U)) $, 
 we have for any $f\in C^\infty(U)$
 $$
 \Op^{\bX,\chi}(\widehat T )f(x)= L_{T_x} (f \circ \exp_x )(0) =\sum_{[\alpha]\leq m} c_\alpha(x)
 L_{\langle \bX\rangle_x}^\alpha (f \circ \exp_x )(0).
 $$
In particular, 
if $T=\id$, then $\Op^{\bX,\chi}(\id ) = \id$.
 \end{proposition}

 \begin{proof}[Proof of Proposition~\ref{ex_OpT}]
 The convolution kernel of $\widehat T$ in the $\bX$-coordinates is $L_{T_x} \delta_0$, so the quantization formula yields:
	\begin{align*}
	\Op^{\bX,\chi}(\widehat T )f(x)
	&=\int_{v\in \bR^n} L_{T_x}  \delta_0\ (\chi_x f) (\exp_x^\bX -v) dv	\\
	&=L_{T_x}^t (\chi_x f) (\exp_x^\bX -v)|_{v=0} = L_{T_x} (f \circ \exp_x )(0)
	\end{align*}
\end{proof}

It follows from Proposition \ref{ex_OpT}
that 
$$
\forall T\in \Gamma (\sU(\fg)), \;\;\forall  \alpha\in \bN_0,\qquad 
\Op^{\bX,\chi}(\widehat T ) \Op^{\bX,\chi}(\widehat{\langle \bX\rangle} ^\alpha  ) =
\Op^{\bX,\chi}(\widehat T \, \widehat{\langle \bX\rangle} ^\alpha) .
$$
However, as in the Euclidean setting, this will not be the case for the composition of two symbols in general.

\subsection{$L^2$ boundedness for the local quantization}
\label{sec:boundedness}
In this section, we show  the $L^2$-boundedness of 
the operators with symbols in $S^0(\widehat GM)$,  the action of the operators with symbols of orders $m\neq 0$ being discussed later in Section~\ref{sec:symb_cal}. Actually, we will consider the  local $L^2$ space defined on $M$ as 
\begin{equation}\label{def:L2loc}
L^2_{loc}(M)=\{ f\in \cD'(M) \ : \ \theta f \in L^2(M) \ \mbox{for any} \; \theta\in C_c^\infty (M)\}.
\end{equation}
We recall that $L^2_{loc}(M)$
is a topological vector space naturally equipped with a Fr\'echet structure given by the seminorms
$$
f\longmapsto \max_{j\leq N} \| \theta_j f\|_{L^2(M)},
\qquad N\in \bN,
$$
having fixed a sequence
$(\theta_j)_{j\in \bN}$ 
 in $C_c^\infty(M)$  valued in $[0,1]$ 
       such that 
       $M=\cup_{j \in \bN}\{\theta_j=1\} .$

\begin{theorem}
\label{prop_L2bddX} 
 Let $\bX$ be an adapted frame on $U$ and $\chi$ a cut-off for $\bX$. Let $\sigma\in S^0(\widehat GM)$.
 Then $\Op^{\bX,\chi}(\sigma)$ is continuous on $L^2_{loc}(U)$ and the linear map $\sigma\mapsto \Op^{\bX,\chi}(\sigma)$ is continuous $S^0(\Gh M)\to \sL(L^2_{loc}(U))$.
	\end{theorem}

 Slightly more precisely, we will show  that for any  $\chi_1\in C_c^\infty(U)$ and   $\sigma\in S^0(\widehat{G  }M)$, 
  the operator $\chi_1 \Op^{\bX, \chi}(\sigma)$ is bounded on $L^2(U)$, and furthermore that there exist $C>0$ and $N\in \bN$ (both independent of $\sigma$) such that 
$$
\forall \sigma\in S^0(\widehat{G  }M),\qquad
	\|\chi_1 \Op^{\bX, \chi}(\sigma)\|_{\sL(L^2(U))}\leq C  
	\|\sigma\|_{S^0(\widehat G M),(\bX,U), \supp \chi_1, N}\,.
	$$

The rest of this section is devoted to the proof of Theorem \ref{prop_L2bddX}.
The proof relies 
on the kernel estimates stated in Theorem~\ref{thm_kernelM} which show that the operators $\Op^{\bX, \chi}(\sigma)$ have singularities \`a la Calder\'on-Zygmund.
It  is inspired by  Stein's argument in \cite[ch.\!\!\! VI \S 2]{Ste93} (see also \cite[Section 5.4.4]{R+F_monograph}),

\begin{proof}[Proof of Theorem~\ref{prop_L2bddX}]
If $\sigma\in S^m(\Gh M)$ with $m<-Q$, then 
using the Cauchy-Schwartz inequality, we have:
\begin{align*}
\|\chi_1 \Op^{\bX, \chi}(\sigma)f\|_{L^2(U)}^2
 &\leq 
    \int_{x\in U}|\chi_1(x) |^2
    \int_{u_1 \in \bR^n}  |\kappa_{\sigma,x}^{\bX} ( u_1)|^2 du_1  \int_{u_2 \in \bR^n} |(\chi_x f)(\exp^\bX_x - u_2)|^2 \ du_2 \  dx\\
    &\lesssim  \|\sigma\|^2_{S^m(\widehat GM), (\bX, U),\cC_1, N'} \|f\|^2_{L^2(U)},
\end{align*}
by  the kernel estimates in Theorem~\ref{thm_kernelM} (1).
Hence, the linear map $\sigma\mapsto \Op^{\bX,\chi}(\sigma)$ is continuous $S^m(\Gh M)\to \sL(L^2_{loc}(U))$.
\smallskip

The proof will use several $U$-locally compactly supported functions, and we will assume that their supports are small enough so that the expression involving the geometric exponential makes sense.
For instance, we fix    a $U$-locally compactly supported function $\chi_2\in C^\infty(U\times \bR^n)$ valued in $[0,1]$ such that 
$\exp_x(-v)$ exists where $\chi_2(x,v)\neq 0$, and
$\chi_2(x,v)=1$ when $\chi(x,\exp_x (-v))\neq 0$. 

By Corollary \ref{cor_GM_Deltaq}, 
 $\Delta_{1-\chi_2}\sigma\in S^{-\infty}(\Gh M)$, so  $\Op^{\bX,\chi}(\Delta_{1-\chi_2}\sigma)$
is bounded on $L^2_{loc}(U)$ by the result for $m<-Q$ above. 
Hence, we now focus on $\sigma_0:=\Delta_{\chi_2}\sigma$.
We denote by $\kappa_{0,x}^\bX$ its convolution kernel in the $\bX$-coordinates.
Let $\chi_1\in C_c^\infty (U)$.
Up to decomposing $\chi_1$ into a finite sum of functions with smaller supports and modifying $\chi_2$, we may assume that 
$$
\chi_1(x) \chi_2(x,u)\chi_x(\exp_x^\bX -u)
= \chi_1(x) \chi_2(x,u).
$$

We then have
with $\cC_1:= \supp \, \chi_1$,
\begin{align*}
 I_0&:= \|\chi_1 \Op^{\bX, \chi}(\sigma_0)f\|_{L^2(U)}^2 
 \\
&  \leq \int_{x\in U} |\chi_1(x)|^2  \sup_{x_1\in \cC_1}
  \left|\int_{u \in \bR^n}  \kappa_{0,x_1}^{\bX} ( u) \ f(\exp^\bX_x - u) \ du\right|^2 dx\\
    &\lesssim \sum_{|\alpha|\leq 1+n/2} \int_{x\in U} |\chi_1(x)|^2  \int_{x_1\in \cC_1}
  \left|\int_{u \in \bR^n}  \bX_{x_1}^\alpha \kappa_{0,x_1}^{\bX} ( u) \  f(\exp^\bX_x - u) \ du\right|^2 dx_1 dx, 
\end{align*}
having used  the Sobolev inequality on a compact manifold of dimension $n$: 
\[
\| g\|_{L^\infty(M)}^2\leq C \sum_{|\alpha|\leq 1+n/2} \| \mathbb X^\alpha g\|_{L^2}^2.
\]
Swapping the integral over $x$ and $x_1$, we obtain 
$$
I_0
 \lesssim\sum_{|\alpha|\lesssim 1+n/2} \sup_{x_1\in \cC_1} I_{\alpha}(x_1),
$$
where
\begin{align*}
I_{\alpha}(x_1)
&:=
\int_{x\in U} |\chi_1(x)|^2  
  \left|\int_{u \in \bR^n}  \bX_{x_1}^\alpha \kappa_{0,x_1}^{\bX} ( u) \  f(\exp^\bX_x - u) \ du\right|^2 dx.
  \end{align*}
  
  The integrals $I_{\alpha}(x_1)$  are compactly supported in $x$ and $u$ and  
 we now transform them with changes of variables. The first one aims at identifying all the points of the integral in terms of the base point $x_1$ and the map $\exp_{x_1}^\bX$. In the variable $x$, we set $x=\exp^\bX_{x_1} v$, which gives
  \begin{align*}
  I_{\alpha}(x_1)
  &=\int_{v\in \bR^n} |\chi_1(\exp^\bX_{x_1} v)|^2  
  \left|\int_{u \in \bR^n}  \bX_{x_1}^\alpha \kappa_{0,x_1}^{\bX} ( u ) \  f( \exp^\bX_{\exp^\bX_{x_1} v} -u )  du\right|^2 \jac_v \exp_{x_1}^\bX dv,
\end{align*}
and in the variable $u$, we  make the substitution
$$
u':=\varphi_v(u):= \ln^\bX_{x_1}\exp^\bX_{\exp^\bX_{x_1} v} -u,
$$
and obtain:
$$
I_{\alpha}(x_1)
=\int_{\bR^n} |\chi_1(\exp^\bX_{x_1} v)|^2  
  \left|\int_{ \bR^n}  \bX_{x_1}^\alpha \kappa_{0,x_1}^{\bX} (\varphi_v^{-1} (u') ) \  f( \exp^\bX_{x_1} u' ) \ \jac_{u'} \varphi_v^{-1}  du'\right|^2 \jac_v \exp^\bX _{x_1} dv.
$$
We now introduce a convolution structure via the  change of variables $u'\mapsto w := (-u')*_{x_1}^\bX v$. We
write 
$$
\varphi_v^{-1} (u') = -\ln^\bX_{\exp^\bX_{x_1} v} \exp^\bX_{x_1} u' = -\ln^\bX_{\exp^\bX_{x_1} v} \exp^\bX_{x_1} v*_{x_1}^\bX(-w) =  w *_{x_1}^\bX r_{x_1}(v,w),
$$
with $r_{x_1}(v,w)=r_1(x_1,v,w)$ as in  Corollary \ref{cor_lem_Cq1BCH} (1). 
At the end, we are left with
\begin{align*}
I_{\alpha}(x_1)
&=\int_{\bR^n} 
  \left|\int_{ \bR^n}  \bX_{x_1}^\alpha \kappa_{0,x_1}^{\bX} (w *_{x_1}^\bX r_{x_1}(v,w) ) \  \tilde f(v*_x^\bX (-w) ) J(x_1;v,w) dw\right|^2 \jac_v \exp^\bX_{x_1} dv,
\end{align*}
where
\begin{align*}
  \widetilde f(v)&:=\widetilde f_{x_1}(v):= f( \exp^\bX_{x_1} v) \chi_3(x_1,v),\\
J(x_1;v,w)&:=\chi_4 (x;v,w )\,
\jac_{v*_{x_1}(-w)} \varphi_v^{-1} \, \jac_v \exp^\bX_{x_1},  
\end{align*}
and $\chi_3\in C^\infty(U\times \bR^n)$, $\chi_4\in C^\infty(U\times \bR^n\times \bR^n)$
are $U$-locally compactly supported functions satisfying  
$\chi_3=1$ on a neighbourhood of $\supp\, \chi_2$
and 
$\chi_4 =1$ on a neighbourhood of the support of $ (x_1,v,w)\mapsto \chi_2 (x_1,w *_{x_1}^\bX r_{x_1}(v,w) ) 
\chi_3(x_1,v*_x^\bX (-w) ).$

Using the notation of Proposition \ref{prop:kernel_higher_order}, 
we observe
$$
\bX_{x_1}^\alpha \kappa_{0,x_1}^{\bX} (w *_{x_1} r_{x_1}(v,w) ) = \kappa_{(D_\bX^\alpha \sigma)^{(\tilde r_v,\chi_2)},x_1}^\bX
$$
with $\tilde r_v(x,w) =r_x(v,w)$ at $v$ fixed.
Therefore, using the analysis in the proof of Proposition~\ref{prop:kernel_higher_order}, 
more precisely the inequality \eqref{eqpfprop:kernel_higher_orderEST}, 
we obtain $v$-uniform estimates leading to:
$$
I_{\alpha}(x_1)\leq 
\sum_{[\alpha'] \leq N}I_{\alpha,\alpha'}(x_1)
\quad +\quad C_{\sigma} R
$$
where 
\begin{align*}
I_{\alpha,\alpha'}
&:=\int_{\bR^n} 
  \left|\int_{ \bR^n} 
  \bX_{x_1}^\alpha L_{\langle\bX\rangle_x}^{\alpha'} 
  \kappa_{0,x_1}^{\bX} (w  ) \  \tilde f(v*_x^\bX (-w) ) J_{\alpha'} (x;v,w) dw\right|^2 \jac_v \exp_{x_1} dv,
\\
R&:=
   \left|\int_{ \bR^n}  |\tilde f(v*_x^\bX (-w) ) J(x;v,w)| dw\right|^2 \jac_v \exp_{x_1} dv 
\end{align*}
for $N$ with $\upsilon_n N-Q\geq 0$, with a constant $C_\sigma$ given by kernel estimates uniformly in $v$ and locally uniformly in $x_1\in U$, 
and with
$$
J_{\alpha'} (x;v,w):=q_{\alpha'}(x,r(x;v,w))J (x;v,w),
$$
having used Notation~\ref{notation:q_alpha}.
We readily check 
$$
R\lesssim \|\tilde f\|_{L^2(\bR^n)}
\lesssim \|f\|_{L^2(U)},
$$
and we have 
$$
	I_{\alpha,\alpha'}
\lesssim
\| \kappa_{D_\bX ^\alpha \widehat {\langle\bX\rangle_x}^{\alpha'}\sigma, \tilde f, J_{\alpha'}x_1}\|_{L^2(\bR^n)},
$$
with the notation of Lemma~\ref{lem:symbol_perturbconvolution}.
Applying  this lemma, we obtain:
\begin{align*}
	\| \kappa_{D_\bX ^\alpha \widehat {\langle\bX\rangle_x}^{\alpha'}\sigma, \tilde f, J_{\alpha'}x_1}\|_{L^2(\bR^n)}
	\lesssim
	\sum_{\alpha''}
	\| \tilde f  \star_{G_{x_1} M}^\bX (\, \cdot\,  ^{\alpha''} \bX ^\alpha_{x_1} L_{\langle\bX\rangle_x}^{\alpha'} \kappa_{\sigma,x_1}^\bX ) \|_{L^2(\bR^n)}
	+C'_\sigma \|\tilde f \|_{L^2(\bR^n)},
\end{align*}
with $C'_\sigma$ a constant depending on some seminorm in $\sigma\in S^0(\Gh M)$;
above the sum is over $\alpha''$ such that 
$[\alpha']< [\alpha'']\leq Q/2$ when  $\alpha'\neq0$ and  $[\alpha'']\leq Q/2$ when $\alpha'=0.$
For the terms in this sum over $\alpha''$, 
we notice (see Section~\ref{subsec:distrib} for the Fourier transform of distributions)
$$
\|\tilde f \star_{G_{x_1} M}^\bX (\, \cdot\, ^{\alpha''}   \bX ^\alpha_{x_1}L_{\langle\bX\rangle_x}^{\alpha'}\kappa_{\sigma,x_1}^\bX )  \|_{L^2(\bR^n)}
\leq
\|\Delta^{\alpha''}_\bX 
D_\bX ^\alpha \widehat {\langle\bX\rangle_x}^{\alpha'}\sigma(x_1,\cdot) \|_{L^\infty (\widehat G_{x_1}M)} \|\tilde f\|_{L^2(\bR^n)}.
$$
Gathering the estimates above yields with $N_1\in \bN_0$ large enough
$$
 I_\alpha (x_1)\lesssim \| \sigma \|_{S^0 (\widehat GM),\bX,\mathcal C_1,N_1}^2 \|f\|_{L^2(M)}^2.
 $$
This concludes the proof.
\end{proof}

\subsection{Independence of the local quantization in the cut-off and the frame }\label{sec:ind_cut-off}
  In this section, we use the kernel estimates to analyse the consequences of a change of cut-off or of frames in the local quantization process.

\subsubsection{Independence  in the cut-off}
 The local quantization  is independent in the cut-off 
 modulo smoothing symbols in the following sense:

\begin{proposition}
\label{prop_Qindepchi}
  Let $\bX$ be an  adapted frame  on an open subset $U\subset M$.
  Let $\chi_1$ and $\chi_2$ be two cut-offs for $\bX$.   
  Then for any $\sigma\in S^{m}(\widehat GM)$, $m\in \bR\cup\{-\infty\}$,
  there exists $\rho\in S^{-\infty} (\widehat GM)$ such that 
  $$
  \Op^{\bX, \chi_1}(\sigma)
  =
  \Op^{\bX, \chi_2}(\sigma)
  +\Op^{\bX, \chi_2}(\rho).
  $$
Moreover, we can construct $\rho$ such that the map $\sigma \mapsto \rho$ is continuous $S^m(\widehat G M) \to S^{-\infty}(\widehat G M)$.
\end{proposition}

This proposition is a consequence of the following more technical lemma: 
\begin{lemma}
\label{lem_prop_Qindepchi}
  Let $\bX$ be an  adapted frame on an open subset $U\subset M$. Let $m\in \bR$.
\begin{enumerate}
    \item Let $\chi_1$
 and $\chi_2$ be $\bX$-cut-offs such that $ \chi_1$ is supported in the interior of $\{\chi_2 =1\}$.
 For any $\sigma \in S^m(\widehat G M)$,
defining the symbol  $\tau\in S^m(\widehat G M|_U)$
via its convolution kernel in the $\bX$-coordinates:
$$
\kappa_{\tau,x}^\bX (u):=
\kappa_{\sigma,x}^\bX (u)  \chi_{1,x}(\exp_x^\bX -u),
$$
we have 
$$
\Op^{\bX,\chi_1} \sigma = \Op^{\bX,\chi_2} \tau.
$$
Moreover, the map $\sigma\mapsto \tau$ is continuous $S^m(\widehat GM) \to S^m(\widehat G M|_U)$.

\item Let $\underline \chi, \chi, \overline \chi$ be three $\bX$-cut-offs  
such that
$\underline\chi$ is supported in the interior of $\{\chi=1\}$ and 
$\chi$ is supported in the interior of $\{\overline \chi =1\}$.
 For any $\sigma \in S^m(\widehat G M)$,
defining the symbol  $\rho\in S^{-\infty}(\widehat G M|_U)$
via its convolution kernel in the $\bX$-coordinates:
$$
\kappa_{\rho,x}^\bX (u):=
\kappa_{\sigma,x}^\bX (u) ( \chi_x - \underline \chi_x )
(\exp_x^\bX -u),
$$
we have 
$$
\Op^{\bX,\chi} \sigma = \Op^{\bX,\underline \chi} \sigma
+ \Op^{\bX,\overline \chi} \rho.
$$
Moreover, the map $\sigma\mapsto \rho$ is continuous $S^m(\widehat GM) \to S^{-\infty}(\widehat G M|_U)$.
 \end{enumerate}
\end{lemma}

The proof of Lemma \ref{lem_prop_Qindepchi} follows from 
the kernel estimates in Corollary \ref{cor2_kernelM} and Proposition~\ref{prop_smoothingM} (1), keeping in mind that the map $x\mapsto \exp_x^\bX u$ defined in~\eqref{eq_mapxexpxv} is smooth.  
\smallskip

Then, the proof of Proposition \ref{prop_Qindepchi} follows from Lemma \ref{lem_prop_Qindepchi} together with constructing two $\bX$-cut-offs $\underline \chi,\overline \chi$ that, having been given the two  $\bX$-cut-offs $\chi_1$ and $\chi_2$  in the statement, \textcolor{blue}{are such that  
$\underline \chi$ is supported in the interior of   $\{\chi_1=1\}\cup\{\chi_2=1\}$ and
$\supp \chi_1 \cup \supp \chi_2 $ is included in the interior of $\{\overline \chi =1\}$,}

\subsubsection{Independence in the frame}
The independence of the local quantization in the frame only holds up to lower order terms. It is also the case when working on Riemanian manifolds with the standard pseudodifferential calculus via change of charts (see e.g.~\cite{Zwobook}).

\begin{proposition}[Independence in the adapted frames]
\label{prop_indepframe}
Let $\bX$ and $\bY$ be two  adapted frames on the same open subset $U\subset M$. 
Let  $\chi_1$ and $\chi_2,\chi$ be  cut-off functions for $\exp^{\bX}$ and 
$\exp^{\bY}$ respectively. 
 Then for any $\sigma\in S^{m}(\widehat GM)$, $m\in \bR\cup\{-\infty\}$,
  there exist a symbol $\rho \in S^{m-1} (\Gh M|_U)$, 
  such that
   \[
  \Op^{\bY, \chi_2}(\sigma)=
  \Op^{\bX, \chi_1}(\sigma)
  +\Op^{\bX, \chi}(\rho).
  \]
Moreover, we can construct $\rho$ such that 
  the map 
$$
\sigma\longmapsto \rho,\quad S^m(\Gh M|_U)\to S^{m-1}(\Gh M|_U),
$$
is continuous. 
Furthermore,
  we can also construct 
  a sequence of symbols $\rho_j\in S^{m-1-j}(\Gh M|_U)$, $j\in \bN_0$, such that $\rho$ admits the  asymptotic expansion $\rho\sim \sum_{j\in \bN_0} \rho_j$ and   the following maps are continuous  
$$
 \sigma \longmapsto \rho - \sum_{j=0}^{N}\rho_j, 
 \qquad
  S^m(\Gh M)\longrightarrow S^{m-1-(N+1)} (\Gh M|_U), 
  \quad N\in \bN_0.
  $$
\end{proposition}

\begin{proof}
It suffices to consider the case of a finite order $m\in \bR$. 
By Proposition~\ref{prop_Qindepchi}, we may assume $\chi=\chi_1=\chi_2$.
We adopt the notation of Lemma \ref{lem_uYX}: with the
 two  frames  $\bX$ and $\bY$ on the open subset $U\subset M$ is associated an open neighbourhood
 $U_{\bX,\bY}\subset U_{\bX}\cap U_{\bY}\subset U\times \bR^n$ 
such that for any $(x,v)\in U_{\bX,\bY}$, the point $(x,tv)$   lie in $U_{\bX,\bY}$ for $t$ in some neighbourhood of $[0,1]$, and we
consider 
 the function $u := u^{\bY,\bX}$ defined via
$$
(x,u(x,v)) = \ln^\bY (\exp^\bX (x,v)), 
\qquad 
(x,v)\in U_{\bX,\bY}.
$$
We fix a function $\bar \chi\in C^\infty(U\times \bR^n)$ $U$-locally compactly supported such that $\bar \chi(x,-v) = 1$ 
on a zero section of $U\times \bR^n$ containing the support of $(x,v)\mapsto \chi (x,\Exp_x^\bX v)$.

We set 
$$
\kappa^{\bX}_{\rho,x} (v):=
 \bar \chi^2(x,v)\left(\kappa^{\bY}_x\left (- u_x (-  v) \right) 
\jac_{- v}(u_x^{\bY, \bX})-\kappa_x^{\bX}(v)\right)
$$ 
We check readily that $\kappa^{\bX}_{\rho,x} (v)$ is well defined. 
Moreover, 
the formula between operators given in the statement holds formally after the  change of variable given by $u_x$ since 
$\kappa^{\bX}$ and $\kappa^{\bY}$ correspond to the same convolution kernel $\kappa$ viewed in the coordinates for $\bX$ and $\bY$ respectively. It remains to check that the symbol $\rho$ associated with $\kappa^{\bX}_{\rho,x} (v)$ is  well defined in $S^{m-1}(\widehat G M)$ and depends continuously on $\sigma$.

By Lemma \ref{lem_XYExp}, we have
$$
\kappa_x^{\bY} (v)=\kappa_x^\bX\left({\rm{diag}}(T_x)^{-1}v\right)\big|\det T_x\big|^{-1},
$$
and therefore,
\begin{align*}
\kappa^{\bX}_{\rho,x} (v)
&=
 \bar \chi^2(x, v) \left(\kappa^{\bX}_x\left (
- {\rm{diag}}(T_x)^{-1} u_x(-v)\right) 
\jac_{- v}(u_x )|\det T_x|^{-1} -  \kappa_x^{\bX}(v)\right).
\end{align*}
We set
\begin{align*}
  \kappa_{\tau,x}^\bX(v)
  &:=\bar \chi(x, v) \kappa^{\bX}_x\left (
- {\rm{diag}}(T_x)^{-1} u_x(-v)\right) 
\\& = \bar \chi(x, v) \kappa^{\bX}_x(v*_x^\bX r(x,v)),  
\end{align*}
by  Proposition~\ref{prop_T_Taylor}, and using its notation.
By  Proposition~\ref{prop:kernel_higher_order}, 
this is the convolution kernel in the $\bX$-coordinates of a symbol $\tau\in S^m(\Gh M|_U)$.
We can now write
$$
\rho = \Delta_q \tau - \Delta_{\bar \chi^2} \sigma, 
$$
where $q\in C^\infty(U\times \bR^n)$ is the function $U$-locally compactly supported given by:
$$
q(x,  v):=  \bar \chi(x,v)\jac_{- v}(u_x) |\det T_x|^{-1}.
$$
By Proposition~\ref{prop:kernel_higher_order} and Corollary \ref{cor_GM_Deltaq}, 
since $q(x,0)=1$ and $\bar\chi^2(x,0)=1$, 
both $\Delta_q \tau$ and $\Delta_{\bar \chi^2} \sigma$
admit asymptotic expansions whose first terms are $\sigma$.
The  properties of continuity follow from the ones in Proposition~\ref{prop:kernel_higher_order} and Corollary \ref{cor_GM_Deltaq}. 
\end{proof}

\section{Symbolic calculus 
for the local quantization}\label{sec:symb_cal}

In this section, we prove that the operators obtained by the local quantization on $M$  enjoy a symbolic calculus. The proof relies on the notion of amplitudes that notably simplify the arguments and is presented in Section~\ref{sec:amplitudeDef+ex} while
Section \ref{sec:amplitudeQ+asymp} is devoted to their quantization and main asymptotic properties. Then, we analyse the adjoints of pseudodifferential operators in Section~\ref{sec:adjoint}, and their compositions in Section~\ref{sec:composition}.

\subsection{Amplitudes}\label{sec:amplitudeDef+ex}

In this section, we define the notion of amplitudes and give important examples.

\subsubsection{Definition}
\begin{definition}
    An {\it amplitude} $\tau$ of order $m$ is a smooth map  $$
    y\longmapsto \tau_y,\qquad M\longrightarrow S^m(\Gh M).$$
    We denote by $\cT^m (\Gh M)$ the Fr\'echet space of amplitudes of order $m$. 
\end{definition}

The topology on  $\cT^m(\Gh M)$ is given by the family of seminorms  
\begin{equation}\label{def:semi_norms_amplitudes}
\|\tau\|_{\cT^m(\widehat G M),(\bX,U), \cC, N}
:=
\max_{|\beta|\leq N}\sup_{y\in \cC}
\|\bX^\beta_y\tau_y  \|_{S^m(\Gh M),(\bX,U), \cC, N}, 
\end{equation}
where $N\in \bN_0$, 
   $\bX$ is an adapted frame on an open set $U\subset M$ and $\cC$ is a compact subset of $U$; 
   the seminorms on $S^m (\Gh M)$ are defined in \eqref{def:semi_norm}. 
We will often  view an amplitude $\tau \in \cT^m (\Gh M)$ as being given by a family of symbols 
$$
\tau(x,y)=\{\tau (x,y,\pi), \pi \in \Gh_x M \} \in S^m(\Gh_x M)
$$
smoothly parametrised by $x,y\in M$.

If $\mu$ is a smooth Haar system over $M$, 
the corresponding convolution kernels are the distributions 
$$
\cS'(G_x M) \ni \kappa^{\mu}_{\tau,x,y} := \cF_{G_x M,\mu_x}^{-1} \tau(x,y,\cdot),\qquad x,y\in M.
$$ 
As in Section \ref{subsec_symbol+k}, 
if $\mu'$ is another  smooth Haar system, then 
 $\mu' = c\mu$ for a positive function $c\in C^\infty(M)$ and we have
 $$
 \kappa^{\mu}_{\tau,x,y} = c(x)\kappa^{\mu'}_{\tau,x,y}, 
 \qquad x,y\in M.   
 $$ 
This allows us 
to define the convolution kernel of $\tau$ independently of any Haar system as  distributional vertical densities smoothly parametrised by $y\in M$:
  $$
\kappa_\tau : y\longmapsto \kappa_{\tau,y} \quad \mbox{in}\quad  C^\infty\left(M, \cS'(GM; |\Omega|(GM))\right).
$$
If  $\bX$ is an adapted frame on an open set $U\subset M$, the distributions 
$$
\kappa_{\tau,x,y}^\bX \in \cS'(\bR^n), \qquad
\kappa_{\tau,x,y}^\bX (v):=  \kappa_{\tau,x,y}^{\mu^\bX} (\Exp^\bX_x v), 
\qquad x,y\in M,
$$
define a map
$$
\kappa^\bX_\tau  \in C^\infty(U\times M, \cS'(\bR^n)), 
$$
 called the \textit{convolution kernel} of $\tau$ in the $\bX$-coordinates.
 At least formally, we have 
 $$
 \tau(x,y,\pi) = \cF_{G_x M,\mu^\bX} \,\kappa^\bX_{\tau,x,y}(\pi), 
 \qquad x,y\in U,  \ \pi\in \Gh_x M.
 $$

The following properties are easily checked using the relations of symbols and amplitudes with their convolution kernels:
\begin{lemma}
\label{lem_relAsigma}
\begin{enumerate}
    \item  If $\tau\in \cT^m (\Gh M)$ then the symbol $\sigma=\sigma_\tau$ given by 
    $$
    \sigma(x,\pi)= \tau(x,x,\pi), 
    \quad x\in M, \ \pi\in \Gh_x M,
    $$
    is in $S^m (\Gh M)$
    with  convolution kernel in the $\bX$-coordinates  given by 
$\kappa_{\sigma,x}= \kappa_{\tau,x,x}$. Moreover, the map 
$\tau\mapsto \sigma_\tau$ is continuous $\cT^m (\Gh M)\to S^m (\Gh M)$.
\item Let $\bX$ be an adapted frame on an open set  $U\subset M$.
If $\tau\in \cT^m (\Gh M)$ and $\alpha,\beta\in \bN_0^n$, then the amplitude $\tau^{(\alpha,\beta)}$ defined via
$$
\tau^{(\alpha,\beta)}(x,y,\pi)=
\Delta_\alpha^\bX D_{\bX, y}^\beta \tau (x,y,\pi) 
\qquad x,y\in U, \, \pi\in G_x M, 
$$
is in $\cT^{m-[\alpha]}(\Gh M|_U)$
with  convolution kernel in the $\bX$-coordinates formally given by 
$$
\kappa_{\tau^{(\alpha,\beta)},x,y}(v)=v^\alpha \bX^\beta_y
\kappa_{\tau,x,y}(v), \qquad x,y\in U, \, v\in \bR^n.
$$
Moreover, the map $\tau\mapsto \tau^{(\alpha,\beta)}$ is continuous $\cT^m (\Gh M) \to \cT^{m-[\alpha]} (\Gh M|_U)$.
\end{enumerate}
\end{lemma}

\subsubsection{Some examples of amplitudes}
\label{subsec_ex_amplitude:convolution}

\begin{lemma}
	\label{lem_ex_amplitude:convolution}
	Let $\sigma_1\in S^{m_1}(\widehat GM)$ and $\sigma_2\in S^{m_2}(\widehat GM)$.
	Then the amplitude $\tau_{\sigma_1,\sigma_2}$ defined via 
	$$
	\tau_{\sigma_1,\sigma_2} (x,y) = \sigma_1(x)\sigma_2(y)
	$$
	makes sense and is in $\cT^{m_1+m_2}(\Gh M)$.
	Moreover, the map
	$$
	(\sigma_1,\sigma_2)\longmapsto\tau_{\sigma_1,\sigma_2},\qquad 
	 S^{m_1}(\widehat GM)\times S^{m_2}(\widehat GM)\longrightarrow \cT^{m_1+m_2}(\Gh M)
	 $$ 
	 is continuous.
\end{lemma}

Continuing with the setting of Lemma \ref{lem_ex_amplitude:convolution}, let  $\kappa_1$, $\kappa_2$ and $\kappa_\tau$  denote the  convolution kernels densities of $\sigma_1,\sigma_2,\tau$ respectively.
For any smooth Haar system $\mu=\{\mu_x\}_{x\in M}$, 
we have for any $x,y\in M$,
$$
\kappa_{\tau,x,y}^\mu
=
\kappa_{2,y}^\mu\star_{G_x M,\mu_x} \kappa^\mu_{1,x},
$$
that is, with $v\in G_x M$
$$
\kappa_{\tau,x,y}^\mu(v)
=
\int_{G_x M}
\kappa_{2,x}^\mu(w)  \kappa^\mu_{1,x}(w^{-1}v) d\mu_x(w)
=
\int_{G_x M}
\kappa_{2,x}^\mu(vw^{-1})  \kappa^\mu_{1,x}(w) d\mu_x(w).
$$
The meaning of this convolution is similar to the one for the product of symbols, see the proof  of Proposition \ref{prop_comp+adj}:
 this is a convolution between distributions which are sums of a Schwartz function and a compactly supported distribution. Alternatively, this can be defined first for Schwartz kernels $\kappa_{1,x}^\mu$, $\kappa_{2,x}^\mu$, i.e. for smoothing symbols $\sigma_1,\sigma_2$, and 
then extended to any symbols in some symbol classes $S^{m_j}(\Gh M)$, $j=1,2$,  by density (see Remark 
\ref{rem_altpf_prop_comp+adj}).
\smallskip

We now generalise Lemma \ref{lem_ex_amplitude:convolution}
with an amplitude described by the perturbation of the convolution
of convolution kernels. Without details, we straightforwardly extend the notion of asymptotic expansion stated for  symbols in Definition~\ref{def:asymptotics} to the case of amplitudes.  

\begin{proposition}
\label{prop_ex_amplitude:convolution}
Let $\sigma_1\in S^{m_1}(\widehat GM)$ and $\sigma_2\in S^{m_2}(\widehat GM)$ with convolution kernels $\kappa_1$ and $\kappa_2$ respectively. Let $(\bX,U)$ a local frame and $(x,y;v,w)\mapsto J(x,y;v,w)$ a $U\times U$-locally compactly supported function
in $C^\infty(U\times U\times \bR^n\times \bR^n)$. Then, the convolution kernel 
\[
\kappa^\bX_{\tau,x,y} (v)
=\int_{\bR^n} \kappa_{1,x}^\bX(w) \kappa_{2,y}^\bX(v*^\bX_x (-w) )J(x,y;v,w) dw ,\quad\; x,y\in U,\;\; v\in\bR^n,
\]
defines an amplitude 
\begin{equation}\label{eq:tau_amplitude_conv}
\tau = \tau(\sigma_1,\sigma_2,J)\in \cT^{m_1+m_2}(\Gh M)
\end{equation}
admitting the asymptotics
$$
\tau \sim \sum_{j\in \bN_0} \tau_j ,
\qquad\mbox{with}\  \tau_j= \tau_j(\sigma_1,\sigma_2,J)\in \cT^{m_1+m_2-j}(\Gh M), \ j\in \bN_0.
$$ 
Moreover, we may choose $\tau_0:=J(x,y;0,0)\sigma_1(x)\sigma_2(y)$, and more generally
$$
\tau_j(x,y) := \sum_{[\alpha_1]+[\alpha_2]=j} c_{\alpha_1,\alpha_2} (x,y)\Delta_\bX^{\alpha_1}\sigma_1(x) \Delta_\bX^{\alpha_2}\sigma_2(y).
$$	
Above, the coefficients $c_{\alpha_1,\alpha_2} (x,y)$
come from the Taylor expansion at $(v,w)\sim 0$ of
$$
\tilde J(x,y; v,w):= J(x,y; v*^\bX_x w,w), 
\qquad
\tilde J(x,y;v,w) \sim \sum_{\alpha_1,\alpha_2} c_{\alpha_1,\alpha_2}(x,y) {v}^{\alpha_2} w^{\alpha_1}.
$$
With these choices, the following maps are continuous:
\begin{align*}
	(\sigma_1,\sigma_2)&\longmapsto  \tau(\sigma_1,\sigma_2,J), 
\qquad S^{m_1}(\widehat GM)\times S^{m_2}(\widehat GM) \longrightarrow \cT^{m_1+m_2}(\Gh M), \\
	(\sigma_1,\sigma_2)&\longmapsto  \tau_j(\sigma_1,\sigma_2,J), 
\qquad S^{m_1}(\widehat GM)\times S^{m_2}(\widehat GM) \longrightarrow \cT^{m_1+m_2-j}(\Gh M), \\
(\sigma_1,\sigma_2)&\longmapsto \tau - \sum_{j=0}^N\tau_j, 
\qquad S^{m_1}(\widehat GM)\times S^{m_2}(\widehat GM) \longrightarrow \cT^{m_1+m_2-(N+1)}(\Gh M),
\end{align*}
\end{proposition}

\begin{proof}[Strategy of the proof of Proposition \ref{prop_ex_amplitude:convolution}]
By Lemma~\ref{lem_ex_amplitude:convolution}
and
the properties of
$\Delta_\alpha^\bX$ as stated in Lemma~\ref{lem:prop_diff},
the amplitude 
$\tau_j=\tau_j(\sigma_1,\sigma_2)$ defined in the statement is in $S^{m_1+m_2-j}(\Gh M)$. Moreover, the map 
$(\sigma_1,\sigma_2)\mapsto \tau_j(\sigma_1,\sigma_2)$ is continuous from $S^{m_1}(\widehat GM)\times S^{m_2}(\widehat GM)$ to  the set $ \cT^{m_1+m_2-j}(\Gh M)$.
 
 We may assume that $\sigma_1$ and $\sigma_2$ are smoothing, so that $\kappa^\bX_{\tau,x,y} (v)$ is a Schwartz function in~$v$ depending smoothly in $x,y$. 
 Then, we will extend the result to any symbols with the density property of smoothing symbols  (see Proposition~\ref{prop_smoothingM})
 as well as 
by proving the continuities of the maps $(\sigma_1,\sigma_2)\mapsto \tau$ and $(\sigma_1,\sigma_2)\mapsto \tau - \sum_{j=0}^N\tau_j$.

The strategy is to prove kernel estimates using Lemma \ref{lem:symbol_perturbconvolution} together with the  properties of $\tau(\sigma_1,\sigma_2,J)$ obtained in Lemma \ref{lem_algprop_ex_amplitude:convolution} below. We postpone the development of this proof  after stating and proving  Lemma \ref{lem_algprop_ex_amplitude:convolution}.
\end{proof}

We first prove some algebraic properties for the amplitude $\tau$ defined in Proposition \ref{prop_ex_amplitude:convolution}:
\begin{lemma}
\label{lem_algprop_ex_amplitude:convolution}
	We continue with the setting and notation of Proposition \ref{prop_ex_amplitude:convolution}
 and consider an amplitude $\tau$ given by~\eqref{eq:tau_amplitude_conv} with $\sigma_1,\sigma_2\in S^{-\infty}(\Gh M)$.
	\begin{enumerate}
		\item For any $\alpha\in \bN_0^n$,
$$
	\Delta^\alpha_\bX \tau = \sum_{[\alpha_1]+[\alpha_2]=[\alpha]} 
c^{(\alpha)}_{\alpha_1,\alpha_2} \
\tau (\Delta_\bX^{\alpha_1} \sigma_1,
\Delta_\bX^{\alpha_2}\sigma_2, J)$$
	where the functions $c^{(\alpha)}_{\alpha_1,\alpha_2}$ are the same as in 
 Corollary \ref{cor_GM_Leib}.
   \item For any $j=1,\ldots,n$,
$$
 	X_{j,y} 
 \tau 
 = 
\tau( \sigma_1 ,D_{X_j} \sigma_2,J)
+\tau( \sigma_1 ,\sigma_2,X_{j,y}J)
.$$
 \item For any $j=1,\ldots,n$,
 $$
 	\widehat{\langle X_j\rangle}^\beta 
 \tau 
 = 
\tau(\widehat {\langle X_j\rangle}^{\beta} 
 \sigma_1 ,\sigma_2,J).$$

  \item For any $j=1,\ldots,n$,
 $$
D_{X_{j,x}} \tau
 = 
\tau (D_{X_{j}} \sigma_1,
\sigma_2, J)+
\tau ( \sigma_1,
\sigma_2, X_{j,x}J)
+ \sum_{\ell=1}^n \tau ( \sigma_1,
\widehat{\langle X_j\rangle} \sigma_2, g_\ell J),
$$
where $g_\ell (x;v,w)$ is a polynomial in $v,w$ with smooth coefficients in $x\in U$
  vanishing in $(v,w)$  at order $\max (\upsilon_\ell -\upsilon_j,-1)$.
\item Let $\cR$ be a symmetric positive Rockland operator on $GM$ of homogenenous degree $\nu$ (see Definition~\ref{def:Rockland}). 
 $$
 \tau = \sum_{[\alpha_1]+[\alpha_2]= \nu N}
\tau(
 \sigma_1 \widehat {\langle \bX\rangle}^{\alpha_2} ,(\id+\pi (\cR_x))^{-N}  \sigma_2,
 c_{\cR,\alpha_1,\alpha_2}  \tilde \bX^{\alpha_2}_x J),
 $$
 where the functions $c_{\cR,\alpha_1,\alpha_2}\in C^\infty(U)$
 come from 
 $$
 (\id+\tilde \cR_x)^N (g_1 \times g_2)= \sum_{[\alpha_1]+[\alpha_2]=\nu N}
 c_{\cR,\alpha_1,\alpha_2} (x) \tilde \bX^{\alpha_1}_x g_1 \times \tilde \bX^{\alpha_2}_x g_2,
 \qquad g_1,g_2\in C^\infty(\bR^n).$$
	\end{enumerate}
\end{lemma}

Recall that the notation $\tilde\cR_x$ denotes the right-invariant operator corresponding to $\cR_x$ (see Notation~\ref{notation_right_left}).

\begin{proof}[Proof of Lemma \ref{lem_algprop_ex_amplitude:convolution}]
Part (1) follows from a 
modification of  the proof of  Corollary \ref{cor_GM_Leib}  which relies on  Lemma \ref{lem_LeibnizG}.
Part (2) is a 
consequence of the Leibniz property of vector fields  
For Part (3), 
we write 
$$
\kappa_{\tau,x,y}^\bX (v)
  := 
\int_{\bR^n} \kappa_{1 ,x}^\bX (w ) \ F(x,y;v*^\bX_x (-w),w ) \  dw, 
\qquad x\in U, \ v\in \bR^n,
$$
where 
$$
F(x,y;v',w) := \kappa_{2,y}(v')\, J (x,y;v'*_x^\bX w,w).
$$
Let us recall the following general property valid for any smooth  function $\varphi:G\to \bC$:
\begin{align*}
L_{\langle X_j\rangle_x,v}\left(\varphi(v*^\bX_x (-w)) \right)
&=\partial_{t=0}	\varphi(v*^\bX_x t e_j *^\bX_x (-w))\\
\nonumber
&=\partial_{t=0}	\varphi(v*^\bX_x (-(w *^\bX_x - t e_j )))\\
\nonumber
&= - L_{\langle X_j\rangle_x,w}\left(\varphi(v*^\bX_x (-w)) \right).
\end{align*}
Hence, with an integration by parts, this leads to
\begin{align*}
L_{\langle X_j\rangle_x,v} \kappa_{\tau,x,y}(v)
&=
\int_{\bR^n} \kappa_{1,x}^\bX (w ) \ L_{\langle X_j\rangle_x,v}\left(F(x,y;v*^\bX_x (-w),w )\right)dw\\	
&=\int_{\bR^n} L_{\langle X_j\rangle_x,w}\kappa_{\sigma ,x}^\bX (w ) \ F(x,y;v*^\bX_x (-w),w )dw.
\end{align*}
This shows Part (3).

 Let us show Part (4).
 We have 
\begin{align*}
X_{j,x}\kappa_{\tau,x,y}(v)=
& \int_{\bR^n}  X_{j,x}\kappa_{1 ,x}^\bX (w)\ \kappa_{2,y}^\bX ( v*_x^\bX (-w)) \ J(x,y;v,w) dw.\\
&+ \int_{\bR^n} \kappa_{2,x}^\bX (w ) \ \kappa_{2,y}^\bX( v*_x^\bX (-w)) \ X_{j,x}J(x,y;v,w) dw.\\
&+\int_{\bR^n} \kappa_{1,x}^\bX (w ) \ X_{j,x}\left(\kappa_{2,y}^\bX( v*_x^\bX (-w)) \right)\ J(x,y;v,w) dw.
\end{align*}
We recognise the first and second term as the convolution kernels of $\tau (D_{X_{j}} \sigma_1,
\sigma_2, J)$ and $
\tau ( \sigma_1,
\sigma_2, X_{j,x}J)$ respectively. It remains to analyse the third. 
We write for any function $\varphi$:
\begin{align*}
X_{j,x} \left(\varphi((-w)*_x v)\right)
&=\partial_{t=0} 
\varphi( v*_{\exp_x^\bX te_j}^\bX (-w))
\\&=\sum_{\ell=1}^n g_\ell (x;v,w ) (R_{\langle X_\ell\rangle_x} \varphi)(v*_x^\bX (-w))
\end{align*}
with  
$$
(g_\ell(x;v,w))_{\ell=1}^n
:= g (x;v,w)
:=\partial_{t=0} (w*_x^\bX (-v))*_x^\bX (v*_{\exp_x^\bX te_j}^\bX (-w)).
$$
The BCH formula implies that $g_\ell(x;v,w)$ is a polynomial in $v,w$ with smooth coefficients in $x\in U$. 
We also observe 
$$
\eps^{\upsilon_j} \delta_\eps^{-1} g (x;\delta_\eps(v,w))
=g (x;v,w),\qquad \eps>0,
$$
so that 
$$
	\lim_{\eps\to0}\eps^{\upsilon_j -\upsilon_\ell} g_\ell(x;\delta_\eps(v,w))   =0
	\quad \mbox{for}\ \upsilon_\ell> \upsilon_j.
$$
This shows that the third term is as stated, and concludes the proof of Part (4).

Let us show Part (5). We use the notation from Definition~\ref{def:Rockland} and identify  $\cR_x$ with a differential operator on $\bR^n \sim G_x M$ satisfying $\cR_x^t=\cR_x$. Then,  for any $\varphi\in \cS(\bR^n)$, we have 
$$
(\id+\cR_x)^{N}(\id+\cR_x)^{-N}\varphi=\varphi, 
$$
so
$$\varphi(v*^\bX_x (-w) =
 (\id+\tilde \cR_x)|_w^N \left( (\id+\cR_x)^{-N}\varphi\right)(v*^\bX_x (-w) ).
 $$
Hence, 
\begin{align*}
\kappa^\bX_{\tau,x,y} (v)
&=\int_{\bR^n} \kappa_{1,x}^\bX(w) \ (\id+\tilde \cR_x)|_w^N \left( (\id+\cR_x)^{-N}\kappa_{2,y}^\bX\right)(v*^\bX_x (-w) )J(x,y;v,w) dw \\
\\
&=\int_{\bR^n} (\id+\tilde \cR_x)^{N}\left( \kappa_{1,x}^\bX(w) \ J(x,y;v,w) \right)
 \left( (\id+\cR_x)^{-N}\kappa_{2,y}^\bX\right) (v*^\bX_x (-w) )dw ,
 \end{align*}
 after integrating by parts.  
 Developing the differential operator $(\id+\tilde \cR_x)^N$ 
 and using the Leibniz properties of vector fields, 
 yield the result.
\end{proof}

Let us now show Proposition \ref{prop_ex_amplitude:convolution}.

\begin{proof}[Proof of Proposition \ref{prop_ex_amplitude:convolution}]
Let $m_1,m_2\in \bR.$
It suffices to show that 
for any compact subset $\cC\subset U$,
given $\alpha_0,\beta_0,\beta_0',\beta_0''\in \bN_0^n$, there exists $N_0\in \bN_0$ such that for any $N\geq N_0$, there exist $C>0$, $N'_1,N'_2\in \bN$ for which the following holds 
\begin{align*}
&\biggl\|\bX^{\beta_0''}_y\biggl(\tau(\cdot, y) - \sum_{j=0}^N  \tau_j(\cdot, y)\biggr)
 \biggr\|_{(\bX,U), \cC, \alpha_0,\beta_0,\beta_0',\infty}
\\&\qquad \leq C \|\sigma_1\|_{S^{m_1}(\Gh M), (\bX,U), \cC, N'_1}\|\sigma_2\|_{S^{m_2}(\Gh M), (\bX,U), \cC, N'_2} ,
\end{align*}
for any $\sigma_1,\sigma_2\in S^{-\infty}(\Gh M)$.

We observe that the properties described in  Lemma \ref{lem_algprop_ex_amplitude:convolution} hold for $\tau$ but  also for each~$\tau_j$;
indeed, the latter can be viewed as a finite sum of particular cases for $J(x,y,v,w) = c_{\alpha_1,\alpha_2} w^{\alpha_1} (v*_x^\bX (-w))^{\alpha_2}$ for which the results in Lemma \ref{lem_algprop_ex_amplitude:convolution} hold with a proof identical to the one given for this lemma.
Furthermore, Parts (1), (2), (3) and (4) of Lemma \ref{lem_algprop_ex_amplitude:convolution} imply that we may assume $\alpha_0=\beta_0=\beta_0'=\beta_0''=0$ while Part (5) and the symbolic properties of the calculus (see Proposition~\ref{prop_smoothingM}
as well as Theorem \ref{thm_SmGh}) imply that it suffices to consider the case of $m_2$ as negative as we want. The case of $m_2<-Q$ and $\alpha_0=\beta_0=\beta_0'=\beta_0''=0$ follows from Lemma  \ref{lem:symbol_perturbconvolution} and the kernel estimates
(see Theorem \ref{thm_kernelM}), concluding the proof.
\end{proof}

\subsubsection{Some deformations of amplitudes}
%\footnote{V:new lemma.}
In Lemma \ref{lem:symbol_perturbconvolution}, we analysed some deformation of the symbol obtained by convolution of two convolution kernels. In our proof of composition (Proposition \ref{prop_comploc}), we will also need to consider  further deformations of amplitudes. They  will be covered by the following lemma:
\begin{lemma}
\label{lem_amplitudedepv_1}
Let $m\in \bR$.
Let $v_1\mapsto \tau(v_1)$ be a smooth map $\bR^n\to \cT^m(\Gh M)$.
We assume that $\tau_{x,y}(v_1)$ is compactly supported in $v_1\in \bR^n$ locally in $(x,y)\in M\times M$.
Let $(\bX,U)$ be a local frame.
Consider the amplitude $\tau_0$ given   via
$$
\kappa_{\tau_0,x,y}^\bX(v) =
\kappa_{\tau(v_1),x,y}^\bX(v)|_{v_1=v}, 
\qquad x,y\in U, \ v\in \bR^n.
$$
  \begin{enumerate}
      \item The amplitude $\tau_0$ is  in $\cT^m(\Gh M|_U)$ and admits the  asymptotics
$$
\tau_0 \sim \sum_j \sum_{[\alpha]=j} \Delta_{q_\alpha} L_{\langle\bX\rangle,v_1}^\alpha \tau(v_1)|_{v_1=0}
\qquad\mbox{in}\  \cT^m(\Gh M|_U),
$$
with $q_\alpha$ as in Notation \ref{notation:q_alpha}.
  \item For each $j'\in \bN_0$, we consider a sequence of smooth
  maps 
  $$
  v_1\longmapsto \tau_{j'}(v_1),\qquad  \bR^n\longrightarrow \cT^{m-j'}(\Gh M),
  $$
  with $\tau_{j',x,y}(v_1)$
compactly supported in $v_1\in \bR^n$ locally in $(x,y)\in M\times M$.
We assume that    the asymptotics
  $\tau(v_1) \sim \sum_{j'} \tau_{j'} (v_1)$ holds in $\cT^m(\Gh M)$ for each $v_1\in \bR^n$, and furthermore that for each $N'\in \bN_0$, the maps
  $  v_1 \mapsto \tau(v_1) - \sum_{j'=0}^{N'} \tau_{j'} (v_1)$
  are smooth $\bR^n\to \cT^{m-(N+1)}(\Gh M)$.
Then  $\tau_0$ admits the  asymptotics 
$$
\tau_0 \sim \sum_{j',j} \sum_{[\alpha]=j} \Delta_{q_\alpha} L_{\langle\bX\rangle,v_1} \tau_{j'}(v_1)|_{v_1=0}
\qquad \mbox{in}\  \cT^m(\Gh M|_U).
$$
  \end{enumerate}  
\end{lemma}

\begin{proof}
Before starting the proof, let us observe that the asymptotics make sense since 
$$
\Delta_{q_\alpha} L_{\langle\bX\rangle,v_1}^\alpha \tau(v_1)|_{v_1=0}\in   \cT^{m-[\alpha]}(\Gh M|_U).
$$
We apply  the Taylor estimates due to Folland and Stein~\cite{folland+stein_82} on the group $G_x M$,
and recalled in Theorem~\ref{thm_MV+TaylorG} (2) 
(see also  Remark \ref{remthm_MV+TaylorG} for the local uniformity in $x$). 
We will use the notation  $\lceil N \rfloor$ and $\eta$ of Theorem~\ref{thm_MV+TaylorG}.
We have
\begin{align*}
&|\kappa^\bX_{\tau(v),x,y}(v) -
\sum_{[\alpha]\leq N} q_\alpha(v) L_{\langle\bX\rangle_x,v_1}^{\alpha} \kappa^\bX_{\tau(v_1)x,y}(v)|_{v_1=0} | 
\\&\qquad \qquad \leq C \sum_{\substack{|\alpha|\leq \lceil N \rfloor +1\\ [\alpha]> N}}|v|_\bX^{[\alpha]}\sup_{|v_1|_\bX \leq \eta^{\lceil N \rfloor+1}|v|_\bX}\bigg| (L_{\langle\bX\rangle_x,v_1}^\alpha \kappa^\bX_{\tau(v_1)x,y}(v)
\bigg|,    
\end{align*}
with a constant $C=C_{N,x}$ depending locally uniformly in $x\in U$. 
Using the kernel estimates from  Theorem \ref{thm_kernelM} and the uniformity with respect to $v_1\in \bR^n$, we obtain for any compact subset $\cC\subset U$ and any $N$ sufficiently large,
$$
\biggl\|\tau_0 - \sum_{[\alpha]\leq N}  \Delta_{q_\alpha} L_{\langle\bX\rangle,v_1}^\alpha \tau(v_1)|_{v_1=0}
 \biggr\|_{(\bX,U), \cC, \alpha_0,\beta_0,\beta_0',\infty}
 \leq C' \sup_{v_1\in \bR^n, |\alpha'|\leq N'}\|\partial_{v_1}^{\alpha'} \tau(v_1) \|_{S^{m}(\Gh M), (\bX,U), \cC, N'_1}
$$
for $\alpha_0=\beta_0=\beta'_0=\beta''_0=0$; above, $N'$ is an integer depending on $N$ and the structural constant of $M$ while $C'>0$ is  constant depending on $N$, on the structural constant of $M$ on $\cC$.
This also holds for any given multi-indices $\alpha_0,\beta_0,\beta_0',\beta_0''\in \bN_0^n$ by applying the above estimate to 
$\Delta_\bX^{\alpha_0} \langle \bX \rangle^{\beta_0}D_{\bX,x} ^{\beta_0'}
D_{\bX,y} ^{\beta_0''} \tau_{x,y}(v_1)$
which satisfy similar hypotheses to $\tau(v_1).$
This proves Part (1). 

Under the hypotheses of
Part (2), 
we observe 
\begin{align*}
&|\kappa^\bX_{\tau(v)x,y}(v) -
\sum_{j'\leq N'}\sum_{[\alpha]\leq N} q_\alpha(v) L_{\langle\bX\rangle_x,v_1}^{\alpha} \kappa^\bX_{\tau_{j'}(v_1)x,y}(v)|_{v_1=0} | 
\\&\qquad \leq 
|\kappa^\bX_{\tau(v)x,y}(v) -
\sum_{[\alpha]\leq N} q_\alpha(v) L_{\langle\bX\rangle_x,v_1}^{\alpha} \kappa^\bX_{\tau(v_1)x,y}(v)|_{v_1=0} |
\\&\qquad\qquad \qquad +
\sum_{[\alpha]\leq N}
\left| q_\alpha(v) L_{\langle\bX\rangle_x,v_1}^{\alpha} 
\Bigl(\kappa^\bX_{\tau(v_1)x,y} (v)- 
\sum_{j'=0}^{ N'} \kappa^\bX_{\tau_{j'}(v_1)x,y}(v)\Bigr)\Bigl|_{v_1=0}\right|.
\end{align*}
Using the estimate from the proof of Part (1) for the first term, and the kernel estimates together with the smoothness of $  v_1 \mapsto \tau(v_1) - \sum_{j'=0}^{N'} \tau_{j'} (v_1)$, we can conclude as for Part~1.
\end{proof}

%%%%%%%%%%%%%%%%%%%%%%%%%%%%%%%%%%%%%%%%%%%%%%%%%%%%%%%%%%%%%%%%

\subsection{Local quantization of amplitudes}\label{sec:amplitudeQ+asymp}

The local quantization of amplitudes is defined as follows. 
Let $\bX$ be a frame on an open set $U\subset M$ and let  $\chi$ be a $\bX$-cut-off. 
For an amplitude $\tau\in \cT^m(\widehat{G  }M)$,
we define the operator $\AOp^{\bX,  \chi}(\tau)$  at any function $f\in C_c^\infty(M)$ and any point $x\in U$ via
\begin{align*}
    \AOp^{\bX,  \chi}(\tau)f(x)&:=
\int_{M}  \kappa_{\tau,x,y}^{\bX} (- \ln^\bX_x y) \ (\chi_x f)(y)\ \jac_{y} (\ln^\bX_x) \ dy\\
&=
\int_{\bR^n} \kappa^\bX_{\tau,x,\exp^\bX_x(-  w )} (w) \ (\chi_x f)(\exp^\bX_x(-  w )) \ dw;
\end{align*}
as before,
$\jac_{y} (\ln^\bX_x) $
is the (smooth) Radon-Nykodym derivative $\frac{d \ln^\bX_x}{dy} $  of $(\exp^\bX_x - u)^*du$ against the measure $dy$ that we have fixed on $M$.
In other words, the integral kernel of 
the operator $\AOp^{\bX,\chi}(\tau)$  is given by 
\begin{equation}\label{eq:kernel_AOp}
(x,y)\mapsto \kappa^\bX_{\tau,x,y}(-\ln_xy)\chi_x(y)\jac_y(\ln^\bX_x).
\end{equation}

The local quantizations for amplitudes and symbols are related by the following asympotic property:
\begin{proposition}
\label{prop:kxy}
   Let $\bX$ be an adapted frame on an open set $U\subset M$ and $\chi$ an adapted cut-off. 
   Let $\tau\in \cT^m(\Gh M|_U)$ be an amplitude.
   Then there exists a  symbol $\sigma \in S^m (\Gh M|_U)$ such that  the quantizations for amplitudes and symbols are related by
   $$
   \AOp^{\bX,\chi}(\tau)= \Op^{\bX,\chi}(\sigma),
   $$
and the asymptotic expansion $\sigma\sim \sum_{j\in \bN_0} \sigma_j$ for some $\sigma_j\in S^{m-j}(\Gh M|_U)$ holds. Moreover, the first symbol $\sigma_{0}(x,\cdot )$ coincides with $ \tau(x,x,\cdot)$, 
and for each $N\in \bN$, the  map
$$
 \tau \longmapsto \sigma - \sum_{j=0}^{N-1}\sigma_j, \qquad \cT^m(\Gh M)\longrightarrow S^{m-N} (\Gh M|_U)
 $$
  is continuous. 
\end{proposition}

Proposition \ref{prop:kxy} is a consequence of the following more technical property and its proof:
\begin{lemma}
  \label{lem_prop:kxy}
  We continue with the setting of Proposition \ref{prop:kxy}.
  We fix an $\bX$-cut-off function~$\widetilde \chi$ such that $\widetilde \chi \equiv 1$ on the support of $\chi$ and we assume that the support of $\chi$ and $\widetilde \chi$ are sufficiently close to the diagonal.
  The symbol $\sigma$  in Proposition \ref{prop:kxy} may be constructed via  its convolution kernel in the $\bX$-coordinates which is given by:
$$
\kappa^\bX_{\sigma, x}(u)= \kappa^\bX_{\tau,x,\exp^\bX_x (-u)}(u)\, \widetilde \chi (x,\exp_x^\bX (-u)), \qquad x\in U, \ u\in \bR^n.
$$
In this case, the $\sigma_j$'s are given by the following finite sum over the indices $\alpha,\beta$
\begin{align*}
\sigma_j (x,\pi)&= \sum_{[\alpha]=j,\;|\beta|\leq |\alpha| }  c_{\alpha,\beta}(x)  \Delta_\bX^\alpha D_{\bX, y=x}^\beta \tau(x,x,\pi),
\end{align*}
where the  coefficients $ c_{\alpha,\beta}\in C^\infty(U)$, $\alpha,\beta\in \bN_0^n$, are determined by
$$
(\sum_{i=1}^n -u_i X_{i,y})^k = \sum_{|\alpha|=k,|\beta| \leq k}  c_{\alpha,\beta}(y)u^\alpha \bX_y^\beta;
$$
in particular, 
 $
 c_{0,0}=1,$ and $ c_{0,\beta}=0$ for $\beta\neq 0$.
\end{lemma}

\begin{proof}[Proof of Lemma \ref{lem_prop:kxy}]
At least formally, the quantization for the amplitudes, especially \eqref{eq:kernel_AOp}, 
shows that  $
   \AOp^{\bX,\chi}(\tau)= \Op^{\bX,\chi}(\sigma),
   $ with $\sigma$ as in Lemma \ref{lem_prop:kxy}. 
We will prove that  $\sigma$ is a well-defined symbol in $S^m(\Gh M|_U)$ that  satisfies  the properties  stated in Lemma \ref{lem_prop:kxy} and Proposition \ref{prop:kxy}.
The hypothesis in Lemma \ref{lem_prop:kxy} on the cut-offs imposes no restriction on the proof of Proposition \ref{prop:kxy} because of the 
independence of the quantization in the cut-off, see Proposition \ref{prop_Qindepchi}.
In particular, we may assume that the function 
$$
\chi_0(x,u):= \widetilde \chi(x,\exp_x - u), \quad x\in U, \ u\in \bR^n,
$$
has $U$-local compact support where 
$\exp^\bX$ is defined and such that that if $(x,v)$ is in $\supp \, \chi_0$   then so is $(x,tv)$ for any $t\in [-1,1]$. Moreover,
by Lemma \ref{lem_relAsigma},  each 
$\sigma_j$ is a well defined symbol in $S^{m-j}(\Gh M)$.
As the distribution $\kappa^\bX_{\sigma,x}(u)$ is  
$U$-locally compactly supported
on $\supp \, \chi_0$
and $\sigma_j\in S^{m-j}(\Gh M)$, 
Lemma \ref{lem_prop:kxy} and Proposition \ref{prop:kxy} will be proved once we have shown  the following estimates:
for any $\beta,\beta'\in \bN_0$, 
there exist $N_0,N'_0\in \bN_0$ such that
for any compact subsets $\cC,\cC_0\subset U$
with $\cC$ being a neighbourhood of $\cC_0$, 
for any $N\geq N_0$,
there exist  $C>0$ and $N_1\in \bN_0$ for which we have  for  $x\in \cC_0$
and $u\in V_{\cC_0}$:
\begin{equation}
\label{eq_estpflem_prop:kxy}
        \left|L_{\langle \bX\rangle_x,u}^{\beta}
\bX^{\beta'}_x\left(\kappa^\bX_{\sigma,x}(u)
-
\sum_{j=0}^{N} \kappa_{\sigma_j,x}^\bX(u)\right)\right|
\leq C |u|_\bX^{N-N'_0}
 \|\tau\|_{\cT^m(\widehat G M),(\bX,U), \cC, N_1};
\end{equation}
above  $V_{\cC_0}$ is the neighbourhood 
$V_{\cC_0}:= \cup_{x\in \cC_0} \supp \, \chi_0(x,\cdot)$  of $0$ in $\bR^n$ and 
$|\cdot|_\bX$ is the fixed quasinorm associated with the gradation of $\bR^n$, see Notation \ref{notation:q_alpha}.
\smallskip

 Let us start by showing \eqref{eq_estpflem_prop:kxy} for $\beta=\beta'=0$.
Lemma~\ref{lem_Taylorexpxu} yields the following Taylor expansion in $v\sim 0$ (for $x_0,u_0$ fixed)  
$$
\kappa^\bX_{\tau,x_0,\exp^\bX_{x_0} (v)}(u_0)
\sim_{v\sim 0} 
\sum_{k}
\frac 1{k!} ( \sum_{i=1}^n v_i X_{i,y})^k 
\kappa_{\tau,x_0,y}^\bX(u_0)\Big|_{y=x_0},
$$
and we define the corresponding (Euclidean) Taylor polynomial at order $N$:
\begin{align*}
  \bP_N(\tau,x_0; u_0, v)
  &:= \sum_{k\leq N}
\frac 1{k!} ( \sum_{i=1}^n v_i X_{i,y})^k 
\kappa_{\tau,x_0,y}^\bX(u_0)\Big|_{y=x_0} \\
&=\sum_{k\leq N}
\frac 1{k!} \sum_{|\alpha|=k,|\beta| \leq k}  c_{\alpha,\beta}(x)(-v)^\alpha \bX_{y=x}^\beta\kappa_{\tau,x,y}^\bX(u_0),
\end{align*}
with the notation of the statement.
With $\cC,\cC_0, V_{\cC_0}$ as above,
the same lemma gives  the following (Euclidean) estimate for the remainder for any $x\in \cC_0$, $v\in V_{\cC_0}$
\begin{equation}
\label{eq_kappabPtau}
    |\kappa^\bX_{\tau,x_0,\exp^\bX_{x_0} (v)}(u_0)
-\bP_N(\tau,x_0; u_0, v)|
\lesssim_{N,\cC,\cC_0} |v|^{N+1}
\max_{|\alpha|=N+1}\sup_{y\in \cC}
 \Bigl| \bX^\alpha_y \kappa^\bX_{\tau,x_0,y}(u_0)\Bigr|,
\end{equation}
where $|v|$ is the Euclidean norm of $v$.
This implies for any $N\geq Q+m$, $x\in \cC_0$ and $u\in V_{\cC_0}$,
\begin{align*}
& |\kappa^\bX_{\sigma,x}(u)
-
\sum_{j=0}^{N} \kappa_{\sigma_j,x}^\bX(u)|
\leq 
|\chi_0 (x,u) | |\kappa^\bX_{\tau,x_0,\exp^\bX_{x_0} (-u)}(u)
-\bP_N(\tau,x_0; u, -u)|
\\&\qquad\qquad  +
 |(1-\chi_0) (x,u) | \sum_{|\alpha|<N, |\beta| \leq |\alpha|} \frac 1{|\alpha|!}|  c_{\alpha,\beta}(x)u^\alpha \bX_{y=x}^\beta\kappa_{\tau,x,y}^\bX(u)|
\\&\qquad\qquad  +|\chi_0 (x,u) | 
\sum_{[\alpha]>N, |\alpha|<N, |\beta| \leq |\alpha|} \frac 1{|\alpha|!}|  c_{\alpha,\beta}(x)u^\alpha \bX_{y=x}^\beta\kappa_{\tau,x,y}^\bX(u)|\\
&\quad \lesssim_{N,N_2, \cC}  (|u|^{N+1 }  +
|u|_\bX^{N+1})
|u|_\bX^{- \max (Q+m,0)}
 \|\tau\|_{\cT^m(\widehat G M),(\bX,U), \cC, N_1},
\end{align*}
for some $N_1\in \bN_0$ obtained from using the kernel estimates (see Theorem \ref{thm_kernelM}).
Since $|u|\lesssim |u|_\bX$ for $u$ in a bounded neighbourhood of 0, 
this yields \eqref{eq_estpflem_prop:kxy} 
in the case $\beta=\beta'=0$.
\smallskip

To consider the general case, we first  prove that an estimate similar to \eqref{eq_kappabPtau} holds after derivation in $x_0,v,u_0$. 
By the properties (especially uniqueness) of Taylor polynomials,  
such a Taylor estimate holds,  however with  a bound for which we know nothing of its dependence  in $\tau$. 
This is not sufficient for our proof: it is essential to know that the constant is given by a kernel estimate for $\tau$. 
To obtain this dependence,  we observe the following properties:
\begin{align*}
L_{\langle \bX\rangle_x,u}^\beta  \kappa^\bX_{\tau,x,\exp^\bX_{x} (v)}(u)
&=
\kappa^\bX_{\widehat {\langle \bX\rangle_x}^\beta\tau,x,\exp^\bX_{x} (v)}(u),\\
\bX^{\beta'}_x 
\kappa^\bX_{\tau,x,\exp^\bX_{x} (v)}(u)
&=
\sum_{|\beta'_1|+|\beta'_2|\leq |\beta'|} c'_{\beta',\beta'_1,\beta'_2}(x,v) 
\kappa^\bX_{D_{\bX_x}^{\beta'_1} \bX_y^{\beta'_2}  \tau,x,\exp^\bX_{x} (v)}(u),\\
R_{\langle \bX\rangle_x,v}^\gamma
\kappa^\bX_{\tau,x,\exp^\bX_{x} (v)}(u)
&=
\sum_{|\gamma_1|\leq|\gamma|} c_{\gamma,\gamma_1}(x,v) \kappa^\bX_{ \bX_y^{\gamma_1} \tau,x,\exp^\bX_{x} (v)}(u),
\end{align*}
for some families of functions $c'_{\beta',\beta'_1,\beta'_2},c_{\gamma,\gamma_1}\in C^\infty (U\times \bR^n)$
that we can construct by induction.
Consequently, denoting by $\bP_N^{f}(v)$ the (Euclidean) Taylor polynomial at order $N$ of a smooth function $f$ in $v\sim 0$, 
the (Euclidean) Taylor polynomial at order $N$ in $v\sim 0$ of 
$$
R_{\langle \bX\rangle_x,v}^\gamma
\bX^{\beta'}_x 
L_{\langle \bX\rangle_x,u}^\beta  \kappa^\bX_{\tau,x,\exp^\bX_{x} (v)}(u)
=
\sum_{\substack{|\beta'_1|+|\beta'_2|\leq |\beta'|\\ |\gamma_1|\leq|\gamma|}}
(c'_{\beta',\beta'_1,\beta'_2}c_{\gamma,\gamma_1})(x,v)
\kappa^\bX_{\bX_y^{\gamma_1} D_{\bX_x}^{\beta'_1} \bX_y^{\beta'_2}  \widehat {\langle \bX\rangle_x}^\beta
\tau,x,\exp^\bX_{x} (v)}(u),
$$
is 
\begin{align*}
 & \sum_{\substack{|\beta'_1|+|\beta'_2|\leq |\beta'|\\ |\gamma_1|\leq|\gamma|}}
\sum_{N_1+N_2=N}
\bP_{N_1}^{(c'_{\beta',\beta'_1,\beta'_2}c_{\gamma,\gamma_1})(x,\cdot)} (v) \
\bP_{N_2}(\bX_y^{\gamma_1} D_{\bX_x}^{\beta'_1} \bX_y^{\beta'_2}  \widehat {\langle \bX\rangle_x}^\beta,x; u, v),\\
&\qquad\qquad =R_{\langle \bX\rangle_x,v}^\gamma
\bX^{\beta'}_x 
L_{\langle \bX\rangle_x,u}^\beta 
\bP_{N+|\gamma|}(\tau,x; u, v),
\end{align*}
 by uniqueness of Taylor polynomials; 
 furthermore,  by the preceding step of the proof, its remainder for $x\in \cC_0$ and $v\in V$ is estimated by 
$$
\lesssim_{N,\cC, V} |v|^{N+1}
\max_{\beta_1\leq |\beta| }
\max_{|\alpha|\leq N+1+|\gamma|}\sup_{y\in \cC}
 \Bigl| \bX^\alpha_y \bX^{\beta_1}_x L_{\langle \bX\rangle_x}^\beta\kappa^\bX_{\tau,x,y}(u)\Bigr|.
 $$
 This gives the generalisation of \eqref{eq_kappabPtau} for derivatives.
This allows us to estimate 
\begin{align*}
 &    \left|L_{\langle \bX\rangle_x,u}^{\beta}
\bX^{\beta'}_x\left(\kappa^\bX_{\sigma,x}(u)
-
\sum_{j=0}^{N} \kappa_{\sigma_j,x}^\bX(u)\right)\right|
\\&\qquad \leq
|\chi_0 (x,u) | 
\left| L_{\langle \bX\rangle_x,u}^{\beta}
\bX^{\beta'}_x
\left(\kappa^\bX_{\tau,x,\exp_x^\bX -u}(u)
-
\sum_{j=0}^{N} \kappa_{\sigma_j,x}^\bX(u)\right)\right|
 \\&\qquad\qquad + 
  \sum_{
  \substack{\beta_1+\beta_2 = \beta\\
  \beta'_1+\beta'_2=\beta'\\
  \beta_1+\beta'_1>0}}
\binom {\beta_1}{\beta}
\binom {\beta'_1}{\beta'}
\left|
L_{\langle \bX\rangle_x}^{\beta_1}
\bX^{\beta'_1}_x\chi_0 (x,u) 
\right|
\left|L_{\langle \bX\rangle_x,u}^{\beta_2}
\bX^{\beta'_2}_x
\left(\kappa^\bX_{\tau,x,\exp_x^\bX -u}(u)\right)\right|.
\end{align*}
For the first term on the RHS, we proceed as in the case $\beta=\beta'=0$ above using the generalisation of \eqref{eq_kappabPtau} for derivatives we just obtained.
Due to the derivatives on $\chi_0$, the second term is supported 
on a $U$-locally compact neighbourhood of 0 and away from a smaller $U$-locally compact neighbourhood of 0. 
The kernel estimates (see Theorem \ref{thm_kernelM}) allow us  to  obtain \eqref{eq_estpflem_prop:kxy} for any $\beta,\beta'$.
This concludes the proof.
\end{proof}

\subsection{The adjoints of locally quantized operators}\label{sec:adjoint}

This section is devoted to showing the following  property for the local quantization:

\begin{proposition}
\label{prop_adjloc}
Let $\bX$ be an adapted frame on an open subset $U\subset M$.  
  If  $\sigma\in S^{m}(\widehat G M)$, 
 then  there exists a unique  symbol $\sigma^{(*)}   \in S^{m}(\widehat G M|_U)$ 
 such that  
$$
\Op^{\bX,\chi}(\sigma)^*=
\Op^{\bX,\chi}(\sigma^{(*)}).
$$
Moreover, the following map is continuous
$$
\sigma\longmapsto \sigma^{(*)},
\qquad S^{m}(\widehat G M) \longrightarrow S^{m}(\widehat G M|_U),
$$
  and $\sigma^{(*)}$ admits an asymptotic expansion 
  $$
  \sigma^{(*)}\sim \sum_{j\geq 0} \rho_j
  $$
  such that  $\rho_0=\sigma^*$ and for each $N\in\bN_0$, the following map is continuous
$$
\sigma\longmapsto \sigma ^{(*)}-\sum_{j=0}^N\rho_j, 
\qquad S^{m}(\widehat GM)\longrightarrow S^{m-N-1}(\widehat GM).
$$
\end{proposition}

Before discussing the proof of Proposition~\ref{prop_adjloc}, 
we observe that 
it implies together with the properties of the adjoint and 
the continuous action of $\Op^{\bX,\chi}(\sigma)$, $\sigma\in \cup_m S^m(\Gh)$, on $C_c^\infty(U)$
(see Lemma~\ref{lem_actonDistrib_locQ_1}) that
 the action of these operators can be extended to  distributions:

\begin{corollary}
\label{cor_actonDistrib_loc_Q}
Let $(\bX,U)$ be an adapted local frame  and $\chi$ a cut-off function  for $(\bX,U)$. Let $m\in \bR$.
If $\sigma\in S^m(\widehat GM|_{U})$, the operator $\Op^{\bX, \chi} (\sigma)$ acts continuously on $C_c^\infty(U)$ and $\mathcal D'(U)$. Furthermore, 
the map  $\sigma \mapsto \Op^{\bX, \chi} (\sigma)$ is an injective morphism of topological spaces 
\[
S^m(\Gh M|_U)\to \sL(C_c^\infty(U))\;\;\mbox{and}\;\; S^m(\Gh M|_U)\to \sL(\cD'(U)).
\]
   \end{corollary}

   \begin{proof}
If $f\in \cD'(U)$, the distribution $\Op^{\bX,\chi}(\sigma) f$ is defined by 
       \[
       \langle \Op^{\bX,\chi}(\sigma) f,g\rangle= \langle f, \overline{\Op^{\bX,\chi}(\sigma^{(*)}) \bar g}\rangle,\;\; g\in \cD(M):=C^\infty_c(M).
       \]
       Here $\langle \cdot,\cdot\rangle$ denotes the duality bracket in distributions. 
Hence the statement follows by duality and Lemma~\ref{lem_actonDistrib_locQ_1}. 
   \end{proof}

The proof of  Proposition \ref{prop_adjloc} follows  from 
the properties of invariance by change of cut-off functions (see Proposition~\ref{prop_Qindepchi})  and  the following lemma:

\begin{lemma}
\label{lem_adj}
Let $\bX$ be an  adapted frame  on an open subset $U\subset M$, and  $\chi$ be a $\bX$-cut-off and we have:

\begin{enumerate}
    \item The function  $\chi^*$ given by  
\[
\chi^*(x,y)= \chi(y,x), 
\qquad x,y\in U,
\]
 is also a cut-off for $\bX$ and if  $\sigma\in S^{m}(\widehat G M)$, then 
    $$
\Op^{\bX, \chi}(\sigma)^*
=
\Op^{\bX, \chi^*}(\sigma^{(*)}),
$$
where $\sigma^{(*)}\in S^{m}(\widehat G M|_U)$
is the symbol whose convolution kernel in the $\bX$-coordinates is given by:
$$
\kappa_{\sigma^{(*)},x}^\bX (u)= 
\overline \kappa_{\sigma,\exp^\bX_x -u}\left(-u\right) \frac{\jac_x(\ln^\bX_{\exp_x^\bX -u})}{\jac_{\exp_x^\bX -u}(\ln_x^\bX)}.
$$
\item  If  $\sigma\in S^{m}(\widehat G M)$, then
$$
\sigma^{(*)} \sim \sum_{j \in \bN} 
\sum_{[\alpha]=j,\beta}
d_{\alpha,\beta}(x) 
\Delta_\alpha^\bX D_\bX^\beta \sigma^*,
$$
where 
  the coefficients $d_{\alpha,\beta}$ are obtained 
from the coefficients $c_{\alpha,\beta}$ of Lemma~\ref{lem_prop:kxy}:
$$
c_{\alpha,\beta}(x)\bX^{\beta}_{y=x}
\left( \frac{ \jac_x(\ln^\bX_y)}{\jac_y(\ln^\bX_x)} f(y)\right)
=
\sum_{|\beta'|\leq |\beta|}
d_{\alpha,\beta'}(x) 
\bX^{\beta'}f(x),
\quad x\in U, \ f\in C^\infty(U).
$$
Moreover, for each $N\in \bN$, 
the following map is continuous
$$
\sigma \longmapsto \sigma^{(*)} - \sum_{j\leq N} 
\sum_{[\alpha]=j,\beta}
d_{\alpha,\beta}(x) 
\Delta_\alpha^\bX D_\bX^\beta \sigma^*,
\quad
S^m(\Gh M)\longleftrightarrow S^{m-N-1}(\Gh M|_U).
$$
\end{enumerate}
\end{lemma}

\begin{proof}[Proof of Lemma~\ref{lem_adj}]
{\it Part (1)}. For any $f\in C^\infty_c(U)$ and $x\in M$
\[
\Op^{\bX,\chi} (\sigma)^*f(x)=\int_M 
\overline{\kappa}_{\sigma, y}^\bX(-\ln_y^\bX x) (\chi^*_xf)(y)\jac_x(\ln_y^\bX)dy.
\]
Therefore, the convolution kernel $\kappa^\bX_{\sigma^{(*)},x}$ of the operator $\Op^{\bX,\chi} (\sigma)^*$ satisfies for $x,y\in M$
\[
\kappa^\bX_{\sigma^{(*)},x}(-\ln_x^\bX y)
= \overline{\kappa}_{\sigma,y}^\bX(-\ln_y^\bX x) \frac{\jac_x(\ln_y^\bX)}
{\jac_y(\ln_x^\bX)}.
\]
The change of variable $u=-\ln_x^\bX y$ gives 
\[
\kappa^\bX_{\sigma^{(*)},x}(u)
= \overline{\kappa}_{\sigma,\exp_x^\bX -u}^\bX(-\ln_{\exp_x^\bX -u}^\bX x) \frac{\jac_x(\ln_{\exp_x^\bX -u}^\bX)}
{\jac_{\exp_x^\bX -u}(\ln_x^\bX)}.
\]
We then conclude by observing that 
\[
-\ln_{\exp_x^\bX -u}^\bX x=\ln_{x}^\bX \exp_x^\bX-u=-u.
\]

\smallskip

{\it Part (2)}.  By Part (1), 
we can write at least formally $\Op^{\bX,\chi}(\sigma)^*=\AOp^{\bX,\chi}(\tau)$ where $\tau$ is the amplitude given by the convolution kernel in the $\bX$-coordinates:
\begin{align*}
\kappa^\bX_{\tau,x,y}(v)
= \overline  \kappa_{\sigma,y}^\bX (-v) \frac{ \jac_x(\ln^\bX_y)}{\jac_y(\ln^\bX_x)},
\qquad x,y\in U.
\end{align*}
We check readily  that
$\tau\in \cT^m(\Gh M|_U)$
with  
$$
c_{\alpha,\beta}(x)\Delta^\bX_\alpha D_{\bX, y=x}^\beta \tau(x,y,\pi)
=
\sum_{|\beta'|\leq |\beta|}
d_{\alpha,\beta'}(x)
\Delta^\bX_\alpha D_\bX^\beta \sigma^*(x,\pi),
\quad x\in U, \ \pi\in G_x M,
$$
by Lemma \ref{lem_relAsigma}.
We then conclude with  Proposition~\ref{prop:kxy}.
\end{proof}

Usually, $\sigma^{(*)}$ is different from $\sigma^*$.
  However,  the   symbols $\widehat {\langle X_j\rangle}$ have special properties: 

\begin{lemma}
\label{lem_adjXj}
Let $\bX$ be an  adapted frame  on an open subset $U\subset M$, 
and  $\chi$ be a cut-off for~$\bX$.
Then $$
\Op^{\bX, \chi}(\widehat {\langle X_j\rangle})^*=
-\Op^{\bX, \chi}(\widehat {\langle X_j\rangle}).
$$	
\end{lemma}
In other words,   $\widehat {\langle X_j\rangle}^{(*)}=\widehat {\langle X_j\rangle}^* = - \widehat {\langle X_j\rangle}$.

\begin{proof}[Proof of Lemma \ref{lem_adjXj}]
Let $f,g\in C^\infty_c(U)$.
 Proposition \ref{ex_OpT} implies
\begin{align*}
(\Op^{\bX, \chi}(\widehat {\langle X_j\rangle}) f,g)_{L^2(U)}
&=
\partial_j\big|_{v=0}
\int_U  f (\exp_x^\bX v )\bar g(x)dx
\\&=\partial_j\big|_{v=0}
\int_U  f (y )\bar g(\exp_y^\bX -v) \jac_y \exp_y^\bX (-v) dy
\\
&=	-
\int_U  f (y )  
\partial_j\big|_{v=0}\bar g(\exp_y^\bX -v)
\ dy,
\end{align*}
since
$\jac_y \exp_y ^\bX0=1$ and $\partial_j\big|_{v=0} \jac_y \exp_y^\bX (-v)=0$ because we have 
 \[
  D_{v=0}D_y \exp_y^\bX(-v) = D_y D_{v=0}\exp_y^\bX(-v) = -D_y (\id ) =0.
 \]
	\end{proof}

 \subsection{The composition of locally quantized operators}\label{sec:composition}

This section is devoted to showing   properties regarding composition  for the local quantization:
\begin{proposition}
\label{prop_comploc}
Let $\bX$ be an adapted frame on an open subset $U\subset M$
and let $\chi$ be a cut-off function for $\bX$. 
If  $\sigma_1\in S^{m_1}(\widehat G M)$ and  
   $\sigma_2\in S^{m_2}(\widehat G M)$, 
 then  there exists a unique symbol
 $$
 \sigma_1 \diamondsuit \sigma_2=\sigma_1 \diamondsuit_{\bX,\chi} \sigma_2   \in S^{m_1+m_2}(\widehat G M|_U)
 $$
 such that  
$$
\Op^{\bX,\chi}(\sigma_1 \diamondsuit \sigma_2)=
\Op^{\bX,\chi}(\sigma_1)\Op^{\bX,\chi}(\sigma_2).
$$
Moreover, the following properties hold:
\begin{enumerate}
    \item
The following map is continuous
$$
(\sigma_1,\sigma_2)\longmapsto  \sigma_1 \diamondsuit \sigma_2, 
\qquad 
S^{m_1}(\widehat G M)\times S^{m_2}(\widehat G M) \longrightarrow S^{m_1+m_2}(\widehat G M|_U).
$$
\item 
 The symbol  $\sigma_1 \diamondsuit \sigma_2$
 admits an asymptotic expansion 
 $$
 \sigma_1 \diamondsuit \sigma_2\sim \sum_{j\geq 0} \rho_j\quad\mbox{in} \ S^{m_1+m_2}(\Gh M|_U),
 $$
 satisfying   $\rho_0=\sigma_1\sigma_2$ and such that for each $N\in\bN$, the following map is continuous
$$
   (\sigma_1,\sigma_2)\longmapsto \sigma_1 \diamondsuit \sigma_2-\sum_{j=0}^N\rho_j, \quad 
   S^{m_1}(\widehat GM)\times S^{m_2}(\widehat GM)\longrightarrow S^{m_1+m_2-N-1}(\widehat GM).
   $$
\end{enumerate}
\end{proposition}

Let us first explain the strategy of the proof of Proposition \ref{prop_comploc} (which follows from standard ideas developed in the Euclidean case). By  Proposition \ref{prop_adjloc}, there exists a symbol $\sigma_2^{(*)}$ such that 
\[
\Op^{\bX,\chi}(\sigma_2)= (\Op^{\bX,\chi}(\sigma_2)^*)^*
= \Op^{\bX,\chi}(\sigma_2^{(*)})^*.
\]
Therefore, it is enough to be able to prove that for all symbol $\sigma_3\in S^m(\widehat GM)$, the operator 
$\Op^{\bX,\chi}(\sigma_1)\, \Op^{\bX,\chi}(\sigma_3)^*$ is  pseudodifferential  in the sense that it  has a symbol in a suitable class, and that  
 this symbol can be written as an asymptotic sum of symbols with suitable continuity properties. 
The advantage of considering this new problem  is that a simple amplitude is associated with the operator $\Op^{\bX,\chi}(\sigma_1)\, \Op^{\bX,\chi}(\sigma_3)^*$. Then Proposition~\ref{prop:kxy} will allow us to conclude and to identify the first term of the asymptotic expansion.

With this strategy in mind,  Proposition \ref{prop_comploc} will derive from the following lemma:
\begin{lemma}
\label{lem_compadj2}
We continue with the setting and notation of Proposition \ref{prop_comploc}.
  The symbol  $\sigma_1 \diamondsuit (\sigma_2^{(*)})$
  admits an asymptotic expansion 
  $$
  \sigma_1 \diamondsuit (\sigma_2^{(*)})\sim \sum_{j\geq 0} \rho_j\quad\mbox{in} \ S^{m_1+m_2}(\Gh M|_U),
  $$
satisfying $\rho_0=\sigma_1 \sigma_2^*$ and  such that for each $J\in\bN$, the following map is continuous
$$   (\sigma_1,\sigma_2)\longmapsto \sigma_1 \diamondsuit (\sigma_2^{(*)})-\sum_{j=0}^J \rho_j,\qquad S^{m_1}(\widehat GM)\times S^{m_2}(\widehat GM)\longrightarrow S^{m_1+m_2-J-1}(\widehat GM).
$$
Moreover,
there exist two families of smooth functions $(c_{\alpha_1,\alpha_2,\gamma_1,\gamma_2})_{\alpha_1,\alpha_2,\gamma_1,\gamma_2\in \bN^n}$ in $C^\infty(U)$ and $(p_{\alpha_3})_{\alpha_3\in\bN^n}$ in $C^\infty(U\times \bR^n)$ such that 
\[
\rho_j=\sum_{[\alpha_1]+[\alpha_2]+|\alpha_3|_{min} =j}\; \sum _{\gamma_1+\gamma_2=\alpha_3,\; |\beta|\leq |\gamma_2|}
c_{\alpha_1,\alpha_2,\gamma_1,\gamma_2}(x)\Delta_{ p_{\alpha_3} }\widehat{\langle X\rangle}^{\alpha_3}
\Delta_\bX^{\alpha_1+\gamma_1} \sigma_1 \Delta_\bX ^{\alpha_2+\gamma_2} D_\bX^\beta  \sigma_2^*.
\]
\end{lemma}

\begin{proof}[Proof of Lemma \ref{lem_compadj2}]
By density of smoothing symbols (see Proposition~\ref{prop_smoothingM}), and, 
since we will prove estimates for continuous seminorms, 
we may assume that the symbols $\sigma_1,\sigma_2$ are smoothing; this justifies  the manipulations below. 

The integral kernel of the operator $\Op^{\bX,\chi}(\sigma_1)\Op^{\bX,\chi}(\sigma_2)^*$ is given   by 
\begin{align*}
(x,y) \mapsto &\int_M \kappa_{1,x}^\bX (-\ln_x^\bX z) \overline \kappa_{2,y}^\bX(-\ln^\bX_y z) \jac_z(\ln^\bX_x)\jac_z(\ln_y^\bX)\chi_x(z)\overline \chi_y(z)dz\\
&= \widetilde \chi(x,y)\int_{\bR^n} \kappa_{1,x}^\bX (-v) \overline \kappa_{2,y}^\bX(-\ln^\bX_y (\exp^\bX_x v)) \jac_{\exp^\bX_x v}
(\ln^\bX_y) \, (\chi_x\overline \chi_y)(\exp^\bX_x v)dv,
\end{align*}
where $\widetilde \chi$ is an $\bX$-cut-off function  with \textcolor{blue}{$\widetilde \chi =1$ on a neighbourhood of  $\cup_{z\in U} \supp (\chi(\cdot,z) \bar \chi(\cdot,z)) $}.
Consequently, 
$$
\Op^{\bX,\chi}(\sigma_1) \, \Op^{\bX,\chi}(\sigma_2)^* = \Op^{\bX,\widetilde \chi}(\sigma)
=
\AOp^{\bX,\widetilde \chi}(\tau),
$$
where  $\tau $ is the amplitude whose  convolution kernel in the $\bX$-coordinates is given by 
$$
    \kappa_{\tau,x,y}^\bX (u)= 
\int_{\bR^n} \kappa_{1,x}^\bX (-v) \overline \kappa_{2,y}^\bX(-\ln^\bX_{\exp^{\bX}_x -u} (\exp^\bX_x v))
\frac{ \jac_{\exp^\bX_x v}
(\ln^\bX_y)}{\jac_y(\ln^\bX_x)}
(\chi_x\overline \chi_y)(\exp^\bX_x v)
dv.
$$
After
 the change of variable 
$$
w =\varphi_{x,u}(v)= \ln^\bX_{\exp^{\bX}_x -u} (\exp^\bX_x v) ,
\qquad
v=\varphi_{x,u}^{-1}(w)= \ln^\bX_x \exp^\bX_{\exp^\bX_x -u} w,
$$
we obtain:
$$
\kappa_{\tau,x,y}^\bX (u)= 
\int_{\bR^n} \kappa_{1,x}^\bX (-\varphi^{-1}_{x,u}(w)) \kappa_{2*,y}^\bX(w)
 \, 
(\chi_x\overline \chi_y)(\exp^\bX_{\exp^\bX_x -u} w)
\frac{ \jac_{\exp^\bX_{\exp^\bX_x -u} w}
(\ln^\bX_y)}{\jac_y(\ln^\bX_x)} \jac_w \varphi_{x,u}^{-1}
dw,
$$
 where $\kappa_{2*,y}^\bX (u) := \overline \kappa_{2,y}^\bX (-u)$ is
  the convolution kernel of the symbol $\sigma_2^*$ in the $\bX$-coordinates.
We then make the substitution $w=u*_x(-w_1)$, 
and define as in Corollary \ref{cor_lem_Cq1BCH} (2)
$$
\check r(x;u,w_1) =(- w_1)*_x
(-\ln^\bX_{x} \exp^\bX_{\exp^\bX_x -u} u*_x(-w_1)),
$$
to obtain
$$
\kappa_{\tau,x,y}^\bX (u) = \kappa_{\tau(u_1),x,y}(u)\big|_{u_1=u},
$$
where
$$	\kappa_{\tau(u_1),x,y}^\bX (u)
  := 
\int_{\bR^n} \chi(x;u_1,w_1)\kappa_{1,x}^\bX (w_1 *_x \check r(x;u_1,w_1)) \kappa_{2*,y}^\bX( u*_x(-w_1)) \ J(x,y;u,w_1) dw_1.
$$
Above, we adopted the short-hand
$$
J(x,y;u,w_1):=(\chi_x\overline \chi_y)(\exp^\bX_{\exp^\bX_x -u} u*_x(-w_1)) \
\frac{ \jac_{\exp^\bX_{\exp^\bX_x -u} u*_x(-w_1)}
(\ln^\bX_y)}{\jac_y(\ln^\bX_x)} 
\jac_{u*_x(-w_1)} \varphi_{x,u}^{-1},
$$
for the expression involving the cut-off functions and the Jacobians; we also 
 inserted a suitable $U$-locally  compactly supported function 
  $\chi_1\in C^\infty(U\times \bR^n\times \bR^n)$. 
Hence, with the notation of Lemma \ref{lem_amplitudedepv_1}, we have 
$$\tau=\tau_0
\qquad\mbox{where}\qquad
\tau(u_1) :=\tau(\sigma_1^{(\check r_{u_1},\chi_{1,u_1} )},\sigma_2^*,J),
$$
having used the notation of Propositions \ref{prop_ex_amplitude:convolution} and \ref{prop:kernel_higher_order}, and set
$$
\check r_{u_1} (x;w_1):=\check r(x;u_1,w_1), \qquad \chi_{1,u_1}(x;w_1):=\chi_1(x;u_1,w_1).
$$
These statements, together with the link between amplitudes and symbols (see Proposition~\ref{prop:kxy}) and  the continuity of  taking the adjoint of symbols (see Proposition \ref{prop_comp+adj}), allow us to conclude the proof. 
\end{proof}

\section{Pseudodifferential calculus on $M$}\label{sec:global_sobolev}

In this section we define a pseudodifferential calculus on filtered manifolds. 
In Section~\ref{sec:global} we introduce   a  passage 
from  the  local quantization to the definition of  operators defined on the whole manifold~$M$. It follows classical lines and is inspired by the construction of quantization on  manifolds  presented in~\cite{Zwobook,Hintz}. 
Then, we verify that the operators we have introduced satisfy the expected properties: composition rules when they are properly supported (Section~\ref{sec:properly_supported}),  a notion of principal symbol (Section~\ref{subsec:princ_symbol}), and the existence of parametrices when the symbols satisfy adequate conditions (Section~\ref{subsec:parametrices}).
As usual, associated with this set of pseudodifferential calculus comes a notion of (local) Sobolev spaces that we describe in Section~\ref{sec:sobolev}, and on which  pseudodifferential operators act continuously. Finally,  in Section~\ref{subsec:groupoid}, we discuss what is the classical pseudodifferential calculus deriving from our construction and prove in Section~\ref{sec:groupoid} that it coincides with 
the abstract class of pseudodifferential operators constructed by 
van Erp and Yuncken calculus in~\cite{VeY1,VeY2}
via groupoids.

\subsection{Pseudodifferential operators on $M$}\label{sec:global}

We recall that  the smooth manifold $M$ is assumed second countable, paracompact Hausdorff, 
and 
that we have fixed a measure $dx$ on~$M$. 

\subsubsection{Definition}
\label{subsubsec_defPDOM}

\begin{definition}\label{def:Q0}
    Let $M$ be a filtered manifold. Let $m\in\mathbb{R}$. 
    A linear operator 
    $$P: C_c^\infty(M)\longrightarrow C^\infty(M)$$
    is a pseudodifferential operator of order $m$ on the filtered manifold $M$ when 
    \begin{itemize}
         \item[(1)] 
         for any adapted frame $\bX$ on an open set $U\subset M$, for any $\bX$-cut-off $\chi$, 
there exists a symbol $\sigma\in S^m(\widehat{G}M|_U)$ such that 
      $$
      \varphi P(\psi f)=\varphi \Op^{\bX, \chi}(\sigma)(\psi f),
      $$
         for any $f\in C_c^\infty(M)$ and 
         $\varphi, \psi\in  C_c^\infty (U)$, 
\item[(2)] there exists a smooth function 
$K\in C^\infty (M\times M  \setminus D_M)$
defined away from the diagonal $D_M:=\{(x,x)\colon x\in M\}$
such that 
$$
\varphi_1 P(\varphi_2)(x) =\varphi_1 (x)\int_M K(x, y) \varphi_2(y)dy, \qquad x\in M,
$$
for any $\varphi_1,\varphi_2 \in C_c^\infty(M)$ with disjoint supports (i.e. $\supp\,\varphi_1\cap \supp\,\varphi_2=\emptyset$).
     \end{itemize}
\end{definition}

The properties in the definition above and  those of the local quantization (especially Lemma~\ref{lem_actonDistrib_locQ_1}) imply that $P$ is a continuous operator $C_c^\infty (M)\to \cD'(M)$.
By the Schwartz kernel theorem, it admits an integral  kernel $K_0\in \cD'(M\times M)$. 
Part (2) above says that $K_0$ coincides with a smooth function on $M\times M \setminus D_M$, i.e. away from the diagonal.
When this integral kernel is smooth on $M\times M$,  we will follow the traditional convention 
(as well as  in coherence with what we have done in the preceding sections)
of calling the operator smoothing:

\begin{definition}
\label{def_Psi-infty}
    Let $M$ be a smooth manifold. An operator $P$ on $M$ is {\it smoothing }when it admits a smooth integral kernel, that is, there exists $K\in C^\infty(M\times M)$ such that
    $$
    Pf(x) = \int_M f(y)K(x,y) dy, \qquad x\in M.
    $$
    \end{definition}

We denote by $\Psi^m(M)$  the  space of pseudodifferential operators of order $m\in \bR$ on  $M$
and by $\Psi^{-\infty}
(M)$ the  space of smoothing operators  on $M$. 
We check readily for $m,m_1,m_2\in \bR$ with $m_1\leq m_2$
$$
\Psi^{-\infty}(M)\subset \Psi^m (M)
\quad\mbox{and}\quad
\Psi^{m_1} (M) \subset \Psi^{m_2} (M).
$$
We also denote by 
$$
\Psi(M):=\cup_{m\in \bR}\Psi^m(M)
$$
the  space of pseudodifferential operators (of any order) on $M$.

From Definitions \ref{def:Q0} and \ref{def_Psi-infty} together with the properties of the local quantization, 
we check readily that 
     $\Psi^m (M)$ for each  $m\in \bR\cup\{-\infty\}$ is a vector space and furthermore a left and right $C^\infty(M)$-module:
    $$
\forall \varphi,\psi\in C^\infty (M), \quad \forall P\in \Psi^m(M)
\qquad \varphi P\psi \in \Psi^m (M).
$$
Moreover, it is stable under the formal adjoint:
\begin{lemma}
\label{lem_adjPsi}
 If $P\in \Psi^m(M)$ with $m\in \bR\cup\{-\infty\}$, then its formal adjoint $P^*$, that is, the operator defined via
 $$
 \forall f,g\in C_c^\infty (M)\qquad
 (P^* f,g)_{L^2(M)} = (f,Pg)_{L^2(M)},
 $$
 is a well defined operator  in $\Psi^m(M)$.
\end{lemma}
  \begin{proof}
  We observe that if $P$ is an operator  $C_c^\infty(M)\to C^\infty(M)$ with integral kernel $K\in \cD'(M\times M)$, then  $P^*$ is also  an operator  $C_c^\infty(M)\to \cD'(M)$ with integral kernel $K^*\in \cD'(M\times M)$ given by $K^*(x,y)= \bar K(y,x)$.
  This observation implies 
  the statement for  smoothing operators (i.e. the case $m=-\infty$) 
  as well as the second part in Definition \ref{def:Q0} 
  for the formal adjoint $P^*$ of an operator
 $P\in \Psi^m(M)$ with $m\in \bR$. 
The first part in Definition \ref{def:Q0} for $P^*$ is a consequence of Proposition \ref{prop_adjloc}. This concludes the proof.
  \end{proof}

From Definition \ref{def:Q0} and the properties of the local quantization (especially Lemma \ref{lem_actonDistrib_locQ_1} and Corollary \ref{cor_actonDistrib_loc_Q}), we check readily that a pseudodifferential operator  on $M$ acts linearly continuously $C_c^\infty(M)\to  C^\infty (M)$
and also $\sE'(M) \to \cD'(M)$ by Lemma \ref{lem_adjPsi};  here $\sE'(M)$ stands for the space of  distributions with compact support. 
Note that the space $\Psi^{-\infty}(M)$ may  be  described equivalently as
 the space of continuous  linear operators $\sE'(M) \to C^\infty(M)$, as well as 
$$
\Psi^{-\infty}(M)
=\cap_{m\in \bR} \Psi^m (M).
$$

\subsubsection{Examples}\label{sec:ex_global_pseudo}
Here, we give some examples of pseudodifferential operators.
Our first class of examples  results from the local quantizations defined in Section~\ref{sec:local_quant}:

\begin{lemma}
\label{lem:local-to-global}
Let $m\in \bR\cup \{-\infty\}$. 
Let  $(\bX,U)$ be an adapted frame with $\mathbb{X}$-cut-off $\chi$.
\begin{enumerate}
    \item If  $\sigma\in S^m(\widehat{G}M|_U)$, then $\Op^{\mathbb{X}, \chi}(\sigma)\in \Psi^m(U)$.
    \item Moreover, if $\sigma\in \cup_{m'\in \bR\cup\{-\infty\}} S^{m'}(\widehat{G}M|_U)$ satisfies $\Op^{\mathbb{X}, \chi}(\sigma)\in \Psi^m(U)$, then $\sigma\in S^m(\Gh M|_U).$
\end{enumerate}
\end{lemma}
\begin{proof}
The characterisation of smoothing symbols (see Proposition \ref{prop_smoothingM} (1), see also Corollary \ref{cor:ker_smoothing}) implies the case $m=-\infty$. We now prove the case $m\in \bR$.

Let $\sigma\in S^m(\widehat{G}M)$.
The kernel estimates (see Theorem \ref{thm_kernelM}) imply that  the Schwartz kernel $K(x, y)$ of $\Op^{\bX, \chi}(\sigma)$ is
smooth away from the diagonal when $m\in \bR$. Hence,
Part (2) of Definition  \ref{def:Q0} holds.

    Let $\bY$ be another adapted frame on an open set $V\subset U$, with $\bY$-cut-off $\widetilde{\chi}$. Then by Propositions  \ref{prop_Qindepchi} and  \ref{prop_indepframe}, 
    there exists $\widetilde{\sigma}\in S^{m}(\widehat{G}M|_U)$ such that
$\Op^{\bX, \chi}(\sigma) = \Op^{\bY, \widetilde{\chi}}(\widetilde{\sigma}).$
    This shows that Part (1) of Definition  \ref{def:Q0} holds.
    Therefore $\Op^{\bX, \chi}(\sigma)\in \Psi^m(U)$.

 Let $m'\in \bR$ and $\sigma\in S^{m'}(\widehat{G}M|_U)$.   
 We assume $\Op^{\bX, \chi}(\sigma)\in \Psi^m(U)$. We may assume $m'<m$. 
 By Part (1) of Definition  \ref{def:Q0}, 
 there exists $\widetilde{\sigma}\in S^{m'}(\widehat{G}M|_U)$ such that 
 $$
 \forall \varphi, \psi\in  C_c^\infty (U)\qquad
\varphi  \Op^{\bX, \chi}(\sigma) \psi = \varphi \Op^{\bX, \chi}(\widetilde{\sigma}) \psi.
 $$
 This readily implies that $\kappa_{\sigma}^\bX -\kappa_{\sigma'}^\bX$ is smooth on $U\times \bR^n$, so $\sigma-\widetilde \sigma \in S^{-\infty}(\Gh M)$ by Corollary~\ref{cor:ker_smoothing}. This shows Part (2).
\end{proof}

We will often identify a  function  with compact support in $U\subset M$ with a function trivially extended to $M$. And conversely, if a function on $M$ is supported in $U$, we may view it as a function on $U$. With these conventions in mind, Lemma \ref{lem:local-to-global} and its proof imply readily:
\begin{corollary}
\label{corlem:local-to-global}
Let $m\in \bR\cup \{-\infty\}$. 
Let  $( \bX,U)$ be an adapted frame with $\mathbb{X}$-cut-off $\chi$.
For any $\sigma\in S^m(\widehat{G}M|_U)$ and 
$\psi\in C_c^\infty (U)$, 
the operator given by 
$$
f\longmapsto \Op^{\bX, \chi}(\sigma)( \psi f), \qquad C_c^\infty(M)\longrightarrow C^\infty (M).
$$ 
is in $\Psi^m (M)$. That is, 
   $$
\forall \psi\in C_c^\infty (U),\qquad 
\Op^{\bX, \chi}(\sigma)\,  \psi \in \Psi^m(M).
$$
Moreover, 
if $\sigma\in \cup_{m'\in \bR\cup\{-\infty\}} S^{m'}(\widehat{G}M|_U)$ satisfies $\Op^{\mathbb{X}, \chi}(\sigma)\psi\in \Psi^m(M)$ for any $\psi\in C_c^\infty (U)$, then $\sigma\in S^m(\Gh M|_U).$
\end{corollary}

\begin{remark}\label{remlem:local-to-global}
    Lemma \ref{lem:local-to-global} and its proof imply that  if the property described in Part (1) of Definition~\ref{def:Q0} holds for a family of open subsets $(U_\alpha)_\alpha$ of $M$, each $U_\alpha$ being equipped with an adapted frame $\bX_\alpha$ and a $\bX_\alpha$-cut-off, with the family of open subsets  covering  $M$, then Part (1) of Definition  \ref{def:Q0} holds for any $(\bX,U)$.
\end{remark}

Our second class of examples comprises  differential operators:
\begin{lemma}
\label{lem_exdiffPDO}
Any differential operator is pseudodifferential,
and its  integral kernel is supported on the diagonal. Moreover, 
 if $P$ is a differential operator, then $P\in \Psi^N(M)$ with $N$ being its homogeneous order.   
\end{lemma}
The homogeneous order was defined in Definition \ref{def_homorder}. This statement implies:
$$
\DO(M)\subset \Psi(M),
\quad\mbox{and more precisely,}\quad 
\DO^{\leq N}(M)\subset \Psi^N(M).
$$
\begin{proof}
A differential operator $P$ is a linear operator $P: C_c^\infty(M)\to C^\infty(M)$ with  integral kernel supported on the diagonal. 
    It remains to check Part (1) of Definition \ref{def:Q0}.
    We denote by $N$ the homogeneous order of $P$ and by $\sigma_N\in \Gamma(\sU_N(\fg M))$  its principal part  in the sense of Section \ref{subsec:diff_op}. 
 
         Let $\bX$ be an adapted frame  on an open set $U\subset M$ and let $\chi$ be a $\bX$-cut-off.
The operator $\Op^{\bX, \chi}(\sigma_N)$ is differential by 
         Proposition \ref{ex_OpT} 
         and in $\Psi^N(M)$
by Lemma \ref{lem:local-to-global}.
Moreover, Proposition \ref{ex_OpT}  implies that $P|_U-\Op^{\bX, \chi}(\sigma_N)$ is a differential operator on $U$ of homogeneous order $\leq N-1$.
Inductively, we construct a finite sequence of symbols 
$$
\sigma_N,\, \sigma_{N-1},\ldots,\, \sigma_0\quad
         \mbox{respectively in} \  \Gamma(\sU_N(\fg M)),\ \Gamma(\sU_{N-1}(\fg M)), \ldots, \ \Gamma(\sU_0(\fg M)),
         $$
         such that the differential operator $P$ coincides on $U$ with 
         $\Op^{\bX, \chi}(\sigma_N +\sigma_{N-1}+\ldots +\sigma_0).
         $
This implies Part (2) of Definition \ref{def:Q0} and concludes the proof.        
\end{proof}

\subsubsection{Characterisation via atlases}
Remark \ref{remlem:local-to-global} above
 leads us  to adapt the notion of atlas to the setting of filtered manifold in the following way:

\begin{definition}\label{def:atlas}
The collection    $\mathcal A= (U_\alpha, \bX_\alpha, \chi_\alpha,\psi_\alpha)_{\alpha\in A}$ is   an  {\it atlas} of the filtered manifold~$M$ when it consists in 
    \begin{itemize}
        \item[(a)] a countable family of open subsets $U_\alpha$ providing a locally finite covering of $M = \cup_\alpha U_\alpha$, and
        \item[(b)] for each $\alpha\in A$, an adapted frame $\bX_\alpha$ on $U_\alpha$, a  $\bX_\alpha$-cut-off $\chi_\alpha$
and a function $\psi_\alpha\in C_c^\infty(U_\alpha:[0,1])$ 
\textcolor{blue}{
such that the functions
$(x,y)\mapsto \psi_\alpha(x)\psi_\alpha(y)$ are supported in the interior of  the set $\{\chi_\alpha(x,y)=1\}$, and satisfying
$\displaystyle{
\sum_\alpha \psi_\alpha^2=1}$.
}
    \end{itemize}
\end{definition}
It is possible to construct such atlases via 
routine arguments ommited here.

This notion of atlas on $M$ allows us to characterise the pseudodifferential operators on~$M$ in the following way:
\begin{proposition}
\label{caracterization_psi(M)}
Let $P$ be a continuous linear operator  $C_c^\infty(M)\to \mathcal{D}'(M).
    $ Then $P\in \Psi^m(M)$ if and only if 
    there exists one (and then any) atlas $\cA= (U_\alpha, \bX_\alpha, \chi_\alpha,\psi_\alpha)_{\alpha\in A}$ for the filtered manifold $M$, a family of symbols $(\sigma_\alpha)_{\alpha\in A}$ with  $\sigma_\alpha\in S^m(\Gh M|_{U_\alpha})$ and a smoothing operator~$R\in\Psi^{-\infty}(M)$  such that 
    \begin{equation}
        \label{eq:global-char}
           P=R+ \sum_{\alpha}\psi_\alpha\, \Op^{\bX_\alpha, \chi_\alpha}(\sigma_\alpha) \, \psi_\alpha.
    \end{equation}
\end{proposition}
The equality in \eqref{eq:global-char} above means that for any $f\in C_c^\infty (M)$ and $x\in M$, we have:
$$
Pf(x)=Rf (x) + \sum_{\alpha}\psi_\alpha (x) \Op^{\bX_\alpha, \chi_\alpha}(\sigma_\alpha)(\psi_\alpha f)(x).
    $$
    We will see in Corollary \ref{corcorthm:princ_symbol} that we can in fact consider one symbol $\sigma$ instead of a family $(\sigma_\alpha)_{\alpha\in A}$.
In any case, the family of symbols $(\sigma_\alpha)_{\alpha\in A}$  and the smoothing operator~$R$ depend on the atlas $\cA$. However, even if we fix the atlas $\cA$, they are not unique. 
Indeed,  
we may modify a symbol $\sigma_\alpha $
in \eqref{eq:global-char} 
with a smoothing symbol on $U_\alpha$.
Moreover,  we observe that the  integral kernel of $\psi_\alpha\, \Op^{\bX_\alpha, \chi_\alpha}(\sigma_\alpha) \, \psi_\alpha$ is supported near the diagonal, 
more precisely in  $\supp \chi_\alpha \cap \supp\left( (x,y)\mapsto \psi_\alpha(x)\psi_\alpha (y)\right)$. 
As the integral kernel of $P$ may not be supported near the diagonal, the presence of 
the smoothing operator $R$ is essential for the characterisation of $\Psi(M)$ in terms of the local quantization.

\begin{proof}[Proof of Proposition \ref{caracterization_psi(M)}]
If $P$ is of the form \eqref{eq:global-char}, then $P\in \Psi^m(M)$ by Lemma \ref{lem:local-to-global}.  Let us prove the converse. Let $P\in \Psi^m(M)$.
    We denote by $\sigma_\alpha$ the symbol in $S^m(\Gh M|_{U_\alpha})$ obtained from Part (1) of Definition  \ref{def:Q0} for the open set $U_\alpha$. 
    We start the analysis with writing $P$ as
$$
P=
\sum_{\alpha,\alpha'} \psi_{\alpha'}^2 P \psi_\alpha ^2
=
\sum_{\alpha'} \psi_{\alpha'}^2 P \varphi_{\alpha'}
\ + \ R_1
$$
with 
$$
R_1 := \sum_{\substack{\alpha,\alpha'\\ U_\alpha\cap U_{\alpha'}= \emptyset}} \psi_{\alpha'}^2 P \psi_\alpha ^2, 
\qquad\mbox{and}\qquad  \varphi_{\alpha'}:=
\sum_{\alpha:U_\alpha\cap U_{\alpha'}\neq \emptyset} \psi_\alpha ^2 \  \in \ C^\infty_c (M).
$$
Part (2) of Definition  \ref{def:Q0} implies that $R_1$ is smoothing.
We now write the sum over $\alpha'$ as 
$$
\sum_{\alpha'} \psi_{\alpha'}^2 P \varphi_{\alpha'}
=
\sum_{\alpha'} \psi_{\alpha'}^2 P \psi_{\alpha'}^2\varphi_{\alpha'}
+R_2
$$
 where $R_2$ is the smoothing operator given by:
$$
R_2 :=\sum_{\alpha'}\psi_{\alpha'}^2 P (1-\psi_{\alpha'}^2)\varphi_{\alpha'}  .
 $$
 We check readily that the symbol
$$
\sigma'_{\alpha'} := \sum_{\alpha,\alpha'} \psi_{\alpha'} \ (\sigma_{\alpha'} \diamond _{\bX_{\alpha'},\chi_{\alpha'}} \psi_\alpha  \varphi_{\alpha'} ) 
$$
 is in $S^m(\Gh M|_{U_\alpha})$ and that we have
\begin{align*}
 \sum_{\alpha'}   \psi_{\alpha'}\,  
 \Op^{\bX_{\alpha'}, \chi_{\alpha'}} (\sigma'_{\alpha'})\,\psi_{\alpha'} 
 &=\sum_{\alpha'}   \psi_{\alpha'}^2\,  
 \Op^{\bX_{\alpha'}, \chi_{\alpha'}} (\sigma_{\alpha'})\,\psi_{\alpha'} ^2 \varphi_{\alpha'}
 =\sum_{\alpha'} \psi_{\alpha'}^2 P \psi_{\alpha'}^2\varphi_{\alpha'},
 \end{align*}
 concluding the proof.
\end{proof}

This characterisation together with the properties of the local quantization  imply that pseudodifferential operators of non-negative order act continuously on $L^2_{loc}(M)$:
\begin{corollary}
\label{cor_charPsiM}
 If $P\in \Psi^m(M)$ with $m\in (-\infty,0]\cup\{-\infty\}$, then $P$ acts continuously on the space  $L^2_{loc}(M)$ (see~\eqref{def:L2loc} for a precise definition).
\end{corollary}

\begin{proof}
Consider the case of  $P\in \Psi^{-\infty}(M)$ smoothing. 
This means that its integral kernel $K(x,y)$ is a smooth function in $x,y\in M$.
This implies that $P$ acts continuously on $L^2_{loc}(M)$ for instance by the Cauchy-Schwartz inequality.
The case of $m\in (-\infty,0]$  follows from 
 the characterisation in Proposition \ref{caracterization_psi(M)} together with the $L^2$-boundedness (see Theorem \ref{prop_L2bddX}).
\end{proof}

\subsection{Properly supported pseudodifferential operators and composition}
\label{sec:properly_supported}
Already in the traditional case of a trivial filtration on a manifold $M$, 
two pseudodifferential operators may not be composable as they are maps between functional spaces which are not compatible with composition;
for instance, one of them is an operator $C_c^\infty(M)\to C^\infty (M)$ while the other is  $\sE'(M)\to \cD'(M)$.
The same technical obstruction appears here, and we adapt the usual line of approach to solve this issue by   considering properly supported pseudodifferential operators.

\subsubsection{Properly supported pseudodifferential operators.}

\begin{definition}
\label{def_properlysupp}
A pseudodifferential operator $P$ is said to be \emph{properly supported} when its integral kernel $K\in \cD'(M\times M)$ is properly supported or, in other words, when the two projection maps $\supp \, K \to M$ are proper (i.e. preimages of compact sets are compact).
\end{definition}

With Definition \ref{def_properlysupp} in mind
together with Lemma \ref{lem_adjPsi}, it follows readily that a properly supported pseudodifferential operator on $M$ acts on $C_c^\infty(M)$:
\begin{lemma}
\label{lem_properlysuppPDOacts}
  If $P$ is a properly supported pseudodifferential operator on $M$, then $P$ is a linear continuous operator $C_c^\infty(M)\to C_c^\infty(M)$, $\cD'(M)\to \cD'(M)$ and $\sE'(M)\to \sE'(M)$.
  If~$P$ is moreover smoothing, then $P$  is also a linear continuous operator $\sE'(M)\to C_c^\infty(M)$.
\end{lemma}
Consequently, properly supported pseudodifferential operators on $M$ are composable with other pseudifferential operators, and we will study their composition soon.
 However, one drawback of considering properly supported pseudodifferential operators is that,
 by definition, their integral kernels are supported close to the diagonal. Hence, unless the manifold is compact, they comprise a strictly smaller subspace of the space of pseudodifferential operators.
\smallskip

Here are some fundamental examples of properly supported pseudodifferential operators:
\begin{ex}\label{ex:prop_supp}
\begin{enumerate}
\item If the manifold $M$ is compact, then any pseudodifferential operator is properly supported. 
    \item   If $P\in \Psi^m(M)$ and if $\varphi,\psi\in C_c^\infty(M)$, then 
    $\psi P \varphi\in \Psi^m(M)$ is  properly supported.
    \item 
    Whether $M$ is compact or not, 
any differential operator is a properly supported pseudodifferential operator by Lemma \ref{lem_exdiffPDO}.
    \item Considering an 
    atlas $\cA= (U_\alpha, \bX_\alpha, \chi_\alpha,\psi_\alpha)_{\alpha\in A}$ for the filtered manifold $M$ and a family of symbols $(\sigma_\alpha)_{\alpha\in A}$ with $\sigma_\alpha\in S^m(\Gh M|_{U_\alpha})$, then the pseudodifferential operator
    $$
    \sum_{\alpha}\psi_\alpha\, \Op^{\bX_\alpha, \chi_\alpha}(\sigma_\alpha) \, \psi_\alpha \in \Psi^m(M)
    $$
    is properly supported.
\end{enumerate}
\end{ex}

As a consequence of the last example and the characterisation of pseudodifferential operators on $M$ in 
Proposition \ref{caracterization_psi(M)}, any pseudodifferential operator on $M$ is properly supported modulo a smoothing operator:

\begin{lemma}
\label{lem_P=ps+smooting}
For any $P\in \Psi^m(M)$, there exists a properly supported pseudodifferential operator $P'\in \Psi^m(M)$ and a smoothing operator $R\in \Psi^{-\infty}(M)$ such that $P=P'+R$.
\end{lemma}

\subsubsection{Composition of pseudodifferential operators}
We can now consider the composition of pseudodifferential operators on the filtered manifold $M$:
\begin{theorem}
\label{thm_compPsiM}
Let $m_1,m_2\in \bR\cup \{-\infty\}$.
Let $P_1\in \Psi^{m_1}(M)$ and $P_2\in \Psi^{m_2}(M)$. 
If one of them is properly supported then their composition $P_1P_2$ makes sense and is a pseudodifferential operator $P_1P_2 \in \Psi^{m_1+m_2}(M)$.
If they are both properly supported then so is $P_1P_2$.
\end{theorem}

In particular, on a compact filtered manifold $M$,  $\Psi^{m_1}(M)\circ \Psi^{m_2}(M)\subset \Psi^{m_1+m_2}(M)$.

The proof will use the following observations. 
\begin{lemma}
\label{lem_adjPsiMps}
\begin{enumerate}
 \item If $P\in \Psi^m(M)$ with $m\in \bR\cup\{-\infty\}$ is properly supported, then $P^*\in \Psi^m(M)$ is also properly supported. 
 \item Let $P\in \Psi^m(M)$ with   
\[
P=R+ \sum_{\alpha}\psi_\alpha\, \Op^{\bX_\alpha, \chi_\alpha}(\sigma_\alpha) \, \psi_\alpha.
\]  
for $R\in\Psi^{-\infty} (M)$, an atlas $\cA= (U_\alpha, \bX_\alpha, \chi_\alpha,\psi_\alpha)_{\alpha\in A}$ and a family of symbols $(\sigma_\alpha)_{\alpha\in A}$ with  $\sigma_\alpha\in S^m(\Gh M|_{U_\alpha})$. If $P$ is properly supported, so it is for  $R$.
 \end{enumerate}
\end{lemma}

\begin{proof}[Proof of Lemma \ref{lem_adjPsiMps}]
 Part (1)  follows readily from Lemma \ref{lem_adjPsi} and its proof, while Part (2) comes from Example~\ref{ex:prop_supp} (2).   
\end{proof}

\begin{proof}[Proof of Theorem \ref{thm_compPsiM}]
Lemmata \ref{lem_adjPsi} and \ref{lem_adjPsiMps} (1)  allow us to assume $P_2$ properly supported.  
Then $P_1 P_2$ makes sense and is a continuous linear operator $C_c^\infty (M)\to C^\infty (M)$. 
Denoting by $K,K_1,K_2$ the integral kernels of $P_2P_1$, $P_1$ and $P_2$, we have in the sense of distribution:
\begin{equation}
    \label{eq:pfthm_compPsiM_K}
K(x,y)=\int_M K_1(x,z)K_2(z,y) dz.
\end{equation}

 Lemma~\ref{lem_adjPsiMps} (2) and  the vector space structure of $\Psi(M)$ imply that it suffices to consider the  
three following cases: 
\begin{itemize}
    \item[(a)]  $P_2\in \Psi^{-\infty}(M)$ (i.e. $K_2$ smooth), 
    \item[(b)]  $P_1\in \Psi^{-\infty}(M)$ (i.e. $K_1$ smooth) and $P_2$ given by a local quantization, i.e. of the form $P_2 =\psi\, \Op^{\bX, \chi}(\sigma) \, \psi$,   and 
    \item[(c)] when both $P_1$ and $P_2$ are properly supported pseudodifferential operators and given by a local quantization.
\end{itemize}
Let us analyse these three cases.

(a) If $P_2$ is  also smoothing (in addition of being properly supported), then 
$K_2$ is smooth and properly supported, and for fixed $y\in M$, $K$ is the image of the compactly supported function $z\mapsto K(z,y)$ by $P_1$, and thus smooth since $P_1$ is a pseudodifferential operator. 
In this case, we obtain   $P_1P_2\in \Psi^{-\infty}(M)$.

(b) If  $P_1\in \Psi^{-\infty}(M)$ 
and 
\[P_2 =\psi\, \Op^{\bX, \chi}(\sigma) \, \psi,\] 
where $\bX$ is an adapted frame on an open subset $U\subset M$, $\chi$ is a $\bX$-cut-off,  $\psi\in C_c^\infty (U)$ and $\sigma\in S^m(\Gh M|_{U})$.
Denoting by $\tilde K_2$ the integral kernel of $\Op^{\bX, \chi}(\sigma)$, 
we have
in the sense of distribution
$$
K_2(z,y)= \psi(z) \tilde K_2(z,y) \psi(y).
$$
Hence, the expression in $\eqref{eq:pfthm_compPsiM_K}$ shows that $K$ is smooth. 
In this case, again,   $P_1P_2\in \Psi^{-\infty}(M)$.

(c)
We now consider the case of $m_1,m_2$ finite and $P_1,P_2$, given for $j=1,2$ by 
\begin{equation}
\label{eq:pfthm_compPsiM_casePjsymbol}
P_j =\psi_j\, \Op^{\bX_j, \chi_j}(\sigma_j) \, \psi_j,
\end{equation}
where $\bX_j$ is an adapted frame on an open subset $U_j\subset M$, $\chi_j$ is a $\bX_j$-cut-off,  $\psi_j\in C_c^\infty (U_j)$ and $\sigma_j\in S^m(\Gh M|_{U_j})$.
If $\psi_1\psi_2=0$,  then $P_1P_2=0$. 
Hence, we may assume $\psi_1\psi_2\neq 0$
and we fix a function $\tilde \psi \in C_c^\infty (U_1\cap U_2)$ such that $\tilde \psi=1$ on \textcolor{blue}{a neighbourhood of} $\supp\, (\psi_1\psi_2)$.

Let us show that we can write 
\begin{equation}
    \label{eqpfpfthm_compPsiM_sigma1+rho}
    \psi_1
\Op^{\bX_1, \chi_1}(\sigma_1) \tilde \psi^2
 =
  \psi_1\Op^{\bX_2, \chi_2}(\tilde \psi \sigma_1) \tilde \psi \ + \ 
   \psi_1\Op^{\bX_2, \chi_2}(\rho) \tilde \psi, 
\end{equation}
for some $\rho \in S^{m_1-1}(\Gh M)$.
By the symbolic properties of the local quantization (see Proposition \ref{prop_comploc}), 
there exists $\rho_1 \in S^{m_1-1}(\Gh M)$
$$
\psi_1
\Op^{\bX_1, \chi_1}(\sigma_1) \tilde \psi^2
= \psi_1\Op^{\bX_1, \chi_1}(\tilde \psi\sigma_1 + \rho_1) \tilde \psi
$$
By the properties of invariance of the local quantization for  the frame and the cut-off (see Propositions~\ref{prop_indepframe} and~\ref{prop_Qindepchi}), 
there exists $\rho'_1\in S^{m_1-1}(\Gh M)$ such that
$$
 \psi_1 \Op^{\bX_1, \chi_1}(\tilde \psi \sigma_1 ) \tilde \psi
 =
  \psi_1\Op^{\bX_2, \chi_2}(\tilde \psi \sigma_1+\rho'_1) \tilde \psi.
$$
This implies \eqref{eqpfpfthm_compPsiM_sigma1+rho}.

Using  \eqref{eqpfpfthm_compPsiM_sigma1+rho} in the expression for $P_1P_2$, we obtain
\begin{align*}
  P_1P_2&= \psi_1\, \Op^{\bX_1, \chi_1}(\sigma_1) \, \tilde \psi^2\psi_1
\psi_2\, \Op^{\bX_2, \chi_2}(\sigma_2) \, \psi_2
\\ & = \left(\psi_1\Op^{\bX_2, \chi_2}(\tilde \psi \sigma_1) \tilde \psi \ + \ 
   \psi_1\Op^{\bX_2, \chi_2}(\rho) \tilde \psi\right) \psi_1
\psi_2\, \Op^{\bX_2, \chi_2}(\tilde \psi\sigma_2) \, \psi_2 \\
&= \psi_1\, \Op^{\bX_2, \chi_2}( (\tilde \psi\sigma_1) \diamond_{\bX_2,\chi_2}  (\psi_1
\psi_2\sigma_2)) \, \psi_2  
+
\psi_1\, \Op^{\bX_2, \chi_2}(\sigma_1 \diamond_{\bX_2,\chi_2}   \psi_1
\psi_2 \rho) \, \psi_2.
\end{align*}

By the properties of composition in Proposition \ref{prop_comploc} and Lemma \ref{lem:local-to-global}, 
the first term in the sum of the right-hand side is in $\Psi^{m_1+m_2}(M)$ while the second is in $\Psi^{m_1+m_2-1}(M)$.
This implies the case of two operators $P_1,P_2$ of the form \eqref{eq:pfthm_compPsiM_casePjsymbol}.

The case where they are both properly supported follows readily from \eqref{eq:pfthm_compPsiM_K}.
\end{proof}

\subsubsection{The global quantization $\Op^\cA$}
\label{subssubsec_OpA}

\begin{lemma}
\label{lem_OpA}
  We fix  an 
    atlas $\cA= (U_\alpha, \bX_\alpha, \chi_\alpha,\psi_\alpha)_{\alpha\in A}$ for the filtered manifold $M$.
    The operators defined via 
   \begin{equation}
\label{eq:OpAsigma}
\Op^\cA(\sigma):=\sum_{\alpha}\psi_\alpha\, \Op^{\bX_\alpha, \chi_\alpha}(\sigma) \, \psi_\alpha, \quad \sigma\in \cup_{m\in \bR\cup \{-\infty\}} S^m(\Gh M),
\end{equation}
are properly supported pseudodifferential. 
More precisely, we have for any $m\in \bR\cup\{-\infty\}$,
$$
\forall \sigma\in S^m (\Gh M),\qquad 
\Op^\cA(\sigma) \in \Psi^m (M).
$$
Moreover, it satisfies the following two properties:
\begin{enumerate}
    \item If $\sigma\in \cup_{m'\in \bR\cup\{-\infty\}} S^{m'}(\Gh M)$ satisfies $\Op^\cA(\sigma) \in \Psi^m (M)$, then $\sigma\in S^m (\Gh M)$.
    \item If $\sigma$ and $\sigma'$ are two symbols in $S^m (\Gh M)$ such that $\Op^\cA(\sigma)=\Op^\cA(\sigma')$, then $\sigma-\sigma'\in S^{-\infty}(\Gh M)$ is smoothing.
\end{enumerate}
\end{lemma}
\begin{proof}
From Example \ref{ex:prop_supp} (4), it follows readily that $\Op^\cA(\sigma) \in \Psi^m (M)$ for any $\sigma\in S^m (\Gh M)$.
Part (1) is proved by proceeding as in the proof of Lemma \ref{lem:local-to-global} (2).
Part (2) follows.
\end{proof}

The class of operators $\Op^\cA (S^m (\Gh M))$, $m\in \bR$, enjoy a symbolic calculus modulo lower order terms as 
a consequence of the proof of Theorem \ref{thm_compPsiM} and the consideration on the adjoint. 

\begin{corollary}
\label{cor:OpA}
Let $\cA= (U_\alpha, \bX_\alpha, \chi_\alpha,\psi_\alpha)_{\alpha\in A}$ be an atlas for $M$ and let $\Op^\cA$ be its associated global quantization given via \eqref{eq:OpAsigma}. Let $m_1,m_2,m\in \bR\cup\{-\infty\}$.
\begin{enumerate}
\item 
If $\sigma\in S^{m}(\Gh M) $
then 
$$
    \left(\Op^\cA (\sigma)\right)^* = 
     \Op^\cA (\sigma^{(*)}) =  \Op^\cA (\sigma^{*}) +R,
    $$
    where $R\in \Psi^{m-1}(M)$.
    \item If $\sigma_1\in S^{m_1}(\Gh M) $ and $\sigma_2\in S^{m_2}(\Gh M)$, then 
    $$
    \Op^\cA (\sigma_1)\Op^\cA (\sigma_2) = 
     \Op^\cA (\sigma_1\sigma_2) +S, 
    $$
    where $S\in \Psi^{m_1+m_2-1}(M)$.
\end{enumerate}
\end{corollary}

\begin{proof}[Proof of Corollary \ref{cor:OpA}]
We check readily the properties of the adjoint. 
For the composition, we first observe that continuing with the proof of  Theorem \ref{thm_compPsiM} Case (c), 
we have in $S^{m_1+m_2}(\Gh M|_{U_2})$
$$  (\tilde \psi\sigma_1) \diamond_{\bX_2,\chi_2}  (\psi_1
\psi_2\sigma_2) = 
 \psi_1^2 \psi_2 \sigma_1 \sigma_2 + \rho',
$$
for some $\rho'\in S^{m_1+m_2-1}(\Gh M|_{U_2})$.
Therefore, we have obtained:
$$
 \psi_1\, \Op^{\bX_1, \chi_1}(\sigma_1) \, \psi_1
\psi_2\, \Op^{\bX_2, \chi_2}(\sigma_2) \, \psi_2
=
\psi_2 \Op^{\bX_2, \chi_2}(\psi_1^2 \sigma_1\sigma_2 )\psi_2 + R_{\psi_1,\psi_2}
$$
where $R_{\psi_1,\psi_2}\in \Psi^{m_1+m_2-1}(M)$ is given by
$$
R_{\psi_1,\psi_2}=
\psi_1 \Op^{\bX_2, \chi_2}(\rho')\psi_2 + 
\psi_1\, \Op^{\bX_2, \chi_2}(\sigma_1 \diamond_{\bX_2,\chi_2}   \psi_1
\psi_2 \rho) \, \psi_2.
$$

Performing this for every $\psi_{\alpha_1},\psi_{\alpha_2}$, we have
\begin{align*}
   \Op^\cA (\sigma_1)\Op^\cA (\sigma_1) &= 
\sum_{\alpha_1,\alpha_2}
\psi_{\alpha_1}\, \Op^{\bX_{\alpha_1}, \chi_{\alpha_1}}(\sigma) \, \psi_{\alpha_1}
\psi_{\alpha_2}\, \Op^{\bX_{\alpha_2}, \chi_{\alpha_2}}(\sigma) \, \psi_{\alpha_2}
\\&=  \sum_{\alpha_2,\alpha_1} 
\psi_{\alpha_2} \Op^{\bX_2, \chi_2}( \psi_{\alpha_1}^2 \sigma_1\sigma_2 )\psi_{\alpha_2}
\ + \ S, 
\end{align*}
with $\Op^\cA (\sigma_1\sigma_2) $ for the first term since $\sum_{\alpha_1} 
 \psi_{\alpha_1}^2 =1$, and for the second term,
 $$
S:=\sum_{\alpha_1,\alpha_2}R_{\psi_{\alpha_1},\psi_{\alpha_2}} .$$
 We check readily $S\in \Psi^{m_1+m_2-1}(M)$ and 
 this concludes the proof.
\end{proof}

\subsection{Principal Symbol}\label{subsec:princ_symbol}

We start by defining a local notion of principal symbol, attached to open sets small enough to be included into open sets endowed with a local frame. Then, in a second step, we will define a global notion of principal symbol.

\subsubsection{Local principal symbol}

We   introduce the following notation:
\begin{definition}
\label{def:SymbmV0}
   Let $m\in \bR$, let $P\in \Psi^m(M)$ and let $V_0$ be an open subset of $M$ small enough to be included in an open set  equipped with adapted frames. 
   Then  we denote by 
   $$
   {\rm Symb}_{m,V_0}(P)\in S^m(\widehat{G}M|_{V_0})/S^{m-1}(\widehat{G}M|_{V_0})
   $$
the class of symbols modulo $S^{m-1}(\Gh M|_{V_0})$ of $\sigma|_{V_0}$:
   $$
   {\rm Symb}_{m,V_0}(P)
    := \sigma|_{V_0} \ \mbox{mod} \ S^{m-1}(\widehat{G}M|_{V_0}),
   $$
where $\sigma|_{V_0} \in S^{m}(\widehat{G}M|_{V_0})$ is the restriction to $V_0$ of a symbol $\sigma \in S^m (\Gh M)$
  satisfying
   \begin{equation}
    \label{eq_phiPphiOp}
    \phi P\phi=\Op^{\bX, \chi}(\phi \sigma)\phi,
\end{equation}
for some adapted frame $\bX$ on an open set $U$ with $V_0\subset U\subset M$,   $\bX$-cut-off $\chi$ and  function $\phi\in C_c^\infty(U)$ with $\phi=1$ on a neighbourhood of $V_0$.
\end{definition}

For $P\in\Psi^m(M)$, the  existence of symbols $\sigma$ as in Definition~\ref{def:SymbmV0}
 is guaranteed by (1) of Definition~\ref{def:Q0}. 
The fact that it leads to a well-defined and unique class of symbol
follows from the following property:

\begin{lemma}
\label{lem_symbolaboveopenset}
 Let $m\in \bR$ and  $P\in \Psi^m(M)$. 
Let $V$ be an open subset of an open set $U\subset M$ admitting adapted frames.
Let $\phi,\phi'\in C_c^\infty(U)$ be two functions  such that $\phi=\phi'=1$ on $\overline{V}$, let  $\bX,\bX'$  be two adapted frames on $U$, let $\chi,\chi'$ be two cut-offs for $\bX$ and $\bX'$.
If  $\sigma$ and $\sigma'$ are symbols in $S^m(\Gh M)$  satisfying 
\eqref{eq_phiPphiOp} for $\phi,\bX,\chi$ and $\phi',\bX',\chi'$ respectively, then
$(\phi'\phi)^2 \,  (\sigma-\sigma')
     \in S^{m-1}(\Gh M)$.
\end{lemma}

The proof of Lemma \ref{lem_symbolaboveopenset}
relies mainly on  the invariances of the local quantization stated in Section~\ref{sec:ind_cut-off}.

\begin{proof}[Proof of Lemma \ref{lem_symbolaboveopenset}]
We keep the notation of the statement and we have
$$
0 = \phi'\phi P \phi \phi' - \phi \phi' P \phi' \phi
=\phi'\phi \,  \Op^{\bX, \chi}( \sigma)\,  \phi \phi' - \phi \phi' \, \Op^{\bX', \chi'}(\sigma') \, \phi' \phi.
$$
By invariance of the local quantization in frames and cut-offs (see Section~\ref{sec:ind_cut-off}), we have 
$$
\Op^{\bX', \chi'}(\sigma')
 = \Op^{\bX, \chi}( \sigma' +\rho) ,
$$
for some $\rho\in S^{m-1}(\Gh M|_U)$.
The equality above implies
$$
 \Op^{\bX, \chi}\left ( \left(\phi'\phi \,  (\sigma-\sigma'-\rho)\right)\diamond_{\bX, \chi} (\phi \phi')\right) =0, 
 $$
 so 
by Lemma \ref{lem:local-to-global} (2) applied to the smoothing  operator 0, 
 $$
 \left(\phi'\phi \,  (\sigma-\sigma'-\rho)\right)\diamond_{\bX, \chi} (\phi \phi') \in S^{-\infty}(\Gh M|_U).
$$
Therefore, 
\begin{align*}
    0&=\left(\phi'\phi \,  (\sigma-\sigma'-\rho)\right)\diamond_{\bX, \chi} (\phi \phi') \ \mbox{mod} \ S^{m-1}(\Gh M|_U)\\
    &= (\phi'\phi)^2 \,  (\sigma-\sigma'-\rho)
     \ \mbox{mod} \ S^{m-1}(\Gh M|_U)\\
    &= (\phi'\phi)^2 \,  (\sigma-\sigma')
     \ \mbox{mod} \ S^{m-1}(\Gh M|_U),
 \end{align*}
by the symbolic properties of the local quantization. This concludes the proof.
\end{proof}

The local principal symbol enjoys the following properties:

\begin{lemma}
\label{lem_SymbmV0}
Let  $V_0$ be an open subset of $M$ small enough to be included in an open set  equipped with adapted frames. 
\begin{enumerate}
    \item Let $m\in \bR$. The  map 
$$
{\rm Symb}_{m,V_0}\colon P\longmapsto {\rm Symb}_{m,V_0}(P), 
\qquad 
\Psi^m (M)\longrightarrow S^{m}(\Gh M|_{V_0})/S^{m-1}(\Gh M|_{V_0}),
$$
is  linear and  vanishes on $\Psi^{m-1}(M)$.
\item Let $\sigma\in S^m(\Gh M)$ with $m\in \bR$. Let $\bX$ be an adapted frame on an open set $U\subset M$, let $\chi$ be a $\bX$-cut-off and let  $\psi\in C_c^\infty (U)$. Then 
$$
{\rm Symb}_{m,V_0} ( \Op^{\bX,\chi}(\sigma)\, \psi) = (\psi \sigma)|_{V_0} \ \mbox{mod} \ S^{m-1}(\widehat{G}M|_{V_0}).
$$
\item 
For any $P\in \Psi(M)$, we have
$$
{\rm Symb}_{m,V_0}(P^*)=\left({\rm Symb}_{m,V_0}(P)\right)^*.
$$
\end{enumerate}
\end{lemma}

\begin{proof}[Proof of Lemma \ref{lem_SymbmV0}]
By construction, 
Parts (1) and (3) hold and we have for Part (2), 
$$
    {\rm Symb}_{m,V_0} ( \Op^{\bX,\chi}(\sigma)\psi) = 
    \sigma \diamond \psi  |_{V_0} \ \mbox{mod} \ S^{m-1}(\widehat{G}M|_{V_0}).
 $$
By the symbolic calculus, we have
$$
\sigma \diamond \psi = \psi \sigma \ \mbox{mod} \ S^{m-1}(\widehat{G}M|_U).
$$
Part (2) follows.
\end{proof}

\subsubsection{Principal symbol} 
We can now define the notion of principal symbol. 
\begin{theorem}\label{thm:princ_symbol}
Let $m\in \bR$.
    Let $P\in \Psi^m(M)$. There exists a unique element  
    $$
    \princ_m(P)\in S^{m}(\Gh M)/S^{m-1}(\Gh M)
    $$ 
such that for any open subset $V_0$ of $M$ small enough to be included in an open set equipped with adapted frames, 
${\rm Symb}_{m,V_0}(P)$ is the class obtained by  restriction to $V_0$ of  any representative of $\princ_m(P)$, i.e.
\begin{align}
 \label{eq_defpincmP}
    {\rm Symb}_{m,V_0}(P) &=
    \sigma|_{V_0} \ \mbox{mod} \ S^{m-1}(\Gh M|_{V_0}) \\
    & \mbox{for any}\ \sigma \in S^m (\Gh M) \ \mbox{satisfying}\ 
\sigma \ \mbox{mod} \ S^{m-1} (\Gh M)=\princ_m(P) .   \nonumber
\end{align}
This defines a linear map 
$$
\princ_m :\Psi^m (M)\to S^{m}(\Gh M)/S^{m-1}(\Gh M) 
$$
with the following properties:
\begin{enumerate}
    \item The map $\princ_m$ is surjective.
    \item $\ker \princ_m = \Psi^{m-1}(M)$.
      \item The adjoint is naturally defined on $S^{m}(\Gh M)/S^{m-1}(\Gh M)$ and we have
$$
\forall P\in \Psi(M),
\qquad
\princ_m(P^*)=\left(\princ_m(P)\right)^*.
$$

\item When $m=N\in \bN_0$, 
    the  restriction of $\princ_m$ to the space $\DO^{\leq N}$ of differential operators of homogeneous order at most $N$ coincides with the map also denoted by $\princ_N$ and  defined in Section \ref{subsec:diff_op}:
    $$
     \princ_N|_{\DO^{\leq N}} =\princ_N:\DO^{\leq N} \longrightarrow \Gamma (\sU_N(\fg M)) 
     $$
\end{enumerate}
\end{theorem}

\begin{definition}\label{def:princ_symb_global}
Let $P\in\Psi^m(M)$.   The class of symbols $\princ_m(P)$ is called the {\it principal symbol} of $P$ at order $m$.
\end{definition}

The proof relies on the characterisation of pseudodifferential operators stated in Proposition~\ref{caracterization_psi(M)} and  the  properties of 
${\rm Symb}_{m,V_0}(P)$ given in Lemma~\ref{lem_SymbmV0}.

\begin{proof}[Proof of Theorem \ref{thm:princ_symbol}]
We fix an atlas $\cA= (U_\alpha, \bX_\alpha, \chi_\alpha,\psi_\alpha)_{\alpha\in A} $ on $M$.
Let $P\in \Psi^m(M)$.
By Proposition~\ref{caracterization_psi(M)},
 there exist
 a family of symbols $(\sigma_\alpha)_{\alpha\in A}$ with $\sigma\in S^m(\Gh M|_{U_\alpha})$, and a smoothing operator~$R\in\Psi^{-\infty}(M)$  such that \eqref{eq:global-char} holds.
 We check readily that the symbol given via
 $$
 \sigma_P := \sum_\alpha \psi_\alpha^2 \sigma_\alpha 
 $$
is well defined and in $S^m (\Gh M)$. Moreover, by Lemma \ref{lem_SymbmV0} (1) and (2), it satisfies for any open subset $V_0\subset M$ small enough, 
$$
{\rm Symb}_{m,V_0} (P) = \sigma_P|_{V_0}  \  \mbox{mod} \ S^{m-1}(\Gh M|_{V_0}).
$$
Consequently, $\sigma_P$ represents a class satisfying \eqref{eq_defpincmP} for $P$.

 By definition,    
any two symbols in 
$S^{m}(\Gh M)$ satisfying \eqref{eq_defpincmP} for $P$ must coincide modulo $S^{m-1}(\Gh M|_{V_0})$ on any open subset $V_0\subset M$ small enough;
this shows the uniqueness of the class satisfying \eqref{eq_defpincmP}.
Hence $\princ_m$ is a well defined map.
Lemma \ref{lem_SymbmV0}  also implies that it is linear, with kernel containing $\Psi^{m-1}(M)$
and respecting the adjoint, i.e. Part~(3) holds.

Recall that for the above 
family $(\sigma_\alpha)_{\alpha\in A}$, $\sigma_\alpha$ may be obtained in the proof of Proposition \ref{caracterization_psi(M)}  from Part (1) of Definition  \ref{def:Q0} for the open set $U_\alpha$. 
Hence, if $P\in \ker \princ_m$,  
 then  by \eqref{eq_defpincmP}, $\sigma_\alpha\in S^{m-1}(\Gh M|_{U_\alpha})$, 
and \eqref{eq:global-char} implies  $P\in \Psi^{m-1}(M)$. This shows $\ker \princ_m = \Psi^{m-1}(M)$, i.e. Part (2) holds.

If $\sigma \in S^m(\Gh M)$, we check readily that 
$$
\princ_m (\Op^\cA(\sigma)) = \sigma \ \mbox{mod} \ S^{m-1}(\Gh M)
$$
where $\Op^\cA(\sigma)\in \Psi^m(M)$ is defined via \eqref{eq:OpAsigma}.
This shows the surjectivity of $\princ_m$, i.e. Part (1) holds.

Let us prove Part (4).
Let $P$ be a differential operator on $M$. 
Denote by $N$ its homogeneous degree.
Choose a symbol $\sigma$  representing the class in $S^N(\Gh M)$ for the principal part of $P$ in the sense of Section \ref{subsec:diff_op}.
The considerations in that section
(in fact, the very definition of principal part in Section \ref{subsec:diff_op}, see also the proof of  Lemma \ref{lem_exdiffPDO}) imply for any open subset $V_0\subset M$ small enough
$$
{\rm Symb}_{N,V_0}(P)
= \sigma|_{V_0}
\  \mbox{mod} \ S^{N-1}(\Gh M|_{V_0}).
$$
Consequently, 
$$
\princ_N P = \sigma  \ \mbox{mod} \ S^{N-1}(\Gh M).
$$
This shows Part (4), and concludes the proof of Theorem \ref{thm:princ_symbol}.
\end{proof}

Theorem \ref{thm:princ_symbol} and its proof imply readily the following relation between principal symbols and global quantization (see Section \ref{subssubsec_OpA} for the latter):
\begin{corollary}
\label{corthm:princ_symbol}
We fix an atlas 
$\cA= (U_\alpha, \bX_\alpha, \chi_\alpha,\psi_\alpha)_{\alpha\in A} $ for $M$.
 For any $P\in \Psi^m (M)$, 
if $\sigma\in S^m (\Gh M)$ is a representative of $\princ_m(P)$, then 
$$
P= \Op^\cA(\sigma) \ \mbox{mod} \ \Psi^{m-1}(M).
$$
\end{corollary}
Hence,  when considering principal symbols, that is, when working mod $\Psi^{m-1}(M)$, we may assume that the pseudodifferential operators are given by the global quantization $\Op^\cA(\sigma)$ of a symbol $\sigma\in S^m(\Gh M)$ on $M$.

When considering pseudodifferential operators modulo smoothing operators, we may assume that the pseudodifferential operators are given by the global quantization:
\begin{corollary}
\label{corcorthm:princ_symbol}
We fix an atlas 
$\cA= (U_\alpha, \bX_\alpha, \chi_\alpha,\psi_\alpha)_{\alpha\in A} $ for $M$.
Let $m\in \bR$.
 For any $P\in \Psi^m (M)$, there exist a symbol $\sigma\in S^m (\Gh M)$ and a smoothing operator $R\in \Psi^{-\infty}(M)$ such that 
$$
P= \Op^\cA(\sigma) +R.
$$
Moreover, the symbol $\sigma\in S^m (\Gh M) $ is unique modulo $S^{-\infty}(\Gh M).$
\end{corollary}
\begin{proof}
Using Corollary \ref{corthm:princ_symbol},
 we construct inductively $\sigma_0,\ldots,\sigma_N$ 
 in $S^m(\Gh M),\ldots,S^{m-N}(\Gh M)$
 such that 
 $$
 P - \Op^\cA(\sigma_0+\ldots+\sigma_N) \in \Psi^{m-N-1}(M),
 $$
 with $\sigma_{N+1}\in S^{m-(N+1)}(\Gh M)$ being the representative of the principal symbol of the  
 above operator.
Let $\sigma \in S^m(\Gh M)$ be a symbol with 
 asymptotic $\sigma_0+\ldots +\sigma_N+\ldots$.
By construction, 
$ P - \Op^\cA(\sigma)\in \Psi^{-\infty}(M).$
The uniqueness of the symbol modulo smoothing symbols follows from Lemma \ref {lem_OpA} (2).
\end{proof}

\subsubsection{Examples  of principal symbols}

Our main  examples of pseudodifferential operators in Section~\ref{sec:ex_global_pseudo}  were differential operators and operators obtained via local quantization (see 
Corollary \ref{corlem:local-to-global}). 
The principal parts of the former are described in Part (4) above.
For the latter, we have:
\begin{corollary}
\label{cor_princOp}
    For any $\bX$ on an open subset $U\subset M$ and $\bX$-cut-off $\chi$, we have for any $m\in \bR$ and $\sigma\in S^m(\Gh M)$, and for any $\psi\in C_c^\infty (U)$,
    $$
    \princ_m \left(\Op^{\bX,\chi}(\sigma)\, \psi\right) = \psi \sigma \ \mbox{mod} \ S^{m-1}(\widehat{G}M).
    $$
\end{corollary}
\begin{proof}
This follows from  Lemma \ref{lem_SymbmV0} (2)
and Theorem \ref{thm:princ_symbol}.
\end{proof}

\subsubsection{Principal symbols and composition}

The next proposition states that 
principal symbols  respect composition of operators and symbols modulo lower order terms. 
The sense of the composition of symbols modulo lower order terms comes from the following observation:
the properties of composition between symbols (Proposition~\ref{prop_comp+adj}) implies readily that  the following map 
$$
\left(\sigma_1 \ \mbox{mod} \ S^{m_1-1}(\Gh M),  \sigma_2 \ \mbox{mod} \ S^{m_2-1}(\Gh M) \right)
\longmapsto \sigma_1\sigma_2 \ \mbox{mod} \ S^{m_1+m_2-1}(\Gh M),
$$
is well defined and continuous $\frac{S^{m_1}(\Gh M)}{ S^{m_1-1}(\Gh M) }\times \frac{S^{m_2}(\Gh M)}{S^{m_2-1}(\Gh M)} \to \frac{S^{m_1+m_2}(\Gh M)}{ S^{m_1+m_2-1}(\Gh M)}$.
We have:
\begin{proposition}
    Let $m_1,m_2\in \bR.$
For any two pseudodifferential operators $P_1\in \Psi^{m_1}(M)$ and $P_2\in \Psi^{m_2}(M)$, at least one of them being properly supported, we have
$$
\princ_{m_1+m_2}(P_1P_2)=
\princ_{m_1}(P_1)\, \princ_{m_2}(P_2).
$$
\end{proposition}

\begin{proof}
By Corollary \ref{corthm:princ_symbol}, the property of composition of pseudodifferential operators (Theorem \ref{thm_compPsiM}), 
and the properties of $\princ_m$ in Theorem \ref{thm:princ_symbol}, it suffices to consider Case (c) in the proof of Theorem \ref{thm_compPsiM}. With its setting,  we have 
\begin{align*}
\princ_{m_1+m_2}  (P_1P_2)&= 
\princ_{m_1+m_2} \left( \psi_1\, \Op^{\bX_1, \chi_1}(\sigma_1 \diamond  (\psi_1
\psi_2\sigma_2)) \, \psi_2  
\right)\\
&= \psi_1 \psi_2 \left( \sigma_1 \diamond  (\psi_1
\psi_2\sigma_2) \right)\ \mbox{mod} \ S^{m_1+m_2-1}(\Gh M)
\end{align*}
by Corollary \ref{cor_princOp}.
By the properties of the symbolic calculus (Proposition~\ref{prop_comploc}), we have
$$
\psi_1 \psi_2 \left( \sigma_1 \diamond  (\psi_1
\psi_2\sigma_2) \right)
=\psi_1^2 \psi_2^2 \sigma_1 \sigma_2
\ \mbox{mod} \ S^{m_1+m_2-1}(\Gh M),
$$
concluding the proof.
\end{proof}

\subsection{Parametrices}\label{subsec:parametrices}

In this section, we discuss left parametrices,  following mainly the presentation in 
\cite{fischerMikkelsen}.
First let us recall the notion of left parametrix:
\begin{definition}
	An operator $Q$ admits a \emph{left parametrix} in $\Psi^m(M)$ 
	when $Q\in \Psi^m(M)$ and there exists $P\in \Psi^{-m}(M)$ properly supported such that 
	$PQ = \id$ mod $\Psi^{-\infty}(M)$.
\end{definition}
The existence of a left parametrix for an operator $A$ is an important quality as
it readily implies for instance hypoellipticity and subelliptic estimates in terms of the Sobolev spaces defined below.
We will show below that we are able to construct left parametrices  when principal symbols are invertible  in a sense we now define. 
First, let us recall the notion of invertibility for high frequencies on a fixed  graded nilpotent Lie group (see \cite[Section 4.6]{fischerMikkelsen}):
\begin{definition}
\label{def_symbinvG}
 Let $G$ be a graded nilpotent Lie group. A symbol $\sigma\in S^m (\Gh )$
is \emph{invertible  for the  frequencies of a positive  Rockland symbol $\widehat \cR$ higher than $\Lambda\in\bR$}
(or just {\it invertible for high frequencies})
 when 
    for any $\gamma\in \bR$ and almost-every $\pi\in \Gh$,   we have 
\begin{equation}
\label{eqdef_symbinvG}
\forall v\in \cH_{\pi,\cR,\Lambda}\qquad
    \|(\id+\pi(\cR))^{\frac{\gamma}\nu} \sigma(\pi)v
    \|_{\cH_\pi}
\geq C_\gamma \|(\id+\pi(\cR))^{\frac{\gamma+m}\nu} v
\|_{\cH_\pi}
\end{equation}
	with $C_{\gamma}=C_{\gamma,\sigma,G}>0$ a constant independent of $\pi$.
	Above, $\nu$ is the homogeneous degree of $\cR$ and  $\cH_{\pi,\cR,\Lambda}$ is the subspace $\cH_{\pi,\cR,\Lambda}:= E_\pi[\Lambda ,\infty)\cH_\pi$ of $\cH_\pi$ where $E_\pi$ is the  spectral decomposition of $
\pi(\cR) = \int_\bR \lambda dE_\pi(\lambda)$ as in \eqref{eq_spectralmeas}.
\end{definition}

\begin{remark}
\label{remdef_symbinvG}
    \begin{enumerate}
        \item The definition seems to require  a Plancherel measure to have been fixed on $\Gh$. However, the choice of a Plancherel measure is in one-to-one correspondence with the choice of a Haar measure, and two choices would differ by a constant. Hence, if a symbol is invertible for the high frequency of a Rockland operator for one Plancherel measure, it is so for any. 
        \item Another way of formulating the condition in \eqref{eqdef_symbinvG} is the fact that the operator 
$$
(\id+\pi(\cR))^{\frac{\gamma}\nu} 
\sigma(\pi)E_\pi[\Lambda,\infty) (\id+\pi(\cR))^{-\frac{\gamma+m}\nu} 
$$
is invertible with a bounded inverse whose $\cH_\pi$-operator norm is $\leq C_\gamma^{-1}$. In other words, with the notation in~\eqref{def:norm_introduction}
$$
\|E_\pi[\Lambda,\infty)\sigma(\pi)^{-1}\|_{L^\infty_{\gamma,\gamma+m} }(\Gh ) \leq  C_\gamma^{-1}.
$$
\item As explained in \cite[Section 4.6]{fischerMikkelsen}, 
Definition \ref{def_symbinvG} is the same notion but for invariant symbols as \cite[Definition 5.8.1]{R+F_monograph}, that was called ellipticity there. 
However, we prefer to use the vocabulary regarding  ellipticity for polyhomogeneous symbols, as is customary in the traditional case of a trivially filtered manifold. 
In any case, the properties explained in \cite[Section 8.1]{R+F_monograph} hold.
For instance, in Definition \ref{def_symbinvG}, it suffices to prove the estimate for a sequence of $\gamma=\gamma_\ell$, $\ell\in \bZ$,   with $\lim_{\ell\to \pm \infty} \gamma_\ell =\pm \infty$. 
From the examples given in \cite[Section 5.8]{R+F_monograph}, if $\widehat \cR$ is a positive Rockland symbol of  homogeneous degree $\nu$, then $\widehat\cR\in S^\nu(\Gh)$ and $(\id+\widehat\cR)^{m/\nu}\in S^m(\Gh)$, 
 are operators with invertible symbols for the high frequencies of $\widehat \cR$.
 \item Finally, let us point out that the invertibility we discuss here could be improved by localisation along `directions' of the set $\widehat GM$. Following~\cite{FF0}, these directions are defined via the equivalence relation given by the dilations: for $x\in M$, $\pi_1,\pi_2\in G_xM$, $\pi_1\sim\pi_2$ if and only if there exists $r>0$ such that for all $w\in G_xM$ we have $\pi_1(\delta_r w)=\pi_2(w)$. 
 Refined microlocalisations in conical neighbourhoods of the phase space $\widehat GM$ could be pursued with these ideas in mind, at least on concrete examples. 
    \end{enumerate}
\end{remark}

We can now define the invertibility of a symbol on a filtered manifold.

\begin{definition}
\label{def_symbinvM} 
A symbol $\sigma\in S^m (\Gh M)$ 
is \emph{invertible for the  high frequencies of a positive Rockland symbol $\widehat \cR$ }(see Definition~\ref{def:Rockland}) 
when at every $x\in M$, the symbol $\sigma(x)\in S^m(\Gh_x M)$ is invertible for frequencies higher than $\Lambda_x\in \bR$  
of the Rockland symbol $\widehat \cR_x$,
with both $\Lambda_x$ and the constant  $C_\gamma$
in \eqref{eqdef_symbinvG} depending locally uniformly in $x\in M$ (for one and then any smooth Haar system $\mu=\{\mu_x\}_{x\in M}$).
\end{definition}

\begin{ex}
\label{ex_(1+R)inv}
Let $\widehat \cR$ be a positive Rockland symbol. We denote its homogeneous degree $\nu$. For any $m\in \bR$, the symbol 
$(\id+\widehat\cR)^{m/\nu}\in S^m (\Gh M)$ is invertible    for all the frequencies of  $\widehat \cR$.
\end{ex}
We can now adapt the traditional construction of parametrices to our context under the following hypothesis:
\begin{theorem}
\label{thm:parametrix}
Let $Q\in \Psi^m(M)$ and $\sigma\in S^m(\Gh M)$ such that $\princ_m Q=\sigma$ mod $S^{m-1}(\Gh M)$.
If $\sigma$
 is invertible for the high frequencies of a positive Rockland symbol, then   $Q$ admits a left parametrix.
  \end{theorem}

The proof of Theorem \ref{thm:parametrix} is classical and relies on the following statement:
\begin{lemma}
\label{lem:psisigmainv}
 Let $\sigma\in S^m (\Gh M)$ be invertible for the high frequencies of a positive Rockland symbol $\widehat \cR$. 
 Let us fix $\psi\in C^\infty (M\times \bR^n)$ be such that 
    for every $x\in \bR$, we have with $\psi_x(\lambda):=\psi(x,\lambda)$, 
    \begin{itemize}
        \item $\psi_x(\lambda)=1$ on a neighbourhood of $+\infty$, 
        \item $\psi_x(\lambda)=0$ on a neighbourhood of $[0,\Lambda_x]$ for some $\Lambda_x$ as in Definition \ref{def_symbinvM}.
    \end{itemize} 
    The symbol 
    $$
    \psi(\widehat \cR)\sigma^{-1}=\{\psi_x(\pi(\cR_x)) \, \sigma(x,\pi)^{-1}\colon x\in M, \ \pi\in \Gh_x M\} 
    $$
    makes sense and is in $S^{-m}(\Gh M)$.
    Moreover, it satisfies 
   $$
     \psi(\widehat \cR)\sigma^{-1} \ \sigma = \psi(\widehat \cR) = \id + \rho,
     \quad\mbox{with} \ \rho:=(1-\psi(\widehat \cR))\in S^{-\infty}(\Gh M).
     $$
\end{lemma}

\begin{proof}[Proof of Lemma \ref{lem:psisigmainv}]
  By Corollary \ref{corprop:fct_de_R}, 
$\psi(\widehat \cR)\in S^0(\Gh M)$
and $(1-\psi)(\widehat \cR)\in S^{-\infty}(\Gh M)$.
Moreover, for any adapted frame $(\bX,U)$ and any $\beta\in \bN_0^n$ with $\beta\neq 0$, $D_\bX^\beta\psi(\widehat \cR) = - D_\bX^\beta(1-\psi)(\widehat \cR) \in S^{-\infty}(\Gh M|_U)$.   
By    \cite[Theorem 4.28]{fischerMikkelsen} (see also \cite[Theorem 5.8.7]{R+F_monograph}), for each $x\in M$, $\psi_x (\widehat \cR_x )\sigma(x)^{-1}$ makes sense as a symbol in $S^{-m}(\Gh_x M)$. It remains to analyse its $x$-derivatives. For this, we sketch the adaptation of  the  argument explained in the proof of \cite[Theorem~ 5.8.7]{R+F_monograph} to take into account our manifold setting and the presence of the smoothing terms $D_\bX^\beta \psi(\widehat \cR)$.

We consider an open set $U\subset M$ small enough so  that 
the symbol $\sigma(x)\in S^m(\Gh_x M)$ is invertible for frequencies higher than $\Lambda$ of the Rockland symbol $\widehat \cR_x$,
for some $\Lambda\in \bR$ and~$C_\gamma$ independent of $x\in U$ (see equation~\eqref{eqdef_symbinvG}).  Moreover, we  also assume that $\psi_x(\lambda)=1$ on  $(\Lambda-1,+\infty)$
and that $U$ is equipped with an adapted frame $\bX$. 
By  the Leibniz properties (see Remark~\ref{rem:LeibnizD}), we have (at first in a distributional sense for the $x$-derivatives  on the $\bX$-convolution kernels) with $|\beta|=1$
$$
D_\bX^\beta (\psi (\widehat \cR)\sigma^{-1})\, \sigma = 
D_\bX^\beta (\psi(\widehat \cR) ) - \psi(\widehat \cR) \sigma^{-1} D_\bX^\beta\sigma;
$$
multiplying both side by $E[\Lambda,+\infty)$ and using the invertibility of $\sigma E[\Lambda,+\infty) $, we obtain
$$
D_\bX^\beta (\psi (\widehat \cR)\sigma^{-1}) = 
\left(D_\bX^\beta (\psi (\widehat \cR)) - \psi (\widehat \cR)\sigma^{-1} D_\bX^\beta\sigma\right) E[\Lambda, +\infty)\sigma^{-1}.
$$
By Remark \ref{remdef_symbinvG} (2), we obtain
\begin{align*}
&\|D_\bX^\beta (\psi (\widehat \cR)\sigma^{-1})(x)\|_{L^\infty_{\gamma, m+\gamma} (\Gh_x M)} \\
&\qquad \leq C_{\gamma}^{-1}\left(
\|D_\bX^\beta (\psi(\widehat \cR))\|_{L^\infty_{m+\gamma, m+\gamma} (\Gh_x M)}   +\| \psi\|_{L^\infty} C_\gamma^{-1}
\|D_\bX^\beta\sigma \|_{L^\infty_{m+\gamma, \gamma } (\Gh_x M)} \right).
\end{align*}
More generally, 
by induction on $|\alpha|$ and $|\beta|$, we write 
$\Delta_\bX^\alpha D_\bX^\beta (\psi (\widehat \cR)\sigma^{-1})$ as 
as sum of terms involving $\Delta_\bX^{\alpha_1} D_\bX^{\beta_1} \psi(\widehat \cR)$ and $\Delta_\bX^{\alpha_2} D_\bX^{\beta_2} (\psi (\widehat \cR)\sigma^{-1})$ with $|\alpha_2|<|\alpha|$, $|\beta_2|<|\beta|$ and $|\alpha_2|+|\beta_2|\neq0$. 
A similar analysis as above implies then that $\Delta_\bX^\alpha D_\bX^\beta (\psi (\widehat \cR)\sigma^{-1}) \in L^\infty_{\gamma, m -[\alpha]+\gamma} (\Gh_x M)$ for every $x\in U.$ 

We have obtained $\psi(\widehat \cR)\sigma^{-1}\in S^{-m}(\Gh M)$ and the rest of the statement follows by the symbolic calculus.   
\end{proof}

\begin{proof}[Proof of Theorem \ref{thm:parametrix}]
We keep the notation of the statement and fix an atlas $\cA$ on $M$.
By Corollary \ref{corthm:princ_symbol}, 
we can write 
$$
Q=\Op^\cA(\sigma)+ S,
$$
for some $S\in \Psi^{m-1}(M)$.
The construction of  a function $\psi$ as in Lemma \ref{lem:psisigmainv} follows routine arguments.
By Lemma \ref{lem:psisigmainv}, $\psi (\widehat \cR)\sigma^{-1}\in S^{-m}(\Gh M) $ and we have by the symbolic property of the global quantization (see Corollary \ref{cor:OpA})
\begin{align*}
\Op^\cA(\psi(\widehat \cR)\sigma^{-1}) Q
&=\Op^\cA(\psi(\widehat \cR)\sigma^{-1})\Op^\cA(\sigma)+ \Op^\cA(\psi(\widehat \cR)\sigma^{-1}) S\\
&= \Op^\cA (\id) + R_{-1},  
\end{align*}
for some $R_{-1}\in \Psi^{-1}(M)$. Note that $\Op^\cA(\id)=\id$.
By Corollary~\ref{corthm:princ_symbol}, we may write $R_{-1}$ in the form
$$
R_{-1} = \Op^\cA(\rho_{-1}) + S_{-1},
$$
where $\rho_{-1}\in S^{-1}(\Gh M)$ is a symbol representing $\princ_{-1} R_{-1}$ and $S_{-1}\in \Psi^{-2}(M)$.
We will then have 
$$
\Op^{\cA}\left((1-\rho_{-1}) \ \psi(\widehat \cR)\sigma^{-1}\right) Q
= \id + R_{-2},    
$$
for some $R_{-2}\in \Psi^{-2}(M)$.
Inductively, we construct $\rho_{-1}, \ldots, \rho_{-N}$ in $S^{-1}(\Gh M), \ldots, S^{-N}(\Gh M)$ and $R_{-1}, \ldots R_{-N-1}$
in $\Psi^{-1}(\Gh M), \ldots, \Psi^{-N-1}(\Gh M)$
respectively, such that 
$$
\Op^{\cA}\left((1-\rho_{-N})\ldots (1-\rho_{-1}) \ \psi(\widehat \cR)\sigma^{-1}\right) Q
= \id + R_{-N-1}.   
$$
We observe that the formal expansion of the product 
 $(1-\rho_{-N})\ldots (1-\rho_{-1}) \ \psi (\widehat \cR)\sigma^{-1}$ as $N\to \infty$ provides an asymptotic in $S^{-m}(\Gh M)$, and we consider a symbol $\tau\in S^{-m}(\Gh M)$ with this asymptotic. 
 By construction, we have for any $N\in \bN$
 $$
 \tau = (1-\rho_{-N})\ldots (1-\rho_{-1}) \ \psi\sigma^{-1} \ \mbox{mod} \  S^{-m-N-1}
(\Gh M) .
$$
This implies with $P=\Op^\cA (\tau)$, 
$$
PQ = \id  \ \mbox{mod} \  \Psi^{-m-N-1}
(M) ,
$$
for any $N\in \bN$, concluding the proof.
\end{proof}

The proof above shows how to construct  parametrices locally or globally for the invertible symbol in Example \ref{ex_(1+R)inv}:
\begin{corollary}
\label{cor_parametrix}
Let $m\in \bR$, 
\begin{enumerate}
    \item  
For any atlas $\cA$ on $M$ and for any positive Rockland symbol $\widehat \cR$    on $GM$, 
there exists a symbol $\tau\in S^{-m}(\Gh M)$ such that we have:
$$
  \Op^\cA \left( \tau \right) \Op^\cA \left( (\id+\widehat \cR)^{m/\nu}\right) -\id  \in \Psi^{-\infty}(M),
 $$
 where $\nu$ denotes the homogeneous degree of $\cR$.

\item For any adapted frame $\bX$ on an open subset $U\subset M$, for any $\bX$-cut-off, 
for any positive Rockland symbol $\widehat \cR$  on $GM|_U$, there exists a symbol $\tau\in S^{-m}(\Gh M|_U)$ such that we have:
 $$
  \Op^{\bX,\chi}\left( \tau \right) \Op^{\bX,\chi} \left( (\id+\widehat\cR)^{m/\nu}\right) -\id  \in \Psi^{-\infty}(U),
 $$
 where $\nu$ denotes its homogeneous degree.
 \end{enumerate}
\end{corollary}
\begin{proof}
We consider $\tau$ obtained from the proof of Theorem \ref{thm:parametrix} 
for $\sigma=(\id+\cR)^{m/\nu}$ and $S=0$.
\end{proof}

\begin{remark}
\label{rem_existenceR}
    For Part (2) above, the existence of such positive Rockland symbol $\widehat \cR$ on $GM|_U$ is given by $\cR_\bX$ constructed in Example \ref{ex_RU}. 
    For Part (1) (with its notation), we can construct such a positive Rockland symbol by writing 
    $$
    \widehat \cR = \sum_{\alpha} \psi_\alpha(x)^2\widehat \cR_{\bX_\alpha}.
    $$
On a regular subRiemannian manifold, 
the natural global positive Rockland symbol is  the subLaplacian symbol, see Remark \ref{rem_subLaplacianSymbol}. 
\end{remark}

\subsection{Local Sobolev spaces}
\label{sec:sobolev}

Following the now classical ideas developed  on nilpotent Lie groups~\cite{FRSob,folland75},
we introduce   Sobolev-type spaces that take into account the filtration of the manifold $M$. Naturally, the spaces we define here will be  local in nature, unless the manifold $M$ is compact. An alternative construction was introduced in \cite{Dave+Haller}.

\subsubsection{Definition and first properties}

\begin{definition}[Local Sobolev spaces]\label{def:Sobolev}
Let $r\in \bR$.
 We denote by $L^2_{r,loc}(M)$ the space of 
     distributions $f\in \cD'(M)$ satisfying 
         $$
    \forall P\in \Psi^r (M),\qquad \forall \theta\in C_c^\infty (M),\qquad
    P (\theta f) \in L^2_{loc}(M).
    $$
 We call it the {\it local Sobolev space of order $r$} on the filtered manifold $M$. 
\end{definition}

This definition makes sense since 
pseudodifferential operators map $\sE'(M)$ to $\cD'(M)$, and $L^2_{loc}(M)$ is a subspace of $\cD'(M)$ (see its definition in \eqref{def:L2loc}).
The properties of $\Psi(M)$ then imply readily:

\begin{proposition}
\label{prop_L2locr1stprop} 
   \begin{enumerate}
        \item For any $r\in \bR$, the vector space $L^2_{loc,r}(M)$ is a $C^\infty(M)$-module: if $f\in L^2_{loc,r}(M)$ and $g\in C^\infty(M)$ then $fg\in L^2_{loc,r}(M)$.
        \item 
        We have $L^2_{loc,0}(M) = L^2_{loc}(M)$ and  the inclusion
$$
r_1\geq r_2 \Longrightarrow
L^2_{loc,r_1}(M) \subset L^2_{loc,r_2}(M).
$$
\item Let $r,m\in \bR$.  We have 
$$
\forall P\in \Psi^m (M),\quad 
\forall f\in L^2_{r,loc} (M),\quad
\forall\theta\in C_c^\infty (M),\qquad 
P(\theta f)\in L^2_{r-m,loc} (M).
$$
\item We have for any $r\in \bR$
$$
C^\infty(M)\subset L^2_{loc, r}(M),
$$
and for $P\in \Psi^{-\infty}(M)$,
$$
\forall f\in L^2_{r,loc} (M),\qquad
\forall\theta\in C_c^\infty (M),\qquad 
P(\theta f)\in C^\infty (M).
$$
\end{enumerate}
\end{proposition}
\begin{proof}
        Part (1) follows from $\Psi^m (M)$ being a vector space and a $C^\infty(M)$-module.
If $r=0$, taking $P=\id$ in Definition \ref{def:Sobolev}
implies $L^2_{loc,0} (M)\subset L^2_{loc}(M)$, while the reverse inclusion follows from Corollary \ref{cor_charPsiM}.    
    The  inclusion in Part (2) follows from $\Psi^{r_1}(M)\subset \Psi^{r_2}(M)$.
Part~(3) follows from the properties of properly supported and of pseudodifferential operators and of their composition 
(see Lemma \ref{lem_P=ps+smooting}
and Theorem \ref{thm_compPsiM}).
Part (4) follows from the basic properties of 
pseudodifferential operators, see Section \ref{subsubsec_defPDOM}.
\end{proof}

\subsubsection{Characterisation of local Sobolev spaces}
Local Sobolev spaces may be characterised  by the action of local or global pseudodifferential operators, or by the action of the local or global quantization of powers of positive Rockland operators: 

\begin{theorem}
\label{thm:charSob}
    Let $r\in \bR$ and let $f\in \cD'(M)$.
The following are equivalent:
\begin{itemize}
    \item[(i)] $f\in L^2_{r,loc}(M)$.
     \item[(ii)]  For any adapted frame $\bX$ on an open subset $U\subset M$, for any $\bX$-cut-off $\chi$, we have
        $$
        \forall \sigma\in S^r (\Gh M)
        ,\quad \forall \theta\in C_c^\infty (U),\qquad  \Op^{\bX,\chi}(\sigma)(\theta f) \in L^2(U).
        $$  
 \item[(iii)] For one (and then any) atlas $\cA$ for $M$, 
 we have
 $$
 \forall \sigma\in S^r (\Gh M),\quad \forall \theta\in C_c^\infty (M),\qquad 
  \Op^\cA(\sigma)(\theta f) \in L^2(M).
 $$    
        \item[(iv)] 
         For one (and then any) positive Rockland symbol $\widehat \cR$ on $GM$, $\nu$ being its homogeneous degree, for any adapted frame $\bX$ on an open subset $U\subset M$ and  any $\bX$-cut-off $\chi$, we have 
        $$
         \forall \theta\in C_c^\infty (U),\qquad  \Op^{\bX,\chi}\left( (\id+\widehat \cR)^{\frac r \nu}\right)(\theta f) \in L^2(U),
        $$  
        
        \item[(v)]
    For one (and then any) atlas $\cA$ for $M$, for one (and then any) positive Rockland symbol $\widehat \cR$ on $GM$,
  $\nu$ being its homogeneous degree,
  $$
 \forall \theta\in C_c^\infty (M),\qquad 
  \Op^\cA\left( (\id+\widehat \cR)^{\frac r \nu}\right)(\theta f) \in L^2(M).
$$
        \item[(vi)] For any adapted frame $\bX$ on an open subset $U\subset M$, for any $\bX$-cut-off $\chi$, 
        we have
        $$
         \forall \theta\in C_c^\infty (U),\qquad  \Op^{\bX,\chi}\left( (\id+\widehat \cR_U)^{\frac r \nu}\right)(\theta f) \in L^2(U),
        $$  
        for one (and then any) positive Rockland operator $\cR_U$  on $GM|_U$,
        $\nu$ being its homogeneous degree.
        \end{itemize}
\end{theorem}

The  local and global existence of positive Rockland symbols has been discussed  in Remark~\ref{rem_existenceR}.
 
\begin{proof}
 The implication (i) $\Rightarrow$ (ii)
follows from Corollary \ref{corlem:local-to-global}, 
and the converse by  Proposition~\ref{caracterization_psi(M)}.
Similarly,  (i) $\Rightarrow$ (iii)
follows from Corollary \ref{corlem:local-to-global}, 
and the converse from Corollary~\ref{corthm:princ_symbol}
and the inclusion in Proposition \ref{prop_L2locr1stprop}  (2).

 The implication (iii) $\Rightarrow$ (v) is easy. For the converse, we observe that, using the setting and notation of Corollary \ref{cor_parametrix} and writing:
 $$
 Q_m = \Op^\cA \left( (\id+\cR)^{m/\nu}\right)
  \quad\mbox{and}\quad \quad P_{-m}=
  \Op^\cA\left( \tau \right),
  $$
  the operator 
$R:= P_{-m}Q_m -\id  \in \Psi^{-\infty}(M)$ is smoothing  and properly supported.
 Hence, we may write any $P\in \Psi^m (M)$ as 
 $$
 P = PP_{-m} Q_m -PR, \quad \mbox{with}\quad PP_{-m} \in \Psi^0(M) \ \mbox{and}\ PR \in \Psi^{-\infty}(M).
 $$
This readily implies that if $f\in \cD'(M)$ is as in (v), then $P (\theta f) \in L^2_{loc}(M)$ for any $\theta\in C_c^\infty (M)$ by Corollary \ref{cor_charPsiM}.
We conclude that the equivalence (iii) $\Leftrightarrow$ (v) holds. 
Finally, the equivalence (ii) $\Leftrightarrow$ (iv) is proved in a similar way. 

Let us show that the property in (vi) is indeed independent of the positive Rockland operator $\cR_U$ on $GM|_U$. This will readily imply 
the equivalence (iv) $\Leftrightarrow$ (vi) because of the equivalence (ii) $\Leftrightarrow$ (iv) that is already proved.
If $\cR_1$ and $\cR_2$ are two positive Rockland operators on $GM|_U$ of homogeneous degree $\nu_1$ and $\nu_2$, 
then by Corollary \ref{cor_parametrix}, there exists $\tau_1\in S^{-r}(\Gh M|_U)$ and $R_1\in \Psi^{-\infty}(U)$ such that
$$
\Op^{\bX,\chi}(\tau_1)\, \Op^{\bX,\chi}\left( (\id+\cR_1)^{r/\nu_1}\right)=\id +R_1,
$$
so 
$$
\Op^{\bX,\chi}\left( (\id+\cR_2)^{r/\nu_2}\right)
= P_1 \, \Op^{\bX,\chi}\,(\id+\cR_1)^{r/\nu_1}  +  R', 
$$
with
\begin{align*}
P_1&:=     \Op^{\bX,\chi}\left( (\id+\cR_2)^{r/\nu_2}\right) \Op^{\bX,\chi}(\tau_1) \in \Psi^0(U)\\
R'&=
-R_1\Op^{\bX,\chi}\left( (\id+\cR_2)^{r/\nu_2}\right) \in \Psi^{-\infty}(U).
\end{align*}
This together with Proposition \ref{prop_L2locr1stprop}  (3) and (4) imply readily that for any $f\in \cD'(M)$ and $\theta\in C_c^\infty(U)$, if $\theta f$ satisfies the property in (vi) for $\cR_1$, then it will also do so for $\cR_2$. 
This complete the proof of Theorem \ref{thm:charSob}.
\end{proof}

In the case of an integer that is also a multiple of the dilation's weights, we obtain a further characterisation:

\begin{theorem}
\label{thm:Sobrint}
Let  $r\in \bN$ be a  common multiple of the dilation's weights. Let $f\in \cD'(M)$.
The following are equivalent:
\begin{itemize}
\item[(i)] $f\in L^2_{r,loc}(M)$.
    \item[(vii)] $P f\in L_{loc}^2(M)$ for any $P\in \DO^{\leq r}(M)$.
       \item[(viii)] 
For any adapted frame $\bX$ on an open subset $U\subset M$,  we have $\bX^\alpha f\in L^2_{loc}(U)$ for any $\alpha\in \bN_0^n$ with $[\alpha]\leq r$.
\end{itemize}
\end{theorem}
    
\begin{proof}
Let $r\in \bN$ be a common multiple of the dilation's weights.
Proposition \ref{prop_L2locr1stprop} (2) \& (3) 
and Definition \ref{def:Sobolev} imply readily 
(i) $\Rightarrow$ (vii) and (viii).
We now prove the converses.

Let $f\in \cD'(M)$ be such that (vii) holds.
Choosing $P=\id$ in (vii), we have $f\in L^2_{loc}(M)$.
We consider an adapted frame  $\bX$ on a subset $U\subset M$ and $\theta\in C_c^\infty(U)$, 
and the Rockland symbol $\widehat \cR_\bX \in S^{2r} (\Gh M)$ given in Example \ref{ex_RU}.
We have
by the properties of the local quantization (Propositions~\ref{prop_adjloc} and~\ref{prop_comploc}), 
$$
\left(\Op^{\bX,\chi} \left((\id+\widehat \cR_\bX)^{1/2}\right)\right)^*\Op^{\bX,\chi} \left((\id+\widehat \cR_\bX)^{1/2}\right) 
 = \Op^{\bX,\chi} \left(\id+\widehat \cR_\bX\right) + \Op^{\bX,\chi} (\rho),
$$
with $\rho\in S^{2r-1}(\Gh M|_U)$, and 
by Corollary \ref{cor_parametrix} (2),
$$
\Op^{\bX,\chi}(\tau)\, \Op^{\bX,\chi}\left( \id+\widehat \cR_\bX\right)=\id +R,
$$
with $\tau\in S^{-2r}(\Gh M|_U)$ and $R\in \Psi^{-\infty}(U)$, 
so 
$$
\left(\Op^{\bX,\chi} \left((\id+\widehat \cR_\bX)^{1/2}\right)\right)^*\Op^{\bX,\chi} \left((\id+\widehat \cR_\bX)^{1/2}\right) 
= P'\Op^{\bX,\chi}\left( \id+\widehat \cR_\bX\right) +R'',
$$
with 
\begin{align*}
    P'&:=\id + \Op^{\bX,\chi} (\rho)\Op^{\bX,\chi}(\tau) \in \Psi^0 (U),
    \\
    R''&:=-\Op^{\bX,\chi} (\rho)R \in \Psi^{-\infty}(U),
\end{align*}
by the properties of the calculus.
Therefore, for any $\theta\in C_c^\infty(U)$, 
\begin{align*}
 &\|\Op^{\bX,\chi} \left((\id+\widehat \cR_\bX)^{1/2}\right) (\theta f)\|_{L^2(U)}^2
\\ &\qquad =    \left( \left(\Op^{\bX,\chi} \left((\id+\widehat \cR_\bX)^{1/2}\right)\right)^*\Op^{\bX,\chi} \left((\id+\widehat \cR_\bX)^{1/2}\right) (\theta f),(\theta f)\right)_{L^2(U)}
\\ &\qquad = 
\left(P'\Op^{\bX,\chi}\left( \id+\widehat \cR_\bX\right)(\theta f),(\theta f)\right)_{L^2(U)}
+
\left(R''(\theta f),(\theta f)\right)_{L^2(U)}.
\end{align*}
As $\Op^{\bX,\chi}\left( \id+\widehat \cR_\bX\right)$ is a differential operator (see Lemma~\ref{lem_exdiffPDO}), the hypothesis on $f$ and Proposition \ref{prop_L2locr1stprop} (3)  imply that the first term on the right hand-side is finite. Moreover, Proposition~\ref{prop_L2locr1stprop} (4) and $f\in L^2_{loc}(M)$ imply that the second term in the right hand side is finite.
We have thus obtained 
$\Op^{\bX,\chi} \left((\id+\widehat \cR_\bX)^{1/2}\right) (\theta f)\in L^2(U)$.
By Theorem \ref{thm:charSob}  (vi), $f\in L^2_{loc,r}(M)$. We have proved   (vii) $\Rightarrow$ (i).

With Lemma \ref{lem_PulcXalpha},
the same arguments as above show 
(viii) $\Rightarrow$ (i).
This concludes the proof.
\end{proof}

The same proof as above readily implies the following characterisation in the case of 
$H_1$ generating $TM$ (for instance on a regular subRiemannian manifolds, see Example~\ref{ex:subrie} and Corollary \ref{corlem_PulcXalphasubRie}).
\begin{corollary}\label{cor:sobol_def}
Assume that $H_1$ generates $TM$.
Then for any $r\in \bN$, $L^2_{loc,r}(M)$ is the $C^\infty(M)$-module  of distributions $f\in \cD'(M)$ such that 
for any adapted frame $\bX$ on an open subset $U\subset M$, we have  $X_{i_1}\ldots X_{i_{r'}} f\in L^2_{loc}(U)$ for any $1\leq i_1,\ldots,i_{r'}\leq d_1=\dim H_1$ and $r'\leq r$.
\end{corollary}

For a general  filtered manifold, Theorem \ref{thm:Sobrint} implies readily 
 $$
 \forall r\in M_0 \bN, \qquad 
C^\infty (M)\subset L^2_{loc,r}(M) ,
$$
where $M_0$ is a common multiple of the dilation's weights, 
and furthermore, 
 $$
C^\infty(M)=\bigcap_{r\in M_0 \bN} L^2_{loc, r}(M) .
$$ 
The inclusion in Proposition \ref{prop_L2locr1stprop} (2) then yields 
 $$
\forall r\in \bR,\qquad 
C^\infty (M)\subset L^2_{loc,r}(M),
$$
 and
$$
C^\infty(M)=\bigcap_{r\in\bR} L^2_{loc, r}(M) .
$$

\begin{remark}
    Note that, in this special case where $H_1$ generates $TM$, one can use Corollary~\ref{cor:sobol_def} for defining Sobolev spaces in terms of adapted frames.
\end{remark}

\subsubsection{Topology}
We equip $L^2_{loc,r}(M)$ with the smallest topology ensuring the continuity of 
the map
$$
f\longmapsto P(\theta f), 
\qquad L^2_{loc,r} (M)\longrightarrow L^2_{loc}(M), 
$$
 for every $P\in \Psi^r (M)$
and $\theta\in C_c^\infty (M)$.
The arguments in the proofs of Proposition \ref{prop_L2locr1stprop} and
Theorems \ref{thm:charSob}
and \ref{thm:Sobrint} routinely
imply  the following properties:
\begin{lemma}
\label{lem_topoSob}
\begin{enumerate}
    \item  The topological vector space   $L^2_{loc,r}(M)$ is Fr\'echet. 
  Fixing an atlas $\cA$ for $M$,
  a positive Rockland symbol $\widehat \cR$ on $GM$, $\nu$ denoting its homogeneous degree,
  and a sequence
  $(\theta_j)_{j\in \bN}$ 
 in $C_c^\infty(M :[0,1])$ 
       such that 
       $M=\cup_{j \in \bN}\{\theta_j=1\} $, 
       a generating family of seminorms is given by
$$
f\longmapsto \max_{j\leq N}
\Bigl\|\Op^\cA\left( (\id+\widehat \cR)^{\frac r \nu}\right)(\theta_j f)\Bigr\|_{L^2(M)},
\qquad N\in \bN.
$$
\item  
The inclusion $L^2_{loc,r_1}(M) \subset L^2_{loc,r_2}(M)$ with
$r_1\geq r_2$
is continuous.
\item Let $r,m\in \bR$. Any 
$P\in \Psi^m (M)$ maps 
$L^2_{r,loc} (M)$ to $L^2_{r-m,loc} (M)$ continuously.
\item Assume that $r\in \bN$ is a  common multiple of the dilation's weights.
    Let $(\theta_j)_{j\in \bN}$ be a sequence of functions in $C_c^\infty(M)$ valued in $[0,1]$ such that 
    $\supp\, \theta_j$ is included in an open set~$U_j$ equipped with an adapted frame $\bX_j$
    and such that $M=\cup_{j \in \bN}\{\theta_j=1\} .$ 
Then the maps 
$$
f\longmapsto \|\bX^\alpha (\theta_j f)\|_{L^2(M)},\quad j\in \bN,\quad [\alpha]\leq r,
$$
are continuous on the Fr\'echet $L^2_{loc,r}(M)$, and  provide a countable family of seminorms generating its topology.

When $H_1$ generates $TM$, we may consider any integer $r\in \bN$ with 
the family of seminorms 
$$
f\longmapsto \|X_{i_1}\ldots X_{i_{r'}} (\theta_j f)\|_{L^2(M)},\quad j\in \bN,\quad 
1\leq i_1,\ldots,i_{r'}\leq d_1, \ r'\leq r.
$$
\end{enumerate}
\end{lemma}

\begin{corollary}
\label{corlem_topoSob}
\begin{enumerate}
\item 
The natural topology of 
$C^\infty(M)$ coincides with the one obtained as the inductive limit of 
 $$
C^\infty(M) =\bigcap_{r\in \bR} L^2_{loc, r}(M) .
$$   
\item 
For any $r\in \bR$, 
the inclusions
$C^\infty(M) \subset L^2_{loc,r}(M)$
and 
$C_c^\infty(M) \subset L^2_{loc,r}(M)$
are continuous. 
Moreover,  
    $C^\infty(M)$
    and 
    $C_c^\infty(M)$ are dense subspaces of  $L^2_{loc,r}(M)$.
    \end{enumerate}
\end{corollary}
\begin{proof}
Let $M_0$ be the least common multiple of the dilation's weights.
Part (2) when $r\in M_0\bN_0 $ follows readily from Lemma \ref{lem_topoSob} (4). This also shows that the topology of $C^\infty(M)$ coincides with the inductive limit of 
$\bigcap_{r\in M_0 \bN} L^2_{loc, r}(M) $.
    This  together with the continuous inclusions in Lemma \ref{lem_topoSob} (2) imply Part (1) with the intersection over $r\in \bR$ and Part (2) for any $r\in \bR$.
\end{proof}

\subsubsection{Duality}
 If  $f\in L^2_{loc,r'}(M)$ for some $r'\in \bR$, then it is a distribution, meaning that  the map $\varphi \mapsto (f,\varphi)_{\cD'\times C_c^\infty}$ formally given by the distributional duality
$$
(f,\varphi)_{\cD'\times C_c^\infty}=\int_M f(x)\ \varphi(x) \ dx, \qquad \varphi\in C_c^\infty(M),
$$
is a well defined linear functional  $C_c^\infty (M)\to \bC$ that is continuous.
The following statement allows us to identify the topological dual of $L^2_{loc,r}(M)$ with $L^2_{loc,-r}(M)$:
$$
\left(L^2_{loc,r}(M)\right)' \cong L^2_{loc,-r}(M),
$$
via the natural extension of the distributional duality.

\begin{lemma}
\label{lem_dualitySob}
Let $r\in \bR$.
\begin{enumerate}
    \item If  $f\in L^2_{loc,- r}(M)$, then the linear functional $\varphi \mapsto (f,\varphi)_{\cD'\times C_c^\infty}$ extends uniquely to a continuous linear functional  on $ L^2_{loc,r}(M)$.
    \item Conversely, given a continuous linear functional $\ell$ on $ L^2_{loc,r}(M)$, there exists a unique distribution $f\in L^2_{loc,- r}(M)$ such that 
    $$
    \forall \varphi \in C_c^\infty (M)\qquad 
    \ell (\varphi) = (f,\varphi)_{\cD'\times C_c^\infty}.
    $$
\end{enumerate}    
\end{lemma}

\begin{proof}
Part (1)  follows from Corollary \ref{corlem_topoSob} (2).
Conversely, let  $\ell$ be 
a continuous linear functional  on $ L^2_{loc,r}(M)$.
If $r'\in \bN$ is a multiple of the dilation's weights such that $r'\geq r$, then $\ell$ restrict to a continuous linear functional  on $ L^2_{loc,r'}(M)$ by Lemma \ref{lem_topoSob} (2), 
and the seminorms in Lemma \ref{lem_topoSob} (3) shows that $\ell$ also restrict to a continuous linear functional on $C_c^\infty (M)$. 
Therefore, there exists $f\in \cD'(M)$ such that $\ell(\varphi)=(f,\varphi)_{\cD'\times C_c^\infty}$
for any $\varphi\in C_c^\infty (M)$.
Let $\theta\in C_c^\infty (M)$.
We fix an atlas $\cA$.
We have
$\theta f \in \sE'(M)$, 
so $\Op^\cA(\sigma)\theta f\in \sE'(M)$ 
for any $\sigma\in S^{-r} (\Gh M)$, 
with
$$
(\Op^\cA(\sigma)\theta f , \varphi)_{\cD'\times C_c^\infty}
=( f,\theta \overline{\Op^\cA(\sigma)^*} \varphi)_{\cD'\times C_c^\infty} =\ell\left(\theta \overline{\Op^\cA(\sigma)^*}\varphi\right) .
$$
We now use the seminorms for $L^2_{loc,r}(M)$ described in Lemma \ref{lem_topoSob} (1) with its notation.
The continuity of $\ell:L^2_{loc,r}(M)\to \bC$ imply the existence of $C>0$ and $N\in \bN$ such that  
\begin{align*}
 \left|
\ell\left(\theta \overline{\Op^\cA(\sigma)^*}\theta' \varphi\right)\right|
&\leq C 
\max_{j\leq N}
\|P_j (\varphi)\|_{L^2(M)},
\end{align*}
where
$$
P_j:=
\theta_j\Op^\cA\left( (\id+\widehat \cR)^{\frac r \nu}\right) \theta \overline{\Op^\cA(\sigma)^*}.
$$
By the properties  of the global quantization, especially for composition and adjoint (see Corollary~\ref{cor:OpA}), 
$P_j \in \Psi^0(M)$
is a properly supported pseudodifferential operator of order 0.
Therefore, 
$$
\max_{j\leq N}
\|P_j (\varphi)\|_{L^2(M)}
\leq C'
\|\varphi\|_{L^2(M)},
$$
for some $C'>0$ independent of $\varphi$.
This implies that $\Op^\cA(\sigma)\theta f \in L^2(M)$.
By Theorem \ref{thm:charSob}, this shows that $f\in L^2_{loc,-r}(M)$. This concludes the proof of Part (2).
\end{proof}

\subsection{The classical $H$-calculus}\label{subsec:groupoid}
In this section, we adapt to the setting of  the filtered manifold $M$ the traditional lines of approach regarding  polyhomogneous pseudodifferential operators (Section~\ref{sec:polyhom}) and classical calculus (Section~\ref{sec:cla_cal}). We  discuss the link with van Erp and Yuncken calculus~\cite{VeY1,VeY2}
in Section~\ref{sec:groupoid}. 

\subsubsection{Polyhomogeneous operators}\label{sec:polyhom}

\begin{definition}
\label{def_PsiphM}
A pseudodifferential $P\in \Psi^m(M)$ of order $m\in \bR$ on $M$ is \emph{polyhomogeneous} when the  symbol in Part (1) of Definition~\ref{def:Q0} is polyhomogeneous in the sense of Definition~\ref{def:SmphGhM}, 
that is, 
for any adapted frame $\bX$ on an open set $U\subset M$, for any $\bX$-cut-off~$\chi$, 
there exists $\sigma\in S_{ph}^m(\widehat{G}M|_U)$ such that 
$\varphi\, P \,  \psi=\varphi\, \Op^{\bX, \chi}(\sigma) \, \psi$
         for any  
         $\varphi, \psi\in  C_c^\infty (U)$.
  \end{definition}
  
         We denote by $\Psi^m_{ph}(M)$ the space of polyhomogneous pseudodifferential operators of order $m\in \bR$.

\begin{ex}
    \begin{enumerate}
        \item  Any differential operator is polyhomogeneous with a finite asymptotic, see the proof of Lemma \ref{lem_exdiffPDO}: 
  $\DO^{\leq N} \subset \Psi^N_{ph}(M)$.
  \item 
Let $m\in \bR$. 
Let  $(\bX,U)$ be an adapted frame with $\mathbb{X}$-cut-off $\chi$.
If  $\sigma\in S^m_{ph}(\widehat{G}M|_U)$, then $\Op^{\mathbb{X}, \chi}(\sigma)\in \Psi^m_{ph}(U)$ by 
the proof of Lemma \ref{lem:local-to-global}
while 
$\psi\in C_c^\infty (U)$, 
$ \Op^{\bX, \chi}(\sigma)\,  \psi \in \Psi^m_{ph}(M)$ 
 for any $\psi\in C_c^\infty (U)$ by the proof of Corollary \ref{corlem:local-to-global}.
     \end{enumerate}
\end{ex}

The characterisations via atlases of  polyhomogeneous pseudodifferential operators together with the notion of homogeneous principal symbol
are obtained by modifying  the ones in Proposition \ref{caracterization_psi(M)}
and Corollary \ref{corcorthm:princ_symbol}
by considering polyhomogeneous symbols:
 
\begin{proposition}
\label{caracterization_psiph(M)}
Let $P$ be a continuous linear operator  $C_c^\infty(M)\to \mathcal{D}'(M)$.
The following are equivalent. 
\begin{itemize}
    \item[(i)] $P\in \Psi^m_{ph}(M)$.
    \item[(ii)] For one (and then any) atlas $\cA= (U_\alpha, \bX_\alpha, \chi_\alpha,\psi_\alpha)_{\alpha\in A}$ for the filtered manifold $M$, there exists  a family of polyhomogeneous symbols $(\sigma_\alpha)_{\alpha\in A}$ with  $\sigma_\alpha\in S^m_{ph}(\Gh M|_{U_\alpha})$ and a smoothing operator~$R\in\Psi^{-\infty}(M)$  such that 
    $$
    P=R+ \sum_{\alpha}\psi_\alpha\, \Op^{\bX_\alpha, \chi_\alpha}(\sigma_\alpha) \, \psi_\alpha
    $$
    \item[(iii)] For one (and then any) atlas $\cA$ on $M$, there exists  a polyhomogeneous symbol $\sigma\in S^m_{ph}(\Gh M)$ and a smoothing operator~$R\in\Psi^{-\infty}(M)$  such that 
    $P=R+ \Op^{\cA}(\sigma).$
\end{itemize}
Moreover, in this case, there exists a unique symbol $\sigma_m\in \dot S^m (\Gh M)$ such that $\princ_m (P) = \sigma_m \phi (\widehat \cR)$
for some positive Rockland symbol $\widehat\cR$ on $\Gh M$ and
for  some function $\phi\in C^\infty(M\times \bR)$ 
 satisfying $\phi(x,\lambda)=0$ for $\lambda$ in a neighbourhood of $(-\infty,0)$
and $\phi(x,\lambda)=1$ on a neighbourhood of $+\infty$, 
these two $\lambda$-neighbourhoods depending locally uniformly on $x\in M$.
\end{proposition}  

Recall that by Corollary \ref{corprop:fct_de_R}, 
 the symbol
$ \phi(\widehat \cR) = \{\phi (x,\pi(\cR_x)):(x,\pi)\in \widehat G M\}$ is 
   in $S^{0}(\widehat G M)$
   while $(1-\phi)(\widehat \cR) \in S^{-\infty}(\widehat G M).$

 \begin{proof}[Beginning of the proof of Proposition~\ref{caracterization_psiph(M)}]
 For the equivalence (i)$\Longleftrightarrow$ (ii), we readily modify the proof of Proposition \ref{caracterization_psi(M)}.
 \end{proof}
 
For the equivalence (i) $\Longleftrightarrow$ (iii), 
the modification in the proof of Corollary \ref{corcorthm:princ_symbol} requires to refine the notion of principal symbols which we start here.
By Definitions \ref{def_PsiphM} and \ref{def:SymbmV0}, for open subset $V_0\subset M$ small enough, we have for some $\sigma\in S^m_{ph}(\Gh M)$:
\begin{align*}
  {\rm Symb}_{m,V_0} 
  &=\sigma |_{V_0} \ \mbox{mod} \ S^{m-1}(\Gh M|_{V_0}) \\
  &=\sigma_{m,V_0} \phi(\widehat \cR) |_{V_0} \ \mbox{mod} \ S^{m-1}(\Gh M|_{V_0}) 
\end{align*}
where $\sigma_{m,V_0}$ is the first term in the expansion of $\sigma$ with $\phi,\widehat \cR$ as in the statement. Moreover, 
by the uniqueness of a polyhomogeneous asymptotic expansion (see Section \ref{subsubsec_SmphGhM}), $\sigma_{m,V_0}$ is uniquely determined on $V_0$.
Proceeding as in  the first part of the proof of Theorem \ref{thm:princ_symbol}, this allows us to construct $\sigma_m \in \dot S^m (\Gh M)$ such that $\sigma_m|_{V_0} = \sigma_{m,V_0}|_{V_0}.$

\begin{definition}
    For $P\in \Psi^m_{ph}(M)$, the unique symbol $\sigma_m\in \dot S^m(\Gh M)$ as in (the above proof of) Proposition~\ref{caracterization_psiph(M)} is denoted by 
  $$
\sigma_{m} =: \princ_{m,h} (P).
$$  
It is called the homogeneous principal symbol of $P$.
\end{definition}

\begin{ex} 
\label{ex_pincOpAsigma=sigmam}
If $\cA$ is an atlas and if $\sigma\in S^m_{ph}(\Gh M)$ with $\sigma\sim_h \sum_{j\geq 0} \sigma_{m-j}$, then 
$$
\princ_{m,h}\left(\Op^\cA(\sigma) \right)= \sigma_m .
$$
\end{ex}

From Theorem \ref{thm:princ_symbol} and its proof
together with Example \ref{ex_pincOpAsigma=sigmam},
we obtain readily that  the map 
$$
\princ_{m,h}:\Psi^m_{ph}(M)\to \dot S^m(\Gh M)
$$
is surjective, with kernel 
 $ \Psi^{m-1}_{ph}(M)$.
 Furthermore, 
$$
\forall P\in \Psi_{ph}(M),
\qquad
\princ_{m,h}(P^*)=\left(\princ_{m,h}(P)\right)^*,
$$
and for any $P_1\in \Psi^{m_1}_{ph}(M),\ P_2\in \Psi^{m_2}_{ph}(M)$ one of them being properly supported, 
    $$
 P_1P_2 \in \Psi^{m_1+m_2}_{cl}(M)
\quad\mbox{with}\quad
\princ_{m,h}(P_1P_2) = \princ_{m,h}(P_1)
\princ_{m,h}(P_2).
    $$
When $m=N\in \bN_0$, 
    the  restriction of $\princ_{m,h}$ to the space $\DO^{\leq N}$ of differential operators of homogeneous order at most $N$ coincides with the map also denoted by $\princ_N$ and  defined in Section \ref{subsec:diff_op} and in Theorem \ref{thm:princ_symbol}.

 \begin{proof}[End of the proof of  Proposition~\ref{caracterization_psiph(M)}]
From  Example \ref{ex_pincOpAsigma=sigmam} and the properties of $\princ_{m,h}$, 
 we obtain the equivalence (i)$\Longleftrightarrow$ (iii).
 \end{proof}
 
 We also obtain  the uniqueness of the polyhomogeneous expansion:
    \begin{corollary}
        \label{cor:caracterization_psiph(M)}
Let $P\in \Psi^m_{ph}(M)$.
We fix an atlas $\cA$, a 
positive Rockland symbol $\widehat\cR$ on $\Gh M$ and a  function $\phi\in C^\infty(M\times \bR)$ 
 as in Proposition \ref{caracterization_psiph(M)}. 
 Then the polyhomogeneous expansion of  $\sigma\in S^m_{ph}(\Gh M)$ in Proposition \ref{caracterization_psiph(M)} (3) is unique.
    \end{corollary}
    \begin{proof}
        By  Example \ref{ex_pincOpAsigma=sigmam},
        $\princ_{m,h}( P -\Op^\cA(\sigma_m \phi(\widehat \cR)))=0$, so $P_1=P -\Op^\cA(\sigma_m \phi(\widehat \cR) )\in 
    \Psi^{m-1}_{ph}(M)$ and we set $\sigma_{m-1}=\princ_{m-1,h}(P_1)$.
    Inductively, we construct 
    $$
    \sigma_m, \ldots, \sigma_{m-N} \  \mbox{in} 
    \ \dot S^m(\Gh M), \ldots, \dot S^{m-N}(\Gh M),
    $$
    such that
    $$
    P-\Op^\cA\left( \left(
    \sigma_m +\ldots + \sigma_{m-N}\right) \phi(\widehat \cR))\right)
    \in \Psi^{m-(N+1)}_{ph}(M).
    $$
   This concludes the proof of Proposition \ref{caracterization_psiph(M)}. 
   If $P$ is smoothing, then every $\sigma_{m-j}$ is zero, and this shows the uniqueness of the polyhomogeneous expansion in Proposition \ref{caracterization_psiph(M)} (3).
    \end{proof}

 We conclude this paragraph with the notion of homogeneous elliptic symbol.  
 
\begin{definition}
    \begin{enumerate}
        \item A $m$-homogeneous symbol $\sigma$ is \emph{elliptic} when 
for one (and then any) positive  Rockland symbol $\widehat \cR$ and
    for any $\gamma\in \bR$, $x\in M$ and almost-every $\pi\in \Gh$,    the inverse of 
   $\pi(\cR_x)^{\frac{\gamma}\nu} \sigma(x,\pi) \pi(\cR_x)^{-\frac{\gamma+m}\nu}$ exists as a bounded operator on $\pi$ with 
    $$
    \sup_{(x,\pi)\in \Gh M}
\|    \left(\pi(\cR_x)^{\frac{\gamma}\nu} \sigma(x,\pi) \pi(\cR_x)^{-\frac{\gamma+m}\nu}\right)^{-1}\|_{\sL(\cH_\pi)} <\infty.
    $$
\item An operator $P\in \Psi^m_{ph}(M)$ is \emph{elliptic} when its homogeneous principal symbol is elliptic.
    \end{enumerate}
\end{definition}

\begin{proposition}
\label{prop_par_ellipticop}
    Any elliptic operator $Q\in \Psi^m_{ph}(M)$  admits a left parametrix $P\in \Psi^{-m}_{ph}(M)$. 
\end{proposition}
\begin{proof}
Let $Q\in \Psi^m_{ph}(M)$ be elliptic and set  $\sigma_m = \princ_{m,h}(Q)\in \dot S^m (\Gh M).$
We fix $\phi$ and $\cR$ as in 
Proposition \ref{caracterization_psiph(M)}.
We check readily that $\sigma_m \phi (\widehat \cR)$ is invertible for the high frequencies of $\widehat \cR$.
We construct a parametrix as in the proof of  Theorem \ref{thm:parametrix},
and we check that it is  polyhomogeneous.
\end{proof}

\subsubsection{Definition and properties of the classical $H$-calculus}
\label{sec:cla_cal}

We denote by $\Psi^m_{cl}(M)$
the space of properly supported polyhomogeneous operator of order $m$, 
and by $\Psi_{cl}(M)=\cup_{m\in \bR}\Psi^m_{cl}(M)$ the space of properly supported polyhomogeneous operators of any order.
Clearly, $\Psi^m_{cl}(M)$ is a left and right $C^\infty(M)$-module. Moreover, 
the space 
$\Psi_{cl}(M)$ is a classical pseudodifferential calculus in the following sense:
\begin{enumerate}
    \item 
    It  is  stable under composition and taking the formal adjoint. More precisely, we have: 
    \begin{align*}
   & \forall P_1\in \Psi^{m_1}_{cl}(M),\ P_2\in \Psi^{m_2}_{cl}(M),
    \quad P_1P_2 \in \Psi^{m_1+m_2}_{cl}(M)\;\;
   \mbox{and}\;\; 
     P^*_1 \in \Psi^{m_1}_{cl}(M).
    \end{align*}
    \item It contains the differential calculus: $\DO(M)\subset \Psi_{cl}(M) $. More precisely, $\DO^{\leq N} \subset \Psi_{cl}^N(M)$ for any $N\in \bN_0$.
    \item It acts on the adapted local Sobolev spaces, and the loss of Sobolev derivatives is bounded by the order. More precisely, any 
    $ P\in \Psi^{m}_{cl}(M) $ acts continuously $L^2_{s}(M)\to L^2_{s-m}(M)$ for any $m,s\in \bR.$
    \item The smoothing operators in the classical $H$-calculus are the operators with smooth integral kernels that are properly supported. 
    Denoting this space by
    $$
    \Psi^{-\infty}_{cl} (M) := \Psi^{-\infty} (M) \cap \Psi_{cl} (M), 
    $$
    it may be equivalently described as 
    $$
    \Psi^{-\infty}_{cl} (M) = \cap_{m\in \bR} \Psi^m_{cl}(M).
    $$   
    \item Any elliptic operator admits a left parametrix. More precisely, if $Q\in \Psi^m_{cl}(M)$ is elliptic, then we can construct $P\in \Psi^{-m}_{cl}(M)$ such that $PQ-\id\in\Psi^{-\infty}_{cl} (M) $.
\end{enumerate}

\begin{ex}
\label{ex_ellipticinHcalculus}
Considering Example \ref{ex_I+RmnuphM}, 
an example of elliptic operator is $\Op^\cA \left((\id+\widehat \cR)^{m/\nu}\right),$
where $\widehat \cR$ is a positive Rockland symbol on $\Gh M$, $\nu$ its homogeneous degree and $\cA$ an atlas for $M$.
By Proposition \ref{prop_par_ellipticop}, it admits a left parametrix in the classical $H$-calculus.
\end{ex}

\subsection{H-calculus and tangent groupoid}\label{sec:groupoid}

\subsubsection{The tangent groupoid and H-pseudodifferential kernels}

Before briefly presenting the tangent groupoid adapted to the filtered manifold $M$ as constructed by Van Erp and Yuncken in \cite{VeY1,VeY2}, let us explain what they mean by a pseudodifferential calculus.
We will allow ourselves to identify densities and functions as we have fixed a measure $dx$ on $M$.

Van Erp and Yuncken
 blur the distinction between a linear continuous operator $P:C_c^\infty(M)\to \cD'(M)$ and its integral kernel given by the distribution $K_P$.
Moreover, they consider pseudodifferential operators modulo smoothing operators, that is, modulo operators with smooth integral kernels:
$$
P \ \mbox{mod} \ \Psi^{-\infty}(M)
\ \longleftrightarrow \
K_P \ \mbox{mod} \ C^\infty (M\times M).
$$
Hence, instead of describing  classes of operators, they describe classes of properly supported integral kernels. 

Let us now recall some elements of their construction, the references to their works being \cite{VeY1,VeY2}.
 The $H$-tangent groupoid (the letter~$H$ refers to the filtration $M$ is equipped with) is the  union of the osculating group bundle $GM$ with $M\times M$ together with a dilation parameter $t\in \bR_+ = [0,+\infty)$:
\[
\bT_HM=GM \times \{0\}\,\bigcup \, (M\times M)\times \bR_+^*
\]
where $\bR_+^*=\bR_+\setminus \{0\}$.
Van Erp and Yuncken prove in~\cite{VeY1} that $\bT_HM$ may be equipped with a natural  smooth structure  that makes it a Lie groupoid.
The description of this smooth structure on  $\bT_H M$   is in terms of exponential maps associated with connections. It is equivalent (and perhaps simpler) to express it in terms of geometric exponential associated with an adapted frame.

On $\bT_HM$ is defined the set of $r$-fibred distribution, $ \sE_r' (\bT_H M)$, which is the set of continuous $C^\infty(M\times \bR_+)$-linear maps $C^\infty(\bT_HM) \to C^\infty(M\times \bR_+)$. 
This space is naturally equipped with a convolution product law (with the standard precautions about the fact that the involved distributions are properly supported), and it is called the {\it convolution algebra} of $\bT_H M$. 

The zoom action $\alpha$  of $\bR^*_+$ on $\bT_HM$  is defined  according to 
\[
\left\{
\begin{array}{ccl}
\alpha_\lambda(x,y,t)&=&(x,y,\lambda^{-1}t)\;\;\mbox{if}\;\; (x,y)\in M\times M\;\;\mbox{and} \;\; t\not=0,\\
\alpha_\lambda(x,v,0)&=&(x, \delta_\lambda v,0)\;\;\mbox{if}\;\; x\in M\;\;\mbox{and}\;\;v\in \mathfrak g _xM.
\end{array}
\right.
\]
The zoom action induces
a one-parameter family of automorphisms $\alpha_{\lambda*}$, $\lambda\in\bR_+^*$ on  the convolution algebra $ \sE_r' (\bT_H M)$, and preserves 
the ideal $C^\infty_p(GM;\Omega_r)$ of properly supported smooth densities.

 Van Erp and Yuncken then define the following two notions \cite[Definitions 18 \& 19]{VeY2}: 

\begin{definition}
\begin{enumerate}
    \item  A properly supported $r$-fibred distribution $\bP\in   \sE_r' (\bT_H M)$ is called {\it essentially homogeneous of weight $m\in\bR$} if
\[
\forall \lambda\in\bR^*_+,\;\;\alpha_{\lambda*} \bP-\lambda^m\bP \in C^\infty_p(\bT_H M;\Omega_r).
\]  
\item A {\it $H$-pseudodifferential kernel of order $\leq m$}
is a distribution  $K\in \sE'_r(M\times M)$ 
satisfying $K=\bP|_{t=1}$ for some 
properly supported $r$-fibred distribution $\bP\in   \sE_r' (\bT_H M)$ which is essentially homogeneous of weight $m\in\bR$. 

\item 
We denote by 
$\Psi^m_{vEY}(M) $ the space of properly supported operators 
$P:C_c^\infty (M) \to C^\infty(M)$ whose integral kernels
are   $H$-pseudodifferential of order $\leq m$.
\end{enumerate}
\end{definition}

Let us show that van Erp and Yuncken's construction
coincides with $\Psi_{cl}^m(M)$ in the following sense:
\begin{theorem}
\label{thmvEY}
For any  $m\in \bR,$
$\displaystyle{
\Psi^m_{cl}(M) = \Psi^m_{vEY}(M).
}$
\end{theorem}

In other words,  our classical calculus $\Psi_{cl}(M)$ and the pseudodifferential calculus constructed  via groupoid coincide in exactly the same way that the H\"ormander classical calculus on a manifold corresponds to the abstract operator classes defined via the groupoid approach (see Section~11 in \cite{VeY2}, see also  Debord and Skandalis~\cite{DS_14}).

\smallskip 

The rest of the section is devoted to the proof of Theorem \ref{thmvEY}, that is, for the two inclusions of the two pseudodifferential calculi. We start with the observation that both calculi may be viewed modulo smoothing operators as:
\begin{remark}
\label{rem_vEYsmoothing}
As in the case of $\Psi_{cl}(M)$, the set $\bigcap_{m\in \bR} \Psi^m_{vEY}(M)$ coincides with the space of smoothing properly supported operators 	
\cite[Corollary 53]{VeY2}:
$$
\bigcap_{m\in \bR} \Psi^m_{vEY}(M) = \Psi^{-\infty}_{cl}(M).
$$
\end{remark}

\subsubsection{The proof of the inclusion $
\Psi^m_{cl}(M) \subset  \Psi^m_{vEY}(M)$}

The main idea behind the proof of the inclusion $
\Psi^m_{cl}(M) \subset  \Psi^m_{vEY}(M)$ is the construction of an $r$-fibred distribution associated with a classical symbol, and this will use  the  following lemma: 

\begin{lemma}
\label{lem_dilationconvolutionkernel}
    Let $\sigma\in S^m_{ph}(\Gh M)$.
    \begin{enumerate}
        \item For any $t>0$, the density 
$ \kappa_{\sigma} \circ \delta_t  
- t^{m}\kappa_{\sigma}$
is smooth.
\item Write the  polyhomogeneous expansion of $\sigma$ as $\sigma\sim_h\sigma_m + \sigma_{m-1} +\ldots$
Then, for any adapted frame $(\bX,U)$,  for every $x\in U$,  we have 
the convergence in the sense of distribution 
    $$
t^{m+Q}\kappa_{\sigma,x}^{\bX} (\delta_tv) \longrightarrow_{t\to 0}  \kappa_{\sigma_m,x}^{\bX} ( v) .
    $$
Furthermore, this convergence holds for any $x$-derivatives  locally uniformly in $x$.
    \end{enumerate}
\end{lemma}

\begin{proof}[Proof of Lemma \ref{lem_dilationconvolutionkernel}]
We start the proof with some preliminary analysis. For a symbol $\tau$ in $\cup_{m'\in\bR} S^{m'}(\Gh M)$, we denote by $\tau\circ \delta_{t^{-1}}$ the symbol obtained using the dilations in~\eqref{def:dil_pi}:
\[
\forall (x,\pi)\in \widehat GM,\;\;
\forall t>0,\qquad
\tau\circ\delta_{t^{-1}}
(x,\pi)= \tau (x,\pi\circ\delta_{t^{-1}}).\]
We check readily that the convolution kernels 
 satisfy
$$
\kappa_{\tau\circ \delta_{t^{-1}}} = t^{m'}\kappa_{\tau}\circ \delta_t, \qquad \lambda>0,
$$
as densities. If a smooth Haar system $\mu=\{\mu_x\}_{x\in M}$ has been fixed, then this becomes:
$$
\kappa^{\mu}_{\tau \circ \delta_{t^{-1}},x} 
=
t^{m'+Q}\kappa^{\mu}_{\tau ,x} \circ \delta_t.
$$
We fix $\widehat \cR$ and $\phi$ as in Proposition \ref{caracterization_psiph(M)}, that is, $\widehat \cR$ is
a positive Rockland symbol  on $\Gh M$  and  $\phi\in C^\infty(M\times \bR)$ 
 satisfies $\phi(x,\lambda)=0$ for $\lambda$ in a neighbourhood of $(-\infty,0]$
and $\phi(\lambda)=1$ on a neighbourhood of $+\infty$, 
these two $\lambda$-neighbourhoods depending locally uniformly on $x\in M$. 
We denote by $\nu$ the homogeneous degree of $\widehat \cR$. 
Then 
$$
\phi (\widehat \cR) \circ \delta_{t^{-1}} = \phi (t^{-\nu} \widehat \cR).
$$
If now $\tau\in \dot S^{m'}(\Gh)$ is $m'$-homogeneous for some $m'\in \bR$, then 
$$
t^{-m'}(\tau \phi (\widehat \cR)) \circ \delta_{t^{-1}}
= \tau \phi (t^{-\nu} \widehat \cR) 
=  \tau \phi (\widehat \cR)
+\rho_{\tau,t}
$$
where $\rho_{\tau,t}$ is the  symbol given by 
\begin{equation}\label{def:notation_rho_tx}
\rho_{\tau,t}:=
\tau 
\left(
\phi (t^{-\nu} \widehat \cR)  - \phi (\widehat \cR)\right).
\end{equation}
Since $\lambda\mapsto \phi(t^{-\nu}\lambda)-\phi(\lambda)$ is compactly supported, 
Corollary~\ref{corprop:fct_de_R} implies that  
$\rho_{\tau, t}$ is smoothing for all $t>0$.

Let us now prove (1). We consider $\sigma\in S^m_{ph}(\Gh M)$. Writing its homogeneous expansion $\sigma\sim_h\sigma_m + \sigma_{m-1} +\ldots$, it admits the inhomogeneous expansion 
$\sigma\sim \sigma_m \phi (\widehat \cR) +\sigma_{m-1} \phi (\widehat \cR)+\ldots $
By the  kernels estimates of Theorem~\ref{thm_kernelM}, the function 
$$
R_x^N(v):=
\kappa_{\sigma,x}(v) - \sum_{j=0}^N \kappa_{\sigma_{m-j}\phi(\widehat \cR),x}(v)
$$
is continuous in $v$ when $m-(N+1) <-Q$ for each $x\in M$.
In fact, Theorem~\ref{thm_kernelM}  also shows that it is in $L^2_s(G_x M)$ for any $s\in\bR$ such that $2(s+m-(N+1)+Q)<Q$, i.e. for $s<-Q/2-m+(N+1)$.
With these properties, we have
\begin{align*}
    t^Q \kappa_{\sigma,x} \circ \delta_t  
- t^m \kappa_{\sigma,x}
&=\sum_{j=0}^N \left( t^Q \kappa_{\sigma_{m-j}\phi(\widehat \cR),x}\circ \delta_t  
- t^m \kappa_{\sigma_{m-j}\phi(\widehat \cR),x} \right)
\ + \ t^Q R^N_x \circ \delta_t  -t^m R^N_x\\
&=\sum_{j=0}^N  \kappa_{ (\sigma_{m-j}\phi(\widehat \cR))\circ \delta_{t^{-1},x}- t^m \sigma_{m-j}\phi(\widehat \cR)}
\ + \ t^Q R^N_x \circ \delta_t  -t^m R^N_x
\\
&=\sum_{j=0}^N 
t^{j} \kappa_{\rho_{\sigma_{m-j},t},x} \ + \ t^Q R^N_x \circ \delta_t  -t^m R^N_x,
\end{align*}
 using the notation~\eqref{def:notation_rho_tx}. By the preliminary analysis, the kernel $\sum_{j=0}^N 
t^{j} \kappa_{\rho_{\sigma_{m-j},t},x}$ is smoothing. Moreover, 
the kernel $t^Q R^N_x \circ \delta_t  -t^m R^N_x$ 
is in $L^2_s(G_x M)$ for all $s$ such that  $s+m<-Q/2+(N+1)$.
The properties of Sobolev spaces on groups imply that $t^Q \kappa_{\sigma,x} \circ \delta_t  
- t^m \kappa_{\sigma,x}$ is in $C^{r}$ on $G_x M$ with $r= -m-Q+N$, and this for all $N\in\bN$, thus smooth.
One concludes the proof by
applying this to $D_\bX^\beta\sigma$ for any $\beta\in \bN_0^n$, which shows (1).
\smallskip 

Taking $t\to 0$ shows the convergence at every $x\in U$. 
Again, applying this to $D_\bX^\beta\sigma$ for any $\beta\in \bN_0^n$ shows (2).
\end{proof}

\begin{proof}[Proof of the inclusion $
\Psi^m_{cl}(M) \subset  \Psi^m_{vEY}(M)
$ in Theorem \ref{thmvEY}]
We fix an atlas $\cA$ on $M$.
Let $P\in \Psi^m_{cl}(M)$. 
By Proposition~\ref{caracterization_psiph(M)}, we may assume $P=\Op^\cA(\sigma)$ for some $\sigma\in S^m_{ph}(\Gh M)$.
It follows from our construction that  its integral kernel $K_P$ 
is given by 
$$
K_P(x,y)=\sum_{\alpha\in A} 
\kappa_{\sigma,x}^{\bX_\alpha} (-  \ln^{\bX_\alpha}_x y) \ \psi_\alpha(x)\psi_\alpha(y) \jac_{y} (\ln^{\bX_\alpha}_x), \qquad x,y\in M.
$$
It is natural to deform this convolution kernel into the following family of distributions on $M\times M$ parametrised by $t>0$:
$$
\bP (x,y,t):= \sum_{\alpha\in A} 
t^{m}\kappa_{\sigma,x}^{\bX_\alpha} (- \delta_{t} \ln^{\bX_\alpha}_x y) \  
\psi_\alpha(x)
\psi_\alpha\left(\exp^{\bX_\alpha}_x\left(\delta_{t} \ln^{\bX_\alpha}_x y\right)\right)
\jac_{\exp^{\bX_\alpha}_x(\delta_{t} \ln^{\bX_\alpha}_x y)} (\ln^{\bX_\alpha}_x).
$$
We observe for any $v\in \bR^n$,
\begin{align*}
D_{\exp^{\bX_\alpha}_x(\delta_{t} \ln^{\bX_\alpha}_x y)} (\ln^{\bX_\alpha}_x) (v)
 & =  
 \partial_{s=0}
 \ln^{\bX_\alpha}_x \exp^{\bX_\alpha}_x\left((\delta_{t} \ln^{\bX_\alpha}_x \exp_y^{\bX_\alpha} (s v)\right) \\
  & =  
 \delta_{t} \partial_{s=0}
  \ln^{\bX_\alpha}_x \exp_y^{\bX_\alpha} (s v) \\
  & =  
 \delta_{t} D_y \ln_x^{\bX_\alpha} (v), 
 \end{align*}
 so 
 $$
\jac_{\exp^{\bX_\alpha}_x(\delta_{t} \ln^{\bX_\alpha}_x y)} (\ln^{\bX_\alpha}_x)
 =
 t^Q \jac_y \ln_x^{\bX_\alpha}.
 $$
Therefore, we have 
\begin{equation}\label{lien_Pt_sigma}
\bP (x,y,t)= \sum_{\alpha\in A} 
t^{m+Q}\kappa_{\sigma,x}^{\bX_\alpha} (- \delta_{t} \ln^{\bX_\alpha}_x y) \  
J_\alpha(x,y,t),
\end{equation}
where
\begin{equation}\label{lien_P0_sigma}
J_\alpha(x,y,t):= 
\psi_\alpha(x)
\psi_\alpha\left(\exp^{\bX_\alpha}_x\left(\delta_{t} \ln^{\bX_\alpha}_x y\right)\right)
\jac_y (\ln^{\bX_\alpha}_x).
\end{equation}
Hence, we are lead to extend $\bP$  to $\bT_H M$ 
via
$$
\bP(x,v,0) = \sum_{\alpha\in A} \psi_\alpha(x)^2 \kappa_{\sigma_m,x}^{\bX_\alpha}(v).
$$
By Lemma \ref{lem_dilationconvolutionkernel} (2), the equations~\ref{lien_Pt_sigma} and~\ref {lien_P0_sigma} define a distribution $\bP\in \sE_r'(\bT_H M)$.
It remains to show that $\bP$ is essentially homogeneous of weight $m\in \bR$.

We already have  for $t=0$
$$
\alpha_{\lambda*}\bP (x,v,0)
	=\lambda^{m+Q}\bP (x,v,0).
	$$ 
We now  consider for $t,\lambda>0$
\begin{align*}
\nonumber 
\alpha_{\lambda*}\bP (x,y,t)
	&=\lambda^{Q}\bP (x,y,\lambda^{-1} t) 	\\
	&=
	\lambda^m\sum_{\alpha\in A} 
t^{Q+m} \kappa_{\sigma,x}^{\bX_\alpha} (- \delta_{\lambda^{-1} t} \ln^{\bX_\alpha}_x y) 
J_\alpha(x,y,\lambda^{-1}t). 
\end{align*}
We have
\begin{align}
\nonumber
&(\alpha_{\lambda*}\bP - \lambda^{m} \bP)(x,y,t)
\\
\label{somme_sur_alpha}
&\qquad	=R (x,y,t,\lambda)	 +\lambda^m\sum_{\alpha\in A} 
t^{Q+m} 
\left(\kappa_{\sigma,x}^{\bX_\alpha} (- \delta_{\lambda^{-1} t} \ln^{\bX_\alpha}_x y) 
- \kappa_{\sigma,x}^{\bX_\alpha} (- \delta_{ t} \ln^{\bX_\alpha}_x y) \right) 
J_\alpha (x,y,\lambda^{-1}t),
\end{align}
with  
$$
R (x,y,t,\lambda):=
\lambda^m\sum_{\alpha\in A} 
t^{-(Q+m)} \kappa_{\sigma,x}^{\bX_\alpha} (- \delta_{ t} \ln^{\bX_\alpha}_x y)
\left(   J_\alpha (x,y,\lambda^{-1} t)- J_\alpha (x,y, t)\right).
$$
Moreover, the sum over $\alpha\in A$ in~\eqref{somme_sur_alpha} is smooth 
by Lemma \ref{lem_dilationconvolutionkernel} (1). 
 We are left with showing the smoothness of $R(x,y,t)$.
Recall that $\kappa_{\sigma,x}$ is smooth away from 0. Since $J_\alpha (x,y,\lambda^{-1} t)- J_\alpha (x,y, t)=0$ indentically when  $x,y$ are near the diagonal, 
$R$ is smooth on $M\times M\times \bR_+^*\times \bR_+^*$.
We deduce that 
$\alpha_{\lambda*}\bP - \lambda^{m} \bP$ is smooth on $M\times M$.
\end{proof}

\subsubsection{The reverse inclusion}
For the proof of the reverse inclusion, 
we need to discuss the notion of 
{\it co-principal symbol}  in the sense of van Erp and Yuncken  \cite[Section 6]{VeY2}: for any  $m\in \bR$, 
the  co-principal symbol of an $H$-pseudodifferential kernel  $K\in \sE_r'(M\times M)$  of order $\leq m$
 may be represented uniquely by a vertical distributional density  $\kappa \in \cS'(GM;|\Omega|(GM))$ that is $m$-homogeneous in the sense of Definition \ref{def_mhomS'GMvert}; in this paper,  $\kappa$ is called the \emph{homogeneous co-principal symbol} of $K$.
  This follows from choosing  an essentially homogeneous distribution $\bP\in   \sE_r' (\bT_H M)$  of weight $0$ satisfying $K=\bP|_{t=1}$ and that is homogeneous on the nose 
 outside of $(-1,1)$ \cite[Proposition 42]{VeY2}, and then applying \cite[Proposition 43]{VeY2}.
 Moreover, $\kappa$ is smooth away from the 0-section of $GM$.
 \smallskip

The main point of the proof of the reverse inclusion of  Theorem \ref{thmvEY} is the natural correspondence between the 0-homogeneous co-principal symbols in $\Psi_{vEY}(M)$
and the 0-homogeneous symbols in $\Psi_{cl}(M)$:
\begin{lemma}
  \label{lem_vEYcosymbol}
If $\kappa\in \cS'(GM;|\Omega|(GM))$ is $0$-homogeneous and is smooth away from the 0-section of $GM,$
then it is the convolution kernel of a symbol $\sigma\in \dot S^0(\Gh M).$
\end{lemma}

We will not need the converse, which is true, and follows   from the properties of symbols in $\dot S^0(\Gh M)$:
if the symbol $\sigma$ is in $S^0(\Gh M)$, then  its convolution kernel $\kappa$ is $0$-homogeneous and is smooth away from the 0-section of $GM.$
 The proof of both implications follow from the study of the group case in \cite{FF0}.

\begin{proof}[Proof of Lemma \ref{lem_vEYcosymbol}]
Let $\kappa \in \cS'(GM;|\Omega|(GM))$ be a  0-homogeneous distribution density  that is smooth away from the 0-section of $GM$.
It follows from Lemma 2.25 and Corollary 5.4 in \cite{FF0}
that we may define a symbol $\sigma_0$ on $M$ with convolution kernel $\kappa$
such that 
 $\sigma_0(x,\cdot)\in \dot S^0(\Gh_x M) $  with seminorms depending locally uniformly in $x\in M$. 
 We then check that, for any adapted frame $(\bX, U)$ and $\beta\in \bN_0^n$, the distribution 
 $\bX^\beta \kappa^\bX\in\cS'(GM|_U;|\Omega|(GM|_U)) $ is still 0-homogeneous and smooth away from the origin, 
 so by uniqueness of the Fourier transform, above each $x\in M$, it is the convolution kernel 
 of  $D_\bX^\beta \sigma_0 (x,\cdot) \in \dot S^0(\Gh_x M)$. This implies $\sigma\in \dot S^0(\Gh M)$.
\end{proof}

The properties of $\Psi_{cl}(M)$ and the inclusion $\Psi_{cl}(M)\subseteq \Psi_{vEY}(M)$ then  imply:
\begin{corollary}
\label{corlem_vEYcosymbol}
	We fix  an atlas $\cA$ on $M$, and  $\widehat \cR$ and $\phi$ as in Proposition \ref{caracterization_psiph(M)}.
	Let $m\in \bR$ and let $P\in \Psi^m_{vEY}(M)$.
Then there exists a symbol $\sigma_m\in \dot S^m(\Gh M)$
such that 
$$
P-\Op^\cA (\sigma_m \phi(\widehat\cR))\in \Psi^{m-1}_{vEY}(M).
$$
\end{corollary}

The proof of Corollary  \ref{corlem_vEYcosymbol} follows from Lemma \ref{lem_vEYcosymbol} and standard manipulations of pseudodifferential theory. 
Before giving it, let us conclude the proof of Theorem \ref{thmvEY}.
\begin{proof}[End of the proof of Theorem \ref{thmvEY}]
It remains to show  the inclusion $
\Psi^m_{vEY}(M) \subset  \Psi^m_{cl}(M)$ for any $m\in \bR.$
Let $P=P_m\in \Psi^m_{vEY}(M) $.
Let $\sigma_m$ as in Corollary \ref{corlem_vEYcosymbol}.
Inductively, we construct a sequence of symbols
$\sigma_{m-j}\in \dot S^{m-j},$ $j\in \bN_0,$ such that 
$$
P_{m-(j+1)}=
	P_{m-j}-\Op^\cA (\sigma_{m-j} \phi(\widehat\cR))\in \Psi^{m-(j+1)}_{vEY}(M),
$$
By construction, we have for any $N\in \bN$, 
$$
P=\sum_{j=0}^N \Op^\cA (\sigma_{m-j} \phi(\widehat\cR)) \ + \ P_{m-(N+1)}
$$
Consider
a symbol $\sigma\in S^m(\Gh M)$ with  asymptotic expansion   
$\sigma \sim \sum_j \sigma_{m-j} \phi(\widehat\cR)$. 
Then $\sigma\in S^m_{ph}(\Gh M)$ and for any $N\in \bN$,
the operator
$$
P - \Op^\cA (\sigma) 
=  \Op^\cA \Bigl(\sum_{j=0}^N \sigma_{m-j} \phi(\widehat\cR) \ - \ \sigma\Bigr) \ + \ P_{m-(N+1)} ,
$$
is in $\Psi_{vEY}^{m-(N+1)}(M)$ since $\sum_{j=0}^N \sigma_{m-j} \phi(\widehat\cR) \ - \ \sigma \in S^{m-(N+1)}(\Gh M)$.
By Remark \ref{rem_vEYsmoothing}, this implies that 
$P - \Op^\cA (\sigma)\in \Psi_{cl}^{-\infty}(M).$ 
\end{proof}

It remains to show Corollary \ref{corlem_vEYcosymbol}.
The case $m=0$ is essentially a consequence of Lemma~ \ref{lem_vEYcosymbol}:

\begin{proof}[Proof of Corollary \ref{corlem_vEYcosymbol} when $m=0$.] 
Let $\kappa\in \cS'(GM;|\Omega|(GM))$ be the homogeneous principal co-symbol of the integral kernel $K_P$ of $P$. 
As we are in the  case $m=0$, 
by Lemma \ref{lem_vEYcosymbol}, 
$\kappa$ is the convolution of a symbol
$\sigma_0\in \dot S^0(\Gh M)$.
We now observe that the proof 
  of the implication (2) $\Rightarrow $ (1) in Theorem~\ref{thmvEY} above
shows $\Op^\cA (\sigma_0 \phi(\widehat\cR))\in \Psi^m_{vEY}(M)$, and moreover that the homogeneous co-principal symbol of its integral kernel  is $\kappa_{\sigma_0} = \kappa\in \cS'(GM;|\Omega|(GM)).$
This implies that the  homogeneous
co-principal symbol of the integral kernel of $P-\Op^\cA (\sigma_0 \phi(\widehat\cR))$ 
is zero. 
Therefore, by \cite[Corollary 38]{VeY2},  
 $P-\Op^\cA (\sigma_0 \phi(\widehat\cR))\in \Psi^{-1}_{vEY}(M)$.  
\end{proof}

Traditional manipulations in pseudodifferential theory allow us to pass from the case $m=0$ to the case of general $m.$

\begin{proof}[Proof of Corollary \ref{corlem_vEYcosymbol} for any $m\in \bR$.]
As we have proved $\Psi^{m'}_{cl}(M)\subseteq \Psi_{vEY}^{m'}(M)$ for any $m'\in \bR$, 
$\Op^\cA ((\id+\widehat \cR)^{-m/\nu} )\in \Psi_{vEY}^{-m}(M)$.
By the composition properties of the calculus built by van Erp and Yuncken   	\cite[Corollary 49]{VeY2}, 
$\Op^\cA ((\id+\widehat \cR))^{-m/\nu})P \in \Psi^0_{vEY}(M)$.
Denoting $\widetilde \sigma_0$ the corresponding symbol as in the case $m=0$, we have
$$
\Op^\cA ((\id+\widehat \cR))^{-m/\nu}) P-\Op^\cA (\widetilde \sigma_0 \phi(\widehat\cR))\in \Psi^{-1}_{vEY}(M).
$$
Applying Corollary \ref{cor_parametrix}  to $-m$, we obtain a symbol $\tau \in S^{m}_{ph}(\Gh M)$ satisfying 
$$
\Op^\cA (\tau ) \Op^\cA ((\id+\widehat \cR)^{-m/\nu})-\id \in \Psi_{cl}^{-\infty}(M).
$$
Moreover, an inspection of the construction shows that $\princ_{m,h} (\tau) =\widehat \cR^{m/\nu}$.
Hence, we have
\begin{align*}
\princ_{h,m}\left(\Op^\cA (\tau ) \Op^\cA (\widetilde \sigma_0 \phi(\widehat\cR))\right) 
&=\princ_{h,m}\left(\Op^\cA (\tau )  \right)\princ_{h,m}\left(\Op^\cA (\widetilde \sigma_0 \phi(\widehat\cR))\right) 
= \widehat \cR^{m/\nu} \widetilde \sigma_0.
\end{align*}
Consequently, 
$$
\Op^\cA (\tau ) \Op^\cA (\widetilde \sigma_0 \phi(\widehat\cR)) = \Op^\cA (\widehat \cR^{m/\nu} \widetilde \sigma_0 \phi(\widehat\cR))
\ \mbox{mod}\ \Psi^{m-1}_{cl}(M).
$$
The inclusion $\Psi^{m'}_{cl}(M)\subseteq \Psi_{vEY}^{m'}(M)$ for any $m'\in \bR$ together with  the properties of the calculus built by van Erp and Yuncken, especially its composition  	\cite[Corollary 49]{VeY2} and for the smoothing operators (see 
Remark \ref{rem_vEYsmoothing}), imply
\begin{align*}
	P&=\Op^\cA (\tau ) \Op^\cA ((\id+\widehat \cR)^{-m/\nu})P  \ \mbox{mod} \ \Psi_{cl}^{-\infty}(M)\\
	&=\Op^\cA (\tau ) \Op^\cA (\widetilde \sigma_0 \phi(\widehat\cR))  \ \mbox{mod} \ \Psi_{vEY}^{m-1}(M)\\
	&=\Op^\cA (\widehat \cR^{m/\nu} \widetilde \sigma_0 \phi(\widehat\cR))  \ \mbox{mod} \ \Psi_{vEY}^{m-1}(M),
\end{align*}
which terminates the proof. 
\end{proof}

%%%%%%%%%%%%%%%%%%%%

\appendix

\section{Combinatoric}\label{app:combi}

In this section we prove equation~\eqref{eq:combinatoric} and analyse the form of the terms after reorganisation by the number of elements in the set~$I$, in order to obtain~\eqref{eq:ZNk} with the afferent properties.

\smallskip

The proof of~\eqref{eq:combinatoric} is done by induction. The formula~\eqref{eq:combinatoric} holds for $k=1$ since, in that case, there is only $n$ configuration of indices sets that satisfy the assumptions:
$I=\{j\}$, $I_1=\emptyset$, which
gives the term $z_jX_j$, for $j=1,\cdots, n$.
 It also holds for $k=2$ with the sets of indices: 
 \begin{align*}
&I=\{ j_1,j_2\}, \;\;I_1=I_2=\emptyset,
\;\;\mbox{whence\;\;the \;\;term }\;\;z_{j_1}z_{j_2} X_{j_1}X_{j_2},\\
& I=\{j_1\}, \;I_1=\{j_2\}, \; I_2=\emptyset,
\;\;\mbox{whence}\;\;(X_{j_2}z_{j_1})z_{j_2} X_{j_1},
 \end{align*}
 with $j_1,j_2$ describing $1,\cdots,n$. 
 The configuration
 $I=\{j_1\}, \;I_1=\emptyset, \; I_2=\{j_2\}$ is forbidden.
\smallskip 

 Assume now that the formula~\eqref{eq:combinatoric} holds for some $k\geq 1$, let us prove that the result also holds for $k+1$.  We derive 
 \begin{align*}
(Z^{(N)})^{k+1}  =\sum_{j=1}^n 
\sum_{(j_1,\ldots,j_k)=[I,I_1,\ldots,I_k]}
\Bigl(&
z_j (\bX_{I_1}z_{j_1})\ldots (\bX_{I_k}z_{j_k}) \bX_{ [\{j\}, I]}\\
&+ \sum_{1\leq p\leq k}  z_j (\bX_{I_1}z_{j_1})\ldots (\bX_{[\{j\},I_p]}z_{j_p})\ldots (\bX_{I_k}z_{j_k})
\Bigr).
 \end{align*}
 Denoting by $j_{k+1}$ the new index, we obtain
  \begin{align*}
(Z^{(N)})^{k+1}  =
\sum_{(j_1,\ldots,j_k)=[I,I_1,\ldots,I_k]}\sum_{j_{k+1}=1}^n 
\Bigl(&
(\bX_{I_1}z_{j_1})\ldots (\bX_{I_k}z_{j_k})z_{j_{k+1}}  \bX_{[\{j_{k+1}\}, I]}\\
&+ \sum_{1\leq p\leq k}  (\bX_{I_1}z_{j_1})\ldots (\bX_{[\{j_{k+1}\}, I_p]}z_{j_p})\ldots (\bX_{I_k}z_{j_k}) z_{j_{k+1}}
\bX_{I}
\Bigr).
 \end{align*}
 We then observe the equality of the set of indices $(j_1,\ldots, j_{k+1})=[I,I_1,\ldots, I_{k+1}]$ (equipped with the associated constraints), with the disjoint union 
 \begin{align*}
     \bigcup_{j=1}^n\Bigl\{& (j_1,\ldots, j_{k+1})=[[\{j_{k+1}\},I],I_1,\ldots, I_{k},\emptyset]\Bigr\}\\
     \cup &\; \bigcup_{p=1}^k\bigcup_{j=1}^n \Bigl\{
  (j_1,\ldots, j_{k+1})=[I,I_1,\ldots,I_{p-1}, [\{j_{k+1}\},I_p],I_{p+1},  \ldots, I_{k},\emptyset]\Bigr\}  . 
 \end{align*}
This concludes the proof. 

\medskip 

Let us now  reorganise the terms of $(Z^{(N)})^k$ in terms of the cardinal  of the indices set~$I$ and write 
    \begin{equation}\label{hyp_recurrence}
(Z^{(N)})^k=\sum_{\ell=1}^k\sum_{I=(j_1,\cdots , j_\ell)} a_I^{(\ell)} \bX_I.  
    \end{equation}
    Our aim is to prove that for all $\ell\geq 1$,
the term $\sum_{I=(j_1,\cdots , j_\ell)} a_I^{(\ell)} \bX_I$ writes as the sum of terms having the property~\eqref{def:product_form}, that is such that there exists $v_1^{\{\ell\}},\cdots, v_\ell^{\{\ell\}}$ such that $a_I^{(\ell)}= v_{j_1}^{(\ell)}\cdots v_{j_\ell}^{(\ell)}$ for $I=(j_1,\cdots, j_\ell)$, $1\leq \ell\leq k$. 
We argue by a recursion and starts by observing that the result holds when $k=1$. Let us now assume that the property holds  for $k\in\bN$, with 
\[
a_{j_1,\ldots ,j_\ell}= v_{j_1}^{(\ell)}\ldots v_{j_\ell}^{(\ell)},\;\; 1\leq \ell\leq k.
\]
We want to prove that the same property holds for $(Z^{(N)})^{k+1}$. It is enough to apply $Z^{(N)}$ to one of the terms of the sum and to prove that it has the desired structure. 
We write 
\begin{align*}
&Z^{(N)}\left(\sum_{I=(j_1,\cdots , j_\ell)} a_I^{(\ell)} \bX_I\right)
=I_1+I_2\\
&I_1=
\sum_{j=1}^n \sum_{I=(j_1,\cdots , j_\ell)} z_j^{(N)}a_I^{(\ell)} X_{(\{j\}, I)}\;\;\mbox{ and}\;\;
I_2= \sum_{I=(j_1,\cdots , j_\ell)} (Z^{(N)} a_I ^{(\ell)}) X_I
\end{align*}
where $I_1$ contains  derivatives of length $\ell +1$ and $I_2$ derivatives of length $\ell$.

The term $I_1$ writes 
\[
I_1= \sum_{I=(j_1,\cdots , j_\ell)} (z_1^{(N)} a_I ^{(\ell)}X_1X_I+\ldots + z_n^{(N)} a_I^{(\ell)} X_n X_I)
= \sum_{I'=(j_1,\cdots , j_{\ell+1})} a_{I'}^{(\ell+1)} \bX_{I'}
\]
where the coefficients $a_{I'}^{(\ell+1)}$ are constructed as follows: let $I'=(j_1,\ldots, j_{\ell+1})$, $I'=(\{j_1\}, I)$ 
then $a_{I'}^{(\ell+1)}=v_{j_1}^{(\ell+1)}\ldots v_{j_{\ell+1}}^{(\ell+1)} $ with $v_{j_1}^{(\ell+1)}=z_{j_1}$ and $v_{j_p}^{(\ell+1)}=v_{j_{p-1}}^{(\ell)}$ for $p=2,\cdots ,\ell+1$.
Therefore, $I_1$ has the property~\eqref{def:product_form}. 

Let us now consider the term $I_2$.  It is the sum of $\ell$ terms of the form 
\[
\sum_{I=(j_1,\cdots , j_\ell)} v_{j_1}^{(\ell)}\ldots (Z^{(N)}v_{j_p}^{(\ell)})\ldots v_{j_\ell}^{(\ell)} X_I,\;\; 1\leq p\leq \ell.
\]
We observe 
\[
\sum_{I=(j_1,\cdots , j_\ell)} v_{j_1}^{(\ell)}\ldots (Z^{(N)}v_{j_p}^{(\ell)})\ldots v_{j_\ell}^{(\ell)} X_I=
\sum_{I=(j_1,\cdots , j_\ell)} v_{j_1}^{(\ell+1,p)}\ldots v_{j_\ell}^{(\ell+1,p)} X_I
\]
with $v_{j_{p'}}^{(\ell+1,p)}= v_{j_{p'}}^{(\ell)}$ if $p'\not=p$ and $v_{j_{p}}^{(\ell+1,p)}= Z^{(N)}v_{j_{p}}^{(\ell)}$.
Therefore, $I_2$ is the sum of $k$ terms with the property~\eqref{def:product_form}, which concludes the proof.

\section{Functional spaces on a vector bundle}
\label{secApp_VecBundleS}

Here, we  recall the definitions of  spaces of Schwartz functions and tempered distributions associated with a (smooth, real, finite dimensional) vector bundle $E$ over a (smooth) manifold~$M$.
We denote by $d=\dim E$ the dimension of (the fiber of) the vector bundle $E$.
The functional spaces we want to introduce will use the concept of vertical densities we first recall.

\subsection{Vertical densities on $E$}
\label{subsec_vertdensityE}
We denote by $|\Omega|(E)$ the space of 1-densities on $E$ whose definition we recall:
\begin{definition}
\label{def_1density}
\begin{itemize}
    \item    A {\it 1-density} on $\bR^d$ is a map  $\omega:(\bR^d)^d\to \bC$ 
satisfying 
$$
\forall v_1,\ldots,v_d\in \bR^d, \ \forall A\in {\rm GL}_d(\bR), \qquad 
\omega (A v_1,\ldots,A v_d ) = |\det A| \, \omega (v_1,\ldots,v_d).
$$
\item A {\it 1-density} on $E$ is map  $\omega:E^{\times d}\to \bC$ such that $\phi^* \omega (x,\cdot)$ is a 1-density on $\bR^d$
for every trivialising chart $\phi:U\times \bR^d \to E$ and every $x\in U.$
\end{itemize}
\end{definition}

A 1-density on $E$ is  sometimes called a {\it vertical density} as it yields above each $x\in M$ a density on  the vector space  $E_x$.

\begin{ex}
\label{ex_1density}
\begin{enumerate}
    \item If $\tau \in$ is an $n$-form on $M$ with $n=\dim M$, then $|\tau| \in |\Omega|(\bT M)$ is a density on the tangent bundle $E=\bT M$.
    \item More generally, for any vector bundle $E$, any $\tau \in \Gamma(\bigwedge^d E)$ yields the 1-density $|\tau|\in |\Omega|(E)$ on $E$.
\end{enumerate}
\end{ex}

The space  $|\Omega|(E)$ is a line bundle over $M$. Moreover, if $\omega\in |\Omega|(E)$ never vanishes, then any other 1-density $\omega'\in |\Omega|(E)$ differs by a smooth function $c\in C^\infty(M)$:
\begin{equation}
    \label{eq_omega12c}
\forall (x,v)\in E\qquad     \omega'(x,v) = c(x) \omega(x,v).
\end{equation}

We denote  by $|\Omega|^+(E)$ the set of  positive 1-densities on $E$, i.e. the set of $\omega\in|\Omega|(E) $ with $\omega (x,v)>0$  for any $(x,v)\in E.$
We can readily construct positive 1-densities
using local trivialising charts and subordinated covering of the manifold.
As Example \ref{ex_1density} suggests, 
 a positive 1-density on $E$ may be viewed as an element of volume $|dv|$ of the fiber $E_x$ depending smoothly on $x\in M$.
 In the case of two positive 1-densities, the smooth function $c\in C^\infty(M)$ they differ from in \eqref{eq_omega12c}
 is  positive: $c(x)>0$ for all $x\in M$.

\subsection{Schwartz functions and vertical tempered densities on $E$}

\begin{definition}
A  {\it Schwartz function} $f$ on $E$ is a smooth function $f:E\to \bC$ such that for every trivialising chart $\phi:U\times \bR^d \to E$, 
$\phi^* f$ may be written as $\phi^* f (x,v) =f_\phi(x,v)$ with
\begin{equation}
    \label{eq_fphi}
    f_\phi\in C^\infty (U\times \bR^d)
\qquad\mbox{and}\qquad
x\mapsto \phi^* f(x,\cdot) \in C^\infty (U; \cS(\bR^d)).
\end{equation}  
\end{definition}
 We denote by $\cS(E)$ the set of Schwartz functions on $E$. It is naturally equipped with a structure of topological vector space.

\smallskip

\begin{definition}
A {\it vertical tempered  density} on $E$ is a continuous map $\cS(E)\to C^\infty (M)$ that is $C^\infty(M)$-linear.
  \end{definition}
 A vertical tempered density $\kappa$ is described in  any trivialising chart $\phi:U\times \bR^d \to E$ by 
$\phi^* \kappa (x,v) = \kappa_\phi (x,v) |dv| $ with 
$$
\kappa_\phi\in \cD'(U\times \bR^d)
\qquad\mbox{and}\qquad  
x\mapsto \kappa_\phi (x,\cdot)\in C^\infty (U, \cS'(\bR^d)). $$ 
 We denote by $\cS'(E; |\Omega|(E))$  the set of vertical tempered densities on $E$.
 It is naturally equipped with a structure of topological vector space.
 By definition, 
the  map
$$
\cS(E)\times \cS'(E; |\Omega|(E))
\longrightarrow C^\infty(M), \qquad
(\varphi,\kappa)\longmapsto
(x\mapsto \int_E \varphi(x,v) \kappa(x,v))  ,
$$
is continuous and $C^\infty(M)$-bilinear.

\begin{lemma}
\label{lem_kappaomegax}
We assume that  a positive 1-density $\omega\in |\Omega|^+(E)$ has been fixed on $E$.
\begin{enumerate}
    \item A  vertical tempered density $\kappa\in \cS'(E; |\Omega|(E))$ is described above every $x\in M$ by a unique distribution $\kappa_x^\omega \in \cS'(E_x)$
   satisfying
   $$
   \forall \varphi\in \cS(E)\qquad
   \int_E \varphi(x,v) \kappa(x,v) = \int_{E_x} \varphi(x,v) \kappa_x^\omega(v) d\omega_x(v).
   $$
\item Moreover, if $\omega'\in |\Omega|^+(E)$ is another positive 1-density on $E$, then 
$$
\kappa_x^{\omega} = c(x) \kappa_x^{\omega'} \in \cS'(E_x),
$$
where $c$ is the function defined via \eqref{eq_omega12c}.
\end{enumerate}
\end{lemma}

\begin{proof}
The distribution $\kappa_x^\omega \in \cS'(E_x)$ is defined invariantly via trivialisation charts, showing Part (1).
Part (2) follows by uniqueness.
\end{proof}

The following statement is readily checked and defines  the notion of $m$-\textit{homogeneous  vertical distributional density}.
\begin{lemma}
\label{lem_mhomcS'Evert}
Here, we assume that the vector bundle $E$ is graded, i.e.  each fiber $E_x$ admits a gradation in the sense of Definition \ref{def_gradedvs}, and the corresponding dilations $\delta_{E_x}$ yield  a map $\delta :E\to E$ that is smooth (and hence is  a morphism of the vector bundle $E$).

Let $m\in \bR$. The following properties of a vertical tempered density $\kappa\in \cS'(E;|\Omega|(E))$ are equivalent:
\begin{enumerate}
    \item For one (and then any) positive 1-density $\omega\in |\Omega|^+(E)$ on $E$, for every $x\in M$, the distribution  $\kappa_x^\omega\in \cS(E_x)$ is $m$-homogeneous for the dilations $\delta_{E_x}.$
    \item 
    For every trivialising chart $\phi:U\times \bR^d \to E$ and ever $x\in M$, the distribution $\kappa_\phi(x,\cdot)$ is $m$-homogeneous for the dilations for $\phi^*\delta_{E_x}$
\end{enumerate}
  \end{lemma}

\subsection{Vertical Schwartz densities}
\begin{definition}
A {\it Schwartz vertical density} on $E$ is a smooth section $f\in C^\infty(E; |\Omega|(E))$ such that for every trivialising chart $\phi:U\times \bR^d \to E$, we may write
    $\phi^* f (x,v) := f_\phi(x,v) |dv|$ with $f_\phi$ as in \eqref{eq_fphi}.  
\end{definition}
The set $\cS(E; |\Omega|(E))$  of  Schwartz vertical  densities on $E$
 is naturally equipped with a structure of topological vector space and may be viewed naturally as a subspace of $\cS'(E; |\Omega|(E))$.
 In fact, it is dense in $\cS'(E; |\Omega|(E))$.

%%%%%%%%%%%%%%%%%%%%%%%%%%%%%%%%%%%%%%%%%%%%%%%%%%%%%%%%

\end{document}